\documentclass[10pt, oneside]{amsart}   	
\usepackage{geometry}                		
\usepackage{graphicx}				
					\usepackage{color}			
\usepackage{amssymb,mathrsfs}

\theoremstyle{plain} 
\newtheorem{theorem}{Theorem}[section]
\newtheorem{lemma}[theorem]{Lemma}
\newtheorem{proposition}[theorem]{Proposition}
\newtheorem{corollary}[theorem]{Corollary}
\newtheorem{remark}[theorem]{Remark}

\newcommand{\Mat}{{\rm Mat}}

\title{The Capelli Eigenvalue Problem for Quantum Groups}
\author{Gail Letzter}
\address{Gail Letzter, Mathematics Research Group, National Security Agency}
\email{ gletzter@verizon.net}
\author{Siddhartha Sahi}
\address{Siddhartha Sahi, Department of Mathematics, Rutgers University}
\email{sahi@math.rutgers.edu}
\author{Hadi Salmasian}
\address{Hadi Salmasian, Department of Mathematics and Statistics, University of Ottawa}
\email{ hadi.salmasian@uottawa.ca}

\begin{document}
\maketitle
\begin{abstract} 

We introduce and study quantum Capelli operators inside 
 newly  constructed  quantum Weyl algebras
 associated to three families of symmetric pairs (\cite{LSS}). 
  Both the center of a particular quantized enveloping algebra and the Capelli operators act semisimply on the   polynomial part of these quantum Weyl algebras.   We show how to transfer well-known properties of the  center arising from the theory of quantum symmetric pairs to the Capelli operators.  Using this information, we provide a natural realization of  Knop-Sahi interpolation polynomials as functions that produce eigenvalues for quantum Capelli operators. 
 \end{abstract}
\tableofcontents

\section{Introduction} In the mid 1980's, Macdonald introduced a new  family of parametrized orthogonal polynomials, referred to as Macdonald polynomials. These polynomials can be viewed as generalizations of the orthogonal polynomials that appear as 
zonal spherical functions for real and $p$-adic symmetric spaces (\cite{M}).  About a decade later, Knop (\cite{Kn}) and Sahi (\cite{S})  defined parameterized versions  of interpolation   polynomials (in Type A) 
with close connections to Macdonald polynomials.  For certain parameters, the Knop-Sahi interpolation  polynomials  can be viewed as $q$-analogs of polynomials that produce eigenvalues for Capelli operators (\cite{KS1991},\cite{KS1993},\cite{S2}).  Knop-Sahi   interpolation  polynomials  are referred to by a number of other names in the literature including interpolation Macdonald polynomials (\cite{Ol}),  quantum Capelli polynomials (\cite{Kn}), and shifted Macdonald polynomials (\cite{O}).

Drinfeld and Jimbo discovered quantized enveloping algebras, which are Hopf algebra deformations of universal enveloping algebras of Lie algebras,  around the same time as Macdonald initiated the study 
of   his new family of orthogonal polynomials.
Shortly afterwards,  the quest began for realizing Macdonald polynomials as zonal sperical functions on quantum symmetric spaces.
 This realization was ultimately carried out in a series of papers (\cite{DS}, \cite{L2003},\cite{L2004},  \cite{N}, \cite{NDS}, \cite{NS}) where the theory relies on definitions of quantum symmetric pairs  using
special coideal subalgebras.  They were constructed first for classical Lie algebras using  generators via $L$-functionals and solutions to  reflection equations  (\cite{M}, \cite{NS})
and then,  in general, via expressions derived from the Drinfeld-Jimbo generators (\cite{L1999}, \cite{L2002}, \cite{K}).  

In this paper, we complete another part of the quantum picture, namely finding a natural realization of the Knop-Sahi interpolation  polynomials as functions that produce eigenvalues for quantum Capelli operators.  The crucial piece in the story is identifying quantum Capelli operators, which are invariants with respect to a particular quantized enveloping algebra, inside three families of newly  constructed  quantum Weyl algebras $\mathscr{PD}_{\theta}$ (\cite{LSS}).  
 In other words, we formulate and solve the Capelli eigenvalue problem in the quantum case, thus providing 
quantum analogs of results in \cite{S2}.  
Similar results  for Capelli operators and their eigenvalues have been obtained in the Lie superalgebra setting (\cite{ASS}, \cite{SS2016}, \cite{SSS2020}, \cite{SSS2021}). 
 We also show that these two quantum realizations, one for Macdonald polynomials and the other for Knop-Sahi interpolation polynomials, are closely connected.  In particular, 
we use results from the Macdonald polyomial setting concerning central elements of the quantized enveloping algebra in order to establish basic properties of quantum Capelli operators.

The quantum Weyl algebra $\mathscr{PD}_{\theta}$ is associated  to one of the following three  (infinitesimal) symmetric pairs
 $(\mathfrak{g}, \mathfrak{k})$ where $\mathfrak{k}$ is the Lie subalgebra fixed by the involution $\theta$  of $\mathfrak{g}$:
\begin{itemize}
\item $(\mathfrak{g}, \mathfrak{k})=(\mathfrak{gl}_n, \mathfrak{so}_n)$ 
\item  $(\mathfrak{g}, \mathfrak{k})=(\mathfrak{gl}_{2n}, \mathfrak{sp}_{2n})$
\item  $(\mathfrak{g}, \mathfrak{k})=(\mathfrak{gl}_n\oplus \mathfrak{gl}_n, \mathfrak{gl}_n)$
\end{itemize}
We refer to the first pair as Type AI, the second as Type AII, and the third as the Type A diagonal case, thus matching the language used in the classification of symmetric pairs via Satake diagrams. 
These correspond to the Jordan algebras of Hermitian matrices over $\mathbb{R}$,  $\mathbb{H}$, and $\mathbb{C}$ in Type AI, Type AII, and the Type A diagonal case respectively. Indeed, in each case, the   Riemannian symmetric space $G/K$ can be identified naturally with the open subset of positive definite Hermitian matrices. 
We focus on these three families  precisely because  the Jordan algebra 
setting  is central in  Kostant and Sahi's study 
of Capelli identities, Capelli operators, and interpolation polynomials (\cite{KS1991}, \cite{KS1993}, \cite{S2}).

As explained in \cite{LSS}, the quantum Weyl algebra $\mathscr{PD}_{\theta}$ is a deformation of the twisted tensor product of  two algebras, the polynomial part $\mathscr{P}_{\theta}$ and the constant coefficient differential part $\mathscr{D}_{\theta}$. The polynomial part  $\mathscr{P}_{\theta}$ is 
the algebra of quantized functions on the space of $n\times n$ symmetric matrices for Type AI, $n\times n$ skew symmetric matrices for Type AII, and all $n\times n$ matrices in the Type A diagonal case.  The algebra $\mathscr{D}_{\theta}$ is isomorphic to the opposite algebra of $\mathscr{P}_{\theta}$.

Let $\mathcal{B}_{\theta}$ denote the right coideal subalgebra of $U_q(\mathfrak{g})$ which is a quantum analog of $U(\mathfrak{k})$ inside of $U(\mathfrak{g})$ as defined in  \cite{L2002} (see also \cite{K}).  Recall that there is a second root system, called the 
restricted root system, associated to a symmetric pair $\mathfrak{g}, \mathfrak{k}$. 
For  each of the symmetric pairs under consideration in this paper, the restricted root system $\Sigma$ is of type $A_{n-1}$.  Denote by  $\Lambda^+_{\Sigma}$ the set of partitions of length $n$ realized as weights associated to $\Sigma$.  Note that these partitions defined by $\Sigma$ form a subset of the partitions defined by the root system of $\mathfrak{g}$.  Thus with respect to this inclusion,  elements of $\Lambda^+_{\Sigma}$ are also weights for $\mathfrak{g}$.

Let $M={\rm Mat}_{n}$ in type AI, $M={\rm Mat}_{2n}$ in type AII, and  $M={\rm Mat}_n\times{\rm Mat}_n$ in the diagonal case.  Initially, see 
\cite{LSS},    $\mathscr{P}_{\theta}$ is defined as  a  subalgebra of $\mathcal{B}_{\theta}$-invariants inside the quantized function algebra  $\mathcal{O}_q(M)$ on $M$.  Moreover, $\mathscr{P}_{\theta}$ inherits a left $U_q(\mathfrak{g})$-module structure from $\mathcal{O}_q(M)$. 
 As  a left $U_q(\mathfrak{g})$-module, $\mathscr{P}_{\theta}$ is multiplicity free and is isomorphic to a  direct sum of simple highest weight modules $L(2\lambda)$ where $\lambda\in \Lambda^+_{\Sigma}$.  The algebra $\mathscr{P}_{\theta}$ can be viewed as a subalgebra of    the quantum homogeneous space $\mathcal{O}_q[G/K]$, which in turn is a subalgebra of $\mathcal{O}_q[G]$, where $G,K$ is the symmetric pair of Lie groups   associated to $\mathfrak{g}, \mathfrak{k}$.   The algebra $\mathscr{D}_{\theta}$ also admits a decomposition into left  $U_q(\mathfrak{g})$-modules that  is isomorphic to a direct sum of the $L^*(2\lambda)$, for  $\lambda\in \Lambda^+_{\Sigma}$ where $L^*(2\lambda)$ is a left $U_q(\mathfrak{g})$-module dual of $L(2\lambda)$.

The quantum Weyl algebra  $\mathscr{PD}_{\theta}$ inherits the structure of a left $U_q(\mathfrak{g})$-module via its construction that is compatible with the module actions on the subalgebras $\mathscr{P}_{\theta}$ and $\mathscr{D}_{\theta}$.  Thus $\mathscr{PD}_{\theta}$ is  isomorphic to the direct sum of modules of the form $L(2\lambda)\otimes L^*(2\mu)$ where $\lambda, \mu$ both run over partitions in $\Lambda^+_{\Sigma}$. Note that 
 $L(2\lambda)\otimes L^*(2\lambda) \cong {\rm End}\ L(2\lambda)$. Let $C_{\lambda}$ be the vector corresponding to the identity in   $L(2\lambda)\otimes L^*(2\lambda)$ via this isomorphism. The space of left $U_q(\mathfrak{g})$-invariants of  $L(2\lambda)\otimes L^*(2\mu)$ is zero if $\lambda \neq \mu$ and equal to the one-dimensional space spanned by  $C_{\lambda}$ for  $\lambda = \mu$. The quantum Capelli operators are the elements 
 $C_{\lambda}, \lambda\in \Lambda^+_{\Sigma}.$  

One can define an action of $\mathscr{PD}_{\theta}$ on $\mathscr{P}_{\theta}$ in a manner similar to the action of  the classical Weyl algebra on its polynomial subalgebra.    Since $\mathscr{P}_{\theta}$ is a $U_q(\mathfrak{g})$-module, elements of  $U_q(\mathfrak{g})$ also act on $\mathscr{P}_{\theta}$.    These actions lead to $U_q(\mathfrak{g})$-module maps from $\mathscr{PD}_{\theta}$ and $U_q(\mathfrak{g})$ into ${\rm End}\ \mathscr{P}_{\theta}$ where the module action for $\mathscr{P}_{\theta}$ is the left action and the module action for $U_q(\mathfrak{g})$ is the (left) adjoint action.  
Let $\mathcal{F}(U_q(\mathfrak{g}))$ denote the locally finite subalgebra of $U_q(\mathfrak{g})$ with respect to the adjoint action.  
We define a $U_q(\mathfrak{g})$-module map  that takes certain elements of  $\mathcal{F}(U_q(\mathfrak{g}))$, including most of the center  of $U_q(\mathfrak{g})$, to the quantum Weyl algebra that is compatible with the maps into ${\rm End}\ \mathscr{P}_{\theta}$. 
  This enables us to establish the following connection between the center and the Capelli operators.

\medskip
\noindent
{\bf Theorem A.}
 {\it  There is an isomorphism between  a  polynomial subring $Z$ of the center of $U_q(\mathfrak{g})$ and the algebra generated by the quantum 
Capelli operators so that the action of the two agree on $\mathscr{P}_{\theta}$.}

 \medskip
 In \cite{LSS0}, quantum Weyl algebras over $m\times n$ matrices where $m\neq n$ are studied.  These quantum Weyl algebras
come equipped with a left action of $U_q(\mathfrak{gl}_m)$ and a right action of $U_q(\mathfrak{gl}_n)$.  A key tool in \cite{LSS0} is a mapping of most of  the locally finite subalgebras of $U_q(\mathfrak{gl}_m)$ and of $U_q(\mathfrak{gl}_n)$ to subalgebras of the quantum Weyl algebra using very similar arguments to the ones found here. 

   By \cite{L}, there is a Harish-Chandra type map from the center to a subalgebra of the  Cartan subalgebra of $U_q(\mathfrak{g})$ so that the image of the center $Z(U_q(\mathfrak{g}))$ consists of invariants 
 with respect to a  dotted  action of the Weyl group for the restricted root system.  Note that elements of $Z(U_q(\mathfrak{g}))$ act semisimply on $\mathscr{P}_{\theta}$ with eigenspaces corresponding to simple modules $L(2\lambda)$, $\lambda\in \Lambda^+_{\Sigma}$.  The image of a central element inside of the Cartan subalgebra can be viewed as an eigenvalue  function: the eigenvalue for a central element $z$ on $L(2\lambda)$ is obtained by evaluating the image of $z$ under this Harish-Chandra map at $\lambda$,  or more precisely, at $q^{\lambda}$ (see Section \ref{section:dd} for a definition of this type of evaluation).

  It follows from Theorem A that
the eigenvalue functions  for quantum Capelli operators live inside the same subalgebra of the Cartan part and, moreover,    inherit the desired Weyl group invariance property from the center.  Write $\mathcal{E}_{\lambda}$ for the eigenvalue function associated to the Capelli operator $C_{\lambda}$.  We shall see that there is an interpretation for this eigenvalue function as a polynomial in $n$ variables.  The explicit map from $Z$ to the quantum Capelli operators of Theorem A  ensures that the degree of  $\mathcal{E}_{\lambda}$ is equal to $|\lambda|$. 

 Let  $\mathbb{C}(a,g)[x_1, \dots, x_n]$ be the polynomial  ring in $n$ variables over the field $\mathbb{C}(a,g)$ where $a$ and $g$ are two independent parameters. 
Given a partition $\mu=\mu_1\geq \mu_2\geq\mu_n\geq 0$ and a polynomial $P(x_1,\dots, x_n)$ in $\mathbb{C}(a,g)[x_1, \dots, x_n]$, set $P(a^{\mu}) = P(a^{\mu_1}, \dots, a^{\mu_n})$.  Knop-Sahi interpolation polynomials introduced in \cite{Kn} and \cite{S}, also called shifted Macdonald polynomials in the later paper \cite{O},  are a family of polynomials $P^*_{\lambda}(x; a,g)$ indexed by partitions ${\lambda}$ and contained in this  polynomial ring.  In addition, they     satisfy both an invariance condition and a vanishing condition.  
In particular, the element $P^*_{\lambda}( x;a,g)$ in 
 $\mathbb{C}(a,g)[x_1, \dots, x_n]$ is the unique (up to nonzero scalar) polynomial in the $x_1, \dots, x_n$ of degree $|\lambda|$  such that
\begin{itemize}
\item $P^*_{\lambda}( x;a,g)$ is symmetric viewed as a polynomial in the $n$ terms $x_1g^{-1}, \dots, x_ng^{-n}$
\item $P^*_{\lambda}(a^{\mu};a,g) = 0$  for each partition $\mu\neq \lambda$ with $|\mu|\leq |\lambda|$ and $P^*_{\lambda}(a^{\lambda};a,g) \neq 0$.
\end{itemize}
These polynomials are defined in a slightly different manner -- although it is easy to  convert from one definition to another --  in each of \cite{Kn},\cite{S}, and \cite{O}.  See Section \ref{section:KS} for more details and context.

Let $H_{2\lambda}$ denote the highest weight generating vector for the copy of $L(2\lambda)$ inside of $\mathscr{P}_{\theta}$.  The algebra $\mathscr{P}_{\theta}$ has a natural degree function which turns it into a graded algebra. Using a careful analysis involving the relations of $\mathscr{PD}_{\theta}$, 
we show that $C_{\lambda}\cdot H_{2\mu}=0$ for all  $H_{2\mu}$, $\mu\neq \lambda$,
 of degree less than or equal to that of $C_{\lambda}$. Moreover, $C_{\lambda}\cdot H_{2\lambda}\neq 0$. 
 These results are used to show that  the polynomials $\mathcal{E}_{\lambda}$ satisfy the vanishing  property, as well as a nondegeneracy condition, of the Knop-Sahi interpolation polynomials.  
This leads to 
our main result involving the Knop-Sahi interpolation  polynomials $P^*_{\lambda}({x;a,g})$. 

\medskip
\noindent
{\bf Theorem B.} {\it For each $\lambda\in \Lambda^+_{\Sigma}$, the polynomial  $\mathcal{E}_{\lambda}$  is equal to the
 polynomial $P_{\lambda}^*(x;a,g)$ (up to a normalization scalar) where  }
\begin{itemize}
\item $(a,g) = (q^4, q^2)$ {\it in Type AI}
\item $(a,g) =(q^2,q^4)$ {\it in Type AII},
\item $(a,g) = (q^2, q^2)$ {\it in the Type A diagonal case.}
\end{itemize}

It should be noted that the parameters obtained in Theorem B are precisely the same parameters as those for the realization of Macdonald polynomials as zonal spherical functions. This is due both to the connection between   Macdonald polynomials and Knop-Sahi interpolation polynomials as well as the relationship between eigenspaces of $\mathscr{P}_{\theta}$ and quantum zonal spherical functions. For more details on this connection, see Remark \ref{remark:finalremark}.

An  original motivation for this paper  was to understand and extend Bershtein's results on quantum Capelli operators in \cite{B}.
 One can view the results in  \cite{B} as a quantum analog of the classical setup relying on a Hermitian symmetric pair of Type AIII. By \cite{KS1991}, Section 1, the classical versions of Bershtein's approach and the  Type A diagonal case of this paper produce the same Capelli operator eigenvalues.  This paper shows that the same happens in the quantum setting. 
 Indeed,  the eigenvalues for the quantum Capelli operators in the Type A diagonal case of Theorem B match those in \cite{B}.  
Note that  the fixed Lie subalgebra $\mathfrak{k}$ in \cite{B} contains the entire Cartan subalgebra and so the corresponding symmetric pair is maximally compact. 
An advantage of the approach here is that the symmetric pairs  are  explicitly defined so that they are    in maximally split form.    This  allows us to use the theory of quantum symmetric pairs and  symmetric spaces as developed in \cite{L1999},\cite{L2002},\cite{L2003}, \cite{L2004} and  
\cite{L}.

The remainder of this paper is organized as follows. Section 2 sets notation for root systems, quantized enveloping algebras,  the three types of symmetric pairs and their quantum versions. In Section 3, we give  generators and relations for 
 the quantized function algebras $\mathcal{O}_q({\rm Mat}_N)$ on  $N\times N$ matrices and for the three quantum homogeneous spaces $\mathscr{P}_{\theta}$. In Section 4, we obtain detailed descriptions of highest weight vectors inside $\mathscr{P}_{\theta}$ with respect to the action of $U_q(\mathfrak{g})$ and use them to write down explicit module decompositions of the quantum homogeneous spaces $\mathscr{P}_{\theta}$. 

 In Section 5, we describe the quantum Weyl algebras $\mathscr{PD}_{\theta}$ of \cite{LSS} in terms of generators and relations.  We  then define a $U_q(\mathfrak{g})$ equivariant action of $\mathscr{PD}_{\theta}$ on $\mathscr{P}_{\theta}$ that resembles the action for the classical Weyl algebra on the polynomial subalgebra.   Using the relations for $\mathscr{PD}_{\theta}$, we analyze the action of the quantum Weyl algebra on the highest weight vectors $H_{2\lambda}, \lambda\in \Lambda_{\Sigma}^+,$ inside of $\mathscr{P}_{\theta}$.   This information is used to determine the action of the Capelli operators $C_{\lambda}$ on highest weight vectors $H_{2\mu}$ needed for the proof of Theorem B.

Section 6 shows how to identify certain elements $u$ of the Cartan subalgebra of $U_q(\mathfrak{g})$ with elements $a$ of  $\mathscr{PD}_{\theta}$ so that  $u$ and $a$ agree with respect to their action on  $\mathscr{P}_{\theta}$.   The elements of the Cartan subalgebra that have this property are part of the locally finite subalgebra $\mathcal{F}(U_q(\mathfrak{g}))$ of $U_q(\mathfrak{g})$.  The locally finite subalgebra is a well-studied object for quantized enveloping algebras of semisimple Lie algebras.  In Section 7, we give a complete description of $\mathcal{F}(U_q(\mathfrak{g}))$ by translating the results for the locally finite part of 
$U_q(\mathfrak{sl}_N)$ to the $\mathfrak{gl}_N$ setting. This allows us to define a mapping $\Upsilon$ on almost all of $\mathcal{F}(U_q(\mathfrak{g}))$ into $\mathscr{PD}_{\theta}$ which is compatible with the action on $\mathscr{P}_{\theta}$. 

Section 8 is devoted to the center of $U_q(\mathfrak{g})$ and the image of a special subalgebra $Z$ of the center (as in Theorem A) to $\mathscr{PD}_{\theta}$ via $\Upsilon$. In order to understand the center, we state and review specific results from \cite{L} for the center of $U_q(\mathfrak{sl}_N)$ and use them as a guideline for understanding the center of $U_q(\mathfrak{gl}_N)$ and the related subalgebra $Z$. 
Generators for the image of $Z$ inside ${\rm End}\ \mathscr{P}_{\theta}$ are identified  and the degree of their image inside of $\mathscr{PD}_{\theta}$ are determined.  Using the standard Harish-Chandra map and a restricted version based on those in \cite{L}, we explain  the dotted restricted Weyl group invariance for central elements.  We further show that both $Z$ and its image under the restricted Harish-Chandra map are isomorphic to a polynomial ring in $n$ variables.  

Quantum Capelli operators, are defined in Section 9.  Comparing the degree of quantum Capelli operators with the degree of the  image of generators of $Z$ under $\Upsilon$  yields Theorem A (Theorem \ref{theorem:center_and_capelli}).  Eigenvalue functions associated to quantum Capelli operators are introduced in Section \ref{section:van} and their relation to the image of central elements of $Z$ under the restricted Harish-Chandra map is explained.  We show that these eigenvalue functions  satisfy the defining properties of Knop-Sahi polynomials by combining facts about the center of $U_q(\mathfrak{g})$ with structural properties of $\mathscr{PD}_{\theta}$. This establishes  Theorem B (Theorem \ref{thm:KSpoly}).

\medskip\noindent
{\bf Acknowledgements.} The authors extend special thanks to the referee for a careful reading and many insightful suggestions.  A preliminary version of these results was presented by G. Letzter at the Workshop on Symmetric Spaces, Their Generalizations, and Special Functions, University of Ottawa, 2022, which was sponsored by grants from the Fields Institute and the NSF. 
The research of S. Sahi was partially supported by NSF grants DMS-1939600, DMS-2001537, and Simons foundation grant 509766.  The research of  H. Salmasian  was partially supported by an NSERC Discovery Grant (RGPIN-2018-04004).

\section{Quantized enveloping algebras}
\subsection{Roots and weights}\label{section:roots-and-weights} Let $N$ be a positive integer
and let $\epsilon_1, \dots, \epsilon_N$ denote a fixed  orthonormal basis for $\mathbb{R}^N$ with respect to the standard inner product $(\cdot,\cdot)$.     Let $\Phi_N$ denote the root system of Type $A_{N-1}$ with positive simple roots $\alpha_i =\epsilon_i-\epsilon_{i+1}$ for $i=1, \dots, {N-1}$.  Write $\omega_1, \dots, \omega_{N-1}$ for the fundamental weights associated to $\Phi_N$. Set $P_N=\sum_{i=1}^{N-1}\mathbb{Z}\omega_i$ equal to the weight lattice and $Q_N=\sum_{i=1}^{N-1}\mathbb{Z}\alpha_i$ equal to the root lattice.  Write $P_N^+=\sum_{i=1}^{N-1}\mathbb{N}\omega_i$ for the subset of (non-negative) dominant integral weights and let $Q_N^+$ denote the subset  $\sum_{i=1}^{N-1}\mathbb{N}\alpha_i$ of the root lattice.  

Let $\Lambda_N$ denote the set of vectors $\lambda=\lambda_1\epsilon_1+\cdots + \lambda_N\epsilon_N$ where each $\lambda_i\in \mathbb{Z}$ and $\lambda_1\geq \lambda_2\geq\cdots \geq \lambda_N$.      Set $\Lambda_N^+$ equal to the subset of $\Lambda_N$ consisting of those weights $\lambda\in \Lambda_N$ satisfying $\lambda_1\geq \cdots \geq \lambda_N\geq 0$.  In other words, elements in $\Lambda^+_N $ correspond to the set of partitions $(\lambda_1,\dots, \lambda_N)$.  Recall that $\epsilon_1+\cdots + \epsilon_N$ is orthogonal to each $\alpha_i$ with respect to the given inner product.  Moreover, it is straightforward to check that the fundamental weight $\omega_i$ satisfies
\begin{align}\label{omega} {\omega}_i = \epsilon_1 + \cdots + \epsilon_i -{{i}\over{N}}(\epsilon_1+\cdots + \epsilon_N)
\end{align}
for each $i=1, \dots, N-1$.
 Set $\hat{\omega}_i= \epsilon_1 +\cdots + \epsilon_i$ for $i=1, \dots, N$ and so the above equality becomes $\omega_i = \hat{\omega}_i-{{i}\over{N}}\hat{\omega}_N$. We refer to the weights $\hat{\omega}_1, \dots, \hat{\omega}_N$ as the fundamental partitions associated to the root system $\Phi_N$.   Note that $\Lambda_N^+$ is the $\mathbb{N}$-linear span of the fundamental partitions $\hat{\omega}_1, \dots, \hat{\omega}_{N}$. Note further that $P_N^+\subset \Lambda^+_N+\mathbb{N}(-\hat{\omega}_N
 /N)$ and, as stated above, $P_N^+$ is the $\mathbb{N}$-linear span of the fundamental weights $\omega_i, i=1, \dots, N-1$.

Let $w_0$ denote the longest element of the  Weyl group in type $A_{N-1}$. We have $w_0\alpha_i = -\alpha_{N-i}$ and $w_0\epsilon_j = \epsilon_{N-j}$ for  $i=1, \dots, N-1$ and $j=1, \dots, N$.  It follows that the 
 set $w_0\Lambda_N^+$ is the $\mathbb{N}$-linear span of the elements $w_0\hat{\omega}_i=\epsilon_{N-i+1}+\cdots + \epsilon_N$ for $i=1, \dots, N$. Moreover, for each $i=1,\dots,N-1$,  the image of the fundamental weight $\omega_i$ under $w_0$ is $w_0{\omega}_i =\epsilon_{N-i+1}+\cdots + \epsilon_N -{{i}\over{N}}(\epsilon_1+\cdots \epsilon_N)$.  So
$w_0P_N^+$ is the $\mathbb{N}$-linear span of $w_0{\omega}_1,\dots, w_0\omega_{N-1}$ which in turn equals the $\mathbb{N}$-linear span of  $-\omega_{N-1}, \cdots. -\omega_1.$
 Hence $w_0P^+_N =-P^+_N$.
 
 Let $ \mathfrak{gl}_N$ denote the complex general linear Lie algebra consisting of $N\times N$ matrices and let $\mathfrak{sl}_N$ be the Lie subalgebra equal to the subspace of $N\times N$ matrices with trace $0$.   Recall that $\Phi_N$ is the root system for $\mathfrak{sl}_N$ and the finite-dimensional simple $\mathfrak{sl}_N$-modules are parameterized by their highest weights which are elements of $P_N^+$.  For $\mathfrak{gl}_N$, one uses $\Lambda_N$ to parameterize the finite-dimensional simple $\mathfrak{gl}_N$-modules.  More generally, the relationship between $\omega_i$ and $\hat{\omega}_i$ in (\ref{omega}) will help us translate results from the $\mathfrak{sl}_N$ setting to that of $\mathfrak{gl}_N$.  
Sometimes it will be useful to just consider the first $N-1$ partitions. To do this, we set
 $\hat\Lambda^+_N$ equal to the $\mathbb{N}$-linear span of the first $N-1$ partitions $\hat{\omega}_1,\dots, \hat{\omega}_{N-1}$.
Note that $\hat\Lambda^+_N + \mathbb{N}\hat{\omega}_N=\Lambda^+_N$. Similarly, $w_0\hat{\Lambda}^+_N$ is equal to the $\mathbb{N}$-linear span of the 
 $N-1$ partitions $w_0\hat{\omega}_{i},i=1,\dots, N-1$.

The root system for $\mathfrak{gl}_N\oplus \mathfrak{gl}_N$ is just the disjoint union $\Phi^{(1)}_N \cup \Phi^{(2)}_N$ of two copies of the  root system for $\mathfrak{gl}_N$ and the corresponding Weyl group is just the direct product of two copies of the Weyl group for $\Phi_N$ .  Write $\epsilon_1, \dots, \epsilon_N$ for the orthonormal basis for the first copy of $\mathfrak{gl}_N$ and $\epsilon_{N+1}, \dots, \epsilon_{2N}$ for the second copy.  Set $\hat{\omega}_i =  \epsilon_1 +\cdots + \epsilon_i$  and $\hat{\omega}_{i+N} =  \epsilon_{1+N} +\cdots + \epsilon_{i+N}$ for $i=1, \dots, N$. The longest Weyl group element is $w_0\times w_0$ which simply acts as $w_0$ on each copy of $\Phi_N$.  Other notions are extended from $\mathfrak{gl}_N$ to $\mathfrak{gl}_N\oplus \mathfrak{gl}_N$ in a similar fashion. Sometimes we denote weights for $\mathfrak{gl}_N\oplus \mathfrak{gl}_N$ using a single symbol, say $\lambda$.  In other instances, we use the sum  $\gamma \oplus \gamma'$ to represent the weight $\gamma_1\epsilon_1+\cdots + \gamma_N\epsilon_N+\gamma'_{1}\epsilon_{1+N}+\cdots + \gamma'_{N}\epsilon_{2N}$.

\subsection{The quantized enveloping algebra}\label{section:qea}
Write $e_1, \dots, e_{N-1},$ $ f_1, \dots, f_{N-1},$ $ h_{\epsilon_1},\cdots, h_{\epsilon_N}$ for the standard Chevalley  generators for $\mathfrak{gl}_N$.  For the direct sum $\mathfrak{gl}_N\oplus \mathfrak{gl}_N$, we write $e_i, f_i, h_{\epsilon_j}$ where $i\in \{1, \dots, N-1\}\cup\{N+1,\dots, 2N-1\}$ and $j\in \{1,\dots, 2N\}$ for the standard generators.   Here, the $e_i, f_i, h_{\epsilon_j}$, $i\leq N-1, j\leq  N$ generate the first copy of $\mathfrak{gl}_N$ while 
$e_i, f_i, h_{\epsilon_j}$, $i,j\geq N+1$ generate the second copy.  

Let $q$ be an indeterminate.   The quantized enveloping algebra $U_q(\mathfrak{gl}_N)$   is an algebra over $\mathbb{C}(q)$ generated by $K_{\epsilon_1}^{\pm 1}, \dots, K_{\epsilon_N}^{\pm 1}$,   $E_1, \dots, E_{N-1},$ $ F_1, \dots, F_{N-1}$ subject to the algebra relations as stated in \cite{N} (see also \cite{KS}, Section 10).   Note that $K_{\epsilon_1}\cdots K_{\epsilon_N}$ is a central element of $U_q(\mathfrak{gl}_N)$. 
Set  $K_i = K_{\epsilon_i}K_{\epsilon_{i+1}}^{-1}$ for $i=1, \dots, N-1$.   
The subalgebra of $U_q(\mathfrak{gl}_N)$ generated by $$K_1^{\pm 1}, \dots, K_{N-1}^{\pm 1}, E_1, \dots, E_{N-1}, F_1, \dots, F_{N-1}$$ is the quantized enveloping algebra $U_q(\mathfrak{sl}_N).$

Given an integer linear combination $\beta = \sum_{i=1}^N\beta_j\epsilon_j$, write $K_{\beta}$ for the product $K_{\epsilon_1}^{\beta_1}\cdots K_{\epsilon_N}^{\beta_N}$. The algebra $U_q(\mathfrak{gl}_N)$ is a  Hopf algebra with coproduct $\Delta$, counit $\epsilon$, and antipode $S$ defined on generators by 
\begin{itemize}
\item $\Delta(E_i) = E_i\otimes 1 + K_i\otimes E_i$, $\epsilon(E_i) = 0$ and $S(E_i) = -K_i^{-1}E_i$
\item $\Delta(F_i) = F_i\otimes K_i^{-1} + 1\otimes F_i$, $\epsilon(F_i) = 0$ and $S(F_i) = -F_iK_i$
\item $\Delta(K) = K\otimes K$, $\epsilon(K) = 1$ and $S(K) = K^{-1}$
\end{itemize}
for $i=1, \dots, N-1$ and for all $K=K_{\beta}, \beta\in \sum_j{\mathbb{Z}}\epsilon_j$.  
 It follows from the defining relations  for $U_q(\mathfrak{gl}_N)$ that 
\begin{align}\label{adKaction}
 K_{\beta}E_iK_{\beta}^{-1} = q^{(\beta,\alpha_i)}E_i \quad {\rm and}\quad  K_{\beta}F_iK_{\beta}^{-1}=q^{-(\beta,\alpha_i)}F_i 
\end{align} for $i=1, \dots, N-1$ and all $ \beta\in \sum_j{\mathbb{Z}}\epsilon_j$.  The subalgebra $U_q(\mathfrak{sl}_N)$ is also a Hopf algebra using the same coproduct, couinit, and antipode.  Given an arbitrary element 
$u$ in one of these Hopf algebras, we express the coproduct of $u$ by $\Delta(u) = \sum u_{(1)}\otimes u_{(2)}$.

We also consider the quantized enveloping algebra of $\mathfrak{gl}_N\oplus \mathfrak{gl}_N$.  Note that $$U_q(\mathfrak{gl}_N\oplus \mathfrak{gl}_N) \cong U_q(\mathfrak{gl}_N)\otimes U_q(\mathfrak{gl}_N)$$ as Hopf algebras. Write $E_i, F_i, K_{\epsilon_j}^{\pm 1}$, $1\leq i\leq N-1, 1\leq  j\leq N$ for the generators of the first copy of  $U_q(\mathfrak{gl}_N)$ and 
$E_i, F_i, K_{\epsilon_j}^{\pm 1}$, $N+1\leq i\leq 2N-1, N+1\leq j\leq 2N$ for the generators of the second copy.  Formally, we may identify $E_i$ with $E_i\otimes 1$ and $E_{i+N}$  with $1\otimes E_{i}$ for $1\leq i\leq N-1$ with similar identifications for the $F_i$ and $K_{\epsilon_j}^{\pm 1}$.

Let $U^0(\mathfrak{gl}_N)$ denote the subalgebra of $U_q(\mathfrak{gl}_N)$ generated by $K_{\epsilon_1}^{\pm 1}, \dots, K_{\epsilon_N}^{\pm 1}$.  Similarly, we write
$U^0(\mathfrak{gl}_N\oplus \mathfrak{gl}_N)$ for the subalgebra of $U_q(\mathfrak{gl}_N\oplus \mathfrak{gl}_N)$ generated by $K_{\epsilon_1}^{\pm 1}, \dots, K_{\epsilon_{2N}}^{\pm 1}$. Write $U^0(\mathfrak{sl}_N)$ for $U^0(\mathfrak{gl}_N)\cap U_q(\mathfrak{sl}_N)$ and note that $U^0(\mathfrak{sl}_N)$  is just the Laurent polynomial ring generated by $  K_i^{\pm 1}, i=1,\dots N-1$.

Let $U^+(\mathfrak{gl}_N)$ denote the subalgebra of $U_q(\mathfrak{gl}_N)$ generated by $E_1, \dots, E_{N-1}$.  Similarly, let $U^+(\mathfrak{gl}_N\oplus\mathfrak{gl}_N)$ denote the subalgebra of $U_q(\mathfrak{gl}_N\oplus \mathfrak{gl}_N)$ generated by $E_i,$ for $i=1, \dots N-1$ and $i=N+1, \dots, 2N-1$. 
We frequently drop $\mathfrak{gl}_N$ or $\mathfrak{gl}_N\oplus\mathfrak{gl}_N$ from the notation and simply write $U^+$ for when the associated Lie algebra can be understood from context. 
Define the subalgebra $U^-(\mathfrak{gl}_N)$  in the same way with each $E_i$ replaced by $F_i$ and similarly write $U^-$ when the associated Lie algebra is understood from contexs.
 
In the study of $U_q(\mathfrak{sl}_N)$, it is sometimes necessary to pass to  the simply connected quantized enveloping algebra  $\check U_q(\mathfrak{sl}_N)$.   This Hopf algebra is an extension of $U_q(\mathfrak{sl}_N)$ obtained by enlarging $U^0(\mathfrak{sl}_N)$ to $\check U^0(\mathfrak{sl}_N)$ 
 where $\check U^0(\mathfrak{sl}_N)$ is the group algebra generated by the $K_{\mu}$ where $\mu$ is in the weight lattice $P_N$ (see for example \cite{Jo}, 3.2.10).  In another words,    $\check U_q(\mathfrak{sl}_N)$ is the Hopf algebra generated by $U_q(\mathfrak{sl}_N)$ and $\check U^0(\mathfrak{sl}_N)$ where the generators of $U_q(\mathfrak{sl}_N)$ and the elements  $K_{\beta}, \beta\in P_N$ satisfy  (\ref{adKaction}).
 
 Recall that the augmentation ideal of a Hopf algebra $H$, denoted by $H_+$,  is the kernel of the counit $\epsilon$.  Given a subset $M$ of $H$, we write $M_+$ for the intersection of $M$ with $H_+$.  For example, we write $(U_q(\mathfrak{gl}_N))_+$ for the augmentation ideal of $U_q(\mathfrak{gl}_N)$.  Similarly, we denote the intersection of $U^+$ with the augmentation ideal of $U_q(\mathfrak{gl}_N)$ by $U^+_+$.

We write $L(\lambda)$ for the simple $U_q(\mathfrak{gl}_N)$-module of highest weight $\lambda$.  In other words, $L(\lambda)$ is generated by a highest weight vector $v_{\lambda}$ such that $E_iv_{\lambda} =0$ for all $i=1, \dots, N-1$ and $K_{\beta}v_{\lambda} = q^{(\beta,\lambda)}v_{\lambda}$ for all weights $\beta$.  Recall that $L(\lambda)$ is finite-dimensional viewed as a $U_q(\mathfrak{sl}_N)$-module  if and only if $\lambda \in P_N^+$ and is finite-dimensional viewed as a $U_q(\mathfrak{gl}_N)$-module if and only if  $\lambda\in \Lambda_N$.   These notions extend to the setting of $U_q(\mathfrak{gl}_N\oplus \mathfrak{gl}_N)$ in a straightforward manner and we will similarly denote highest weight modules by $L(\lambda)$ where here $\lambda$ is understood to be a weight for $\mathfrak{gl}_N\oplus \mathfrak{gl}_N$.

\subsection{The adjoint action}\label{section:lfs}

 Given a Hopf algebra $H$ and a two-sided $H$-bimodule $M$,  the bimodule $M$ admits an  $H$-module structure via the adjoint action defined by
 \begin{align*}
 ({\rm ad}\ a) m = \sum a_{(1)}mS(a_{(2)})
 \end{align*} for all $a\in H$ and $m\in M$.   
The locally finite part $\mathcal{F}_H(M)$  of $M$ is the submodule consisting of those elements that generate a finite module with respect to the 
adjoint action of $H$.  More precisely, 
\begin{align*}
\mathcal{F}_H(M) = \{m \in M| \ \dim[({\rm ad}\ H)m]<\infty \}.
\end{align*}
When $M$ is also a  Hopf algebra, the locally finite part $\mathcal{F}_H(M)$ is a subalgebra of $M$ (\cite{JL}). 
We typically drop $H$ from the notation $\mathcal{F}(M)$ when $H$ can be understood from context. 

For $H=U_q(\mathfrak{gl}_N)$, $U_q(\mathfrak{sl}_N)$ or $\check U_q(\mathfrak{sl}_N)$,  the adjoint action is determined by the following formulas:
 \begin{align*} 
&({\rm ad}\ E_i)\cdot m =E_im-K_imK_i^{-1}E_i\cr  &({\rm ad}\ F_i)\cdot m = F_imK_i-mF_iK_i \cr  &({\rm ad}\ K)\cdot m = KmK^{-1}
\end{align*}
for all $i=1, \dots, N-1$, $K=K_{\beta}$ for all weights $\beta$ with $K_{\beta}$ in the specified quantized enveloping algebra, and $m\in M$.  Note that these formulas carry over easily to $H=U_q(\mathfrak{gl}_N\oplus \mathfrak{gl}_N)$.

\subsection{Three symmetric pair families}\label{section:three-families}
We are interested in quantum homogeneous spaces associated to three families of symmetric pairs $\mathfrak{g}, \mathfrak{k}$ where $\mathfrak{k}$ is the Lie subalgebra fixed by the involution $\theta$. Here,  $\mathfrak{g} = \mathfrak{gl}_n$ for the first family, $\mathfrak{g} = \mathfrak{gl}_{2n}$ for the second family, and   $\mathfrak{g}=\mathfrak{gl}_n\oplus \mathfrak{gl}_n$ for the third family.  
Below, we describe the involution for each family and then define the Drinfeld-Jimbo type generators for the  associated quantum analog $\mathcal{B}_{\theta}$ of $U(\mathfrak{g}^{\theta})$.

For each of the three families, we associate an $R$-matrix $R_{\mathfrak{g}}$ and 
a solution $J$  to the reflection equation 
$R_{\mathfrak{g}}J_1R_{\mathfrak{g}}^{t_1} J_2 = J_2 R_{\mathfrak{g}}^{t_1}J_1R_{\mathfrak{g}}$ where $J$ is an $N\times N$ matrix, $J_1 = J\otimes I$ and $J_2 = I\otimes J$, $I$ is the $N\times N$ identity matrix and  $R_{\mathfrak{g}}^{t_1}$ denotes  the transpose in the first column.  These $R$-matrices are closely related to   the following   matrix $R$ 
in $\Mat_N\times \Mat_N$ defined by \begin{align}\label{Rmatrix2}
R
&= \sum_{1\leq i\leq N} 
qe_{ii}\otimes e_{ii} +\sum_{1\leq i<j\leq N}(e_{ii}\otimes e_{jj} + e_{jj}\otimes e_{ii})
+ (q-q^{-1})\sum_{1\leq j<i \leq N} 
e_{ij}\otimes e_{ji}
\end{align}
where $\Mat_N$ is the space of $N\times N$ matrices and the $e_{ij}$ are  matrix units.
 When $R_{\mathfrak{g}}$ is the matrix $R$, then $(R^{t_1})^{ij}_{kl} = r^{kj}_{il}.$ 
 
\medskip
\noindent
{\bf Type AI}:  $\mathfrak{g} = \mathfrak{gl}_n$ and $\theta$ is defined by $\theta(e_i) = -f_i$, 
$\theta(f_i) = -e_i$ and $\theta(h_{\epsilon_j}) = -h_{\epsilon_j}$ each $i=1, \dots, n-1$ and $j=1, \dots, n$.   Hence $\mathfrak{k}$ is generated by $f_i-e_i$ for $i=1, \dots, n-1$.
Passing to the quantum case,   $\mathcal{B}_{\theta}$  is generated by $F_i-E_iK_i^{-1}$, for $i=1, \dots, n-1$.   In this case, $R_{\mathfrak{g}}=R$ as defined in (\ref{Rmatrix2}) with $N=n$. The associated solution to the reflection equation is $J = I_n$, the $n\times n$ identity matrix.

\medskip
\noindent
{\bf Type AII:} $\mathfrak{g} = \mathfrak{gl}_{2n}$ and $\theta$ is defined by \begin{itemize}
\item $\theta(e_i) = e_i, \theta(f_i) = f_i, \theta(h_i) = h_i$ for $i=1, 3, \dots, 2n-1.$
\item $\theta(f_i) = -[e_{i-1}, [e_{i+1}, e_{i}]]$  for $i=2, 4,\dots, 2n-2$.
\item $\theta(h_{\epsilon_{2i-1}}) = -h_{\epsilon_{2i}}$ 
 for $i=1, \dots, n$.
\end{itemize}
Hence $\mathfrak{k}$ is generated by $e_i,f_i, h_i, $ for $i=1, 3, \dots, 2n-1$ and $f_{i}-[e_{i-1}, [e_{i+1}, e_{i}]]$ for $i=2, 4,\dots, 2n-2$.
 Passing to the quantum case, 
$\mathcal{B}_{\theta}$ is generated by 
\begin{itemize}
\item $K_i^{\pm 1}, E_i,F_i$ for $i=1, 3, \dots, 2n-1.$
\item $B_i= F_i -q^3(({\rm ad}\ E_{i-1}E_{i+1})\cdot E_i)K_i^{-1}= F_i-q^3[E_{i-1}, [E_{i+1}, E_i]_q]_qK_i^{-1}$ for $i=2, 4,\dots, 2n-2$ where 
 $[a,b]_q$ denotes  the $q$-commutator $ab-qba$ of $a$ and $b$.
\end{itemize}
In this case, $R_{\mathfrak{g}}=R$ as defined in (\ref{Rmatrix2}) with $N=2n$. The associated solution to the reflection equation is 
$J= \sum_{k=1}^n(e_{2k-1, 2k}-qe_{2k,2k-1})$.  

\medskip
\noindent
{\bf Type A diagonal case}:   $\mathfrak{g} = \mathfrak{gl}_n\oplus \mathfrak{gl}_n$ and $\theta$ is defined by
$\theta(f_i) =- e_{n+i}, \ \theta(f_{n+i}) =- e_i,$ and $ \theta(h_{\epsilon_j}) = -h_{\epsilon_{n+j}}$
for $i=1, \dots, n-1$ and $j=1, \dots, n$ and so $\mathfrak{k}$ is generated by $f_i-e_{n+i}$, $ f_{n+i}-e_i$, and $h_{\epsilon_j}-h_{\epsilon_{n+j}}$ for $ i=1, \dots, n-1$ 
and $j=1,\dots, n$.
Passing to the quantum case, the corresponding quantum symmetric pair coideal subalgebra $\mathcal{B}_{\theta}$ is generated by 
\begin{align*}
B_i = F_i -qE_{n+i}K_{i}^{-1}, \quad B_{n+i}= F_{n+i} -qE_{i}K_{n+i}^{-1},{\rm \quad and\quad}(K_{\epsilon_j}^{- 1}K_{\epsilon_{n+j}})^{\pm 1}
\end{align*}
for $i=1, \dots, n-1$ and $j=1, \dots, n$.  In this case, $R_{\mathfrak{g}}$ is the $4n^2\times 4n^2$ block diagonal matrix with diagonal  $(R, I_{n^2}, I_{n^2}, R)$ were $R$ is the matrix defined by (\ref{Rmatrix2}) with $N=n$.  The associated solution to the reflection equation is $
J=\sum_{k=1}^n(e_{k,n+k} + e_{n+k,k})$.  We often drop the phrase Type A and simply refer to this family  as the diagonal case or diagonal type.

\medskip

 It should be noted that the quantum analog $\mathcal{B}_{\theta}$ can also be defined using $R$-matrices along with  the solutions $J$ to the reflection equation (see \cite{N}, \cite{NS}.)
The fact that the different approaches to defining $\mathcal{B}_{\theta}$ yields the same coideal subalgebra follows from the uniqueness result proved in \cite{L1999}, Sections 5 and 6 (see also \cite{L2002}, Theorem 7.3). 

Note that this uniqueness result is up to isomorphism via a Hopf algebra automorphism of $U_q(\mathfrak{g})$.  
Thus one can introduce parameters both in the solutions $J$ to the reflection equation (see for example \cite{NS}, Section 3) and in the  Drinfeld-Jimbo generators (see for example  \cite{LSS}, Section 5.1).  There is a one-to-one correspondence between the two sets of parameters, thus matching the choice of coideal subalgebra $\mathcal{B}_{\theta}$ in terms of   Drinfeld-Jimbo generators and  a reflection equation solution $J$ as in the above examples. This correspondence can be found in \cite{LSS},  Section 5.  (See especially the set-up in Section 5.1, the description of invariant elements in Section 5.2, their use in connecting  the parameters in Lemmas 5.1 and 5.2,  as well as the discussion at the end of the Section 5.2.)

\subsection{Restricted root systems}\label{section:restricted-root-system}
Write $\Phi$ for the root system generated by the set of positive simple roots for $\mathfrak{g}$. By Section \ref{section:roots-and-weights},  for Types AI and AII,   the set of positive simple roots is $\{\alpha_1, \dots, \alpha_{N-1}\}$ and the root system $\Phi=\Phi_N$ is of type $A_{N-1}$ where $N=n$ in Type AI and $N=2n$ in Type AII.  For the diagonal case, the set of positive roots is $\{\alpha_1, \dots, \alpha_{n-1}\}\cup\{\alpha_{n+1},\dots, \alpha_{2n-1}\}$ where $\{\alpha_1,\dots, \alpha_{n-1}\}$ and $\{\alpha_{n+1}, \dots, \alpha_{2n-1}\}$ each separately generate a root system of Type A$_n$ and together generate a root system $\Phi=\Phi^{(1)}_n\cup \Phi^{(2)}_n$.  

Note that $\theta$ induces an involution, which we also call $\theta$,  on the root system $\Phi$. More generally, $\theta$ can be extended to an involution on $\sum_{i=1}^{N}\mathbb{Z}\epsilon_i$ where $N=n$ in Type AI, and  $N=2n$ in Type AII and the diagonal case.  Set $\tilde{\beta}=(\beta-\theta(\beta))/2$ for each $\beta\in \sum_{i=1}^N\mathbb{Z}\epsilon_i$.    The set of all $\tilde{\beta} $ where $ \beta$ runs over elements in $ \Phi$, forms another root system,  called the restricted root system, which we denote by $\Sigma$.  For all three families under consideration, the root system $\Sigma$ is of type A$_{n-1}$.  We denote the simple roots for $\Sigma$ by $\alpha_1^{\Sigma}, \dots, \alpha_{n-1}^{\Sigma}$ and explain below how these are related to the simple roots for $\Phi$.

 The restricted root system is contained in a  vector space spanned by orthonormal basis vectors that live inside the set $\{\tilde{\beta}|\ \beta \in \sum_{i=1}^n\mathbb{Z}\epsilon_i\}$.  We denote this orthonormal basis 
 by $\epsilon_{1}^{\Sigma}, \dots, \epsilon_{n}^{\Sigma}$.
Write $\eta_1, \dots, \eta_{n-1}$ for the fundamental weights associated to $\Sigma$ and write $\hat{\eta}_1, \dots, \hat{\eta}_{n}$ for the fundamental partitions.  Let $P_{\Sigma}$ denote the weight 
lattice and $P^+_{\Sigma}$ denote the dominant integral weights defined by the root system  $\Sigma$.  In other words $P_{\Sigma}$ (resp. $P^+_{\Sigma}$) is 
just the set of $\mathbb{Z}$-linear (resp.\,$\mathbb{N}$-linear) span of the fundamental weights  $\eta_1, \dots, \eta_{n-1}$.  Let   $\Lambda^+_{\Sigma}$ denote the $\mathbb{N}$-linear span of the fundamental partitions $\hat{\eta}_i, i=1, \dots, n$. Define $\hat{\Lambda}^+
_{\Sigma}$ in a way similar to $\hat{\Lambda}^+_N$.  In particular, $\hat{\Lambda}^+
_{\Sigma}$ is the $\mathbb{N}$-linear span of $\hat{\eta}_1, \dots, \hat{\eta}_{n-1}$.
We identify the $\epsilon^{\Sigma}_i$, $\eta_i$ and $\hat{\eta}_i$ with elements of  $\sum_{i=1}^n\mathbb{Q}\epsilon_i$ below for each of the three families under consideration. With respect to this identification, it is straightforward to check that the longest element  of the Weyl group for $\Phi$ also acts as the longest element in the restricted Weyl group $W_{\Sigma}$.   
Abusing notation slightly, we refer to $w_0$ as the longest element of $W_{\Sigma}$ with the understanding that $w_0$ is the restriction of the longest element of the Weyl group for $\Phi$  to restricted weights.  In particular, $w_0\hat{\eta}_i = \epsilon^{\Sigma}_{n-i+1} + \cdots +\epsilon^{\Sigma}_n$ for each $i$ and for each of the three families.

\medskip
\noindent
{\bf Type AI:} For Type AI, $\tilde{\alpha}_i = (\alpha_i-\theta(\alpha_i))/2 = \alpha_i$ for each $i=1, \dots, n-1$.  Hence $\alpha_i^{\Sigma} = \tilde{\alpha}_i=\alpha_i$ for $i=1, \dots, n-1$  and the set of positive simple roots for the root system $\Sigma$ is just  the set $ \{\alpha_1,\dots, \alpha_{n-1}\}$.   
A similar straightforward computation yields $\tilde{\epsilon}_i = \epsilon_i$ for each $i$.  In other words, the restricted root system $\Sigma$ in this case equals $\Phi$.  Thus $\epsilon^{\Sigma}_j=\epsilon_j$, $\hat{\eta}_j=\hat{\omega}_j$, and ${\eta}_k={\omega}_k$ for $j=1, \dots, n$ and $k=1,\dots, n-1$.    Note further that $\tilde{\omega}_k=(\omega_k-\theta(\omega_k))/2 = \omega_k$ and so $\eta_k$ equals $\tilde{\omega}_k$,  the restricted weight associated to $\omega_k$.  Similarly, $\hat{\eta}_k =(\hat{\omega}_k- \theta({\hat{\omega}}_k))/2$, the restricted weight associated to $\hat{\omega}_k$, for each $k=1, \dots, n$.

\medskip
\noindent
{\bf  Type AII:} For Type AII, $\tilde{\alpha}_{2i-1} = 0$ for   $i=1, \dots, n$  and  $\tilde{\alpha}_{2j} = (\alpha_{2j-1}+2\alpha_{2j}+\alpha_{2j+1})/2$ for  $j=1, \dots, n-1$. The set of positive simple roots for $\Sigma$ is $ \{\tilde{\alpha}_{2j}|\ j=1, \dots, n-1\}$ and $\alpha_j^{\Sigma} = \tilde{\alpha}_{2j}$ for $j=1, \dots, n-1$.  Note that the inner product on $\Phi$ needs to be scaled differently for the restricted root system.  This is because 
\begin{align*}(\tilde{\alpha}_{2j}, \tilde{\alpha}_{2j}) = (\alpha_{2j-1} + 2\alpha_{2j} + \alpha_{2j+1}, \alpha_{2j-1} + 2\alpha_{2j} + \alpha_{2j+1})/4 = 1.
\end{align*}
Hence the inner product for $\Sigma$ takes the form $(\cdot, \cdot)_{\Sigma} = 2(\cdot, \cdot)$ where $(\cdot, \cdot)$ is the usual Cartan inner product for the root system of $\mathfrak{gl}_{2n}$. In other words, $(\cdot, \cdot)_{\Sigma}$ is the normalization of $(\cdot, \cdot)$ chosen so that $(\tilde{\alpha}_{2j}, \tilde{\alpha}_{2j}) = 2$ for each $j$.  Now $\tilde{\epsilon}_{2j} =\tilde{\epsilon}_{2j-1}= (\epsilon_{2j-1} + \epsilon_{2j})/2$ and so $2(\tilde{\epsilon}_{2i}, \tilde{\epsilon}_{2j}) =\delta_{ij}$. It follows that the corresponding set of orthonormal vectors is 
$\{\tilde{\epsilon}_{2}, \tilde{\epsilon}_{4}, \dots, \tilde{\epsilon}_{2n}\}$ and hence $\epsilon^{\Sigma}_i = \tilde{\epsilon}_{2i}$ for $ i=1 \dots, n$. The fundamental partitions associated to this restricted root system take the form  
\begin{align*}\hat{\eta}_r =\epsilon^{\Sigma}_1+\cdots \epsilon^{\Sigma}_r=\tilde{\epsilon}_{2} + \cdots + \tilde{\epsilon}_{2r}=(\epsilon_{1} + \epsilon_{2} + \cdots +\epsilon_{2r})/2=\hat{\omega}_{2r}/2 \end{align*}  for $r=1, \dots, n$. It follows that
the fundamental weights are \begin{align*}\eta_r = \hat{\eta}_r-({{r}/{n}})\hat{\eta}_n=(\epsilon_{1} + \epsilon_{2} + \cdots +\epsilon_{2r})/2 -({{r}/{n}})(\epsilon_{1} + \epsilon_{2} + \cdots +\epsilon_{2n})/2 = {\omega}_{2r}/2
\end{align*} for $r=1, \dots, n-1$.  
Note also that
\begin{align*}(\hat{\omega}_{2j}-\theta(\hat{\omega}_{2j}) )/2=( \tilde{\epsilon}_1+\cdots +\tilde{\epsilon}_{2j}) =2( \tilde{\epsilon}_2  + \tilde{\epsilon}_{4} + \cdots + \tilde{\epsilon}_{2j} )= 2\hat{\eta}_j
\end{align*}
for each $j=1, \dots, n$. Similarly, 
\begin{align*}(\hat{\omega}_{2j-1}-\theta(\hat{\omega}_{2j-1}) )/2=( \tilde{\epsilon}_1+\cdots +\tilde{\epsilon}_{2j-1}) =2( \tilde{\epsilon}_2  + \tilde{\epsilon}_{4} + \cdots + \tilde{\epsilon}_{2j-2} )+\tilde{\epsilon}_{2j}
\end{align*}
for  $j=1, \dots, n$. Hence $(\hat{\omega}_{1}-\theta(\hat{\omega}_{1}) )/2=\hat{\eta}_1$ and 
\begin{align*}
(\hat{\omega}_{2j-1}-\theta(\hat{\omega}_{2j-1}) )/2= \hat{\eta}_j+\hat{\eta}_{j-1}
\end{align*}
for each $j=2, \dots, n$. It follows that $(\omega_1 -\theta(\omega_1))/2 = \eta_1$, $(\omega_{2j}-\theta(\omega_{2j}))/2 =2 \eta_j$ and $(\omega_{2j-1}-\theta(\omega_{2j-1}))/2 = \eta_j+\eta_{j-1}$ for $j=2, \dots, n$.

\medskip
\noindent
{\bf Type A diagonal case:} 
For the diagonal case, we have $\tilde{\alpha}_i = (\alpha_i+\alpha_{n+i}) /2=\tilde{\alpha}_{n+i}$. Thus the set of positive simple roots for the root system 
$\Sigma$ is $\{\tilde{\alpha}_1, \cdots, \tilde{\alpha}_{n-1}\}$ and $\alpha_i^{\Sigma} = \tilde{\alpha}_i$ for $i=1, \dots, n-1$. 
Note that $(\tilde{\alpha}_i,\tilde{\alpha}_i) = ((\alpha_i+\alpha_{n+i})/2, (\alpha_i+\alpha_{n+i})/2) = 1$. Hence, the inner product for the restricted root system in the diagonal case is $(\cdot, \cdot)_{\Sigma} = 2(\cdot, \cdot)$. Here, the corresponding orthonormal basis is $\{\tilde{\epsilon}_1, \dots, \tilde{\epsilon}_n\}$
where $\tilde{\epsilon}_i = (\epsilon_i+\epsilon_{n+i})/2$ for each $i$. In particular, we have $\epsilon^{\Sigma}_i = \tilde{\epsilon}_i$ for $i=1, \dots, n$.
The fundamental partitions  in this case are 
\begin{align*}\hat{\eta}_r 
=\epsilon^{\Sigma}_1+\cdots + \epsilon^{\Sigma}_r
=\tilde{\epsilon}_{1} +\cdots + \tilde{\epsilon}_r=({\epsilon}_{1} +\epsilon_{1+n}+\cdots +{\epsilon}_r+\epsilon_{r+n})/2 
\end{align*} for $r=1, \dots, n$. It follows that the fundamental weights associated to $\Sigma$ satisfy 
\begin{align*}\eta_r =\hat{\eta}_r-({r}/{n})\hat{\eta}_n = (\hat{\omega}_r+\hat{\omega}_{r+n})/2 -({r}/{n})(\hat{\omega}_n+\hat{\omega}_{2n})/2=\omega_r/2+\omega_{r+n}/2
\end{align*} 
for $r=1, \dots, n-1$.
We have 
\begin{align*}(\hat{\omega}_{j}-\theta(\hat{\omega}_{j}) )/2=\tilde{\epsilon}_1+\cdots +\tilde{\epsilon}_{j}= \hat{\eta}_j
\end{align*}
and, similarly, $(\hat{\omega}_{n+j}-\theta(\hat{\omega}_{n+j}))/2=\hat{\eta}_j$.  Thus $\eta_j = \tilde{\omega}_j= \tilde{\omega}_{n+j}$ for $j=1,\dots, n$.

\bigskip

A finite-dimensional simple highest weight module is called spherical if it contains a nonzero $\mathcal{B}_{\theta}$ invariant vector, i.e., a vector $v$ such that $x\cdot v=\epsilon(x)v$ for $x\in\mathcal B_\theta$.
Note that $2\Lambda^+_{\Sigma}=\{2\lambda|\ \lambda \in \Lambda^+_{\Sigma}\}$ is a subset of $\Lambda_N^+$ where $N=n$ in Type AI and $N=2n$ in Type AII. Moreover, given $\gamma\in \Lambda_N^+$, the module  $L(\gamma)$ is spherical if and only if 
$\gamma =2\lambda+s\hat{\eta}_n$ for some $\lambda\in \hat{\Lambda}^+_{\Sigma}$ and  $s\in \{0,1\}$ 
  in Type AI and $\gamma=2\lambda+2s\hat{\eta}_{2n}$ for $\lambda\in \hat{\Lambda}^+_{\Sigma}$ and $s\in \mathbb{N}$ in Type AII (\cite{N}, (3.12)).   (Note that we stick with $\mathbb{N}$ instead of $\mathbb{Z}$ since we are considering functions on matrices rather than symmetric spaces $G/K$.)
  The extra assumption in type AI is because $2\Lambda^+_{\Sigma}$ already contains the even nonnegative multiples of $\hat{\eta}_n$ while all terms of the form $\mathbb{N}\hat{\eta}_n$ show up in the description of spherical modules.

Now consider the diagonal case. For $\gamma,\gamma'\in \Lambda^+_n$, $L({\gamma\oplus \gamma'})$ is spherical  if and only if  $\gamma=\gamma'$.   This  follows from the classification of spherical modules using $\mathfrak{sl}_n$ instead of $\mathfrak{gl}_n$ (see \cite{L2002}, Section 7) along with the fact that  weight vectors admitting a trivial $K_{\epsilon_i}K_{\epsilon_{i+n}}^{-1}$ action for each $i=1, \dots,n$ must have weights of the form $\gamma\oplus \gamma$. Note that the set of $\gamma\oplus \gamma$ with $\gamma\in \Lambda_n^+$ is precisely $ 2\Lambda^+_{\Sigma}$. Thus $L({\gamma\oplus \gamma'})$ is spherical if and only if $\gamma\oplus \gamma'=2\lambda$ for some $\lambda\in \Lambda^+_{\Sigma}$.

\section{Quantized function algebras}
\subsection{Quantized matrix functions}\label{section:qmf}
Let $\Mat_N$ denote the space of $N\times N$ matrices with basis $e_{ij}$, $1\leq i,j \leq N$. The matrix $R$  defined by (\ref{Rmatrix2})
can be written as  $R = \sum_{i,j,k,l} r^{ij}_{kl} e_{ik}\otimes e_{jl}$
where
\begin{itemize}
\item $r^{ii}_{ii} = q,$ $r_{ij}^{ij} = 1$
for all $i,j, $ with $i\neq j$.
\item 
$r_{ji}^{ij} = (q-q^{-1})$
for all $j<i$.
\item $r^{ij}_{kl} = 0$ for all other choices of $i,j,k,l$.
\end{itemize}
We can view $R$ as the matrix in $\Mat_N\times \Mat_N$ with $(i,k)\times (j,l)$ entry $r^{ij}_{kl}$. Given a matrix $A \in \Mat_N\times \Mat_N$, write $A^{t_s}$ for the transpose in the first term when $s=1$ and the second when $s=2$. For example, 
$(R^{t_1})^{ij}_{kl} = r^{kj}_{il}.$

The quantized function algebra $\mathcal{O}_q(\Mat_N)$ is the bialgebra over $\mathbb{C}(q)$ generated by $t_{ij}, 1\leq i,j\leq N$ where the $t_{ij}$ satisfy the following relations
\begin{itemize}
\item[(i)]$t_{ki}t_{kj} = qt_{kj}t_{ki}$, $t_{ik}t_{jk} = qt_{jk}t_{ik}$ ($i<j$)
\item[(ii)] $t_{il}t_{kj}  = t_{kj}t_{il}, t_{ij}t_{kl} -t_{kl}t_{ij} = (q-q^{-1})t_{kj}t_{il}$ ($i<k;j<l$)
\end{itemize}
Set $T=(t_{ij})$, the $N\times N$ matrix with $ij$ entry equal to $t_{ij}$ and set $T_1=T\otimes I$ and $T_2 = I\otimes T$ where $I$ is the $N\times N$ identity matrix.   These relations can be written in matrix form as $RT_1T_2 = T_2T_1R$, or equivalently, as the set of equations
 \begin{align*}\sum_{j,k}r_{jk}^{ld} t_{ja}t_{kb} = \sum_{j,k}t_{dk}t_{lj}r^{jk}_{ab}.
\end{align*}
for all $a,b,d,l$.
A straightforward computation shows that  the map $\iota$ defined by 
\begin{align}\label{iota} \iota(t_{ij}) = t_{ji} 
\end{align} 
for all $i,j=1,\dots, n$ defines an algebra automorphism of $\mathcal{O}_q({\rm Mat}_N)$.  

The algebra  $\mathcal{O}_q(\Mat_N)$ admits the structure of a $U_q(\mathfrak{gl}_N)$-bimodule  algebra where the left action is determined by 
 \begin{align*}
 &E_{k} \cdot t_{ij} =\delta_{i-1,k}t_{i-1,j},
 \quad 
F_{k} \cdot t_{ij} =\delta_{ik}  t_{i+1,j},
\quad K_{\epsilon_r} \cdot t_{ij} = q^{\delta_{ir}} t_{ij}
\end{align*} 
and the 
 right action by \begin{align*} &t_{ij} \cdot E_k=\delta_{jk}t_{i,j+1},\quad
 t_{ij}\cdot F_k=\delta_{j-1,k} t_{i,j-1}, \quad
t_{ij} \cdot K_{\epsilon_r}= q^{\delta_{jr}}t_{ij}
\end{align*} 
for $r= 1, \dots, N$, $i,j=1, \dots, N$ and $k=1, \dots, N-1$. Here, we are using the notation $t_{uv} = 0$ for $u\in \{0,N+1\}$ or $v\in \{0, N+1\}$.

Let $\mathcal{O}_q(\Mat_N)^{op}$ denote the bialgebra with the same coalgebra structure and opposite algebra structure of the bialgebra $\mathcal{O}_q(\Mat_N)$.   Write $\partial_{ij}, 1\leq i,j\leq N$ for the generators of $\mathcal{O}_q(\Mat_N)^{op}$ so that the algebra map defined by $t_{ij}\mapsto \partial_{ij}$ is an anti-automorphism.  The algebra $\mathcal{O}_q(\Mat_N)^{op}$ is also a $U_q(\mathfrak{gl}_N)$-bimodule algebra  with action defined by 
 \begin{align*}
 &E_{k} \cdot \partial_{ij} =-\delta_{ik}q^{-1}\partial_{i+1,j},
 \quad 
F_{k} \cdot \partial_{i,j} =-\delta_{i-1,k}  q\partial_{i-1,j},
\quad K_{\epsilon_r} \cdot \partial_{ij} = q^{-\delta_{ir}} \partial_{ij}
\end{align*} 
and right action  defined by 
\begin{align*} &\partial_{ij} \cdot E_k=-\delta_{j-1,k}q\partial_{i,j-1},\quad
 \partial_{ij}\cdot F_k=-\delta_{jk} q^{-1}\partial_{i,j+1}, \quad
\partial_{ij} \cdot K_{\epsilon_r}= q^{-\delta_{jr}}\partial_{ij}
\end{align*} 
for $r= 1, \dots, N$, $i,j=1, \dots, N$ and $k=1, \dots, N-1.$ Just as for the $t_{uv}$, we are using the notation $\partial_{uv} = 0$ for $u\in \{0,N+1\}$ or $v\in \{0, N+1\}$.

Let $\mathcal{O}_q(\Mat_N)\otimes \mathcal{O}_q(\Mat_N)$ denote the algebra generated by two copies of $\mathcal{O}_q(\Mat_N)$.  The first is generated by $t_{ij}, 1\leq i,j\leq N$ and the second is generated by $t_{i+N,j+N}, 1\leq i,j\leq N$ and $t_{ij}$ commutes with $t_{k+N,l+N}$ for all 
$i,j,k,l\in \{1, \dots, N\}$.  Formally, $t_{ij}$ can be identified with $t_{ij}\otimes 1$ and $t_{i+N,j+N}$ with $1\otimes t_{ij}$. The algebra $\mathcal{O}_q(\Mat_N)\otimes \mathcal{O}_q(\Mat_N)$ is a $U_q(\mathfrak{gl}_N\oplus \mathfrak{gl}_N)\cong U_q(\mathfrak{gl}_N)\otimes U_q(\mathfrak{gl}_N)$-bimodule algebra. Here,  the left action is given by  $(g\otimes h)\cdot (a\otimes b) = (g\cdot a)\otimes (h\cdot b)$ for all $g\otimes h\in U_q(\mathfrak{gl}_N)\otimes U_q(\mathfrak{gl}_N)$ and $a\otimes b\in \mathcal{O}_q(\Mat_N)\otimes \mathcal{O}_q(\Mat_N)$. The right action is defined in the same way with the action on the right instead of the left.   Similar notions apply for the algebra $\mathcal{O}_q(\Mat_N)^{op}\otimes \mathcal{O}_q(\Mat_N)^{op}$.

Consider the  three families as described in Section \ref{section:three-families}.  Set $\mathscr{P}=\mathcal{O}_q(\Mat_N)$ in Type AI with $N=n$  and in Type AII with $N=2n$.  For the diagonal case, set $\mathscr{P}= \mathcal{O}_q(\Mat_n)\otimes \mathcal{O}_q(\Mat_n)$.  Similarly, set $\mathscr{D}=\mathcal{O}_q(\Mat_N)^{op}$ in Type AI with $N=n$  and in Type AII with $N=2n$.  For the diagonal case, set $\mathscr{D}= \mathcal{O}_q(\Mat_n)^{op}\otimes \mathcal{O}_q(\Mat_n)^{op}$.

\subsection{Functions on homogeneous spaces}  \label{section:qhs}

Let $\mathfrak{g}, \mathfrak{k}$ be a symmetric pair corresponding to one of the three  families described in Section \ref{section:three-families}.  Recall that $J$ is the associated solution to the reflection equation. Write $J_{r,s}$ for the coefficient of $e_{rs}$ in $J$.  Define elements $x_{ij}$ and $d_{ij}$ by
  \begin{align*}x_{ij}= \sum_{r,s}t_{ir}J_{r,s}t_{js} {\rm \ and \ }d_{ij} = \sum_{r,s}q^{-2\hat{s}}\partial_{ir}J_{r,s}\partial_{js}
  \end{align*}
for $1\leq i,j\leq N$ where $N=n$ in Type AI, $N=2n$ in Type AII and diagonal type, $\hat{s} =  s$ in Types AI and AII,  and  for the diagonal type   $\hat{s} =s$ for $s\leq n$, and   $\hat{s} = s - n$  when $s\geq n+1$. 

 Let $\mathscr{P}_{\theta}$ be the subalgebra of $\mathscr{P}$ generated by the $x_{ij}, 1\leq i,j\leq N$, and let 
 $\mathscr{D}_{\theta}$ be the subalgebra of $\mathscr{D}$   generated by the $d_{ij}, 1\leq i,j\leq N$.  As explained in \cite{N} and \cite{LSS}, the quantum homogeneous space $\mathcal{O}_q[G/K]$ associated to $\mathfrak{g}, \mathfrak{k}$ is generated by $\mathscr{P}_{\theta}$ and powers of quantum determinants.  
  Set 
 $X=(x_{ij})_{1\leq i,j\leq N}$ and 
 $D=(d_{ij})_{1\leq i,j\leq N}$. Define $X_1, X_2, D_1,$ and $D_2$ in the same way as $T_1, T_2$ (i.e. $X_1 = X\otimes I$, etc.).  The following theorem summarizes  properties of $\mathscr{P}_{\theta}$ from \cite{LSS}.
 \begin{theorem}\label{theorem:polypart}
The relations satisfied by the generators $x_{ij}$  of $\mathscr{P}_{\theta}$ are determined by the following equations:
 \begin{itemize} 
 \item[(i)] The matrix relation $R_{\mathfrak{g}}X_1R_{\mathfrak{g}}^{t_1}X_2 = X_2R_{\mathfrak{g}}^{t_1}X_1R_{\mathfrak{g}}$.
  \item[(ii)] The linear relations $x_{ij} = \gamma x_{ji}$ for $i<j$ and $x_{ab} = 0$ for all $(a,b)\in \mathcal{S}$ where
 \begin{itemize}
 \item[$\bullet$]  $\gamma = q$ in Type AI, $\gamma=-q^{-1}$ in Type AII, $\gamma =1$ in the diagonal type
 \item[$\bullet$] $\mathcal{S}$ is the empty set in Type AI, $\mathcal{S}=\{(i,i),i=1, \dots, 2n\}$ in Type AII, and $\mathcal{S}=\{(i,j), 1\leq i,j\leq n\}\cup\{(i,j), n+1\leq i,j\leq 2n\}$ in diagonal type.
 \end{itemize}
  \end{itemize}
  Moreover,  $\mathscr{P}_{\theta}$ is a  left $U_q(\mathfrak{g})$-module  and (trivial) right $\mathcal{B}_{\theta}$-module subalgebra of $\mathscr{P}$.
  \end{theorem}
  In the diagonal case, it turns out that the matrix relations for $\mathscr{P}_{\theta}$ reduce to $R\hat{X}_1\hat{X}_2 = \hat{X}_2\hat{X}_1R$ where $\hat{X}=(x_{i,j+n})_{1\leq i,j\leq n}$ and $R$ is the matrix defined by (\ref{Rmatrix2}) with $N=n$ (\cite{LSS}, Lemma 5.12).  In particular, $\mathscr{P}_{\theta}\cong \mathcal{O}_q(\Mat_n)$ as algebras via the map that sends $x_{i,j+n}$ to $t_{ij}$ all $1\leq i,j\leq n$.   
For the other two families, Types AI and AII, explicit relations are given in \cite{LSS}, Lemma 5.8.  Here we provide some of these relations in special cases that will be needed later in this paper.  In particular, for Type AI we have  
  \begin{align} \label{xreln1} x_{en}x_{nn} = q^2x_{nn}x_{en} \quad{\rm and\quad} x_{an}x_{en} = qx_{en}x_{an}
  \end{align}
  for all $1\leq a<e<n$.  Similarly, for Type AII, we have 
    \begin{align} \label{xreln2} x_{a,2n}x_{e,2n} = qx_{e,2n}x_{a,2n}
  \end{align}
  for $1\leq a<e<2n$.
  More generally, it follows from \cite{LSS}, Lemma 5.8 that the defining relations are $q$-analogs of commutativity relations and take the form 
  \begin{align}\label{quadformula} x_{ij}x_{kl} = q^sx_{kl}x_{ij} + \sum_{\{a,b,c,d\}=\{i,j,k,l\}}(q-q^{-1})wx_{ab}x_{cd}\end{align}
  for some integer $s$ and some element $w\in \mathbb{C}[q,q^{-1}]$ where  neither $x_{ij}x_{kl}$ nor $x_{kl}x_{ij}$  appear in the final sum of the right hand side.  It follows from the formulas for the relations satisfied by the $t_{ij}$ and the fact that $\mathscr{P}_{\theta}\cong \mathcal{O}_q(\Mat_n)$ as algebras in the diagonal case, that (\ref{quadformula}) holds for the diagonal family as well.

The following result, also from \cite{LSS}, holds for $\mathscr{D}_{\theta}$ and shows that as an algebra, $\mathscr{D}_{\theta}$ is isomorphic to $\mathscr{P}_{\theta}^{op}$. 
 In analogy to the situation for $\mathscr{P}_{\theta}$,  the map sending $d_{i,j+n}$ to $\partial_{ij}$ for $1\leq i,j\leq n$ defines an isomorphism of $\mathscr{D}_{\theta}$ onto $\mathcal{O}_q(\Mat_n)^{op}$ for the diagonal family. 

 \begin{theorem}\label{theorem:iso}
The relations satisfied by the generators $d_{ij}$  of $\mathscr{D}_{\theta}$ are determined by the following equations:
 \begin{itemize} 
 \item[(i)] The matrix relation $R_{\mathfrak{g}}D_2R_{\mathfrak{g}}^{t_1}D_1 =D_1R_{\mathfrak{g}}^{t_1}D_2R_{\mathfrak{g}}$. \item[(ii)] The linear relations $d_{ij} = \gamma d_{ji}$ for $i<j$ and $d_{ab} = 0$ for all $(a,b)\in \mathcal{S}$ where
 \begin{itemize}
 \item[$\bullet$]  $\gamma = q^{-1}$ in Type AI, $\gamma=-q$ in Type AII, $\gamma =1$ in the diagonal type
 \item[$\bullet$]$\mathcal{S}$ is the empty set in Type AI, $\mathcal{S}=\{(i,i),i=1, \dots, 2n\}$ in Type AII, and $\mathcal{S}=\{(i,j), 1\leq i,j\leq n\}\cup\{(i,j), n+1\leq i,j\leq 2n\}$ in diagonal type.
 \end{itemize}
  \end{itemize}
  Moreover,  $\mathscr{D}_{\theta}$ is a  left $U_q(\mathfrak{g})$-module  and (trivial) right $\mathcal{B}_{\theta}$-module subalgebra of $\mathscr{D}$.
  \end{theorem}
  
  Note that the defining relations for  $\mathscr{P}_{\theta}$ are homogeneous and thus $\mathscr{P}_{\theta}$ has a natural  degree function defined by $\deg (x_{ij}) = 1$ for all $i,j$.  
Let  $\mathcal{J}$ be the filtration on $\mathscr{P}_{\theta}$ defined by $\deg$. In particular, for each $r$, we have 
 \begin{align*}\mathcal{J}_r(\mathscr{P}_{\theta}) = \{x\in \mathscr{P}_{\theta}|\ \deg(x)\leq r\}.
 \end{align*}
 Since all the relations for $\mathscr{P}_{\theta}$ are homogeneous with respect to degree, $\mathscr{P}_{\theta}$ is isomorphic to the graded algebra defined by this filtration. 
Just as for $\mathscr{P}_{\theta}$, we can define a degree function on $\mathscr{D}_{\theta}$ such that $\deg(d_{ij}) = 1$ for all $i,j$.   The resulting filtration yields a graded algebra isomorphic to $\mathscr{D}_{\theta}$.  
For each $r$, set $\mathscr{P}^r_{\theta}$ equal to the homogeneous subspace of $\mathscr{P}_{\theta}$ consisting of elements of exactly degree $r$ and $\mathscr{D}_{\theta}^r$ equal to the homogeneous subspace of $\mathscr{D}_{\theta}$ consisting of elements of exactly degree $r$. 
 
  The left $U_q(\mathfrak{g})$-module structures  can be deduced directly from the left actions of $U_q(\mathfrak{g})$ on $\mathscr{P}$ and $\mathscr{D}$.  The action of the generators of $U_q(\mathfrak{g})$ on the generators of $\mathscr{P}_{\theta}$ and $\mathscr{D}_{\theta}$ is explicitly given in \cite{LSS}, Lemma 5.4.  
As noted in \cite{LSS}, this action can be described as follows.  The vector space spanned by the generators $x_{ij}$ of $\mathscr{P}_{\theta}$ forms a simple left $U_q(\mathfrak{g})$-module generated by a highest weight vector $x_{1r}$ of weight $L(\epsilon_1+\epsilon_r)$ where $r=1$ in Type AI, $r=2$ in Type AII, and $r=n+1$ in the diagonal type. Similarly, $\sum_{i,j}\mathbb{C}(q)d_{ij}$ is a simple left $U_q(\mathfrak{g})$-module generated by the lowest weight vector $d_{1r}$ of weight $-\epsilon_1-\epsilon_r$ where $r=1$ in Type AI, $r=2$ in Type AII, and $r=n+1$ in the diagonal type.  
  Moreover, the action of the Cartan elements on $\mathscr{P}_{\theta}$ and $\mathscr{D}_{\theta}$ is determined by 
  \begin{align}\label{Kaction} K_{\epsilon_s}\cdot x_{ij} = q^{\delta_{is} + \delta_{js}}x_{ij} \quad \quad K_{\epsilon_s}\cdot d_{ij} = q^{-\delta_{is}-\delta_{js}}d_{ij}
  \end{align}
  for all $i,j,s$ in Types AI and AII and 
    \begin{align}\label{Kaction2} K_{\epsilon_s}\cdot x_{i,j+n} = q^{\delta_{is} + \delta_{j+n,s}}x_{i,j+n} \quad \quad K_{\epsilon_s}\cdot d_{i,j+n} = q^{-\delta_{is}-\delta_{j+n,s}}d_{i,j+n}
  \end{align}
   for all $i,j,s$ in the diagonal setting.

 Note that the isomorphisms $\mathscr{P}_{\theta}\cong \mathcal{O}_q(\Mat_n)$ and 
 $\mathscr{D}_{\theta}\cong \mathcal{O}_q(\Mat_n)^{op}$ in the diagonal setting as described above  are actually $U_q(\mathfrak{gl}_n)$-bimodule isomorphisms where the left action of $U_q(\mathfrak{gl}_n)$ on $\mathcal{O}_q(\Mat_n)$ (resp. $\mathcal{O}_q(\Mat_n)^{op}$) is the same as the action of the first copy of $U_q(\mathfrak{gl}_n)$ inside 
  $U_q(\mathfrak{gl}_n\oplus\mathfrak{gl}_n)$ on $\mathscr{P}_{\theta}$ (resp. $\mathscr{D}_{\theta}$).  Similarly, the right action for $\mathscr{P}_{\theta}$ (resp. $\mathscr{D}_{\theta})$ goes over to the action of the second copy of $U_q(\mathfrak{gl}_n)$ on $\mathcal{O}_q(\Mat_n)$ (resp. $\mathcal{O}_q(\Mat_n)^{op}$). (See \cite{LSS}, Lemma 5.12 for details.)

  It is well-known that the algebra $\mathcal{O}_q(\Mat_N)$ admits a PBW basis using monomials in the $t_{ij}$.  The next result from \cite{LSS} shows that the same is true for $\mathscr{P}_{\theta}$.  Using the fact that $x_{ij}\mapsto d_{ij}$ defines an antiautomorphism from $\mathscr{P}_{\theta}$ to $\mathscr{D}_{\theta}$, the next lemma also holds for $\mathscr{D}_{\theta}$ with each $x_{ij}$ term replaced by $d_{ij}$.
  
  \begin{lemma} \label{lemma:PBW} The following monomials form a basis for $\mathscr{P}_{\theta}$ where $N=n$:
\begin{itemize}
\item[(i)] Type AI: $$x_{11}^{m_{11}}x_{12}^{m_{12}}\cdots x_{1n}^{m_{1n}}x_{22}^{m_{22}}x_{23}^{m_{23}}\cdots x_{2n}^{m_{2n}}\cdots x_{n-1,n-1}^{m_{n-1,n-1}}x_{n-1,n}^{m_{n-1,n}}x_{nn}^{m_{nn}}$$
\item[(ii)] Type AII: 
$$x_{12}^{m_{12}}x_{13}^{m_{13}}\cdots x_{1,2n}^{m_{1,2n}}x_{23}^{m_{23}}x_{24}^{m_{24}}\cdots x_{2,2n}^{m_{2,2n}}\cdots x_{2n-1,2n}^{m_{2n-1,2n}}$$
\item[(iii)] Diagonal type: $$
x_{1,n+1}^{m_{11}}x_{12}^{m_{12}}\cdots x_{1,2n}^{m_{1n}}x_{2,n+1}^{m_{21}}x_{2,n+2}^{m_{22}}\cdots x_{2,2n}^{m_{2n}}\cdots x_{n,n+1}^{m_{n1}}\cdots x_{n,2n}^{m_{nn}}$$
\end{itemize}
as each $m_{ij}$ runs over nonnegative integers.  Moreover, we also get a basis if the order of the monomials in the terms above are reversed. \end{lemma}

\subsection{Module structure}\label{section:modules}

It is well-known that $\mathcal{O}_q(\Mat_N)$ admits a decomposition as $U_q(\mathfrak{gl}_N)$-bimodules 
  \begin{align}\label{decompO}
\mathcal{O}_q(\Mat_N) \cong \bigoplus_{\lambda\in \Lambda^+_N}L(\lambda)\otimes L(\lambda)^*
 \end{align}
 where $L(\lambda)$ is a left $U_q(\mathfrak{gl}_N)$-module and $L(\lambda)^*$ is a right $U_q(\mathfrak{gl}_N)$-module.
By restriction, we can also view all these modules as $U_q(\mathfrak{sl}_N)$-modules.  Note that $\mathcal{O}_q(\Mat_N)$ contains a nontrivial bi-invariant submodule  with respect to the action of 
$U_q(\mathfrak{sl}_N)$.  In particular, the submodule of $\mathcal{O}_q(\Mat_N)$ consisting of $U_q(\mathfrak{sl}_N)$-bi-invariants corresponds to
$$ \bigoplus_{m\in \mathbb{N}}L(m\hat{\omega}_N)\otimes L(m\hat{\omega}_N)^*.$$
Decomposition (\ref{decompO}) implies
  \begin{align}\label{decompO2}
\mathcal{O}_q(\Mat_N) \otimes \mathcal{O}_q(\Mat_N)\cong\bigoplus_{(\lambda, \lambda')\in \Lambda^+_N\times \Lambda^+_N}L(\lambda\oplus \lambda')\otimes L(\lambda\oplus\lambda')^*
 \end{align}
 as $U_q(\mathfrak{gl}_N\oplus \mathfrak{gl}_N)$-bimodules.   Moreover, the submodule of $\mathcal{O}_q(\Mat_N) \otimes \mathcal{O}_q(\Mat_N)$ consisting of $U_q(\mathfrak{sl}_N\oplus \mathfrak{sl}_N)$-bi-invariants is 
$$ \bigoplus_{m,m'\in \mathbb{N}}L(m\hat{\omega}_N \oplus m'\hat{\omega}_N)\otimes L(m\hat{\omega}_N \oplus m'\hat{\omega}_N)^*.$$

Setting $N=n$ in Type AI and the diagonal case and $N=2n$ in Type AII, we can read off of  (\ref{decompO}) and (\ref{decompO2}) the $U_q(\mathfrak{g})$-module decomposition of $\mathscr{P}$.
  It follows from the description of spherical weights in Section \ref{section:restricted-root-system} that the right $\mathcal{B}_{\theta}$-invariants $\mathscr{P}^{\mathcal{B}_{\theta}}$ of $\mathscr{P}$ admits the following decomposition as left $U_q(\mathfrak{g})$-modules:
 \begin{align*} 
\bigoplus_{\lambda\in \Lambda^+_{\Sigma}, s\in \{0,1\}} &L(2\lambda+s\hat{\omega}_n) {\rm \  in\  Type\  AI,}\cr
\bigoplus_{\lambda\in \Lambda^+_{\Sigma}}&L(2\lambda){\rm \  in\ Type\ AII, \ and}\cr
\bigoplus_{\lambda\in \Lambda^+_{\Sigma}}&L(2\lambda)=\bigoplus_{\gamma\in \Lambda^+_n} L(\gamma\oplus\gamma){\rm \  in\  the\ diagonal\ type.\ }
 \end{align*}
By \cite{VSS} Lemma 5.3, the fact that $\mathscr{P}_{\theta}$ is generated by right $\mathcal{B}_{\theta}$-invariant elements ensures that $\mathscr{P}_{\theta}$ is a subalgebra and submodule of 
$\mathscr{P}^{\mathcal{B}_{\theta}}$.  We see in Section \ref{section:detailed} that $\mathscr{P}^{\mathcal{B}_{\theta}}$ agrees with 
$\mathscr{P}_{\theta}$ in both Type AII and the diagonal case while it is slightly larger in Type AI.

\section{Detailed module decompositions}\label{section:detailed}

\subsection{Quantum determinants}\label{section:q-det}
There is a quantum analog of the determinant, denoted by $\det_q(T)$, which  is a central element in $\mathcal{O}_q(\Mat_N)$.   The quantum determinant  can be expressed explicitly in terms of the $t_{ij}$ as
\begin{align}\label{qdetdef}{\det}_q (T)= \sum_{s\in \mathcal{S}_n} (-q)^{l(s)} t_{s(1),1}\cdots t_{s(N),N}
\end{align}
and satisfies $\iota(\det_q(T)) = \det_q(T)$  where $\iota$ is the antiautomorphism defined in Section \ref{section:qmf} (see (1.26) of \cite{N}).
It is straightforward to check that the quantum determinant $\det_q(T)$ satisfies $\det_q(T)\cdot K_i = \det_q(T)$, 
$\det_q(T)\cdot F_i = 0$ and $\det_q(T)\cdot E_i=0$ for $i=1, \dots, N-1$. Hence $\det_q(T)$ is right invariant with respect to the action of $U_q(\mathfrak{sl}_N)$.   The same is true with respect to the left action of $U_q(\mathfrak{sl}_N)$ and can be easily verified with a similar computation using the fact that $\iota(\det_q(T)) = \det_q(T)$.
Hence $\det_q(T)$ is both a right and left invariant element with respect to the action of $U_q(\mathfrak{sl}_N)$.   However, the same is not true upon passing to $U_q(\mathfrak{gl}_N)$.  Indeed, we have
\begin{align}\label{kaction}
{\det}_q(T)\cdot K_{\epsilon_i} = K_{\epsilon_i}\cdot {\det}_q(T) = q{\det}_q(T).
\end{align} 
Let $\det_q(T')$ be the quantum determinant defined using the elements $t_{i+N,j+N}$ instead of the $t_{ij}$ viewed as elements of $\mathcal{O}_q(\Mat_N)\otimes \mathcal{O}_q(\Mat_N)$.  The properties for $\det_q(T)$ carry over to $\det_q(T')$ where each subscript $i$ is replaced by $i+N$.

Note that the weight of $\det_q(T)$ with respect to the action of $U_q(\mathfrak{gl}_N)$ is $\hat{\omega}_N = \epsilon_1+\cdots +\epsilon_N$.  Hence by (\ref{decompO}), the submodule of 
$U_q(\mathfrak{sl}_N)$-bi-invariants inside $\mathcal{O}_q(\Mat_N)$ is the polynomial ring $\mathbb{C}(q)[\det_q(T)]$.  Similarly,
 the submodule of $\mathcal{O}_q(\Mat_N) \otimes \mathcal{O}_q(\Mat_N) $  consisting of $U_q(\mathfrak{sl}_N\oplus \mathfrak{sl}_N)$-bi-invariants is the polynomial ring
 $\mathbb{C}(q)[\det_q(T), \det_q(T')]$.

\begin{lemma}\label{lemma:intersection}The intersection of $\mathbb{C}(q)[\det_q(T)]$ and $\mathscr{P}_{\theta}$ is $\mathbb{C}(q)[(\det_q(T))^2]$ in Type AI and equals 
$\mathbb{C}(q)[\det_q(T)]$ in type AII.   The intersection 
of $\mathbb{C}(q)[\det_q(T),\det_q(T')]$ and $\mathscr{P}_{\theta}$ is $\mathbb{C}(q)[\det_{q}(T)\det_q(T')]$ in the diagonal case.
\end{lemma}
\begin{proof} Consider Type AI.  Note that any element in $\mathscr{P}_{\theta}$ can be written as a linear combination of monomials, say $t_{i_1,j_1}\cdots t_{i_m,j_m}$.  Moreover, by the form of the elements $x_{ab}$, each  right  index $j_k$ must appear an even number of times in a particular monomial.   Thus examining $\det_q(T)$, we see that $\det_q(T)\notin \mathscr{P}_{\theta}$ since each right  index $j_k$ shows up exactly once in $t_{s(1),1}\cdots t_{s(N),N}$ where $N=n$ in Type AI and $N=2n$ in Type AII.  The same holds for $(\det_q(T))^m$ for $m$ odd.  On other hand, by Remark 4.12 of \cite{N}, $(\det_q(T))^2\in \mathscr{P}_{\theta}$. Hence the first assertion for Type AI 
holds. For Type AII, note that Remark 4.12 of \cite{N} also shows that $\det_q(T)\in \mathscr{P}_{\theta}$ in this case.  This completes the proof of the  first assertion of the lemma.

For the diagonal case, note that $\det_q(T)\det_q(T')$ is right invariant with respect to action of $U_q(\mathfrak{sl}_n\oplus \mathfrak{sl}_n)$ as well as with respect to the action of $K_{\epsilon_i}K_{\epsilon_{n+i}}^{-1}$ for $i=1, \dots, n$ where here we have $N=n$. 
Since $\mathcal{B}_{\theta}$ is a subalgebra of the algebra generated by both $U_q(\mathfrak{sl}_n\oplus \mathfrak{sl}_n)$ and the $K_{\epsilon_i}K_{\epsilon_{n+i}}^{-1}$ for $i=1, \dots, n$, it follows that 
$\det_q(T)\det_q(T')$ is an element of $\mathscr{P}^{\mathcal{B}_{\theta}}$.  Now $\det_q(T)\det_q(T')$ has weight $\hat{\omega}_n\oplus
\hat{\omega}_n$  with respect to the left action of $U_q(\mathfrak{gl}_n\oplus \mathfrak{gl}_n)$
and, moreover, generates a trivial left $U_q(\mathfrak{sl}_n\oplus \mathfrak{sl}_n)$-module.  Hence from the description of $\mathscr{P}^{\mathcal{B}_{\theta}}$ in the diagonal case given in Section \ref{section:modules}, we see that $(\det_q(T)\det_q(T'))^m$ is a basis vector for the one dimensional left module $L(m\hat{\omega}_n\oplus m\hat{\omega}_n)$.  Thus $\mathscr{P}^{\mathcal{B}_{\theta}}\cap   \mathbb{C}(q)[\det_q(T),\det_q(T')]= \mathbb{C}(q)[\det_q(T)\det_q(T')]$.  Since $\mathscr{P}_{\theta}\subseteq \mathscr{P}^{\mathcal{B}_{\theta}}$, we also have that the intersection $\mathscr{P}_{\theta}\cap   \mathbb{C}(q)[\det_q(T),\det_q(T')]$ is a subset of $ \mathbb{C}(q)[\det_q(T)\det_q(T')]$.

Now
consider the element $\det_q(X)$ defined by replacing each $t_{ij}$ in the definition of $\det_q(T)$ by $x_{ij}$, again in the diagonal case.
Recall that $\mathscr{P}_{\theta}$ is isomorphic as an algebra and $U_q(\mathfrak{gl}_n)$-bimodule to $\mathcal{O}_q(\Mat_n)$ (see the discussions following Theorems \ref{theorem:polypart}  and \ref{theorem:iso}).  It follows that $\det_q(X)$ is an element of $\mathscr{P}^{\mathcal{B}_{\theta}}$ invariant with respect to the left action of $U_q(\mathfrak{sl}_n\oplus \mathfrak{sl}_n)$. Moreover, 
it is straightforward to see that the weight of $\det_q(X)$ is $\hat{\omega}_n\oplus\hat{\omega}_n$.  By the previous paragraph, 
 $\det_q(X)$ must be a nonzero scalar multiple of 
 $\det_q(T)\det_q(T')$.  This guarantees the inclusion  $\mathbb{C}(q)[\det_q(T)\det_q(T')] \subseteq \mathscr{P}_{\theta}$ which combined with the previous paragraph yields  the desired equality. \end{proof}

Set $H_{n}=\det_q(T)\det_q(T')$ in the diagonal case,  $H_n = \det_q(T)^2$ in Type AI, and $H_n=\det_q(T)$ in  Type AII.  By Lemma \ref{lemma:intersection}, $H_n\in \mathscr{P}_{\theta}$.  Moreover, since $\det_q(T)$ is a central element in $\mathcal{O}_q({\rm Mat}_n)$  in Type AI, is a central element in  $\mathcal{O}_q({\rm Mat}_{2n})$ in Type AII, and is a central element of  $\mathcal{O}_q(\Mat_n) \otimes \mathcal{O}_q(\Mat_n) $ in the diagonal setting, 
 we must have $H_{n}$ is central in $\mathscr{P}_{\theta}$.

\subsection{Chains of algebras}\label{section:explicit}
Consider the chain of 
quantized enveloping algebras 
  \begin{align*} 
  U_q(\mathfrak{g}_1)\subset 
  U_q(\mathfrak{g}_2) \subset \cdots \subset U_q(\mathfrak{g}_{n}) 
  \end{align*}
  where 
  $\mathfrak{g}_r = \mathfrak{gl}_{r}$ in Type AI, $\mathfrak{g}_r = \mathfrak{gl}_{2r}$ in Type AII, and
  $\mathfrak{g}_r = \mathfrak{gl}_{r}\oplus \mathfrak{gl}_{r}$ in the diagonal case.  This means that $U_q(\mathfrak{g}_n)= 
  U_q(\mathfrak{gl}_{n})$ in Type AI, $U_q(\mathfrak{g}_n)= 
  U_q(\mathfrak{gl}_{2n})$ in Type AII, and   $U_q(\mathfrak{g}_n)=  U_q(\mathfrak{gl}_{n}\oplus \mathfrak{gl}_{n})$ in the diagonal case. Here $U_q(\mathfrak{g}_r)$ is identified with the subalgebra of $U_q(\mathfrak{g_n})$ generated by 
  $E_{i}, F_{i}, K_{\epsilon_{j}}$ where 
  \begin{itemize}
  \item $i\in \{1,\dots, r-1\}$ and $j\in \{1, \dots, r\}$ in Type AI
  \item $i\in \{1,\dots, 2r-1\}$ and $j\in \{1, \dots, 2r\}$ in Type AII
  \item  $i\in \{1, \dots,r-1\}\cup\{n+1, \dots, n+r-1\}$, $j\in \{1, \dots, r\}\cup \{n+1, \dots, n+r\}$ in the diagonal case.
  \end{itemize} 
  Note that $U_q(\mathfrak{g}_1)$ is just a commutative Laurent polynomial ring over $\mathbb{C}(q)$ in Type AI and the diagonal case.   In Type AI, this Laurent polynomial ring is generated by $K_{\epsilon_1}$, in the diagonal case, it is generated by 
  $K_{\epsilon_1}$ and $K_{\epsilon_{2}}$. In Type AII, $U_q(\mathfrak{g}_1)$ is the  quantized enveloping algebra of $\mathfrak{gl}_2$ generated by $E_1,F_1,$ and $K_{\epsilon_1}^{\pm 1}, K_{\epsilon_2}^{\pm 1}$.
    
  Similarly, we have a chain  of quantized function algebras  
  \begin{align*} 
\mathscr{P}(\mathfrak{g}_1) \subset \mathscr{P}(\mathfrak{g}_2)\subset \cdots \subset \mathscr{P}(\mathfrak{g}_n) =\mathscr{P}\end{align*} where 
$\mathscr{P}(\mathfrak{g}_r)\cong \mathcal{O}_q(\Mat_r)$ in Type AI, $\mathscr{P}(\mathfrak{g}_r)\cong \mathcal{O}_q(\Mat_{2r})$ in Type AII, and 
$\mathscr{P}(\mathfrak{g}_r)\cong \mathcal{O}_q(\Mat_r)\otimes \mathcal{O}_q(\Mat_r)$ in the diagonal case. 
Moreover, 
this isomorphism is an equality for $r=n$ and so  $\mathscr{P}(\mathfrak{g}_{n}) = \mathscr{P}$.
For $r<n$, 
$\mathscr{P}(\mathfrak{g}_r)$  is equal to  the subalgebra of $\mathscr{P}$ generated by 
\begin{itemize}
\item $t_{ij}$ for $1\leq i,j\leq r$ in Type AI
\item $t_{ij}$ for $1\leq i,j\leq 2r$ in Type AII
\item $t_{ij}$ for  $1\leq i,j\leq r$ and $n+1\leq i,j\leq n+r$ in the diagonal case.
\end{itemize}
Note that 
$\mathscr{P}(\mathfrak{g}_r)$  is a $U_q(\mathfrak{g}_r)$-bimodule and, moreover, this bimodule structure is compatible with the $U_q(\mathfrak{g})$-bimodule structure on $\mathscr{P}$.

 Set $\mathcal{B}_{\theta}^r = \mathcal{B}_{\theta}\cap U_q(\mathfrak{g}_r)$ for $r=1, \dots, n$.   Note that this gives us a chain of subalgebras  \begin{align*}
 \mathcal{B}_{\theta}^1 \subset \mathcal{B}_{\theta}^2 \subset \cdots \subset \mathcal{B}_{\theta}^n=\mathcal{B}_{\theta}.
 \end{align*}
For each $r$, it is straightforward to see that  $\theta$ restricts to an involution on $\mathfrak{g}_r$ with fixed Lie subalgebra $\mathfrak{k}_r = \mathfrak{g}_r^{\theta}$.  Thus $\mathcal{B}^r_{\theta}$ for the right coideal subalgebra of $U_q(\mathfrak{g}_r)$ that is a quantum analog of $U(\mathfrak{k}_r)$.  Note that $\mathcal{B}_{\theta}^r$ belongs to the same family as $\mathcal{B}_{\theta}$ with the only difference being the rank of the underlying Lie algebra $\mathfrak{g}_r$. 

 Using  $\mathcal{B}_{\theta}^r, \mathscr{P}(\mathfrak{g}_r)$, and $U_q(\mathfrak{g}_r)$, one can define the quantized function algebra  $\mathscr{P}_{\theta}(\mathfrak{g}_r)$  generated by elements $x(r)_{ij}$ 
 where $1\leq i,j\leq r$ in Type AI, $1\leq i,j\leq 2r$ in Type AII, and $1\leq i,j-n\leq r$ or $1\leq i-n,j\leq r$ in the diagonal case. Here, we use the notation $x_{ij}$ for the generators when $r=n$ (i.e. $x(n)_{ij} = x_{ij}$). 
 Note that the difference between $x(r)_{ij}$ and $x_{ij}$ has to do with which $t_{kl}$ are involved in the expression of these elements in terms of elements of $\mathscr{P}$.  For example, in Type AI, 
 \begin{align*}
 x(r)_{ij} = \sum_{k=1}^{r }t_{ik}t_{jk}
 \end{align*}
 while 
 \begin{align*}
 x_{ij} = \sum_{k=1}^{n}t_{ik}t_{jk}.
 \end{align*}
 Nevertheless, we see from the next lemma that this distinction is not  important.

 \begin{lemma} \label{lemma:mappings} For each $r,s$ with $1\leq r<s\leq n$, the map $\psi_{r,s}:\mathscr{P}_{\theta}(\mathfrak{g}_r)\rightarrow \mathscr{P}_{\theta}(\mathfrak{g}_s)$ defined by $\psi_{r,s}(x(r)_{ij}) = x(s)_{ij}$ induces an algebra embedding which preserves the left $U_q(\mathfrak{g}_r)$-module and (trivial) right $\mathcal{B}_{\theta}^r$-module structures.  
 \end{lemma}
 \begin{proof} Note that the relations for both algebras are given by Theorem \ref{theorem:polypart}.  Moreover, there are two types of relations: linear and quadratic.  The linear relations take the form $x_{ab}=0$ for various conditions on $a,b$ and $x_{ij} = \gamma x_{ji}$ for an appropriate scalar $\gamma$ and all $i,j$. Clearly these agree for the two algebras.   Hence $x(r)_{ab}=0$ if and only if $x(s)_{ab}=0$ and $x(r)_{ij} = \gamma x(r)_{ji}$ if and only if 
 $x(s)_{ij} = \gamma x(s)_{ji}.$
 
By  (\ref{quadformula}), the quadratic relations correspond to $q$-analogs of commutativity relations between two generators, say $x_{ij}$ and $x_{ld}$.  Moreover, the only terms showing up in these relations are of the form $x_{ab}x_{cd}$ where $\{a,b,c,d\} = \{i,j,l,d\}$.  Now if $i,j$ and $l,d$ satisfy the necessary conditions for $x(r)_{ij}$ and $x(r)_{ld}$ to be generators of $\mathscr{P}_{\theta}(\mathfrak{g}_r)$, then 
 $x(r)_{ab}$ is also a valid generator whenever $\{a,b\} \in \{i,j,l,d\}$.  In other words, when $i,j,l,d$ are chosen so that $x(r)_{ij}, x(r)_{ld}$ are among the generators for $\mathscr{P}_{\theta}(\mathfrak{g}_r)$, then the quadratic relations involving $x(r)_{ij}$ and $x(r)_{ld}$ inside $\mathscr{P}_{\theta}(\mathfrak{g}_r)$ take exactly the same form as the 
 relations satisfied by $x_{ij}$ and $x_{ld}$ inside of $\mathscr{P}_{\theta}$.  The same holds with $r$ replaced by $s$. 
Thus the quadratic relations satisfied by the $x(r)_{ij}$ of $\mathscr{P}_{\theta}(\mathfrak{g}_r)$ agree with the relations coming from the larger algebra $\mathscr{P}_{\theta}(\mathfrak{g}_s)$ for the corresponding elements $x(s)_{ij}$.

 The module structures for both algebras can be deduced directly from the actions of $U_q(\mathfrak{g})$  and $\mathcal{B}_{\theta}$ on $\mathscr{P}$ and these actions do not depend on $r$ or $s$.   These module actions agree, which yields
 the  desired module isomorphisms.
 \end{proof}
 
 An immediate consequence of Lemma \ref{lemma:mappings} is that $\mathscr{P}_{\theta}(\mathfrak{g}_r)$ is isomorphic to a subalgebra and left $U_q(\mathfrak{g}_r)$-submodule of $\mathscr{P}_{\theta}(\mathfrak{g}_n)$ where each generator $x(r)_{ij}$ is mapped to $x_{ij}$.  Moreover, it is straightforward to see that the embeddings of this lemma are all compatible with each other and so $\psi_{s,m} \circ\psi_{r,s}=\psi_{r,m}$ for all $1\leq r<s<m\leq n$. 
 
 \subsection{Highest weight generators}
 
Let $T_{(r)}$ be the submatrix of $T$ with entries $t_{ij}$ where $1\leq i,j\leq r$ in Type AI and  $1\leq i,j\leq 2r$ in Type AII.  Similarly, let $T_{(r)}$ be the submatrix of $T$ with entries $t_{i,j}$ and  $T'_{(r)}$ be the submatrix of $T'$ with entries $t_{n+i,n+j}$ where again $1\leq i,j\leq r$.   Set 
\begin{itemize}
\item $\hat{H}_{r}= (\det_q T_{(r )})^2$ in Type AI
\item $\hat{H}_{r }= \det_q T_{({2r})}$ in Type  AII and 
\item $\hat{H}_{r }= (\det_qT_{(r)})(\det_qT'_{(r)})$ in the diagonal case. 
\end{itemize}
Note that $\hat{H}_{r} \in \mathscr{P}(\mathfrak{g}_r)$. Moreover, by Lemma \ref{lemma:intersection}, we have $\hat{H}_{r }\in \mathscr{P}_{\theta}(\mathfrak{g}_r)$. Set $H_{r} = \psi_{r,n}(\hat{H}_{r})$ for 
 $r=1,\dots, n$.  

For each $r$, $\mathscr{P}_{\theta}(\mathfrak{g}_r)$ has a natural degree function compatible with the degree function (related to the filtration $\mathcal{J}$)  on $\mathscr{P}_{\theta}$.  In particular, we have  $\deg x(r)_{ij} = 1$ for all $r,i,j$.

\begin{proposition} \label{prop:weights} The elements $H_1, \dots, H_{n}$ generate a commutative subring of $\mathscr{P}_{\theta}$ that is isomorphic to a polynomial ring in these  variables.  Moreover, $H_{r}$ is a highest weight vector of weight $2\hat{\eta}_{r}$ with respect to the left action of $U_q(\mathfrak{g})$ and $H_{r}$ is a homogeneous element of degree $r$ in $ \mathcal{J}_{r}(\mathscr{P}_{\theta})$ for $r=1,\dots, n$. 
\end{proposition} 
\begin{proof} Note that $\det_q(T_{(r)})$ is a central element of $\mathcal{O}_q(\Mat_{r})$ in Type AI, $\det_qT_{(2r)}$ is a central element in $\mathcal{O}_q(\Mat_{2r})$ in Type AII and $(\det_qT_{(r)})(\det_q(T'_{(r)})$ is a central element of $\mathcal{O}_q(\Mat_{r})\otimes 
\mathcal{O}_q(\Mat_{r})$ in the diagonal case.  Hence $\hat{H}_{r}$ is in the center of $\mathscr{P}(\mathfrak{g}_r)$ for $r=1,\dots,n$.  
Also $\hat{H}_{r}$ commutes with each $\psi_{m,r}(\hat{H}_m)$ for $m\leq r$.  By induction, we see that $\psi_{2,r}(\hat{H}_{1}), \dots, \psi_{r-1,r}(\hat{H}_{r-1}), \hat{H}_{r}$ generates a commutative subring of 
$\mathscr{P}_{\theta}(\mathfrak{g}_r)$.  When $r=n$, this sequence is simply $H_1, \dots, H_{n}$ and so these elements generate a commutative subring of $\mathscr{P}_{\theta}(\mathfrak{g}_n)$.


The fact that $H_{r}$ generates a one-dimensional  $U_q(\mathfrak{g}_r)$-module follows from the definition of $\hat{H}_{r}$ and properties of quantum determinants (see Section \ref{section:q-det}).
Moreover, this module is invariant with respect to the action of $U_q(\mathfrak{g}_r)\cap U_q(\mathfrak{sl}_{n})$  in Type AI, $U_q(\mathfrak{g}_r)\cap U_q(\mathfrak{sl}_{2n)})$ in Type AII.  Similarly, it  is invariant with respect to the action of $U_q(\mathfrak{g}_r)\cap U_q(\mathfrak{sl}_{n}\oplus  \mathfrak{sl}_{n})$ in the diagonal case.   In Type AI, $\hat{H}_{r}$ only contains terms $t_{kl}$ with $1\leq k,l\leq r$. Hence $H_{r}$ only contains terms $x_{ij}$ for $1\leq i,j\leq r$ in Type AI.  For Type AII, $H_r$ only contains terms $t_{kl}$ with $1\leq k,l\leq 2r$.  Hence $H_{r}$ contains terms $x_{ij}$ for $1\leq i,j\leq 2r$ in Type AII. In the diagonal case, $H_{r}$  contains terms $t_{i,j+n}$ and $t_{i+n,j}$ with $1\leq i,j\leq r$. Thus the $x_{i,j+n}$ and $x_{i+n,j}$ satisfy the same constraints. 

The formulas for the left action of $U_q(\mathfrak{gl}_N)$ on $\mathcal{O}_q(\Mat_N)$ (in Section \ref{section:qmf})
 ensure that $E_s\cdot t_{ij}=0$ and  for all $s$  with $s\geq i$. Hence, $E_s\cdot a = 0$ for all  $a\in \mathscr{P}(\mathfrak{g}_r)$ where 
 $s\geq r$ in Type AI and $s\geq 2r$ in Type AII.  In the diagonal case, we have $E_s\cdot x_{i,j+n}=0$ and $E_s\cdot x_{i+n,j} =0$ when $n\geq s\geq r$ and $s\geq i$ and 
when $2n\geq s\geq n+r$ and $s\geq i+n$. This guarantees that $E_s\cdot H_{r}=0$ for all $E_s\notin U_q(\mathfrak{g}_r)$.  But we also know that $E_s\cdot \hat{H}_{r}=0$ for all $E_s\in U_q(\mathfrak{g}_r)$ because of  the left  invariant  property of $H_{r}$ with respect to the action of the  subalgebra of $U_q(\mathfrak{g}_r)$ described above.  This proves that  $E_s\cdot H_{r}=0$  for all 
$E_s\in U_q(\mathfrak{g})$.

Consider Type AI.  Note that $\det_q(T_{(r)})^2$ has weight $2\epsilon_1+\cdots + 2\epsilon_r$ in terms of the left action of $U_q(\mathfrak{g}_r)$. Thus it is straightforward to see from the definitions of $\hat{H}_{r}$ and of the  restricted weight $2\hat{\eta}_{r}$ (see Section \ref{section:restricted-root-system}) that $\hat{H}_{r}$ has weight $2\hat{\eta}_{r}$ for each $r$.  The weight of $H_{r}$ is the same as that of $\hat{H}_{r}$, hence by the previous paragraph, $H_{r}$ is a highest weight vector of weight $2\hat{\eta}_{r}$ with respect to the left action of $U_q(\mathfrak{g})$. Now consider Type AII. In this case, $\det_q(T_{(2r)}$ has weight $\epsilon_1+\cdots + \epsilon_{2r}$ in terms of the left action of $U_q(\mathfrak{g}_r)$.  Again,  as explained in Section \ref{section:restricted-root-system}, this weight equals $2\hat{\eta}_r$.  In the   diagonal case,   $\det_q(T_{(r)}T'_{(r)})$ has weight $\epsilon_1+\cdots + \epsilon_r +\epsilon_{n+1}+\cdots +\epsilon_{n+r}$.  Using the information in Section \ref{section:restricted-root-system}, this weight  is $2\hat{\eta}_r$.

We finish the proof by arguing that the $H_1, \dots, H_{n}$ are algebraically independent and hence the ring they generate is a polynomial ring in these variables.  Suppose 
$$\sum_ma_mH_1^{m_1}\cdots H_{n}^{m_{n}} = 0$$ where $m=(m_1, \dots, m_{n})$.   Since the monomials $H_1^{m_1}\cdots H_{n}^{m_{n}}$ have distinct weights $\sum_i2m_i\hat{\eta}_i$, we can separate the monomials using the action of $U_q(\mathfrak{g})$.   Thus the above equality implies $$a_mH_1^{m_1}\cdots H_{n}^{m_{n}} =0$$ and hence $a_m=0$ each $m$.
  \end{proof}
  
We frequently write $H_{2\mu}$ for the element $H_1^{m_1}\cdots H_{n}^{m_{n}}$ for each $\mu=\sum_im_i\hat{\eta}_i$, thus labeling this element by its weight.  In particular, by the above proposition, $H_{2\mu}$ is a highest weight vector of weight $2\mu$ with respect to the action of $U_q(\mathfrak{g})$ on 
  $\mathscr{P}_{\theta}$. 

\subsection{Explicit module descriptions}\label{subsection:explicit-module}
Recall that $\mathscr{P}_{\theta}$ is a subalgebra and submodule of $ \mathscr{P}$.  The decompositions of Section \ref{section:modules}  ensure that
 as left $U_q(\mathfrak{g})$-modules, we have  the following inclusions \begin{align}\label{Ptheta2}
 \mathscr{P}_{\theta}\subseteq 
 \left\{\begin{matrix} 
&\bigoplus_{\lambda\in \Lambda^+_{\Sigma}{s\in \{0,1\}}}L(2\lambda+s\hat{\eta}_n) {\rm \  in\  Type\  AI,}\cr
&\bigoplus_{\lambda\in \Lambda^+_{\Sigma}}L(2\lambda){\rm \  in\ Type\ AII, \ and}\cr
&\bigoplus_{\lambda\in \Lambda^+_{\Sigma}}L(2\lambda)=\bigoplus_{\gamma\in \Lambda^+_n} L(\gamma\oplus\gamma){\rm \  in\  the\ diagonal\ type.\ }
 \end{matrix}\right.
 \end{align} In the next theorem, we obtain a precise decomposition of $\mathscr{P}_{\theta}$ into left $U_q(\mathfrak{g})$-modules and trivial right $\mathcal{B}_{\theta}$-modules.
 
\begin{theorem} \label{theorem:explicit-module}
We have 
\begin{align}\label{Ptheta}\mathscr{P}_{\theta}= \bigoplus_{\lambda\in \Lambda^+_{\Sigma}}(U_q(\mathfrak{g}))\cdot H_{2\lambda}
\cong \bigoplus_{\lambda\in \Lambda^+_{\Sigma}}L({2\lambda})
\end{align}
where $(U_q(\mathfrak{g}))\cdot H_{2\lambda}$ is isomorphic to the simple (left) $U_q(\mathfrak{g})$-module  generated by the highest weight vector $ H_{2\lambda}$ with weight
$2\lambda$ and is a trivial right $\mathcal{B}_{\theta}$-module.
 \end{theorem}
 \begin{proof} By Proposition \ref{prop:weights}, $H_{2\lambda}$ generates a finite-dimensional simple (left) $U_q(\mathfrak{g})$-module with highest weight $2\lambda$ for each $\lambda\in \Lambda^+_{\Sigma}$.  Thus $U_q(\mathfrak{g})\cdot H_{2\lambda}\cong L(2\lambda)$ for each $\lambda \in \Lambda_{\Sigma}$ which proves the second part of (\ref{Ptheta}).  (In the discussion below, we identify $U_q(\mathfrak{g})\cdot H_{2\lambda}\cong L(2\lambda)$, which means that we view this second part of (\ref{Ptheta}) as an equality.)
Note that
 \begin{align*}
 \bigoplus_{\lambda\in \Lambda^+_{\Sigma}}(U_q(\mathfrak{g}))\cdot H_{2\lambda}\subseteq \mathscr{P}_{\theta}.
 \end{align*}
 Hence by (\ref{Ptheta2}), we get equality here in Type AII and the diagonal type.  
 
 Now consider Type AI and a module of the form $L(2\lambda+s\hat{\eta}_n)$ with $s\neq 0$.  Note that 
 \begin{align*}L(2\lambda +s\hat{\eta}_n)
 \cong L(2\lambda)L(s\hat{\eta}_n)
\end{align*} 
 where $L(s\hat{\eta}_n)$ is a trivial one-dimensional  $U_q(\mathfrak{sl}_n)$-module inside of 
 $\mathscr{P}_{\theta}$ of weight $s\hat{\eta}_n$ with respect to the left action of $U_q(\mathfrak{g})$.  By Lemma  \ref{lemma:intersection}, $L({s\hat{\eta}_n})$ must be a  multiple of $(\det_q(T))^2$ and so $s$ is an even integer.  Thus by (\ref{Ptheta}),  \begin{align*}
\mathscr{P}_{\theta}\subseteq \bigoplus_{\lambda\in \Lambda^+_{\Sigma}}L(2\lambda) \end{align*} in Type AI 
 and the theorem follows.
 \end{proof}

 The entire construction of this section can be transferred to the differential parts $\mathscr{D}_{\theta}$ by using the antiautomorphisms sending $\mathscr{P}$ to $\mathscr{D}=\mathscr{P}^{op}$. Let $H^*_{2\mu}$ be the image of $H_{2\mu}$  via this antiautomorphism. A comparison of the action of $U_q(\mathfrak{g})$ on $\mathscr{P}$ and 
 on $\mathscr{D}$ yields that $H^*_{2\lambda}$ is a lowest weight vector of weight $-2\lambda$.  Hence, we have a similar decomposition as above for $\mathscr{D}_{\theta}$, again as left  $U_q(\mathfrak{g})$-modules and trivial right $\mathcal{B}_{\theta}$-modules
 \begin{align*} 
\mathscr{D}_{\theta} \cong \oplus_{\lambda\in \Lambda^+_{\Sigma}}(U_q(\mathfrak{g}))\cdot H^*_{2\lambda}.
\end{align*}
Note that $(U_q(\mathfrak{g}))\cdot H^*_{2\lambda}$ can be viewed as the left dual of $(U_q(\mathfrak{g}))\cdot H_{2\lambda}$.

\begin{remark}The generators of each module $L(2\lambda)$ are expressed using formulas in terms of the $t_{ij}$  for Types AI and AII in \cite{N} (see Lemma 4.10.A),  Our approach yields another concrete identification of these generators  that 
also applies to the diagonal case.  Our methods, which rely directly on quantum determinants also lead to formulas in the $t_{ij}$.  However, because these highest weight terms are elements of $\mathscr{P}_{\theta}(\mathfrak{g}_r)$ for various choices of $r$, it is easier to read off the possible $x_{ij}$ that may appear.  This will be helpful in describing and analyzing the quantum Capelli operators in Section \ref{section:defn-and-desc} of this paper.
\end{remark}

\section{Quantum Weyl algebras}
\subsection{Generators and relations} \label{section:gandr} We  associate a  quantum Weyl algebra $\mathscr{PD}_q(\Mat_N)$ with polynomials corresponding to $\mathcal{O}_q(\Mat_N)$ and constant term differentials corresponding to $\mathcal{O}_q(\Mat_N)^{op}$ as defined and studied in \cite{VSS}, \cite{B}, \cite{B2}, and \cite{LSS}.  This Weyl algebra $\mathscr{PD}_q(\Mat_N)$ is  generated by  $t_{ij}$ and $\partial_{ij}$ for $1\leq i,j\leq N$.  The algebra   $\mathcal{O}_q(\Mat_N)$ (resp. $\mathcal{O}_q(\Mat_N)^{op}$)  embeds inside $\mathscr{PD}_q(\Mat_N)$ and can be identified with the subalgebra generated by the $t_{ij}$ (resp. $\partial_{ij}$).  Moreover,    the $t_{ij}$ and $\partial_{ij}$ satisfy the following relation 
 \begin{align*} \partial_{ab}t_{ef} = \sum_{r,l,j,k} (R^{t_2})^{rl}_{ea}(R^{t_2})^{jk}_{fb}t_{rj}\partial_{lk}+\delta_{ae}\delta_{bf}
   \end{align*}
   for all $a,b,e,f$ in $\{1, \dots, N\}$. 
  The quantum Weyl algebra $\mathscr{PD}_q(\Mat_N)$ inherits the structure of a $U_q(\mathfrak{gl}_N)$-bimodule from the bimodules $\mathcal{O}_q(\Mat_N)$ and $\mathcal{O}_q(\Mat_N)^{op}.$ 
  
  We may view 
 $\mathscr{P}_{\theta}$ and $\mathscr{D}_{\theta}$ as right $\mathscr{B}_{\theta}$ invariant subalgebras of $\mathscr{PD}_q(\Mat_N)$.  However, together, they generate an algebra inside of $\mathscr{PD}_q(\Mat_N)$ that is too large to be taken for a quantum analog of the Weyl algebra with polynomial part equal to $\mathscr{P}_{\theta}$. In particular, the subset $\mathscr{P}_{\theta}\mathscr{D}_{\theta}$ consisting of sums of terms of the form $pd$ with $p\in \mathscr{P}_{\theta}$ and $d\in \mathscr{D}_{\theta}$ is strictly smaller than the subalgebra  generated by $\mathscr{P}_{\theta}$ and $\mathscr{D}_{\theta}$ (see \cite{LSS} for more details).  Instead, we use the construction of \cite{LSS} which starts with a twisted tensor product of $\mathscr{P}_{\theta}$ and $\mathscr{D}_{\theta}$ and deforms it so as to add constant terms to some of the relations. The result is the quantum Weyl algebra $\mathscr{PD}_{\theta}$ associated to $\theta$ for each of the three settings of this paper. As an algebra, 
 $\mathscr{PD}_{\theta}$ is   generated by $x_{ij}$, $d_{ij}$, $1\leq i,j\leq N$ where $N=n$ in Type AI and $N=2n$ in Type AII and the diagonal case. The algebra   $\mathscr{P}_{\theta}$ (resp. $\mathscr{D}_{\theta}$)  embeds inside $\mathscr{PD}_{\theta}$ and can be identified with the subalgebra generated by the $x_{ij}$ (resp. $d_{ij}$).  Moreover, the $x_{ij}$ and $d_{ij}$ satisfy the following relation
     \begin{align}\label{twistingreln} d_{ab}x_{ef} =\sum_{w,r,x,q,p,m,y,l}(R_{\mathfrak{g}}^{t_2})^{wr}_{xq} (R_{\mathfrak{g}}^{t_2})^{pq}_{ma}(R_{\mathfrak{g}}^{t_2})^{xy}_{fl}(R_{\mathfrak{g}}^{t_2})^{ml}_{eb}x_{pw}d_{ry} + q^{-\delta_{ef}}\delta_{ae}\delta_{bf}
  \end{align}
  for all $a,b,e,f\in {1,\dots, N}$.
  The map $\mathscr{P}_{\theta}\otimes \mathscr{D}_{\theta}$ to $\mathscr{PD}_{\theta}$ defined by multiplication is a vector space isomorphism of left $U_q(\mathfrak{g})$-modules  and (trivial) right $\mathcal{B}_{\theta}$-modules.
Relations (\ref{twistingreln}) in the diagonal case are equivalent to the simpler relations
  \begin{align}\label{reln:diag} d_{a,b+n}x_{e,f+n} = \sum_{r,l,j,k} (R^{t_2})^{rl}_{ea}(R^{t_2})^{jk}_{fb}x_{r,j+n}d_{l,k+n}+\delta_{ae}\delta_{bf}
   \end{align}
   for all $1\leq a,b,e,f\leq n$ where here we are taking into account linear relations satisfied by the $d_{ij}$ and by the $x_{ij}$. In particular, this relation combined with results from Section \ref{section:qhs} ensure that  in the diagonal case $\mathscr{PD}_{\theta}$ and $\mathscr{PD}_q(\Mat_n)$ are isomorphic as algebras via the map sending $x_{i,j+n}$ to $t_{ij}$ and $d_{i,j+n}$ to $\partial_{ij}$ for each $1\leq i,j\leq n$.  (For additional details, see \cite{LSS}.)

  The following result from \cite{LSS} gives insight into the overall form of the relations coming from  the twisting map.
  
\begin{theorem}\label{theorem:relnsWeyl}
(\cite{LSS}, Corollary 8.11) For each of the three families, the following inclusions hold for  the quantum Weyl algebra $\mathscr{PD}_{\theta}$
\begin{align*}
d_{ab}x_{ef} - q^{\delta_{af}+\delta_{ae}+\delta_{bf} +\delta_{be}} x_{ef}d_{ab}-q^{-\delta_{ef}}\delta_{ae}\delta_{bf} \in  \sum_{(e',f',a',b')>(e,f,a,b)} \mathbb{C}(q)x_{e'f'}d_{a'b'}
\end{align*}
for all $a,b,e,f\in \{1, \dots, {\rm rank}(\mathfrak{g})\}$ where
\begin{itemize}
\item $a\leq b$ and $e\leq f$ in Type AI
\item $a<b$ and $e<f$ in Type AII
\item $a\leq n<b$ and $e\leq n<f$ in  diagonal type
\end{itemize}
and $(e',f',a',b')>(e,f,a,b)$ if and only if $e'\geq e,f'\geq f, a'\geq a, b'\geq b$ and at least one of these inequalities is strict.  
\end{theorem}

It is also helpful for arguments later in the paper to express these relations in special cases.  We do this in the next lemma.

\begin{lemma} \label{lemma:formulas}
In Type AI and for $a<n$,
we have
\begin{itemize}
\item[(i)] $d_{an}x_{en} =q^{1+\delta_{ae}}x_{en}d_{an} -\delta_{ae}\sum_{a'>a}q^{2+\delta_{a'n}}(q^{-2}-1)x_{a'n}d_{a'n} + \delta_{ae}$ where $e<n$.
\item[(ii)] $d_{nn}x_{ef} =q^{\delta_{nf}+2\delta_{ne}}x_{ef}d_{nn} + q^{-\delta_{ef}}\delta_{ne}\delta_{nf}$ where $e\leq f$.
\end{itemize}
In Type AII and for $a<2n,e<2n$, we have 
\begin{itemize}
\item[(iii)] $d_{a,2n}x_{e,2n} =q^{1+\delta_{ae}}x_{e,2n}d_{a,2n} -\delta_{ae}\sum_{a'>a}q^{2+\delta_{a',2n}}(q^{-2}-1)x_{a',2n}d_{a',2n} + \delta_{ae}$.
\end{itemize}
In the diagonal case, we have the following relations as given in \cite{LSS0}, Remark 3.7.4 with the adjustment $\partial_{a,b}\mapsto d_{a,b+n}$ and
$t_{a,b}\mapsto x_{a,b+n}$.  Note that the possible subscripts of both $x$ and $d$ terms are $a,b+n$  with $a=1,\dots, n$ and $b=1,\dots, n$. \begin{itemize}
\item[(iv)] $d_{c,b+n}x_{e,a+n} =x_{e,a+n} d_{c,b+n}$ if $b\neq a $ and $c\neq e$.  
\item[(v)] $d_{c,b+n}x_{c,a+n} =qx_{c,a+n} d_{c,b+n}+ \sum_{c'>c} (q-q^{-1})x_{c',a+n} d_{c',b+n}$ if $b\neq a $.
\item[(vi)] $d_{c,a+n}x_{e,a+n} = qx_{e,a+n} d_{c,a+n}+\sum_{a'>a}(q-q^{-1})x_{e,a'+n} d_{c,a'+n}$ if  $c\neq e$.
\item[(vii)]$d_{c,a+n}x_{c,a+n} = q^2x_{c,a+n} d_{c,a+n} +q \sum_{c'>c}(q-q^{-1})d_{c',a+n} d_{c',a+n}$

 $+q\sum_{a'>a}(q-q^{-1}) x_{c,a'+n} d_{c,a'+n}+1$ 
\end{itemize} 
 
\end{lemma}
\begin{proof} Consider Type AI.  Using the explicit formulas for the entries of $R$ (see Section \ref{section:qmf}), $(R^{t_2})^{ul}_{vn} = r^{un}_{vl} \neq 0$ implies that $n=l$ and $u=v$.  Similarly, $(R^{t_2})^{ul}_{nv}=r^{uv}_{nl}\neq 0$ implies that $n=u$ and $v=l$.
Moreover, if $v<n$ then $(R^{t_2})^{vn}_{vn}=r^{vn}_{vn} =1$ and $(R^{t_2})^{nv}_{nv} =r^{nv}_{nv}=1$ while if $v=n$, then $(R^{t_2})^{nn}_{nn}= q$.  

Suppose that $a$ and $e$ are both strictly less than $n$.  By the above information about the entries for $R$, we get 
\begin{align}\label{Kreln3}
d_{an}x_{en}&=(R^{t_2})^{na}_{na}(R^{t_2})^{ea}_{ea}(R^{t_2})^{nn}_{nn}(R^{t_2})^{en}_{en}x_{en}d_{an}\cr&+ \delta_{ae}\sum_{a'>a}(R^{t_2})^{na'}_{na'}(R^{t_2})^{a'a'}_{aa}(R^{t_2})^{nn}_{nn}(R^{t_2})^{en}_{en}x_{a'n}d_{a'n} +\delta_{ae}
\cr&=q^{1+\delta_{ae}}x_{en}d_{an} - \delta_{ae}\sum_{a'>a}q^{2+\delta_{a'n}}(q^{-2}-1)x_{a'n}d_{a'n} + \delta_{ae}.
\end{align} 
This proves (i).

Using (\ref{twistingreln}) and the above information about the entries for $R$, we see that  
\begin{align}\label{Kreln1}
d_{nn}x_{ef} &= (R^{t_2})^{fn}_{fn}(R^{t_2})^{en}_{en}(R^{t_2})^{fn}_{fn}(R^{t_2})^{en}_{en} x_{ef}d_{nn} + q^{-\delta_{ef}}\delta_{ne}\delta_{nf}
\cr&= q^{2\delta_{nf}+2\delta_{ne}}x_{ef}d_{nn} + q^{-\delta_{ef}}\delta_{ne}\delta_{nf}
\end{align}
for $e\leq f$.
This proves (ii).

The argument for (iii) is the same as for (i) with $n$ replaced by $2n$ everywhere.  As stated in the lemma, (iv)-(vii) are directly from \cite{LSS0}.
\end{proof}

We can define a  filtration on  $\mathscr{PD}_{\theta}$  that is compatible with the filtration $\mathcal{J}$ induced by the degree functions on $\mathscr{P}_{\theta}$ and $\mathscr{D}_{\theta}$ 
(see Section \ref{section:qhs}).  We use the same notation, namely $\mathcal{J}$, to denote this filtration on $\mathscr{PD}_{\theta}$. 
Note that multiplication induces a vector space isomorphism from $\mathscr{P}_{\theta}\otimes \mathscr{D}_{\theta}$ to the twisted tensor product of $\mathscr{P}_{\theta}$ and $\mathscr{D}_{\theta}$.  Since $\mathscr{PD}_{\theta}$  is a PBW deformation of this twisted tensor product (or one can check directly from the relations above) $\mathscr{PD}_{\theta}$ inherits a filtration from $\mathcal{J}$ on $\mathscr{P}_{\theta}$ and $\mathscr{D}_{\theta}$ via $${\mathcal{J}}_r(\mathscr{PD}_{\theta})=\sum_{u+v= r}\mathcal{J}_u(\mathscr{P}_{\theta})\mathcal{J}_v(\mathscr{D}_{\theta}).$$ It follows that 
$$\mathcal{J}_u(\mathscr{PD}_{\theta})\mathcal{J}_v(\mathscr{PD}_{\theta})\subseteq {\mathcal{J}}_{u+v}(\mathscr{PD}_{\theta})$$ for all nonnegative integers $u$ and $v$.
 Note that the filtration $\mathcal{J}$ is preserved by the action of $U_q(\mathfrak{g})$.

\subsection{Action on polynomials}\label{section:action}
Let $\mathcal{L}$ be the left ideal of 
$\mathscr{PD}_{\theta}$ generated by the elements $d_{ij}$ for $1\leq i\leq j\leq N$ where $N=n$ in Type AI and $N=2n$ in the other two cases.  Note that $\mathscr{PD}_{\theta}$ admits a direct sum decomposition
$\mathscr{PD}_{\theta} = \mathscr{P}_{\theta} \oplus \mathcal{L}.$
Let $\pi:\mathscr{PD}_{\theta}\rightarrow \mathscr{P}_{\theta}$ be the projection with kernel
$\mathcal{L}$ 
and note that the map $\pi$ is a $U_q(\mathfrak{g})$-module map.

Recall (Section \ref{section:qhs})  that  $\mathscr{P}^r_{\theta}$ equals the homogeneous space of degree $r$ with respect to the degree filtration $\mathcal{J}$.  Similarly, $\mathscr{D}^r_{\theta}$ equals the homogeneous space of degree $r$ with respect to the degree filtration $\mathcal{J}$.  Note that $\mathscr{P}^0_{\theta} = \mathbb{C}(q)$ and so
$\mathscr{P}_{\theta} = \mathbb{C}(q) \oplus \sum_{r>0}\mathscr{P}_{\theta}^r$.
 This decomposition  can be extended to $\mathscr{P}\mathscr{D}_{\theta}$ using the map $\pi$.   We have
\begin{align*}{\rm ker}\ \pi =\mathcal{L} = \sum_{r>0}\mathscr{P}_{\theta}\mathscr{D}_{\theta}^r
\end{align*} and so 
\begin{align*}
\mathscr{P}\mathscr{D}_{\theta} = \pi(\mathscr{P}\mathscr{D}_{\theta})\oplus \mathcal{L}=\mathscr{P}_{\theta}\oplus  \sum_{r>0}\mathscr{P}_{\theta}\mathscr{D}_{\theta}^r
=\mathbb{C}(q) \oplus \sum_{r>0}\mathscr{P}_{\theta}^r\oplus  \sum_{r>0}\mathscr{P}_{\theta}\mathscr{D}_{\theta}^r
\end{align*}  Write $(b)_0$ for the projection of an element in $\mathscr{P}\mathscr{D}_{\theta}$ onto $\mathbb{C}(q)$ using this direct sum decomposition.  It follows  that $(b)_0=0$ for all $b\in \sum_{r>0}\mathscr{P}_{\theta}^r\oplus   \mathcal{L}$.

Define a bilinear form $\langle \cdot, \cdot \rangle$ from $\mathscr{D}_{\theta}\times \mathscr{P}_{\theta}$ to $\mathbb{C}(q)$ by 
\begin{align}\label{inner}\langle d, p\rangle = \pi(dp)_0 
\end{align}
where $\pi(dp)_0=(\pi(dp))_0$.

\begin{lemma}\label{lemma:inv} The bilinear form $\langle \cdot, \cdot \rangle$ on $\mathscr{D}_{\theta}\times \mathscr{P}_{\theta}$ satisfies \begin{align*} \sum \langle u_{(1)}\cdot d, u_{(2)}\cdot p\rangle = \epsilon(u)\langle d,p\rangle
\end{align*}
for all $u\in U_q(\mathfrak{g})$, $d\in \mathscr{D}_{\theta}$ and $p\in \mathscr{P}_{\theta}$ and hence  is (left)  $U_q(\mathfrak{g})$ invariant.
\end{lemma}
\begin{proof}
Write $\pi(dp) = \langle d, p\rangle + a$ where $a\in\sum_{r>0}\mathscr{P}_{\theta}^r$.   Since $\mathscr{PD}_{\theta}$ is a left $U_q(\mathfrak{g})$-module, 
and $\pi$ is a $U_q(\mathfrak{g})$-module map, we have \begin{align*} 
u\cdot \pi(dp)=\pi(u\cdot dp) 
 = \sum \pi((u_{(1)}\cdot d)(u_{(2)}\cdot p) ).
\end{align*}
Since $(b)_0=0$ for all $b\in {\rm ker}\ \pi = \mathcal{L}$, we have 
\begin{align} \label{eqnbilinear1} (\sum \pi((u_{(1)}\cdot d)(u_{(2)}\cdot p) )_0=  (\sum (u_{(1)}\cdot d)(u_{(2)}\cdot p ))_0
\end{align}
On the other hand 
\begin{align*} u\cdot \pi(dp) = u\cdot  (\langle d,p\rangle +  a) =  u\cdot (\langle d,p\rangle) + u\cdot a = \epsilon(u)\langle d, p \rangle + u\cdot a
\end{align*}
since $\langle d,p\rangle$ is a scalar.  
Since the action of $U_q(\mathfrak{g})$ on $\mathscr{P}_{\theta}$ preserves degree, $\sum_{r>0}\mathscr{P}^r_{\theta}$ is a $U_q(\mathfrak{g})$-module.  Thus $u\cdot a \in   \sum_{r>0}\mathscr{P}^r_{\theta}$ and so 
\begin{align}\label{eqnbilinear2}(\epsilon(u)\langle d, p \rangle + u\cdot a)_0=\epsilon(u)\langle d, p \rangle.
\end{align}
Putting together (\ref{eqnbilinear1}) and (\ref{eqnbilinear2})    yields 
the desired property.   Thus $\langle \cdot, \cdot \rangle$ is (left)  $U_q(\mathfrak{g})$ invariant.
\end{proof}

Note that the map $\pi$ defines an action of $\mathscr{PD}_{\theta}$ on $\mathscr{P}_{\theta}$.  In particular,   the action of the element $a\in \mathscr{PD}_{\theta}$ on  $x\in\mathscr{P}_{\theta}$ yields the element $\pi(ax)$ in $\mathscr{P}_{\theta}$.
This action  of $\mathscr{PD}_{\theta}$ on $\mathscr{P}_{\theta}$ can be viewed as a map of algebras, say $\phi$, from 
$\mathscr{PD}_{\theta}$ into ${\rm End}\ \mathscr{P}_{\theta}$.  (Here we write ${\rm End}\ \mathscr{P}_{\theta}$ for  ${\rm End}_{\mathbb{C}(q)}\mathscr{P}_{\theta}$ which are endomorphisms over the scalars.  The field $\mathbb{C}(q)$ is dropped since it can be understood from context.) Given $a\in \mathscr{PD}_{\theta}$ we frequently write $\phi_a$ for $\phi(a)$ in order to make the exposition below clearer. 
Since $\mathscr{P}_{\theta}$ is a left $U_q(\mathfrak{g})$-module,  ${\rm End}\ \mathscr{P}_{\theta}$ inherits the structure of  a $U_q(\mathfrak{g})$-bimodule in the standard way.  Thus ${\rm End}\ \mathscr{P}_{\theta}$  is an   $({\rm ad\ } U_q(\mathfrak{g}))$-module via 
\begin{align*} (({\rm ad}\ u)\cdot b)(p)  =\left( \sum u_{(1)} b S(u_{(2)})\right) (p) = \sum u_{(1)} b \left(S(u_{(2)})(p)\right)
\end{align*}
for all $u\in U_q(\mathfrak{g})$, $b\in {\rm End}\ \mathscr{P}_{\theta}$ and $p\in \mathscr{P}_{\theta}$.

\begin{proposition}\label{prop:locally-finite} The map $\phi$ is a $U_q(\mathfrak{g})$-module map with respect to the left action on $\mathscr{PD}_{\theta}$ and the left adjoint action on ${\rm End}\ \mathscr{P}_{\theta}$ and so  
\begin{align}\label{phi-reln} \phi_{u\cdot a}= ({\rm ad}\ u)\cdot\phi_a
\end{align}
for all  $a\in  \mathscr{PD}_{\theta}$ and $u\in U_q(\mathfrak{g})$.   
Thus the image of $\mathscr{PD}_{\theta}$ under $\phi$ is an $({\rm ad}\ U_q(\mathfrak{g}))$-submodule of ${\rm End}\ \mathscr{P}_{\theta}$.
\end{proposition}

\begin{proof} From the discussion preceding the proposition,  we have 
\begin{align*}\phi_a(x) = \pi(ax).  
\end{align*} Hence 
\begin{align*}
\phi_{u\cdot a}(x) &= \pi((u\cdot a)x) = \sum \pi((u_{1}\cdot a))(u_{(2)}S(u_{(3)})\cdot x)
\cr &=\sum u_{(1)}\cdot \pi(a (S(u_{(2)})\cdot x))=
(({\rm ad\ }u)\phi_a)(x)
\end{align*}
\end{proof}

\subsection{Orthogonality conditions}

Section \ref{section:gandr} asserts that, as an algebra,  $\mathscr{PD}_{\theta}$ is isomorphic to $\mathscr{PD}_q(\Mat_n)$  in the diagonal case.   Thus  the second assertion of the next result is a generalization of \cite{B},  Proposition 1 to include the other two families.

\begin{proposition}\label{prop:orthogonal} For each $r$ and $s$ with $r\neq s$, $\mathscr{D}^r_{\theta}$ is equal to the vector space dual of $\mathscr{P}^r_{\theta}$ and $\mathscr{D}^r_{\theta}$ is orthogonal to $\mathscr{P}_{\theta}^s$ with respect to the bilinear form defined by (\ref{inner}).  Moreover,    $\mathscr{P}_{\theta}$  is a faithful $\mathscr{PD}_{\theta}$-module with respect to the action defined above. 
\end{proposition}

\begin{proof}  
Note that  the relation between the bilinear form  and the action ensures that the ``moreover" part of the proposition is an immediate consequence of the main assertion.  Hence, we focus on proving the duality result. 

By Theorem \ref{theorem:relnsWeyl}  \begin{align*} 
\mathscr{D}^{r+s}_{\theta}\mathscr{P}^s_{\theta} \subseteq 
\sum_{s\geq k\geq 0}\mathscr{P}^k_{\theta}\mathscr{D}^{r+k}_{\theta}{\rm \ and \ }\mathscr{D}^r_{\theta}\mathscr{P}^{r+s}_{\theta}\subseteq 
 \sum_{s\geq k\geq 0}\mathscr{P}^{r+k}_{\theta}\mathscr{D}^k_{\theta}
 \end{align*} 
for all $r\geq 0$ and $s\geq 0$.  It follows that   $(d\cdot p)_0 = 0$ for $d\in \mathscr{D}^{r}_{\theta}, p\in \mathscr{P}^{s}_{\theta} $ with $r\neq s$.

By Lemma \ref{lemma:PBW}, the set of monomials of the form 
\begin{align*}x_{e_1,f_1}^{m_1}x_{e_2,f_2}^{m_2}\cdots x_{e_r,f_r}^{m_r}
\end{align*} form a PBW basis for $\mathscr{P}_{\theta}^m$ where $m=m_1+\cdots + m_r$ and  $(e_1,f_1)>(e_2,f_2)>\dots >(e_r,f_r)$  where here $``>"$ is the standard lexicographic ordering (from left to right)..  As explained preceding Lemma \ref{lemma:PBW}, the analogous result holds for $\mathscr{D}_{\theta}$ with 
each $x_{ij}$ replaced by $d_{ij}$. Moreover, by Lemma \ref{lemma:PBW}, we can switch the ordering of the subscripts and still get a PBW basis.  We do this below for $\mathscr{D}_{\theta}$.
Consider a sequence of ordered pairs $(e_1,f_1),\dots, (e_r,f_r)$ satisfying 
\begin{align*}
(a,b)\geq (e_1,f_1)>(e_2,f_2)>\dots >(e_r,f_r).
\end{align*} It follows from Theorem \ref{theorem:relnsWeyl} that
\begin{align*}
d_{ab}x_{e_1,f_1}^{m_1}x_{e_2,f_2}^{m_2}\cdots x_{e_r,f_r}^{m_r}\in c\ \delta_{a,e_1}\delta_{b,f_1}x_{e_1.f_1}^{m_1-1}x_{e_2,f_2}^{m_2}\cdots x_{e_r,f_r}^{m_r}+ 
\mathscr{P}_{\theta}^{m-1}+\sum_{(a',b')\geq (a,b)}  \mathscr{P}_{\theta}^{m}d_{a'b'}
\end{align*}
where $m=m_1+\cdots +m_r$ and $c$ is a nonzero scalar.
Hence given $h\in \mathscr{D}_{\theta}^{m-1}$, if
\begin{align*}
(hd_{ab}\cdot x_{e_1,f_1}^{m_1}x_{e_2,f_2}^{m_2}\cdots x_{e_r,f_r}^{m_r})_0 \neq 0
\end{align*}
then $(a,b)= (e_1,f_1)$.  Using induction, we obtain
\begin{align*}(d_{a_k,b_k}^{s_k}d_{a_{k-1},b_{k-1}}^{s_{k-1}}\cdots d_{a_1,b_1}^{s_1}\cdot x_{e_1,f_1}^{m_1}x_{e_2,f_2}^{m_2}\cdots x_{e_r,f_r}^{m_r})_0 \neq 0
\end{align*}
with  \begin{align*}(a_k,b_k)<(a_{k-1}, b_{k-1})<\cdots < (a_1,b_1)\end{align*}
if and only if $r=k, s_i=m_i, e_i=a_i, $ and $f_i = b_i$ for $i=1, \dots, r$.  This proves the desired duality result.
\end{proof}

\subsection{Action on highest weight terms}  
By Proposition \ref{prop:weights},  the highest weight vector $H_{2\mu}, \mu = \sum_im_i\hat{\eta}_i$ is homogeneous of degree 
$\sum_rm_rr$.  Note that this degree is just the size of $\mu$ viewed as a partition.  In other words, $\sum_rm_rr = \sum_r\mu_r=|\mu|$ where $\mu$ is expressed as $\sum_i\mu_i\epsilon^{\Sigma}_i$, a linear combination in terms of the orthonormal basis in the restricted root setting.  Moreover, the action of $U_q(\mathfrak{g})$ preserves the degree.  Thus for  each $\mu \in \Lambda_{\Sigma}^+$,  the  module $U_q(\mathfrak{g})\cdot H_{2\mu}$ sits inside the 
homogeneous component $\mathscr{P}_{\theta}^m$ of degree $m=|\mu|$. Recall the definition of the augmentation ideals $U^+_+$ and $U^-_+$ given in Section \ref{section:qea}.

\begin{proposition}\label{prop:mugamma} For all $\mu$ and $\gamma$ in $\Lambda_{\Sigma}^+$  with $\mu\neq \gamma$, the space  $U_q(\mathfrak{g})\cdot H_{2\mu}$ is equal to the  $U_q(\mathfrak{g})$-module dual of 
$ U_q(\mathfrak{g})\cdot H^*_{2\mu}$ and orthogonal to $U_q(\mathfrak{g})\cdot H^*_{2\gamma}$ with respect to the bilinear form defined by (\ref{inner}).  Moreover, 
$$\langle H_{2\mu}^*, H_{2\mu}\rangle \neq 0$$ whereas $$\langle E\cdot H_{2\mu}^*, H_{2\mu}\rangle =0=\langle H_{2\mu}^*, F\cdot H_{2\mu}\rangle$$ for 
all $E\in U^+_+$ and $F\in U^-_+$.
\end{proposition}
\begin{proof} Let $\mu\in \Lambda_{\Sigma}^+$ and set $m=|\mu|$.  Since the left action of $U_q(\mathfrak{g})$ preserves degree, both $\mathscr{P}_{\theta}^m$ and $\mathscr{D}_{\theta}^m$ are 
left $U_q(\mathfrak{g})$-modules.  By Lemma \ref{lemma:inv}, the dualities of vector spaces in Proposition \ref{prop:orthogonal} are actually dualities  of left $U_q(\mathfrak{g})$-modules. 

Note that  there is only one way to express a weight $\mu\in \Lambda^+_{\Sigma}$ as a linear combination of the 
$\hat{\eta}_r$. This means that $U_q(\mathfrak{g})\cdot H_{2\mu}$ is the unique simple module with highest  weight $2\mu$ inside the 
decomposition of $\mathscr{P}_{\theta}$, and thus inside of $\mathscr{P}^m_{\theta}$.  Similarly, $U_q(\mathfrak{g})\cdot H^*_{2\mu}$ is the unique simple module with lowest  weight $-2\mu$ inside the 
decomposition of $\mathscr{D}_{\theta}$, and thus inside of $\mathscr{D}^m_{\theta}$. 
Hence, by the previous paragraph, $U_q(\mathfrak{g})\cdot H_{2\mu}$ is equal to the  $U_q(\mathfrak{g})$-module dual of 
$ U_q(\mathfrak{g})\cdot H_{2\mu}^*$ with respect to bilinear form defined by (\ref{inner}).

Since $H_{2\mu}^*$ is a lowest weight generating vector for $U_q(\mathfrak{g})\cdot H_{2\mu}^*$, it follows that 
\begin{align*}
U_q(\mathfrak{g})\cdot H_{2\mu}^*=H_{2\mu}^* \oplus U^+_+\cdot H_{2\mu}^*.
\end{align*}
The fact that $ H_{2\mu}$ is a highest weight vector combined with the $U_q(\mathfrak{g})$ invariance of the bilinear form $\langle\cdot, \cdot \rangle$ (Lemma \ref{lemma:inv}) ensures that 
$ H_{2\mu}$ is perpendicular to  $U^+_+\cdot \mathscr{D}_{\theta}$.  The argument showing $H_{2\mu}^*$ is perpendicular to $U_+^- \cdot \mathscr{P}_{\theta}$ follows in a similar fashion. This completes the proof of the proposition.
\end{proof}

Note that by the above proposition,  the pairing $\langle H_{2\mu}^*, H_{2\mu}\rangle$ is nonzero.  It follows that  the projection, $\pi(H_{2\mu}^*H_{2\mu})$, which can be viewed as the action of $H_{2\mu}^*\in \mathscr{PD}_{\theta}$ on $H_{2\mu}\in \mathscr{P}_{\theta}$, is nonzero. In the next result, we obtain more information for when such a pairing  and related projections are possibly nonzero.

\begin{lemma} \label{lemma:mugamma} Given $\mu$ and $\gamma$ in $\Lambda_{\Sigma}^+$ such that $|\gamma|\geq |\mu|$ and $\mu\neq \gamma$, we have 
\begin{align*}\pi((U_q(\mathfrak{g})\cdot H_{2\gamma}^* )(U_q(\mathfrak{g})\cdot H_{2\mu}))=0.\end{align*}
\end{lemma}
\begin{proof} Note that for $\mu=\sum_im_i\hat{\eta}_i$, all elements of $U_q(\mathfrak{g})\cdot H_{2\mu}$ are homogeneous elements of 
$\mathscr{P}_{\theta}$ of degree $\sum_iim_i$.  The analogous assertion holds for $U_q(\mathfrak{g})\cdot H_{2\gamma}$ with $\mu$ replaced by $\gamma$.  We argue that 
\begin{align}\label{plain-eqn}
(U_q(\mathfrak{g})\cdot H_{2\gamma}^* )(U_q(\mathfrak{g})\cdot H_{2\mu})\in \mathcal{L}
\end{align}
  if $|\gamma|>|\mu|$.
 It follows from the defining relations of $\mathscr{PD}_{\theta}$ that when we move the $d_{ij}$ terms to the right past
the $x_{ij}$ terms in the expression of the left hand side of (\ref{plain-eqn}) we end up with an expression of the form 
$ \sum_ju_jv_j
$
where each $v_j$ is in $\mathscr{D}^{r_j}_{\theta}$ and each  $u_j$ is  in $\mathscr{P}^{w_j}_{\theta}$ and $r_j-w_j=|\gamma|-|\mu|$.  In other words, the relations cancel 
out the same number of $x_{ij}$ and $d_{ij}$ terms.
 If $|\gamma|$ is strictly greater than $|\mu|$, then each $v_j$ has degree at least $1$ and so 
each $u_jv_j$ is  in $\sum_{i,j}\mathscr{PD}_{\theta}d_{ij} = \mathcal{L}$.  This proves (\ref{plain-eqn}) for $|\gamma|>|\mu|$.

Now assume that  $|\gamma|=|\mu|$ but $\gamma\neq \mu$.  The same argument as in the previous paragraph yields an expression for $
(U_q(\mathfrak{g})\cdot H_{2\gamma}^* )(U_q(\mathfrak{g})\cdot H_{2\mu})$ of the form $\sum_ju_jv_j$ with each $v_j$ is in $\mathscr{D}^{r_j}_{\theta}$, each  $u_j$ is  in $\mathscr{P}^{w_j}_{\theta}$ and $r_j-w_j=0$. 
  It follows that $\pi((U_q(\mathfrak{g})\cdot H_{2\gamma}^* )(U_q(\mathfrak{g})\cdot H_{2\mu}))\in \mathbb{C}(q)$. By Proposition \ref{prop:mugamma}, the two irreducible $U_q(\mathfrak{g})$ modules $U_q(\mathfrak{g})\cdot H_{2\gamma}^* $  and $U_q(\mathfrak{g})\cdot H_{2\mu}$ with $\gamma\neq \mu$ are not dual to each other. It follows that $(U_q(\mathfrak{g})\cdot H_{2\gamma}^* )(U_q(\mathfrak{g})\cdot H_{2\mu})$ does not contain a copy of the trivial representation and hence its image under $\pi$ vanishes.  The lemma now follows.
  \end{proof}

 \section{Action of Cartan  elements}\label{section:Action}

\subsection{Special Cartan elements}
We turn our attention to understanding the action of the various elements of the Cartan subalgebra  on the generators of $\mathscr{P}_{\theta}$ and then identifying them with elements of the appropriate quantum Weyl algebra.

\begin{lemma} \label{lemma:epsilon} Let $N=n$ in Type AI, $N=2n$ in Type AII and $N\in\{n,2n\}$ in the diagonal type. 
 The element   $(K_{2\epsilon_N}-1)/(q^2-1)$ 
acts the same on $\mathscr{P}_{\theta}$ as the element $X$ in $\mathscr{PD}_{\theta}$ where 
\begin{align*} X=
\left\{\begin{matrix}&(q^3+q) x_{nn}d_{nn} + \sum_{a=1}^{n-1} x_{an}d_{an}{\rm\  in\ Type\ AI}\cr\cr
&(\sum_{a=1}^{2n-1} x_{a,2n}d_{a,2n}){\rm \ in \ Type\ AII}\cr\cr
&\sum_{a=1}^nx_{n,a+n}d_{n,a+n} {\rm \ in \  diagonal\  type\ } (N=n)
\cr\cr&\sum_{a=1}^nx_{a,2n}d_{a,2n} {\rm \ in \  diagonal\  type\ } (N=2n)
\end{matrix}
\right.
\end{align*}
In other words, $((K_{2\epsilon_N}-1)/(q^2-1))\cdot a =\pi(Xa)$ for all $a\in \mathscr{PD}_{\theta}$ where $X$ is given by the above formula depending on type.
\end{lemma}
\begin{proof} 
The action of $K_{2\epsilon_N}$ on $\mathscr{P}_{\theta} $ is given on a basis for the degree $1$ space, $\mathcal{J}_1(\mathscr{P}_{\theta})$, by the formula in (\ref{Kaction}), namely,
$K_{2\epsilon_N}\cdot x_{ij} = q^{2\delta_{iN} + 2\delta_{jN}}x_{ij}$  for all valid choices of $i,j$ in Types AI and AII and 
$K_{2\epsilon_N}\cdot x_{i,j+N} = q^{2\delta_{iN} + 2\delta_{jN}}x_{i,j+N}$ in the diagonal setting. By Section 3.2,  (see also \cite{LSS}, Section 5.2) the value for $N$ is given by \begin{itemize}
\item $i,j=1,\dots,N$, $N=n$ and $i\leq j$ in Type AI
\item $i,j=1,\dots, 2n$ , $N=2n$, and  $ i<j$ in Type AII, 
\item $i,j=1,\dots, N$,  $N=n$ in the diagonal case.
\end{itemize}
Note that in the diagonal case, we are taking advantage of the fact that $x_{i,j+n}=x_{i+n,j}$ for all $i=1,\dots, n, j=1,\dots, n$ (see \cite{LSS}, Section 5.2).
For all three types, we use the description of the monomials that form a basis for $\mathscr{P}_{\theta}$ in Lemma \ref{lemma:PBW} using the reverse order described at the end of the lemma.
 
 We start with the Type AI case.  By Lemma \ref{lemma:PBW}, we can express a basis for $\mathscr{P}_{\theta}$ by considering all terms of the form
\begin{align*} x_{nn}^{s_n}x_{n-1,n}^{s_{n-1}}\cdots x_{1,n}^{s_1}x_{e_1,f_1}\cdots x_{e_b,f_b}
\end{align*}
where  $s_i\in \mathbb{N}$ for all $i=1, \dots, n$, and $e_j,f_j\in \mathbb{N}$  with $e_j\leq f_j<n$ for all  $j=1,\dots, b$.  Applying $(K_{2\epsilon_N}-1)/(q^2-1)$ to a typical basis element yields
\begin{align}\label{KKact}{{(K_{2\epsilon_N}-1)}\over{(q^2-1)}}&\cdot (x_{nn}^{s_n}x_{n-1,n}^{s_{n-1}}\cdots x_{1,n}^{s_1}x_{e_1,f_1}\cdots x_{e_b,f_b})\cr&={{(q^{4s_n+2s_{n-1}+\cdots +2s_1}-1)}\over{(q^2-1)}}(x_{nn}^{s_n}x_{n-1,n}^{s_{n-1}}\cdots x_{1,n}^{s_1}x_{e_1,f_1}\cdots x_{e_b,f_b})
\end{align}
Now let's see what happens when we consider the projection $\pi(Xx_{nn}^{s_n}x_{n-1,n}^{s_{n-1}}\cdots x_{1,n}^{s_1}x_{e_1,f_1}\cdots x_{e_b,f_b})$ where $X=(q^3+q) x_{nn}d_{nn} + \sum_{a=1}^{n-1} x_{an}d_{an}$. We proceed by evaluating each summand of $X$ and its action on $x_{nn}^{s_n}x_{n-1,n}^{s_{n-1}}\cdots x_{1,n}^{s_1}x_{e_1,f_1}\cdots x_{e_b,f_b}$.
The following two formulas are special cases of Lemma \ref{lemma:formulas} (ii):
\begin{align}\label{dd2}d_{nn}x_{nn} =q^4 x_{nn}d_{nn} + q^{-1}
\end{align}
and 
\begin{align}\label{dd3} d_{nn}x_{ef} =q^{2\delta_{fn}}x_{ef}d_{nn}
\end{align} for $e\leq f$ and $e<n$.

Starting with the first summand of $X$  applied to the first  term of the basis gives us \begin{align*}(q^3+q) x_{nn}d_{nn}x_{nn}=(q^3+q) x_{nn}(d_{nn}x_{nn})=q^4x_{nn}^2d_{nn}+q^{-1}x_{nn}
\end{align*}
where here we have used (\ref{dd2}).  Similarly, with another application of (\ref{dd2}), we have
\begin{align*}(q^3+q) x_{nn}d_{nn}x_{nn}^2&=(q^3+q) x_{nn}d_{nn}x_{nn}^2=(q^3+q) ((q^4x_{nn}^2d_{nn})x_{nn}+q^{-1}x_{nn}^2)\cr&=
(q^3+q) (q^4x_{nn}^2(d_{nn}x_{nn})+q^{-1}x_{nn}^2)\cr&=(q^3+q) (q^8x_{nn}^3d_{nn} +q^{-1}(q^4+1)x_{nn}^2).
\end{align*} By induction, repeatedly using (\ref{dd2}), we have
\begin{align*}(q^3+q)x_{nn}d_{nn}x_{nn}^{s_n}
&=(q^3+q)\left(q^{4s_n}x_{nn}x_{nn}^{s_n}d_{nn}+q^{-1}(q^{4(s_{n}-1)}+\cdots +1)x_{nn}^{s_n}\right)\cr&
=(q^3+q)\left(q^{4s_n}x_{nn}x_{nn}^{s_n}d_{nn}+q^{-1}{{(q^{4s_n}-1)}\over{(q^{4}-1)}}x_{nn}^{s_n}\right)
\end{align*}
Note that $(q^3+q)q^{-1}=q^2+1$ and so 
\begin{align*}(q^3+q)q^{-1}{{(q^{4s_n}-1)}\over{(q^{4}-1)}}={{(q^{4s_n}-1)}\over{(q^{2}-1)}}
\end{align*}
Hence 
\begin{align*}
(q^3+q)x_{nn}d_{nn}x_{nn}^{s_n}
&=(q^3+q)q^{4s_n}x_{nn}x_{nn}^{s_n}d_{nn}+{{(q^{4s_n}-1)}\over{(q^{2}-1)}}x_{nn}^{s_n}.
\end{align*}

Now consider $(q^3+q) x_{nn}d_{nn}$ applied to a basis term for $\mathscr{P}_{\theta}$ of the form 
$x_{nn}^{s_n}x_{n-1,n}^{s_{n-1}}\cdots x_{1,n}^{s_1}x_{e_1,f_1}\cdots x_{e_b,f_b}$ where each $e_j\leq f_j<n$. By (\ref{dd3}), 
\begin{align*}d_{nn}x_{n-1,n}^{s_{n-1}}\cdots x_{1,n}^{s_1}x_{e_1,f_1}\cdots x_{e_b,f_b} = q^{(2s_{n-1}+\cdots +2 s_1)} x_{n-1,n}^{s_{n-1}}\cdots x_{1,n}^{s_1}x_{e_1,f_1}\cdots x_{e_b,f_b}d_{nn}.
\end{align*}
This is an element of $\mathscr{P}_{\theta}d_{nn}$ which is  a subset of $\mathcal{L}$.  Hence 
\begin{align*}
((q^3+q) &x_{nn}d_{nn})(x_{nn}^{s_n}x_{n-1,n}^{s_{n-1}}\cdots x_{1,n}^{s_1}x_{e_1,f_1}\cdots x_{e_b,f_b}) \cr&= {{(q^{4s_n}-1)}\over{(q^{2}-1)}}x_{nn}^{s_n}x_{n-1,n}^{s_{n-1}}\cdots x_{1,n}^{s_1}x_{e_1,f_1}\cdots x_{e_b,f_b} +\mathcal{L}.
\end{align*}

We turn our attention to the other summands of $X$. Note that relation (\ref{xreln1}) in Section \ref{section:qhs} satisfied by the $x_{ij}$ ensures that   \begin{align}\label{xformd1}x_{bn}x_{nn} = q^2x_{nn}x_{bn} \quad{\rm and\quad} x_{an}x_{bn} = q^{(1-\delta_{ab})}x_{bn}x_{an}
  \end{align}  for all $a\leq b<n$. Meanwhile, by Lemma \ref{lemma:formulas} (ii),  
  \begin{align*}d_{en}x_{bn} =q^{2\delta_{bn}}x_{bn}d_{en}
  \end{align*} for $e\neq b$ and
   \begin{align}\label{den}d_{en}x_{bn} =qx_{bn}d_{en}
  \end{align} for $e\neq b$, $e<n$, $b<n$.
    Using (\ref{den}) to move $d_{en}$ to the right and then (\ref{xformd1}) to move $x_{en}$ to the right results in
\begin{align*}x_{en}d_{en}&(x_{nn}^{s_n}x_{n-1,n}^{s_{n-1}}\cdots x_{1,n}^{s_1}x_{e_1,f_1}\cdots x_{e_b,f_b})\cr&=q^{2s_n+s_{n-1}+\cdots +s_{e+1}}x_{en}x_{nn}^{s_n}x_{n-1,n}^{s_{n-1}}\cdots x_{e+1,n}^{s_{e+1}}d_{en}x_{en}^{s_e}\cdots x_{1n}^{s_1}x_{e_1,f_1}\cdots x_{e_b,f_b})\cr&=
q^{4s_n+2s_{n-1}\cdots +2s_{e+1}}x_{nn}^{s_n}x_{n-1,n}^{s_{n-1}}\cdots x_{e+1,n}^{s_{e+1}}(x_{en}d_{en})x_{en}^{s_e}\cdots x_{1n}^{s_1}x_{e_1,f_1}\cdots x_{e_b,f_b})
\end{align*}
The following formula is  from  Lemma \ref{lemma:formulas}(ii):
\begin{align}\label{formd1} d_{en}x_{en} =q^{2}x_{en}d_{en} - \sum_{a'>e}q^{2+\delta_{a'n}-\delta_{en}}(q^{-2}-1)x_{a'n}d_{a'n} + 1
\end{align}  for  $e<n$.  Note further that $d_{a'n}x_{en}^{s_e}\cdots x_{1n}^{s_1}x_{e_1,f_1}\cdots x_{e_b,f_b}\in \mathscr{P}_{\theta}d_{a'n}$ for $a'>e$ where here we are using
(\ref{dd3}) for $a'=n$ and (\ref{den}) for $e<a'<n$.  Arguing as we did for the first summand of $X$ using here (\ref{formd1}) instead of (\ref{dd2}), we get 
\begin{align*}
(x_{en}d_{en})&x_{en}^{s_e}\cdots x_{1n}^{s_1}x_{e_1,f_1}\cdots x_{e_b,f_b})=(q^{2(s_e-1)+2(s_{e}-2)+\cdots + 1})x_{en}^{s_e}\cdots x_{1n}^{s_1}x_{e_1,f_1}\cdots x_{e_b,f_b} + \mathcal{L}
\cr&={{(q^{2s_e}-1)}\over{(q^2-1)}}x_{en}^{s_e}\cdots x_{1n}^{s_1}x_{e_1,f_1}\cdots x_{e_b,f_b} + \mathcal{L}
\end{align*}
Hence 
\begin{align*}
x_{en}d_{en}&(x_{nn}^{s_n}x_{n-1,n}^{s_{n-1}}\cdots x_{1,n}^{s_1}x_{e_1,f_1}\cdots x_{e_b,f_b})\cr&=q^{4s_n+2s_{n-1}\cdots +2s_{e+1}}{{(q^{2s_e}-1)}\over{(q^2-1)}}
x_{nn}^{s_n}x_{n-1,n}^{s_{n-1}}\cdots x_{e+1,n}^{s_{e+1}}x_{en}^{s_e}\cdots x_{1n}^{s_1}x_{e_1,f_1}\cdots x_{e_b,f_b}+\mathcal{L}
\end{align*}
It follows that the sum of the coefficients of $x_{nn}^{s_n}x_{n-1,n}^{s_{n-1}}\cdots x_{1,n}^{s_1}x_{e_1,f_1}\cdots x_{e_b,f_b}$ in the projection under $\pi$ of $X$ times this term is
\begin{align*}(q^2-1)^{-1}\left(\sum_{i=1}^{n}  q^{4s_n+2s_{n-1}+\cdots +2s_{i+1}}(q^{2s_{i}}-1)
\right)
=(q^2-1)^{-1}\left( q^{4s_n+2s_{n-1}+\cdots +2s_{2}+2s_{1}}-1
\right)
\end{align*}
This agrees with the action of $(K_{2\epsilon_N}-1)/(q^2-1)$ on the  monomial term above as given in (\ref{KKact}) at the beginning of this proof.

 The argument in Type AII is exactly the same where we omit any terms involving $d_{nn}$ and $x_{nn}$ and replace $n$ with $2n$ everywhere.  Indeed, Lemma 5.1 (iii) for Type AII is basically the same as Lemma 5.1(i) for Type AI.  The only differences are replacing $n$ in Lemma 5.1(i) with  $2n$  in Type AII and insisting that all $x_{ij}$ that appear for Type AII satisfy $i<j$ (instead of $i\leq j$ for Type AI).  Also, 
 the monomials that form a basis of $\mathscr{P}_{\theta}$ in Type AII can be viewed as a subset of those for $\mathscr{P}_{\theta}$ in Type AI.  This monomial basis consists of terms of the form
 \begin{align*}x_{n-1,n}^{s_{n-1}}\cdots x_{1,n}^{s_1}x_{e_1,f_1}\cdots x_{e_b,f_b}
 \end{align*}
 where each $e_b<f_b<n$.  Using the same argument as for Type AI shows that the action of $(K_{2\epsilon_N}-1)/(q^2-1)$ on this  monomial term is the same as in the projection under $\pi$ of $X=\sum_{e=1}^{2n-1}x_{en}d_{en}$ applied to this term.  In particular, both give the following coefficient for the above monomial: $(q^2-1)^{-1}(q^{2s_{n-1}+2s_{n-2}+\cdots + 2s_1}-1)$.

Now consider the diagonal case. Using Lemma \ref{lemma:PBW}, we consider two basis for $\mathscr{P}_{\theta}$, the first associated to $N=n$ and the second to $N=2n.$  In particular, the first is straight from Lemma \ref{lemma:PBW} with the order reversed and consists of 
 all monomials of the form 
\begin{align*}x_{n,2n}^{m_{2n}}x_{n,2n-1}^{m_{2n-1}}\cdots  x_{n,n+1}^{m_{n+1}}x_{e_1,f_1+n}\cdots x_{e_b,f_b+n}
\end{align*}
where  $s_i\in \mathbb{N}$ for all $i=1, \dots, n$, and $e_j,f_j\in \mathbb{N}$  with $e_j\leq f_j<n$ for all  $j=1,\dots, b$. The second basis consists of all monomials of the form
\begin{align*} x_{n,2n}^{s_n}x_{n-1,2n}^{s_{n-1}}\cdots x_{1,2n}^{s_1}x_{e_1,f_1+n}\cdots x_{e_b,f_b+n}
\end{align*}
with  $e_i<n$ for $i=1,\dots,b$. Here, we are taking advantage of the fact that in the diagonal case, $\mathscr{P}_{\theta}$ is isomorphic as an algebra to $\mathcal{O}_q[{\rm Mat}_n]$ via the map $x_{i,j+n}\mapsto t_{ij}$ and the relations satisfied by the $t_{ij}$ (see the beginning of Section \ref{section:qmf}) allow us to choose a different ordering of the terms $x_{i,j+n}$ and get a new basis.

By (\ref{Kaction2}), it follows that $(K_{2\epsilon_{n}}-1)/(q^2-1)$  applied to the first kind of basis term is 
\begin{align*}{{(K_{2\epsilon_{n}}-1)}\over{(q^2-1)}}&\cdot x_{n,2n}^{m_{2n}}x_{n,2n-1}^{m_{2n-1}}\cdots  x_{n,n+1}^{m_{n+1}}x_{e_1,f_1+n}\cdots x_{e_b,f_b+n} \cr&={{( q^{2m_n+2m_{n-1}+\cdots 2m_1}-1)}\over{(q^2-1)}}x_{n,2n}^{m_{2n}}x_{n,2n-1}^{m_{2n-1}}\cdots  x_{n,n+1}^{m_{n+1}}x_{e_1,f_1+n}\cdots x_{e_b,f_b+n}\end{align*}
Using (\ref{Kaction2}) again and the second family of basis elements gives us
\begin{align*}{{(K_{2\epsilon_{2n}}-1)}\over{(q^2-1)}}&\cdot x_{n,2n}^{s_n}x_{n-1,2n}^{s_{n-1}}\cdots x_{1,2n}^{s_1}x_{e_1,f_1+n}\cdots x_{e_b,f_b+n}
\cr&= {{(q^{2s_n+2s_{n-1}+\cdots 2s_1}-1)}\over{(q^2-1)}}x_{n,2n}^{s_n}x_{n-1,2n}^{s_{n-1}}\cdots x_{1,2n}^{s_1}x_{e_1,f_1+n}\cdots x_{e_b,f_b+n}
\end{align*}

We show that the image under  the projection map $\pi$ of $X=\sum_{a=1}^nx_{n,a+n}d_{n,a+n}$ applied to a typical term in the first kind of basis yields the same result as applying  $(K_{2\epsilon_{n}}-1)/(q^2-1)$.
Using the relations for $\mathscr{PD}_{\theta}$ as given in Lemma \ref{lemma:formulas} (iv)-(vii), we get 
\begin{align*}
x_{n,a+n}d_{n,a+n}x_{e,f+n} \in \sum_{a'\geq a}\mathscr{P}_{\theta}d_{n,a'+n}\subset \mathcal{L} 
\end{align*}
whenever $f<n$ or $e<a$.
Hence, by induction, 
\begin{align*}
x_{n,a+n}d_{n,a+n} x_{e_1,f_1+n}\cdots x_{e_b,f_b+n} \in \sum_{a'\geq a} \mathscr{P}_{\theta}d_{n,a'+n}\subset \mathcal{L}
\end{align*}
for any choice of $(e_1,f_1), \cdots, (e_b,f_b)$ with $e_j\leq f_j<n$ each $j=1,\dots,b$.  
On the other hand,  by Lemma \ref{lemma:formulas} (v) 
\begin{align*}
\sum_{a<e}x_{n,a+n}d_{n,a+n}x_{n,e+n}= qx_{n,a+n} x_{n,e+n}d_{n,a+n}
\end{align*}
and so 
\begin{align*}
\sum_{a<e}x_{n,a+n}d_{n,a+n}x_{n,e+n}^{m_e}= \sum_{a<e}q^{m_e}x_{n,a+n}x_{n,e+n}^{m_e}d_{n,a+n}
\end{align*}
Using the relations (i) satisfied by the $x_{a,b+n}$ and derived from those satisfied by the $t_{a,b}$ given at the beginning of Section \ref{section:qmf}, 
we have 
\begin{align*}x_{n,a+n}x_{n,e+n}=qx_{n,e+n}x_{n,a+n}
\end{align*}
for $a<e$.  Hence 
\begin{align}\label{firstagain}
\sum_{a<e}(x_{n,a+n}d_{n,a+n})x_{n,e+n}^{m_e}= q^{2m_e} x_{n,e+n}^{m_e}(x_{n,a+n}d_{n,a+n}).
\end{align}

Repeatedly using relation (vii) of Lemma \ref{lemma:formulas}, we have 
\begin{align}\label{secondagain}
x_{n,a+n}d_{n,a+n}x_{n,a+n}^{m_a} = q^{2m_a}x_{n,a+n}^{m_a}d_{n,a+n} +{{(q^{2m_a}-1)}\over{(q^2-1)}}x_{n,a+n}^{m_{a}}+ \sum_{a'\geq a}\mathscr{P}_{\theta}d_{n,a'+n}.
\end{align} (Note that the argument here is very similar to the same type of calculation used in Type AI and Type AII).
Using both (\ref{firstagain}) and (\ref{secondagain}) and arguing as was done for Type AI and Type AII yields
\begin{align*}
&\sum_{a=1}^nx_{n,a+n}d_{n,a+n}x_{n,2n}^{s_n}x_{n,2n-1}^{m_{n-1}}\cdots x_{n,1+n}^{m_1}x_{e_1,f_1+n}\cdots x_{e_r,f_r+n}\cr&=\sum_{a=1}^{n}q^{2 m_n+\cdots+2m_{a+1}}{{(q^{2m_a}-1)}\over{(q^2-1)}}x_{n,2n}^{s_n}x_{n,2n-1}^{m_{n-1}}\cdots x_{n,1+n}^{m_1}x_{e_1,f_1+n}\cdots x_{e_r,f_r+n}
+\sum_a\mathscr{P}_{\theta}d_{n,a+n}
\cr&={{( q^{2m_n+2m_{n-1}+\cdots 2m_1}-1)}\over{(q^2-1)}}x_{n,2n}^{m_{2n}}x_{n,2n-1}^{m_{2n-1}}\cdots  x_{n,n+1}^{m_{n+1}}x_{e_1,f_1+n}\cdots x_{e_b,f_b+n}+\mathcal{L}.
\end{align*}
This completes the proof for the  $N=n$ case.

A similar argument shows that applying the projection map $\pi$ to $X=\sum_{a=1}^nx_{a,2n}d_{a,2n}$ times a typical basis element of the second kind yields the same result as applying $(K_{2\epsilon_{2n}}-1)(q^2-1)^{-1}$.
\end{proof}

\subsection{Relationships between Cartan elements}\label{subsection:Relationships}

In the next lemma, we write $K_{2\epsilon_r+\cdots + 2\epsilon_N}$ in terms of elements in the $({\rm ad}\ U_q(\mathfrak{g}))$-module generated by $K_{2\epsilon_{r+1}+\cdots +2\epsilon_N}$ and
$K_{2\epsilon_N}$ where $N$ is either $n$ or $2n$ depending on  $\mathfrak{g}$. This is the crucial step in showing that $K_{2\epsilon_r+\cdots +2\epsilon_N}$ acts the same as an 
 element in $\mathscr{PD}_{\theta}$ on $\mathscr{P}_{\theta}$.

A key set of tools in the proof of the next lemma  are Lusztig's braid group automorphisms $T_i$.  We use the formulas from \cite{KS}, Section 6.2.2 for $U_q(\mathfrak{sl}_N)$. The images of $E_{t},F_{t}, $ and $K_t$ under $T_s$ are 
\begin{align*}T_s(E_t) = \left\{\begin{matrix}&(-1)E_sE_t+q^{-1}E_tE_s  &a_{st}=-1\cr &E_t &a_{st}=0\end{matrix}\right.
\end{align*}
\begin{align*}T_s(F_t) = \left\{\begin{matrix}&(-1)F_tF_s+qE_sE_t &a_{st}=-1\cr &F_t &a_{st}=0\end{matrix}\right.
\end{align*}
\begin{align*}T_s(K_t)=K_sK_t^{-a_{st}}
\end{align*}
where $a_{st}$ is the $s,t$ entry in the Cartan matrix for $U_q(\mathfrak{sl}_N)$.  

Set $\beta_{s,t}= \alpha_s+\alpha_{s+1} +\cdots +\alpha_{t}$ for $r\leq s< t\leq N-1$ and write $K_{\beta_{s,t}}$ for $K_sK_{s+1}\cdots K_t$.  These automorphisms are used to define root vectors (\cite{KS}, Section 6.2.3) as follows.  
Set \begin{align}\label{rootvectors}
E_{\beta_{s,t}}=T_sT_{s+1}\cdots T_{t-1}(E_t){\rm \ and \ }F_{\beta_{s,t}}=T_sT_{s+1}\cdots T_{t-1}(F_t).
\end{align} 
Note that $E_{\beta_{s,t}}$ has weight $\beta_{s,t}$ and $F_{\beta_{s,t}}$ has weight $-\beta_{s,t}$.  These notions are extended to $s=t$ with $\beta_{s,s}=\alpha_s$, $E_{\beta_{s,s}}=E_s$,
and $F_{\beta_{s,s}}=F_s$.

\begin{lemma}\label{lemma:K_epsilon}
We have the following equalities
\begin{align}\label{eq1}
({\rm ad}\ E_{r}E_{r+1}\cdots E_{N-2}E_{N-1})K_{2\epsilon_N}
=(1-q^{-2}) (-1)^{N-1-r}E_{\beta_{r,N-1}}K_{2\epsilon_N}
\end{align}
and
\begin{align}\label{eq2}
({\rm ad}\ &F_{N-1}F_{N-2}\cdots F_{r})\cdot K_{2\epsilon_{r+1}+\cdots +2\epsilon_N}
=(1-q^{2}) (-1)^{N-1-r}F_{\beta_{r,N-1}}K_{\beta_{r,N-1}}K_{2\epsilon_{r+1}+\cdots +2\epsilon_N}
\end{align}
where 
$N=n$, $r=1, \dots, n$ in Type AI,
 $N=2n$ and $r=1,\dots, 2n$ for Type AII, and   the two options $N=n$, $r=1,\dots, n$ and $N=2n$, $r=n+1,\dots, 2n$ in the diagonal case. Moreover, 
 \begin{align*}K_{2\epsilon_r+\cdots 2\epsilon_N}&=
(q-q^{-1})q^{2}(E_{\beta_{r,N-1}}K_{2\epsilon_N})(F_{\beta_{r,N-1}} K_{\beta_{r,N-1}}K_{2\epsilon_{r+1}+\cdots + 2\epsilon_N})\cr&- (q-q^{-1})(F_{\beta_{r,N-1}}K_{\beta_{r,N-1}}K_{2\epsilon_{r+1}+\cdots + 2\epsilon_{N}})(E_{\beta_{r,N-1}} K_{2\epsilon_N}) +K_{2\epsilon_{r+1}+\cdots + 2\epsilon_N}K_{2\epsilon_N}
 \end{align*}
\end{lemma}
\begin{proof} 

 Straightforward calculations show that 
\begin{align*}({\rm ad}\ E_{N-1})\cdot K_{2\epsilon_N} = (1-q^{-2}) E_{N-1}K_{2\epsilon_N}
\end{align*} and 
\begin{align*}
({\rm ad}\ E_{N-2}E_{N-1})\cdot K_{2\epsilon_N}&=(1-q^{-2})(E_{N-2}E_{N-1}-q^{-1}E_{N-1}E_{N-2})K_{2\epsilon_N} \cr &= -T_{N-2}(E_{N-1})K_{2\epsilon_N}.
\end{align*} Now assume that 
\begin{align*}
({\rm ad}\ E_{s+1}\cdots E_{N-2}E_{N-1})K_{2\epsilon_N}&=(1-q^{-2}) (-1)^{N-2-s}T_{s+1}\cdots T_{N-2}(E_{N-1})K_{2\epsilon_N}.
\end{align*} 
It follows that 
\begin{align*}
&({\rm ad}\ E_{s}E_{s+1}\cdots E_{N-2}E_{N-1})\cdot K_{2\epsilon_N}\cr&=(1-q^{-2}) (-1)^{N-2-s}\left(E_sT_{s+1}\cdots T_{N-2}(E_{N-1})-q^{-1}T_{s+1}\cdots T_{N-2}(E_{N-1})E_sK_{2\epsilon_N}\right)
\cr&=(1-q^{-2}) (-1)^{N-1-s}T_sT_{s+1}\cdots T_{N-2}(E_{N-1})K_{2\epsilon_N}
\end{align*} 
Hence, by induction, we have
\begin{align*}
({\rm ad}\ E_{r}E_{r+1}\cdots E_{N-2}E_{N-1})K_{2\epsilon_N}&
=(1-q^{-2}) (-1)^{N-1-r}T_rT_{r+1}\cdots T_{N-2}(E_{N-1})K_{2\epsilon_N}.
\end{align*} 
Thus (\ref{eq1}) follows from this equality combined with the definition of  $E_{\beta_{r,N-1}}$ in (\ref{rootvectors}).

We have a similar result for the $F_j's$.  
In particular, we have
\begin{align*}({\rm ad}\ F_{r})\cdot K_{2\epsilon_{r+1}+\cdots +2\epsilon_N} &=F_rK_{2\epsilon_{r+1}+\cdots +2\epsilon_N} K_r-K_{2\epsilon_{r+1}+\cdots +2\epsilon_N} F_rK_r
\cr&= (1-q^{2}) F_{r} K_rK_{2\epsilon_{r+1}+\cdots +2\epsilon_N}
\end{align*} 
and
\begin{align*}({\rm ad}\ F_{r+1}F_{r})\cdot K_{2\epsilon_{r+1}+\cdots +2\epsilon_N} 
&=(1-q^{2}) ({\rm ad}\ F_{r+1})\cdot  F_{r} K_rK_{2\epsilon_{r+1}+\cdots +2\epsilon_N}
\cr&=(1-q^{2})\left(F_{r+1} F_{r} K_rK_{r+1}-F_rK_rF_{r+1}K_{r+1}\right)K_{2\epsilon_{r+1}+\cdots +2\epsilon_N}
\cr&=(1-q^{2})\left(F_{r+1} F_{r} -qF_rF_{r+1}\right)K_{r+1}K_{r}K_{2\epsilon_{r+1}+\cdots +2\epsilon_N}
\cr&=(1-q^2)(-1)T_{r}(F_{r+1})K_{\beta_{r,r+1}}K_{2\epsilon_{r+1}+\cdots +2\epsilon_N}
\end{align*}
Now assume that for $s>r+2$, we have
\begin{align*}
&({\rm ad}\ F_{s-1}\cdots F_{r+2}F_{r+1})\cdot K_{2\epsilon_{r+1}+\cdots +2\epsilon_N}
\cr&=(1-q^{2}) (-1)^{s-2-r}T_{r+1}\cdots T_{s-2}(F_{s-1})K_{\beta_{r+1,s-1}}K_{2\epsilon_{r+1}+\cdots +2\epsilon_N}
\end{align*} 
Since $T_{k+1}(F_{r})=F_r$ for $k>r$, it follows that 
\begin{align*}
&({\rm ad}\ F_{s-1}\cdots F_{r+1}F_r)\cdot K_{2\epsilon_{r+1}+\cdots +2\epsilon_N}\cr&=(1-q^2)(-1)^{s-2-r}(T_{r+1}\cdots T_{s-2}(F_{s-1})F_{r}-qF_{r}T_{r+1}\cdots T_{s-2}(F_{s-1}))K_rK_{\beta_{r+1,s-1}}K_{2\epsilon_{r+1}+\cdots +2\epsilon_N}
\cr&=(1-q^2)(-1)^{s-1-r}(T_rT_{r+1}\cdots T_{s-2}(F_{s-1}))
K_{\beta_{r,s-1}}K_{2\epsilon_{r+1}+\cdots +2\epsilon_N}
\end{align*} 
Hence, by induction, we have
\begin{align*}
({\rm ad}\ &F_{N-1}F_{N-2}\cdots F_{r})\cdot K_{2\epsilon_{r+1}+\cdots +2\epsilon_N}
\cr&=(1-q^{2}) (-1)^{N-1-r}(T_rT_{r+1}\cdots T_{N-2}(F_{N-1}))K_{\beta_{r,N-1}}K_{2\epsilon_{r+1}+\cdots +2\epsilon_N}.
\end{align*} Thus (\ref{eq2}) follows from this equality combined with the definition of $F_{\beta_{r,N-1}}$ in (\ref{rootvectors}).

Using the fact that the $T_i$ are algebra automorphisms of $U_q(\mathfrak{sl}_N)$, it follows that $E_{\beta_{r,N-1}},F_{\beta_{r,N-1}}$ and $K^{\pm 1}_{\beta_{r,N-1}}$ generate a subalgebra isomorphic to $U_q(\mathfrak{sl}_2)$. Therefore, the commutator
\begin{align*}[E_{\beta_{r,N-1}},F_{\beta_{r,N-1}}]=(q-q^{-1})^{-1}(K_{\beta_{r,N-1}}-K_{\beta_{r,N-1}}^{-1})
\end{align*}
Hence
\begin{align*}
q^{2}&(E_{\beta_{r,N-1}}K_{2\epsilon_N})(F_{\beta_{r,N-1}} K_{\beta_{r,N-1}}K_{2\epsilon_{r+1}+\cdots + 2\epsilon_N})- (F_{\beta_{r,N-1}}K_{\beta_{r,N-1}}K_{2\epsilon_{r+1}+\cdots + 2\epsilon_{N}})(E_{\beta_{r,N-1}} K_{2\epsilon_N}) 
\cr &= (E_{\beta_{r,N-1}}F_{\beta_{r,N-1}}-F_{\beta_{r,N-1}}E_{\beta_{r,N-1}})K_{\beta_{r,N-1}}K_{2\epsilon_{r+1}+\cdots + 2\epsilon_N}K_{2\epsilon_N}
\cr&=(q-q^{-1})^{-1}(K_{\beta_{r,N-1}}^2-1)K_{2\epsilon_{r+1}+\cdots + 2\epsilon_N}K_{2\epsilon_N}.
\end{align*}
Note that $2\beta_{r,N-1}=2\alpha_r+\cdots + 2\alpha_{N-1}= 2\epsilon_{r}-2\epsilon_{N}.$ Thus the above simplifies to 
\begin{align*}
(q-q^{-1})^{-1}(K_{2\epsilon_r+2\epsilon_{r+1}+\cdots + 2\epsilon_N}-K_{2\epsilon_{r+1}+\cdots + 2\epsilon_N}K_{2\epsilon_N}).
\end{align*} The final assertion of the lemma now follows.
\end{proof}

\subsection{Acting as quantum Weyl algebra elements}
Let $\psi$ denote the map from $U_q(\mathfrak{g})$ to ${\rm End}\ \mathscr{P}_{\theta}$ that agrees with the action of $U_q(\mathfrak{g})$ on $\mathscr{P}_{\theta}$.  We show that $\psi(K_{2\epsilon_r +\dots+2\epsilon_m})$ agrees with the image under $\phi$ of an element of $\mathscr{PD}_{\theta}$.  Moreover, we determine the degree of these elements using 
the degree function defined in Section \ref{section:gandr}.

\begin{proposition}\label{prop:K_epsilon}
The image  $\psi(K_{2\epsilon_r +\dots+2\epsilon_N})$ of  $K_{2\epsilon_r +\dots+2\epsilon_N}$ inside 
${\rm End}\ \mathscr{P}_{\theta}$ is equal to the image $\phi_{a_r}$  for some $a_r\in \mathscr{PD}_{\theta}$ where
\begin{itemize}
\item  
$N=n$, $r=1, \dots, n$ in Type AI
\item $N=2n$ and $r=1,\dots, 2n$ for Type AII
\item  $N=n$, $r=1,\dots, n$ and $N=2n$, $r=n+1,\dots, 2n$ in the diagonal case. 
\end{itemize} Moreover, $\deg a_r \leq 2(N-r+1)$.
\end{proposition}
\begin{proof}
By Lemma \ref{lemma:epsilon}, $\psi(K_{2\epsilon_N})=\phi_{a_N}$ for an appropriate element $a_N\in \mathscr{PD}_{\theta}$.  Moreover,  one sees from the formulas for $\psi(K_{2\epsilon_N})$ given in this lemma that
$\deg a_N=2 = 2(N-N+1)$.  Now assume that for some $r\leq N$, 
$\psi(K_{2\epsilon_{r+1}+\cdots +2\epsilon_N})=\phi_{a_{r+1}}$ for an  element $a_{r+1}$ of degree at most  ${2(N-r)}$  where $N=n$ in Type AI, $N=2n$ in Type AII, and $N=n$ or $N=2n$ in the diagonal case with $N-n\leq r\leq N-1$.  

By Lemma \ref{lemma:K_epsilon},  $K_{2\epsilon_{r}+ \cdots +2\epsilon_N}$ is a linear combination of  products of the form $xy$ where $x\in 
({\rm ad}\ U_q(\mathfrak{g}))\cdot K_{2\epsilon_N}$ and $y\in ({\rm ad}\ U_q(\mathfrak{g}))\cdot K_{2\epsilon_{r+1}+\cdots +2\epsilon_N}$. 
 Since $\phi(\mathscr{PD}_{\theta})$ is a subalgebra and, by Proposition \ref{prop:locally-finite}, an $({\rm ad}\ U_q(\mathfrak{g}))$-submodule of ${\rm End}\ \mathscr{P}_{\theta}$, it follows that $\psi(K_{2\epsilon_{r}+ \cdots +2\epsilon_N}) \in \phi(\mathscr{PD}_{\theta})$, and so the lemma follows by induction.

We now turn our attention to understanding the degree assertion at the conclusion of the lemma. As explained above, $\deg a_N = 2$. Now assume that 
$\deg a_{r+1} \leq 2(N-(r+1)+1)$.  It follows that  $\deg(a_{r+1}a_N) =\deg(a_{r+1}) + \deg(a_N) \leq  2(N-(r+1)+1) + 2 = 2(N-r+1).$  Since the filtration $\mathcal{J}$ is preserved by the action of $U_q(\mathfrak{g})$, we also have $\deg(E_{\beta_{r,N-1}}K_{2\epsilon_N})=2$ and $\deg(F_{\beta_{r,N-1}} K_{2\epsilon_{r+1}+\cdots + 2\epsilon_N})\leq 2(N-r). $ By Lemma \ref{lemma:K_epsilon},
\begin{align*}&(E_{\beta_{r,N-1}}F_{\beta_{r,N-1}}-F_{\beta_{r,N-1}}E_{\beta_{r,N-1}})K_{\beta_{r,N-1}}K_{2\epsilon_{r+1}+\cdots + 2\epsilon_N}K_{2\epsilon_N}\cr&=(q-q^{-1})^{-1}(K_{\beta_{r,N-1}}^2-1)K_{2\epsilon_{r+1}+\cdots + 2\epsilon_N}K_{2\epsilon_N}
\leq 2(N-r+1)\cr&= (q-q^{-1})^{-1}(K_{2\epsilon_r+2\epsilon_{r+1}+\cdots + 2\epsilon_N}-K_{2\epsilon_{r+1}+\cdots + 2\epsilon_N}K_{2\epsilon_N}).
\end{align*} Therefore, 
\begin{align*}K_{2\epsilon_r+2\epsilon_{r+1}+\cdots + 2\epsilon_N}=(q-q^{-1})(K_{2\epsilon_r+2\epsilon_{r+1}+\cdots + 2\epsilon_N}-K_{2\epsilon_{r+1}+\cdots + 2\epsilon_N})+K_{2\epsilon_{r+1}+\cdots + 2\epsilon_N}
\end{align*} and so $\deg a_r\leq 2(N-r+1)$ which equals $2(2n-r+1)$ in Type AII and equals $2(n-r+1)$ in Type AI and the diagonal case.  The final assertion of the proposition now follows by induction. \end{proof}

It will follow from later results in this paper that this inequality is actually an equality.  This is because Theorem \ref{theorem:center_and_capelli} shows that the $U_q(\mathfrak{g})$-module generated by the  image of $K_{2\epsilon_{r}+ \cdots +2\epsilon_N}$ in $\mathscr{PD}_{\theta}$ contains a central element of degree  $2(N-r+1)$.

\section{The locally finite subalgebra}

\subsection{The simply connected case}\label{subsection:simply-conn}
In \cite{JL} and \cite{JL2}, a complete description of the locally finite subalgebra of $U_q(\mathfrak{sl}_N)$ as  a direct sum of $({\rm ad}\ U_q(\mathfrak{sl}_N))$-modules is given. This result is then generalized to the simply connected quantized enveloping algebra (see Section \ref{section:qea})
in \cite{JL2} (see also \cite{Jo}, 7.1).   In particular, we have  \begin{align}\label{FdecompA}
\mathcal{F}(U_q(\mathfrak{sl}_N)) = \bigoplus_{\lambda \in -P_N^+\cap Q}({\rm ad}\ U_q(\mathfrak{sl}_N))\cdot K_{2\lambda}.
\end{align}
and  \begin{align}\label{Fdecomp}
\mathcal{F}(\check U_q(\mathfrak{sl}_N)) = \bigoplus_{\lambda \in -P_N^+}({\rm ad}\ U_q(\mathfrak{sl}_N))\cdot K_{2\lambda}.
\end{align} 
Note that  ${-P_N^+ }= \{-\lambda|\  \lambda\in P_N^+\} = \{w_0\lambda|\  \lambda\in P_N^+\}=w_0P^+_N$ which is just  the $\mathbb{N}$-linear span of the $w_0\omega_i, i=1, \dots, N-1$.  
More concretely,
\[
{-P_N^+}=\left\{
a_1\epsilon_1+\cdots+a_N\epsilon_N
\,\Big|\,
\sum_{i=1}^Na_i=0,\ a_{i+1}-a_{i}\in\mathbb N,\ a_i\in\frac{1}{N}\mathbb Z
\right\},
\]
and 
\[
-P_N^+\cap Q=\left\{
a_1\epsilon_1+\cdots+a_N\epsilon_N\,\Big|\,
a_i\in\mathbb Z,\ a_1\leq \cdots\leq a_N,\ \sum_{i=1}^Na_i=0
\right\}.
\]
As explained in Section \ref{section:roots-and-weights}, the fundamental weight $ \omega_i$ is equal to  $\hat{\omega}_i=\epsilon_{1}+\cdots + \epsilon_i$ plus a scalar multiple of $\hat{\omega}_N = \epsilon_1+\cdots + \epsilon_N$  while its image under $w_0$, namely $w_0\omega_i$, is equal to $\epsilon_{N-i}+\cdots +\epsilon_{N-1}$ plus a scalar multiple of $\hat{\omega}_N= \epsilon_1+\cdots + \epsilon_N$.
We see from (\ref{omega}) that this  scalar is $i/N$ (for both $\omega_i$ and $w_0\omega_i$) which is not an  integer and so,   the simply connected quantized enveloping algebra
$\check U_q(\mathfrak{sl}_N)$ is not a subalgebra of $U_q(\mathfrak{gl}_N)$.  However, the two algebras are closely related. For instance, we can extend $U_q(\mathfrak{gl}_N)$
 in a similar manner to the  construction of $\check U_q(\mathfrak{sl}_N)$ so that the resulting algebra contains both $\check U_q(\mathfrak{sl}_N)$ and $U_q(\mathfrak{gl}_N)$.  To do this,
we set $\mathcal C=\mathbb C[K_{\hat{\omega}_N}^{\pm1}]$ and $\check{\mathcal C}=\mathbb C[K_{\hat{\omega}_N/N}^{\pm1}]$, and define
  $\check{U}_q(\mathfrak{gl}_n)=U_q(\mathfrak{gl}_n)\otimes_\mathcal C\check{\mathcal C}$. 
The algebra $\check U_q(\mathfrak{gl}_N)$ can be given a Hopf structure by insisting that $K_{\hat{\omega}_N/N}$ satisfies the same formulas for coproduct, counit, and antipode as an element $K\in U_q(\mathfrak{gl}_N)$ (as given in Section \ref{section:qea}).  

 Recall that the subalgebra $U^0(\mathfrak{sl}_N)$ of $U_q(\mathfrak{sl}_N)$ is extended to the subalgebra $\check U^0(\mathfrak{sl}_N)$ of the simply connected quantized enveloping algebra  $\check U_q(\mathfrak{sl}_N)$.  Moreover $\check U^0(\mathfrak{sl}_N)$  is equal  to  \begin{align*}\mathbb{C}(q)[ K_{\lambda}|\ \lambda\in
 P_N].\end{align*} This is the unital $\mathbb{C}(q)$-algebra 
 generated by the elements in square brackets described with set-builder notation.  It can be viewed as a Laurent polynomial ring with generators $K^{\pm 1}_{\omega_1}, \dots, K^{\pm 1}_{{\omega}_{N-1}}$. 

As explained in Section \ref{section:qea}, $U^0(\mathfrak{gl}_N)$ is the Laurent polynomial ring with generators $K^{\pm 1}_{{\epsilon}_i}$ for $i=1,\dots, N$.   It is straightforward to see that $U^0(\mathfrak{gl}_N)$ can be viewed as the Laurent polynomial ring in $K^{\pm 1}_{\hat{\omega}_i}, i=1,\dots, N$ where $\hat{\omega}_i=\epsilon_1+\dots + \epsilon_i$.
  is the $i^{th}$ fundamental partition.
By (\ref{omega}), $\omega_i=\hat{\omega}_i-{(i/N})\hat{\omega}_N$ for $i=1,\dots, N-1$. Thus 
\[
\mathbb Z(\hat{\omega}_N/N)+\sum_{i=1}^{N-1}
\mathbb Z\omega_i
=\mathbb Z(\hat{\omega}_N/N)+
\sum_{i=1}^{N-1}
\mathbb Z\hat{\omega}_i.
\] 
Hence $\check{U}^0(\mathfrak{gl}_N)$, which is generated over $\mathbb C(q)$ by $K_{\hat{\omega}_i}^{\pm 1}\otimes 1$ and $1\otimes K_{(1/N)\hat{\omega}_N}$ 
is isomorphic to 
\[
\mathbb{C}(q)[(K_{\hat{\omega}_i}K^{-1}_{(i/N)\hat{\omega}_N})^{\pm 1},(K_{(1/N)\hat{\omega}_N})^{\pm 1}|\  i=1,\cdots, N-1].\]
Here we have dropped the tensor symbol between the $K_{\hat{\omega}_i}$ and $K_{(1/N)\hat{\omega}_N}$ for better readability. 
 Since $\hat{\omega}_1,\quad \hat{\omega}_2,\quad \dots \quad \hat{\omega}_{N-1}$ are linearly independent, the elements $K_{\hat{\omega}_i}$ for $i=1,\cdots, N-1$ are algebraically independent.  Adding $(1/N)\hat{\omega}_N$ to the list keeps the linear independence property in place.  Hence $K_{\hat{\omega}_i}K^{-1}_{(i/N)\hat{\omega}_N}, i=1,\cdots, N-1$ are also algebraically independent.  Moreover, $\check{U}^0(\mathfrak{gl}_N)$ is a free module over $\mathbb{C}(q)[(K_{\hat{\omega}_i}K^{-1}_{(i/N)\hat{\omega}_N})^{\pm 1}|\ i=1,\cdots, N-1]$ with generators $(K_{(1/N)\hat{\omega}_N})^{\pm 1}$.

It follows from (\ref{omega}) that   the map $\zeta$  from $\check U_q(\mathfrak{sl}_N)$ to $\check U_q(\mathfrak{gl}_N)$ defined by
\begin{itemize}
\item $\zeta(K_{{\omega}_i})=K_{\hat{\omega}_i} K^{-1}_{(i/N)\hat{\omega}_N}$ 
\item $\zeta(E_i)=E_i\otimes 1$
\item $\zeta(F_i)=F_i\otimes 1$
\end{itemize}
for $i=1,\dots, N-1$ defines an injective algebra homomorphism.
 Since $K_{\hat{\omega}_N}$ is in the center of $U_q(\mathfrak{gl}_N)$, and, similarly, $K_{\hat{\omega}_N/N}$ is in the center of $\check U_q(\mathfrak{gl}_N)$, we have 
\begin{align}\label{Komegaformula}
({\rm ad}\ U_q(\mathfrak{sl}_N))\cdot uK_{\hat{\omega}_N/N}^s = [({\rm ad}\ U_q(\mathfrak{sl}_N))\cdot u]K_{\hat{\omega}_N/N}^s
\end{align}  for any  $s\in \mathbb{Z}$ and any $u\in U_q(\mathfrak{gl}_N)$. Moreover, by the discussion above, $\check U_q(\mathfrak{gl}_N)$ is a free $\check U_q(\mathfrak{sl}_N)$-module with basis $K^s_{\hat{\omega}_N/N}$, $s\in \mathbb{Z}$.  Hence 
\begin{align}\label{ddsdecomp} \check U_q(\mathfrak{gl}_N) =\bigoplus_{s\in \mathbb{Z}}\check U_q(\mathfrak{sl}_N)K_{\hat{\omega}_N/N}^s
\end{align}
In what follows we set 
\[
\check{\mathcal M}_N=\{2\lambda+s\hat{\omega}_N/N\ |\
\lambda\in w_0 P^+_N,\ 
s\in\mathbb Z\}
\]
and 
\[
\mathcal M_N=\{
2\lambda+s\hat{\omega}_N/N\ |\
\lambda\in w_0 \hat{\Lambda}^+_N,\ 
s\in\mathbb Z\}.
\]
More concretely, we have
\[
\check{\mathcal M}_N=\left\{\sum_{i=1}^Na_i\epsilon_i
\ \Big|\ 
a_i\in(1/N)\mathbb Z,\ a_{i+1}-a_{i}\in2\mathbb N,\ \sum_{i=1}^N a_i=0
\right\},
\]
and
\[
{\mathcal M}_N=\left\{\sum_{i=1}^Na_i\epsilon_i
\ \Big|\ 
a_i\in(1/N)\mathbb Z,\ a_{i+1}-a_{i}\in2\mathbb N
\right\}.
\]

 \begin{theorem}\label{local-finite}The locally finite subalgebra of $\check U_q(\mathfrak{gl}_N)$ admits the following decomposition into a direct sum of 
 $({\rm ad}\ U_q(\mathfrak{gl}_N))$-modules
 \begin{align}\label{equality1}\mathcal{F}(\check U_q(\mathfrak{gl}_N)) =\bigoplus_{\mu \in \mathcal M_N}({\rm ad}\ U_q(\mathfrak{gl}_N))\cdot K_{\mu}= \bigoplus_{\lambda\in w_0P^+_N,\ s \in {\mathbb{Z}}}[({\rm ad}\ U_q(\mathfrak{gl}_N) )\cdot K_{2\lambda}]K_{\hat{\omega}_N/N}^s.\end{align}
Similarly, the locally finite subalgebra of $\check U_q(\mathfrak{gl}_N\oplus\mathfrak{gl}_N)$ can be written as 
\begin{align}\label{equality2}
\mathcal{F}(\check U_q(\mathfrak{gl}_N\oplus\mathfrak{gl}_N))
&=\quad\bigoplus_{\lambda  \in (w_0P_N^+)\times (w_0P_N^+), s,s'\in \mathbb{Z}}
 ({\rm ad}\ U_q(\mathfrak{gl}_N\oplus\mathfrak{gl}_N))\cdot K_{2\lambda+s\hat{\omega}_N/N+s'\hat{\omega}_{2N}/N}
\cr&= \quad\bigoplus_{\lambda \in (w_0P_N^+)\times (w_0P_N^+), s,s'\in \mathbb{Z}}
({\rm ad}\ U_q(\mathfrak{gl}_N\oplus\mathfrak{gl}_N))\cdot K_{2\lambda}K_{\hat{\omega}_N/N}^sK_{\hat{\omega}_{2N}/N}^{s'}.
\end{align}
Moreover, the above equalities all hold with $({\rm ad} \ U_q(\mathfrak{gl}_N))$ replaced by 
$({\rm ad}\ \check U_q(\mathfrak{gl}_N))$, 
$({\rm ad}\ U_q(\mathfrak{sl}_N))$, or $({\rm ad}\ \check U_q(\mathfrak{sl}_N))$ and $({\rm ad}\ U_q(\mathfrak{gl}_N\oplus\mathfrak{gl}_N))$ replaced by  
$({\rm ad}\ \check U_q(\mathfrak{gl}_N\oplus \mathfrak{gl}_N))$, $({\rm ad}\ U_q(\mathfrak{sl}_N\oplus\mathfrak{sl}_N))$,  or $({\rm ad}\ \check U_q(\mathfrak{sl}_N\oplus\mathfrak{sl}_N))$
\end{theorem}
\begin{proof} 
Note that (\ref{equality2}) follows directly from (\ref{equality1}).  Hence we focus on proving the first equality (\ref{equality1}).  
Since $K_{\hat{\omega}_N}$ is central, the adjoint action respects the direct sum decomposition in  (\ref{ddsdecomp}). Hence
\begin{align}\label{Komegaformula2}\mathcal{F}(\check U_q(\mathfrak{gl}_N))=\bigoplus_{s\in \mathbb{Z}}\mathcal{F}(\check U_q(\mathfrak{sl}_N))K_{\hat{\omega}_N/N}^s
\end{align}
Thus (\ref{equality1}) follows from (\ref{Fdecomp}),(\ref{Komegaformula}), and (\ref{Komegaformula2}).

The final assertion follows from the facts that  the adjoint action of additional elements of the form $K_{\mu}$ in these Hopf algebras is semisimple with the same eigenspaces as that of the original Cartan subalgebra of $U_q(\mathfrak{sl}_N)$.  Thus the action of these extra elements preserve the decomposition into $({\rm ad}\ U_q(\mathfrak{sl}_N))$-modules.  \end{proof}

\subsection{The  ordinary enveloping algebra case}
\label{ordinary}
We use Theorem \ref{local-finite} in order to understand the locally finite part of the ordinary enveloping algebra $U_q(\mathfrak{gl}_N)$ and not just its simply connected version. By \cite{JL}, Lemma 6.1, $K_{\beta}$ admits a locally finite action if and only if $(\beta, \alpha_i)$ is a nonpositive even integer for $i=1, \dots, N-1$.  (Here, we are taking into account the slightly different definition of the quantized enveloping algebra used in \cite{JL}).  For the $\check U_q(\mathfrak{sl}_N)$ setting, this criteria translates to $K_{2\lambda}\in \mathcal{F}(\check U_q(\mathfrak{sl}_N))$ if and only if $\lambda \in P^+_N$.

Recall that $\hat\Lambda^+_N$ equals  the $\mathbb{N}$-linear span of the first $N-1$ partitions $\hat{\omega}_1,\dots, \hat{\omega}_{N-1}$ (see Section \ref{section:roots-and-weights})
and that $\hat\Lambda^+_N + \mathbb{N}\hat{\omega}_N=\Lambda^+_N$. Similarly, since $w_0\hat{\omega}_N=\hat{\omega}_N$, we have $w_0\hat\Lambda^+_N + \mathbb{N}\hat{\omega}_N=w_0\Lambda^+_N$.  Moreover, both of these can be viewed as  direct sums since $\hat{\omega}_N$ is linearly independent with the basis for $\hat{\Lambda}^+_N$.
We have
\begin{align*}
K_{\beta} \in \mathcal{F}(U_q(\mathfrak{gl}_N))
\end{align*}
and so \begin{align*}({\rm ad}\ U_q(\mathfrak{gl}_N))\cdot K_{\beta}\subset \mathcal{F}(U_q(\mathfrak{gl}_N))
\end{align*}
if and only if $\beta=2\gamma+s\hat{\omega}_N$ for some $\gamma  \in w_0{\hat{\Lambda}}^+_N$ and $s\in \mathbb{Z}$.  Here, we use $\mathbb{Z}$ instead of $\mathbb{N}$ since $K_{\tilde\omega_n}$ and its inverse are both in $U_q(\mathfrak{gl}_N)$.

Consider $\lambda =\sum_{i=1}^{N-1}\lambda_iw_0\omega_i$ and $\lambda'=\sum_{i=1}^{N-1}\lambda_iw_0\hat{\omega}_i$. We have the following $({\rm ad}\ U_q(\mathfrak{sl}_N))$-module isomorphism
\begin{align}\label{lambdaprime}
({\rm ad}\ U_q(\mathfrak{sl}_N))\cdot K_{2\lambda} \cong ({\rm ad}\ U_q(\mathfrak{sl}_N))\cdot K_{2\lambda'}
\end{align}
via the map sending $({\rm ad}\ a)\cdot K_{2\lambda}$ to $({\rm ad}\ a)\cdot K_{2\lambda'}$ for all $a\in U_q(\mathfrak{sl}_N)$. Note that both $K_{2\lambda}$ and $K_{2\lambda'}$ are elements of $\check U_q(\mathfrak{gl}_N)$ but they are not equal.  Indeed, they differ by a power of $K_{\hat{\omega}_N/N}$ which is a central element.  Thus we can ignore this difference when analyzing the adjoint module structure.  
Hence  the adjoint action of $U_q(\mathfrak{sl}_N)$ on $K_{2\lambda}$ for $\lambda = \sum_{i=1}^{N-1}\lambda_iw_0\omega_i \in w_0P^+_N$  agrees with the adjoint action of $U_q(\mathfrak{sl}_N)$ on $K_{\lambda'}$   for $\lambda' =\sum_{i=1}^{N-1}\lambda_iw_0\hat{\omega}_i\in w_0\hat{\Lambda}_N$.

   Using (\ref{lambdaprime}), we get an  isomorphism of $({\rm ad}\ U_q(\mathfrak{sl}_N))$-modules
\begin{align}\label{Fdecomp10}
 \bigoplus_{\lambda \in w_0P_N^+}({\rm ad}\ U_q(\mathfrak{sl}_N))\cdot K_{2\lambda}\cong  \bigoplus_{\lambda' \in -w_0\hat{\Lambda}_N^+}({\rm ad}\ U_q(\mathfrak{sl}_N))\cdot K_{2\lambda'}.
\end{align} 
By (\ref{Fdecomp}), the left hand side is just $\mathcal{F}(\check U_q(\mathfrak{sl}_N))$. On the other hand, the right hand side is contained in $\mathcal{F}(U_q(\mathfrak{gl}_N))$. (Moreover, this equality holds for  $({\rm ad}\ U_q(\mathfrak{sl}_N))$ replaced by $({\rm ad}\ U_q(\mathfrak{gl}_N))$ in (\ref{Fdecomp10}) since the result is the same vector space.) We can further enlarge the right hand side so that  it is isomorphic to $\mathcal{F}(U_q(\mathfrak{gl}_N))$. 
This uses the fact that  
$U^0(\mathfrak{gl}_N)$ is a free module over $\mathbb{C}(q)[K_{\lambda}^{\pm 1}|\ \lambda\in  w_0\hat{\Lambda}_N^+]$ with basis $K_{\hat\omega_N}^s, s\in\mathbb{N}$,
along with the  basic facts about the adjoint submodule $({\rm ad}\ U_q(\mathfrak{sl}_N))\cdot K_{2\lambda}$ of the locally finite part of $U_q(\mathfrak{sl}_N)$.

\begin{theorem}\label{ordin}
The locally finite subalgebra of $U_q(\mathfrak{gl}_N)$ admits the following decomposition into a direct sum of $({\rm ad}\ U_q(\mathfrak{gl}_N))$-modules
 \begin{align}\label{equalityoc}\mathcal{F}( U_q(\mathfrak{gl}_N)) &=\bigoplus_{\mu\in\check{\mathcal M}_N}({\rm ad}\ U_q(\mathfrak{gl}_N))\cdot K_{\mu}=\bigoplus_{\lambda \in w_0\hat{\Lambda}_N^+,\ s\in \mathbb{Z}}[({\rm ad}\ U_q(\mathfrak{gl}_N))\cdot K_{2\lambda}]K^s_{\hat{\omega}_N}.
 \end{align}
 Similarly, the locally finite subalgebra of $ U_q(\mathfrak{gl}_N\oplus\mathfrak{gl}_N)$ can be written as 
\begin{align}\label{equalityoc2}
\mathcal{F}(U_q(\mathfrak{gl}_N\oplus\mathfrak{gl}_N))&=\bigoplus_{\lambda  \in (w_0\hat{\Lambda}_N^+)
\times  (w_0\hat{\Lambda}_N^+), s,s'\in \mathbb{Z}}
 ({\rm ad}\ U_q(\mathfrak{gl}_N\oplus\mathfrak{gl}_N))\cdot K_{2\lambda+s\hat{\omega}_N+s'{\hat{\omega}_{2N}}}
\cr&=\bigoplus_{\lambda  \in (w_0\hat{\Lambda}_N^+)
\times  (w_0\hat{\Lambda}_N^+), s,s'\in \mathbb{Z}}
 [({\rm ad}\ U_q(\mathfrak{gl}_N\oplus\mathfrak{gl}_N))\cdot K_{2\lambda}]K^s_{\hat{\omega}_N}K^{s'}_{\hat{\omega}_{2N}}
\end{align}

\end{theorem}

\begin{proof} Note that (\ref{equalityoc2}) follows directly from (\ref{equalityoc}). Moreover, the third (rightmost) equality in (\ref{equalityoc}) follows from (\ref{Komegaformula}).
Hence, we establish the theorem by proving the first equality of (\ref{equalityoc}).

Since $U_q(\mathfrak{gl}_N)$ is a Hopf subalgebra of $\check U_q(\mathfrak{gl}_N)$, it follows that \begin{align*}\mathcal{F}(U_q(\mathfrak{gl}_N))=U_q(\mathfrak{gl}_N)\cap 
\mathcal{F}(\check U_q(\mathfrak{gl}_N)).\end{align*} Using formula (\ref{omega}), for $\lambda\in w_0P^+_N$, we have   \begin{align*}
\lambda = \sum_{i=1}^{N-1}\lambda_iw_0\omega_i =\sum_{i=1}^{N-1}\lambda_iw_0\omega_i-(1/N)\hat{\omega}_N\end{align*}
Thus (\ref{equality1}) of Theorem  (\ref{local-finite}) is equivalent to 
\begin{align*}
\mathcal{F}(\check U_q(\mathfrak{gl}_N))&=\bigoplus_{\lambda \in w_0\hat{\Lambda}^+_N,\ s\in \mathbb{Z}}({\rm ad}\ U_q(\mathfrak{gl}_N))\cdot K_{2\lambda+(s/N)\hat{\omega}_N}.
\end{align*}
If $(s/N)\in \mathbb{Z}$, then $K_{2\lambda+(s/N) \hat{\omega}_N}\in U_q(\mathfrak{gl}_N)$ and hence $({\rm ad}\ U_q(\mathfrak{gl}_N))\cdot K_{2\lambda+(s/N)\hat{\omega}_N}\subseteq\mathcal{F}(U_q(\mathfrak{gl}_N).$ Thus \begin{align*} \bigoplus_{\lambda  \in w_0\hat{\Lambda}_N^+, s\in \mathbb{Z}}({\rm ad}\ U_q(\mathfrak{sl}_N))\cdot K_{2\lambda+s\hat{\omega}_N}\subseteq \mathcal{F}(U_q(\mathfrak{gl}_N)).\end{align*}  On the other hand, if $(s/N)\notin \mathbb{Z}$ then $K_{2\lambda+(s/N) \hat{\omega}_N}\notin U_q(\mathfrak{gl}_N).$ Therefore, we have a strict inclusion
\begin{align}\label{inclusion}({\rm ad}\ U_q(\mathfrak{gl}_N))\cdot K_{2\lambda+(s/N)\hat{\omega}_N}\cap \mathcal{F}(U_q(\mathfrak{gl}_N))\subsetneq ({\rm ad}\ U_q(\mathfrak{gl}_N))\cdot K_{2\lambda+(s/N)\hat{\omega}_N}.
\end{align}
As explained in  \cite{L2002} (see \cite{HK}, Theorem 3.9 and Corollary 3.10), $ ({\rm ad}\ U_q(\mathfrak{sl}_N))\cdot K_{2\lambda}$ is simple as an ad-invariant  left coideal for each $\lambda\in P^+_N$ and hence so is $ ({\rm ad}\ U_q(\mathfrak{gl}_N))\cdot K_{2\lambda}$.  By (\ref{lambdaprime}) and (\ref{Komegaformula}), $ ({\rm ad}\ U_q(\mathfrak{gl}_N))\cdot K_{2\lambda'+(s/N)\hat{\omega}_N}$ is isomorphic to $ ({\rm ad}\ U_q(\mathfrak{gl}_N))\cdot K_{2\lambda}$ as an ad-invariant left coideal where $\lambda =\sum_{i=1}^{N-1}\lambda_iw_0\omega_i$ and $\lambda' =\sum_{i=1}^{N-1}\lambda_iw_0\hat{\omega}_i$. In particular, for each $\lambda'\in w_0\hat{\Lambda}_N^+$ and $s\in \mathbb{Z}$, $ ({\rm ad}\ U_q(\mathfrak{gl}_N))\cdot K_{2\lambda'+(s/N)\hat{\omega}_N}$ is also a simple ad-invariant left coideal.  Since these ad-invariant left coideals are simple and holds for $s/N\notin {\mathbb{Z}}$, this  guarantees  that  the left hand side of (\ref{inclusion}) is equal to zero.  The theorem follows.
\end{proof}

\subsection{A special subalgebra}\label{subsection:special}

Let $U^2_q(\mathfrak{gl}_N)$ denote the subalgebra of $U_q(\mathfrak{gl}_N)$ generated by $$E_i, F_iK_i,K_i^{2} {\rm \ for\ } i=1, \dots, N-1 {\rm\  and\ }K_{2w_0\hat{\omega}_j}
{\rm \ for\ } j=1, \dots, N.$$ Recall (Section \ref{section:roots-and-weights}) that $Q^+_N$ is the $\mathbb{N}$-linear span of the positive simple roots. Note that 
 $E_iF_iK_i - q^{-2}F_iK_iE_i $ is a scalar multiple of $(K_i^2-1)$. Set $U^0_2(\mathfrak{gl}_N)=U^2_q(\mathfrak{gl}_N)\cap U^0(\mathfrak{gl}_N)$.  It follows that $U_2^0(\mathfrak{gl}_N)$ equals
the polynomial  ring \begin{align}\label{A2cartan}U^0_2(\mathfrak{gl}_N)=\mathbb{C}(q)[K_i^2, K_{2w_0\hat{\omega}_j}|\  i=1,\dots, N-1,j=1,\dots, N]=\mathbb{C}(q)[K_{2\lambda}|\ \lambda\in Q^+_N+w_0{\Lambda}^+_N]\end{align}  
 One checks from the formulas for the comultiplication $\Delta$ given in Section \ref{section:qea} that $U^2_q(\mathfrak{gl}_N)$ is a left coideal subalgebra of $U_q(\mathfrak{gl}_N)$.
Moreover, the formulas for the adjoint action (Section \ref{section:lfs}) ensure that $U^2_q(\mathfrak{gl}_N)$ is an $({\rm ad}\ U_q(\mathfrak{gl}_N))$-submodule of $U_q(\mathfrak{gl}_N)$.   It further follows that 
\begin{align} \label{u2expansion} U^2_q(\mathfrak{gl}_N) \subseteq \bigoplus_{\lambda\in Q^+_N+w_0\Lambda_N^+}U^+G^-K_{2\lambda}
\end{align}
where $G^-$ is the subalgebra of $U_q(\mathfrak{gl}_N)$ generated by $F_iK_i$ for $ i=1,\dots, N-1$.  

Since $Q_N$ is the root lattice, it has a nontrivial intersection with the dominant integral weights $P^+_N$, as well as their image, $w_0P^+_N$, under $w_0$.  However, the same does not hold for the intersection with $w_0\Lambda^+_N$.  Indeed, for $\lambda\in Q_N\cap w_0P^+_N$, we can write $\lambda = \lambda' + s{\hat{\omega}}_N$ for an appropriate $\lambda'\in w_0\hat{\Lambda}^+_N$ and $s\in {(1/N)}\mathbb{N}$ based on equation (\ref{omega}).  Multiplying $\lambda$ by $N$, we see that $N\lambda \in w_0\Lambda^+_N$ but $\hat{\omega}_N$ is not in $Q_N$.  Thus  $Q_N\cap w_0\Lambda^+_N=0$.

Set $\mathcal{F}(U^2_q(\mathfrak{gl}_N))$ equal to the locally finite part of $U^2_q(\mathfrak{gl}_N)$.

\begin{lemma}\label{lemma:U2result} The space $\mathcal{F}(U^2_q(\mathfrak{gl}_N))$ equals the intersection $U^2_q(\mathfrak{gl}_N)\cap \mathcal{F}(U_q(\mathfrak{gl}_N))$.  Hence  $\mathcal{F}(U^2_q(\mathfrak{gl}_N))$ is a  left coideal subalgebra of $\mathcal{F}(U_q(\mathfrak{gl}_N))$ and has the direct sum decomposition
\begin{align}\label{form1}
\mathcal{F}( U^2_q(\mathfrak{gl}_N))&=\bigoplus_{\lambda \in w_0{\Lambda}^+_N}({\rm ad}\ U_q(\mathfrak{gl}_N))\cdot K_{2\lambda} = \bigoplus_{\lambda \in w_0{\hat\Lambda}^+_N, s\in {\mathbb{N}}}({\rm ad}\ U_q(\mathfrak{gl}_N))\cdot K_{2\lambda +2s\hat{\omega}_N}.
\end{align}
\end{lemma}
\begin{proof}
Since $U^2_q(\mathfrak{gl}_N)$ is a subalgebra of $U_q(\mathfrak{gl}_N)$, it follows that the locally finite part of  $U^2_q(\mathfrak{gl}_N)$ is the intersection of $U^2_q(\mathfrak{gl}_N)$ with the locally finite part of  $U_q(\mathfrak{gl}_N)$. This proves the first assertion. The second equality follows from the fact that $w_0\Lambda_N^+= w_0\hat{\Lambda}_N^+ + 2\mathbb{N}\hat{\omega}_N$.  

Since both  $U^2_q(\mathfrak{gl}_N)$ and $ \mathcal{F}( U_q(\mathfrak{gl}_N))$ are left coideal subalgebras of  $U_q(\mathfrak{gl}_N)$, so is their intersection.  Hence we can write $ \mathcal{F}( U^2_q(\mathfrak{gl}_N))=U^2_q(\mathfrak{gl}_N)\cap \mathcal{F}( U_q(\mathfrak{gl}_N))$  as a direct sum of simple ad-invariant left coideals.  As explained in the proof of Theorem \ref{ordin}, these 
simple ad-invariant left  coideals take the form $$({\rm ad}\ U_q(\mathfrak{gl}_N))K_{2\lambda+s\hat{\omega}_N}$$ where $\lambda \in \hat{\Lambda}_N^+$ and $s\in \mathbb{N}$. Moreover,
arguing as in the proof of Theorem \ref{ordin},  
\begin{align*}
({\rm ad}\ U_q(\mathfrak{gl}_N))K_{2\lambda+s\hat{\omega}_N}\subset \mathcal{F}(U^2_q(\mathfrak{gl}_N))
\end{align*} if and only if $K_{2\lambda+s\hat{\omega}_N}\in \mathcal{F}(U^2_q(\mathfrak{gl}_N)).$ Thus $K_{2\lambda+s\hat{\omega}_N}$ must be an element of $U^2_q(\mathfrak{gl}_N)$ and generates a locally finite ad-invariant simple module  by applying $({\rm ad}\ U_q(\mathfrak{gl}_N))$ as explained at the beginning of Section \ref{ordinary}. It follows that $2\lambda +s\hat{\omega}_N \in  2w_0\hat{\Lambda}_N^++2\mathbb{N}\hat{\omega}_N=2w_0\Lambda^+_N$. \end{proof}

 Define $U^2_q(\mathfrak{gl}_N\oplus \mathfrak{gl}_N)$ as the subalgebra of $U_q(\mathfrak{gl}_N\oplus \mathfrak{gl}_N)$ generated by $E_i, F_iK_i, K_i^2,$ for $ i=1,\dots, N-1$ and 
$ i=N+1, \cdots,  2N-1, K_{2\hat{\omega}_N},$
and $K_{2\hat{\omega}_{2N}}$. Note that $U^2_q(\mathfrak{gl}_N\oplus \mathfrak{gl}_N)$ can be identified with $U^2_q(\mathfrak{gl}_N)\otimes U^2_q( \mathfrak{gl}_N).$ Using this identification, Lemma \ref{lemma:U2result} ensures that  analogous results holds for $U^2_q(\mathfrak{gl}_N\oplus \mathfrak{gl}_N)$.

\subsection{Mapping to the quantum Weyl algebra}\label{subsection:mapping}

Set $U^2_q(\mathfrak{g})$ equal to $U^2_q(\mathfrak{gl}_n)$ in Type AI, 
$U^2_q(\mathfrak{gl}_{2n})$ in Type AII, 
and $U^2_q(\mathfrak{gl}_n\oplus \mathfrak{gl}_n)$ in the diagonal case. For Type AII, note that $K_{2\omega_N}= K_{4\eta_N}$.  This would suggest that we need a slightly larger algebra in order to  get the correct image under the restricted Harish-Chandra map in Section \ref{section:center}.   However, the arguments in this case show that 
$K_{2\eta_N}=K_{\omega_N}$ is also in this algebra. 
By Lemma \ref{lemma:U2result} and subsequent discussion, we have that $\mathcal{F}(U^2_q(\mathfrak{g}))= \mathcal{F}(U_q(\mathfrak{g}))\cap U^2_q(\mathfrak{g})$
in all three cases.

By Proposition \ref{prop:orthogonal}, $\mathscr{P}_{\theta}$ is a faithful $\mathscr{PD}_{\theta}$-module.  It follows that $\phi$ is an injective algebra map and so as algebras, $\mathscr{PD}_{\theta}$ is isomorphic to 
$\phi(\mathscr{PD}_{\theta})$.  This allows us to define an algebra map directly from $\mathcal{F}(U^2_q(\mathfrak{g}))$ to $\mathscr{PD}_{\theta}$ that is compatible with the action on $\mathscr{P}_{\theta}$.  
In the discussion below, we directly identify the image under $\phi$ with an element of $\mathscr{PD}_{\theta}$, thus dropping the notation $\phi$ going forward.

\begin{theorem} \label{theorem:lfp}The image $\psi(\mathcal{F}(U^2_q(\mathfrak{g}))$ in ${\rm End}\ \mathscr{P}_{\theta}$ is an  $({\rm ad}\ U_q(\mathfrak{g}))$-submodule algebra of $\mathscr{PD}_{\theta}$.
\end{theorem}
\begin{proof} By Lemma \ref{lemma:epsilon} and Lemma \ref{lemma:K_epsilon}, $\psi(K)\in \mathscr{PD}_{\theta}$ for $K = K_{2\epsilon_i+\cdots +2\epsilon_N}$  for $i=1, \dots, N$ where $N=n$  in Type AI and $N=2n$  in Type AII.  For the diagonal 
case, $\psi(K)\in \mathscr{PD}_{\theta}$
 for $K = K_{2\epsilon_i+\cdots +2\epsilon_n}$ and $K=K_{2\epsilon_{i+n}+\cdots + 2\epsilon_{2n}}$ for $i=1, \dots, n$.  The proof now follows from the fact  that for each of the 
three families, 
$\mathcal{F}(U^2_q(\mathfrak{g}))$ is an $({\rm ad}\ U_q(\mathfrak{g}))$-module generated by these elements, $\mathcal{F}(U^2_q(\mathfrak{g}))$ is  an algebra (this is just Lemma \ref{lemma:U2result}), that $\psi$ is an $({\rm ad}\ U_q(\mathfrak{g}))$-module algebra map, and that, by Proposition \ref{prop:locally-finite}, $\mathscr{PD}_{\theta}$ is an $({\rm ad}\ U_q(\mathfrak{g}))$-submodule algebra of ${\rm End}\ \mathscr{P}_{\theta}$. 
\end{proof}

We have the following consequence of the previous theorem which relies on Proposition \ref{prop:locally-finite} relating $U_q(\mathfrak{g})$-module structures.

\begin{corollary} \label{corollary:upsilon} There is a unique $U_q(\mathfrak{g})$-module algebra homomorphism $\Upsilon$ from $\mathcal{F}(U^2_q(\mathfrak{g}))$ to $\mathscr{PD}_{\theta}$ such that $$\Upsilon(a) = \psi(a){\quad {\rm and}\quad }\Upsilon(({\rm ad}\ u)\cdot a) = u\cdot \Upsilon(a)$$ for all $a\in  \mathcal{F}(U^2_q(\mathfrak{g}))$ and $u\in U_q(\mathfrak{g})$
where the module structure on $\mathcal{F}(U^2_q(\mathfrak{g}))$ is defined by the (left) adjoint action and the module structure on $\mathscr{PD}_{\theta}$ comes from the left action.
\end{corollary}

\section{The center of  $U_q(\mathfrak{g})$ and related algebras}
\label{section:center}
\subsection{Basis for the center}\label{section:basis-center}
We recall here basic properties of  $U_q(\mathfrak{sl}_N)$ and its center (a good reference is \cite{Jo}) and then transfer these results to other settings of interest. 
 By Section 7.1 of (\cite{Jo}), each $({\rm ad}\ U_q(\mathfrak{sl}_N))$-module of the form
$({\rm ad}\ U_q(\mathfrak{sl}_N))\cdot K_{2\mu}$ for $\mu\in -P_N^+$ contains a unique (up to nonzero scalar multiple) central element which we denote by $z_{2\mu}$.  Moreover, 
the set $\{z_{2\mu}|\ \mu\in -P^+_N\}$ forms a basis for the center $Z(\check U_q(\mathfrak{sl}_N))$ of $\check U_q(\mathfrak{sl}_N)$.  This extends easily to 
$\check U_q(\mathfrak{sl}_N\oplus \mathfrak{sl}_N)$ with basis for the center equal to  $\{z_{2\mu}|\ \mu\in -(P^+_N\times P^+_N)\}$.

The arguments in \cite{Jo} also apply to the $({\rm ad}\ U_q(\mathfrak{sl}_N))$-modules of the form
$({\rm ad}\ U_q(\mathfrak{sl}_N))\cdot K_{2\mu+c(\hat{\omega}_N/N)}$ where $\mu\in w_0\Lambda^+_N$ and $c\in \mathbb{Z}$. In particular, the $({\rm ad}\ U_q(\mathfrak{sl}_N))$-module 
$({\rm ad}\ U_q(\mathfrak{sl}_N))\cdot K_{2\mu+c\hat{\omega}_N/N}$ for $\mu\in w_0\Lambda^+_N$ and $c\in \mathbb{Z}$ contains a unique (up to nonzero scalar multiple) central element of $\check U_q(\mathfrak{gl}_N)$ which we denote by $z_{2\mu+c(1/N)\hat{\omega}_N}$.
  Moreover, it follows from (\ref{Komegaformula}) that 
  \begin{align*}z_{2\mu+c(1/N)\hat{\omega}_N} = z_{2\mu}K_{(\hat{\omega}_N)/N}^c
  \end{align*}
  for all $\mu\in w_0\Lambda_N^+$ and $c\in \mathbb{Z}$.   
  Hence,  the decomposition of the locally finite subalgebra in Theorem \ref{local-finite} ensures that the set $\{z_{2\mu}K_{(\hat{\omega}_N)/N}^c|\ \mu\in w_0\Lambda^+_N, c\in \mathbb{Z}\}$ forms a basis for the center, $Z(\check U_q(\mathfrak{gl}_N))$, of $\check U_q(\mathfrak{gl}_N).$

Recall the decomposition of the locally finite part of $U^2_q(\mathfrak{g})$ given in Lemma \ref{lemma:U2result}. The arguments in \cite{Jo} also apply to the $({\rm ad}\ U_q(\mathfrak{sl}_N))$-modules of the form
$({\rm ad}\ U_q(\mathfrak{sl}_N))\cdot K_{2\mu+2c\hat{\omega}_N}$ where $\mu\in w_0\hat{\Lambda}^+_N$ and $c\in \mathbb{N}$. In particular, the $({\rm ad}\ U_q(\mathfrak{sl}_N))$-module 
$({\rm ad}\ U_q(\mathfrak{sl}_N))\cdot K_{2\mu+2c\hat{\omega}_N}$ for $\mu\in w_0\hat{\Lambda}^+_N$ and $c\in \mathbb{N}$ contains a unique (up to nonzero scalar multiple) central element of $U_q(\mathfrak{gl}_N)$ which we denote by $z_{2\mu+2c\hat{\omega}_N}$.
  Moreover, it follows from (\ref{Komegaformula}) that 
  \begin{align*}z_{2\mu+2c\hat{\omega}_N} = z_{2\mu}K_{\hat{\omega}_N}^{2c}
  \end{align*}
  for all $\mu\in w_0\hat{\Lambda}_N^+$ and $c\in \mathbb{N}$.   
  Hence,  the decomposition of the locally finite subalgebra of $U^2_q(\mathfrak{g})$ in Lemma \ref{lemma:U2result} ensures that the set $\{z_{2\mu}K_{\hat{\omega}_N}^{2c}|\ \mu\in w_0\hat{\Lambda}^+_N, c\in \mathbb{N}\}$ forms a basis for the center, $Z(U^2_q(\mathfrak{g}))$, of $U^2_q(\mathfrak{g})$ in Types AI and AII.  In the diagonal setting, the basis looks like 
$\{z_{2\mu}K_{\hat{\omega}_n}^{2c}K_{\hat{\omega}_{2n}}^{2c'}\ \mu\in (w_0\hat{\Lambda}^+_N)\times (w_0\hat{\Lambda}^+_N), c,c'\in \mathbb{N}\}$.  We focus here on Types AI and AII and use these observations to establish the same results in the diagonal case.     Indeed, the results described so far in this section extend in a straightforward manner to the center of $U_q(\mathfrak{gl}_N\oplus \mathfrak{gl}_N)$
  with $P^+_N$ replaced by $P^+_N\times P^+_N$, $w_0\hat{\Lambda}^+_N$ replaced by $w_0\hat{\Lambda}^+_N\times w_0\hat{\Lambda}^+_N$, $c$ replaced by a pair of
  natural numbers
   $c$ and $c'$, and $K_{\hat{\omega}_N}^c$ replaced by 
  $K_{\hat{\omega}_N}^cK_{\hat{\omega}_{2N}}^{c'}$.

The description of the basis for the center of $\check U_q(\mathfrak{gl}_N)$  combined with  Lemma \ref{lemma:U2result} and the subsequent discussion implies that 
$Z(U^2_q(\mathfrak{g})) = \sum_{\lambda} \mathbb{C}(q)z_{2\lambda}$
where $\lambda$ runs over elements in $w_0\Lambda^+_n$ in Type AI,   $w_0\Lambda^+_{2n}$ in Type AII and 
$w_0\Lambda^+_n\times w_0\Lambda^+_n $ in the diagonal setting.

\subsection{Harish-Chandra maps}\label{section:HCMs}
We start with the   Harish-Chandra map defined for the simply connected quantized enveloping algebra $\check{U}_q(\mathfrak{sl}_N)$, a projection map based on a direct sum decomposition  in \cite{L}, Chapter 3.  In particular,  the Harish-Chandra map,  $\varphi_{HC}$ of the simply connected version $\check U_q(\mathfrak{sl}_N)$ 
is the projection onto the first component  $\check{U}^0(\mathfrak{sl}_N)$ of the direct sum decomposition 
\begin{align}\label{HC1} \check{U}_q(\mathfrak{sl}_N) = \check{U}^0(\mathfrak{sl}_N)\oplus (G^-_+\check{U}_q(\mathfrak{sl}_N) + \check{U}_q(\mathfrak{sl}_N)U^+_+).
\end{align}
where $\check U^0(\mathfrak{sl}_N), U^+_+$ are defined in Section \ref{section:qea} and $G^-$ is the subalgebra  generated by $F_iK_i, i=1,\dots, N-1$ (with  $G^-_+$ its augmentation ideal)   as defined in Section \ref{subsection:special}.
(This is just \cite{L}, (3.3) where the map $\varphi_{HC}$ is called $\mathcal{P}$.)

Using the  simply connected version of $U_q(\mathfrak{gl}_N$)   introduced in Section \ref{subsection:simply-conn}, we can extend the above decomposition to 
\begin{align}\label{HC2} \check{U}_q(\mathfrak{gl}_N) = \check{U}^0(\mathfrak{gl}_N)\oplus (G^-_+\check{U}_q(\mathfrak{gl}_N) + \check{U}_q(\mathfrak{gl}_N)U^+_+)
\end{align} where $\check{U}^0(\mathfrak{gl}_N)$ is equal to the Laurent polynomial ring $ \mathbb{C}(q)[(K_{\hat{\omega}_i}K^{-1}_{(i/N)\hat{\omega}_N})^{\pm 1}|\ i=1,\cdots, N-1]$.
 Note that this direct sum decomposition restricts to a direct sum decomposition on subalgebras of $\check U_q(\mathfrak{gl}_N)$ including ordinary quantized enveloping algebra $U_q(\mathfrak{gl}_N)$, and  more importantly, the special algebra  $U^2_q(\mathfrak{g})$  introduced in Section \ref{subsection:special}:
 \begin{align}\label{decomp7} {U}^2_q(\mathfrak{g}) = {U}_2^0(\mathfrak{g})\oplus (G^-_+U^2_q(\mathfrak{g}) + {U}^2_q(\mathfrak{g})U^+_+).
\end{align} As explained in Section \ref{subsection:special},  ${U}_2^0(\mathfrak{g})=\mathbb{C}(q)[K_{2\lambda}|\ \lambda\in Q^+_N+w_0{\Lambda}^+_N]$.  Note also that it is straightforward to write similar direct sum decompositions for the analogous algebras in the diagonal setting.

In each of the  decompositions (\ref{HC1}), (\ref{HC2}), (\ref{decomp7}) for the algebra on the left side, we call the first summand its Cartan subalgebra.
 We will denote the projection onto the Cartan subalgebra for each decomposition as the Harish-Chandra map $\varphi_{HC}$. We are using the same notation as for the first projection for  the well studied $\check U_q(\mathfrak{sl}_N)$ because all these decompositions lead to compatible projection maps due to obvious inclusions. 
 
Recall that the restricted root system $\Sigma$ is the root system with set of simple roots $\alpha_1^{\Sigma},\dots, \alpha^{\Sigma}_{n-1}$  and fundamental weights  $\eta_1,\dots, \eta_{n-1}$ as described in detail in Section \ref{section:restricted-root-system}.  
Set $$\check{\mathcal{A} }= \{K_{u}|\mu \in P_{\Sigma}\}.$$ 
This group is defined in Section 3 of \cite{L} (see the top of page 24 
in \cite{L}).  Write $\mathbb{C}(q)[\check {\mathcal{A}}]$
 for the group algebra of $\check{\mathcal{A}}$. 
The Cartan subalgebra $\check{U}^0_q(\mathfrak{sl}_N)$ of $\check U_q(\mathfrak{sl}_N)$ is a subset of the following direct sum decomposition (\cite{L}, (3.5))
\begin{align}\label{rhcmap}
\check{U}^0_q(\mathfrak{sl}_N)\subseteq \mathbb{C}(q)[\check{\mathcal{A}}]\oplus \mathbb{C}(q)[\check{\mathcal{A}}]\mathbb{C}(q)[\check{T}_{\theta}]_+
\end{align}
where $\mathbb{C}(q)[\check{T}_{\theta}]$ is the group algebra associated to the group $\{K_{(\mu +\theta(\mu))/2)}|\ \mu\in P_N\}$. As in \cite{L} (right after (3.5)), let $\tilde{\mathcal{P}} $ be the projection of $\check{U}^0_q(\mathfrak{sl}_N)$ onto  $\mathbb{C}(q)[\check{\mathcal{A}}]$.  

Note that this projection map can be described in the following 
 alternative way:
 \begin{align} \label{tilde-mu}\sum_{\mu}c_{\mu}K_{\mu} \mapsto \sum_{\mu}c_{\mu}K_{\tilde{\mu}}
 \end{align}
 where $\mu$ are elements  in $P^+_N$, the $c_{\mu}$ are scalars and $\tilde{\mu}$ is the restricted weight defined by $\mu$. 
 
 Let $\tilde{\varphi}_{HC}$ denote the projection of $\check U_q(\mathfrak{sl}_N)$ onto $\mathbb{C}(q)[\check{\mathcal{A}}]$ defined by taking the composition of the Harish-Chandra map $\varphi_{HC}$ with  the projection $\tilde{\mathcal{P}}$.  
Note that $\tilde{\varphi}_{HC}$ corresponds to $\tilde{\mathcal{P}}\circ\mathcal{P}$ of \cite{L}  (see beginning of \cite{L} Section 6, especially Lemma 6.1).    The map $\mathcal{P}$ is the Harish-Chandra map for 
the simply connected quantized enveloping algebra $\check{U}_q(\mathfrak{sl}_N)$ in this reference.

It should also be noted that a more general restricted Harish-Chandra map is defined in \cite{L} using a quantum analog of the Iwasawa decomposition.  Upon restriction to the center, this more general restricted Harish-Chandra map agrees with $\tilde{\varphi}_{HC}$ (see \cite{L}, Lemma 6.1).  Since we are only worried about the image of central elements under the restricted Harish-Chandra map, we just use $\tilde{\varphi}_{HC}$ and do not consider the more general version.

Using the  simply connected version of $U_q(\mathfrak{gl}_N$)   introduced in Section \ref{subsection:simply-conn}, we can extend the above decomposition to 
\begin{align}\label{rhcmap5}
\check{U}^0_q(\mathfrak{gl}_N)\subseteq  
\mathbb{C}(q)[\check{\mathcal{A}}(\mathfrak{gl}_N)]
\oplus \mathbb{C}(q)[\check{\mathcal{A}}(\mathfrak{gl}_N)]
\mathbb{C}(q)[\check{T}(\mathfrak{gl}_N)_{\theta}]_+
\end{align}
where $\check{\mathcal{A}}(\mathfrak{gl}_N)$ equals the group generated by 
\begin{itemize} 
\item  $K_{\hat{\eta}}$ for $\eta\in \hat{\Lambda}^+_{\Sigma}$ and $K_{(1/N)\hat{\eta}_n}$ for Type AI 
\item $K_{\hat{\eta}}$ for $\eta\in \hat{\Lambda}^+_{\Sigma}$ and $K_{(1/N)\hat{\eta}_{2n}}$ for Type AII
 \item $K_{\hat{\eta}}$ for $\eta\in \hat{\Lambda}^+_{\Sigma}$, $K_{(1/N)\hat{\eta}_n}$,  $K_{(1/N)\hat{\eta}_{2n}}$ in the diagonal case.
 \end{itemize}
 Also $\check{T}(\mathfrak{gl}_N)_{\theta}=\{K_{(\mu+\theta(\mu))/2}| \mu\in Q^++w_0\Lambda_N^+\}$.  
Let $\tilde{\mathcal{P}}$ denote the projection of the Cartan subalgebra $ \check{U}^0_q(\mathfrak{gl}_N)$ onto $\mathbb{C}(q)[\check{\mathcal{A}}]$ using the direct sum decomposition (\ref{rhcmap5}) and set $\tilde{\varphi}_{HC}=\tilde{\mathcal{P}}\circ \varphi_{HC}$.  Recall that $\check U_q(\mathfrak{gl}_N)$ is a free module over $\check U_q(\mathfrak{sl}_N)$ (see Section \ref{subsection:simply-conn}) and the analogous result holds in the diagonal case.  So $\tilde{\mathcal{P}}$, $\varphi_{HC}$, and $\tilde{\varphi}_{HC}$ restrict to projections by the same name for  $\check U_q(\mathfrak{sl}_N)$ and $\check U_q(\mathfrak{sl}_N\times \mathfrak{sl}_N)$.

We can define a monoid that leads to a new polynomial ring. It is derived from the Cartan subalgebra of $U^2_q(\mathfrak{g})$ 
(notation  from the beginning of Section \ref{subsection:mapping}). 
This algebra will be the main focus of the current section.  Set 
\begin{align*}\mathcal{A}_2=\{K_{2\tilde\lambda}|\  \lambda \in Q^+_N+w_0\Lambda_N^+\}
\end{align*}
in Types AI and AII and 
\begin{align*}\mathcal{A}_2=\{K_{2\tilde\lambda}|\ \lambda \in (Q^+_N\times Q^+_N)+(w_0\Lambda_N^+\times w_0\Lambda_N^+)\}
\end{align*}
in the diagonal case.
 Another way to look at $\mathcal{A}_2$ is to view it as a monoid generated by restricted root system partitions  and the images of the $K_i^2$ under restriction as well.  This leads to the polynomial ring $\mathbb{C}(q)[\mathcal{A}_2]$, which by the description of the restricted weights for each type in Section \ref{section:restricted-root-system}, satisfies
 \begin{itemize} 
\item $\mathbb{C}(q)[\mathcal{A}_2]=\mathbb{C}(q)[K_{2w_0\hat{\eta}_i},K_{2\tilde{\alpha}_i}|\  i=1,\dots, n-1]$ for Type AI.
 \item $\mathbb{C}(q)[\mathcal{A}_2]=\mathbb{C}(q)[K_{2w_0\hat{\eta}_i},K_{2\tilde{\alpha}_{2i}}|\ i=1,\dots, n-1]$ for Type AII. 
 \item  $\mathbb{C}(q)[\mathcal{A}_2]=\mathbb{C}(q)[K_{2w_0\hat{\eta}_i},K_{2\tilde{\alpha}_i}|\ i=1,\dots, n-1]$ in the diagonal case.
\end{itemize}

Set $(T_{\theta})_2 =\{K_{\lambda+\theta(\lambda)}|\ \lambda \in Q^+_N+w_0\Lambda_N^+\} $ in Type AI and Type AII and 
$(T_{\theta})_2 =\{K_{\lambda+\theta(\lambda)}|\  \lambda \in (Q^+_N\times Q^+_N)+(w_0\Lambda_N^+\times w_0\Lambda_N^+)\} $ in the diagonal case.
Recall that $\tilde{\lambda}=(\lambda-\theta(\lambda))/2$ and so $K_{2\tilde{\lambda}}K_{(\lambda+\theta(\lambda))}=K_{2\lambda}$.  It follows that 
$K_{2\lambda}\in \mathcal{A}_2(T_{\theta})_2$. Hence, we have the following inclusion 
\begin{align}
\label{rhcmap13}
{U}^0_q(\mathfrak{g})\subseteq \mathbb{C}(q)[\mathcal{A}_2]\oplus \mathbb{C}(q)[\mathcal{A}_2]\mathbb{C}(q)[({T}_{\theta})_2]_+.
\end{align}
We define versions of  $\tilde{\mathcal{P}}$
 and the restricted Harish-Chandra map $\tilde{\varphi}_{HC}$ associated to $U^2_q(\mathfrak{g})$, keeping the notation  from the simply connected versions above.   
   In particular, set $\tilde{\mathcal{P}}$ equal to the projection of ${U}^0_q(\mathfrak{g})$ onto $\mathbb{C}(q)[\mathcal{A}_2]$ using (\ref{rhcmap13}).  Just as for $ \check U_q(\mathfrak{sl}_N)$, let $\tilde{\varphi}_{HC}$ denote the projection of $U_q^2(\mathfrak{g})$ onto $\mathbb{C}(q)[{\mathcal{A}}_2]$ defined by taking the composition of the Harish-Chandra map $\varphi_{HC}$ with  the projection $\tilde{\mathcal{P}}$.  Once again, we can describe $\tilde{\mathcal{P}}$ using a version of (\ref{tilde-mu}) where in this case, $\mu$ runs over elements in $Q^+_N+w_0\Lambda_N^+$ for Types AI and AII and in $(Q^+_N\times Q^+_N)+(w_0\Lambda_N^+\times w_0\Lambda_N^+)$ for the diagonal setting.

 Recall the description of the center of $U^2_q(\mathfrak{g})$ given at the end of Section \ref{section:basis-center}. It follows that its image in relation to the  ordinary Harish-Chandra map is determined by \begin{align}\label{HC-equality0}{\varphi}_{HC}(z_{2\mu} )={\varphi}_{HC}(z_{2\mu'})K_{\hat{{\omega}}_N}^{2s} 
\end{align}
 and, similarly,  for the restricted Harish-Chandra map, 
\begin{align}\label{HC-equality} \tilde{\varphi}_{HC}(z_{2\mu} )= \tilde{\varphi}_{HC}(z_{2\mu'})K_{\hat{{\omega}}_N}^{2s} 
= \tilde{\varphi}_{HC}(z_{2\mu'})K_{\hat{\eta}_N}^{2ms}
\end{align}
where $\mu\in -P^+_N$ and $\mu'\in w_0\hat{\Lambda}^+_N$ with $2\mu =2 \mu'+2s\hat{\omega}_N$ and $s\in\mathbb{N}$.
Moreover, the final equality in (\ref{HC-equality}) follows from the equality on roots $m\hat{\eta}_N = \hat{\omega}_N$ where $m=1$ in  Type AI  and $m=2$ for Type AII (see Section \ref{section:restricted-root-system}). 

Similar results holds in the diagonal setting.   In particular, we have
\begin{align}\label{HC-equality-diag0} \varphi_{HC}(z_{2\mu} )= \varphi_{HC}(z_{2\mu'})K_{\hat{\omega}_{n}}^{2s} K_{\hat{\omega}_{2n}}^{2s'}
\end{align}
and 
\begin{align}\label{HC-equality-diag} \tilde{\varphi}_{HC}(z_{2\mu} )=
 \tilde{\varphi}_{HC}(z_{2\mu'})
\tilde{\mathcal{P}}(K_{\hat{\omega}_{n}}^{2s}
 K_{\hat{\omega}_{2n}}^{2s'})=
 \tilde{\varphi}_{HC}(z_{2\mu'})K_{\hat{\eta}_n}^{2(s+s')}
\end{align}
where $\mu\in -P_N^+\times P^+_N$, $\mu'\in w_0\Lambda^+_N\times w_0\Lambda^+_N$ with $2\mu =2 \mu'+s\hat{\omega}_n+s'\hat{\omega}_n$. 
The final equality uses the  fact that the restricted weights corresponding to $\hat{\omega}_n$ and $\hat{\omega}_{2n}$  are both equal to $\hat{\eta}_n$.

Note that for $z\in Z(U_q(\mathfrak{g}))$ we actually have $z\in U^0(\mathfrak{g})\oplus U^-_+U^0(\mathfrak{g})U^+_+$ and so $z\cdot v=\varphi_{HC}(z)\cdot v$ whenever $v$ is a highest weight vector.  
Now consider a highest weight generating vector $v_{2\beta}$ for the simple module $L(2\beta)$ where $\beta\in \Lambda_{\Sigma}^+$.  Since $\beta$ is a restricted weight, it follows that $K_{2\mu}\cdot v_{2\beta} = K_{2\tilde{\mu}}\cdot v_{2\beta}$ for all weights $\mu$. Hence  
\begin{align*}
z\cdot v_{2\beta} =\varphi_{HC}(z)\cdot v_{2\beta} =  \tilde{\varphi}_{HC}(z)\cdot v_{2\beta}
\end{align*}
for all $z\in Z(U_q(\mathfrak{g}))$.  Since $H_{2\beta}\in \mathscr{P}_{\theta}$ is a highest weight vector of weight $2\beta$ that generates a $U_q(\mathfrak{g})$-module isomorphic to $L(2\beta)$, we also have 
\begin{align}\label{zvbeta}
z\cdot H_{2\beta} = \tilde{\varphi}_{HC}(z)\cdot H_{2\beta}.
\end{align}
Since central elements act as scalars on all finite-dimensional simple $U_q(\mathfrak{g})$-modules, it follows that the restricted Harish-Chandra map can be used to determine the eigenvalues with respect to the action of $Z(U_q(\mathfrak{g}))$ on $\mathscr{P}_{\theta}$.
 
\subsection{Dotted Weyl group invariance} 

Let $\rho$ denote the half sum of the positive roots for the root system associated to $\mathfrak{sl}_N$ and let $W$ denote the Weyl group for this root system.  Define a dotted Weyl group action on the Cartan subalgebra $\check U^0(\mathfrak{sl}_N)$ of $\check U_q(\mathfrak{sl}_N)$ by (\cite{L},  Chapter 3, (3.1):
 \begin{align*} w\circ q^{(\rho, \mu)}K_{\mu} = q^{(\rho,w\mu)}K_{w\mu}
 \end{align*} 
Recall the following well-known result on the image of the center of $\check{U}_q(\mathfrak{sl}_N)$ under the ordinary Harish-Chandra map $\varphi_{HC}$:
\begin{theorem}\label{L3.1}(\cite{L}, Theorem 3.1, see also \cite{Jo}, Lemma 7.17 and 7.1.25) The ordinary Harish-Chandra map $\varphi_{HC}$ defines an isomorphism from $Z(\check{U}_q(\mathfrak{sl}_N))$ 
onto $\mathbb{C}(q)[K_{2\lambda} |\lambda \in P_N]^{W\circ}$.
\end{theorem}

Note that terms of the form 
 \begin{align}\label{mlambda}m_{2\lambda}=
\sum_{w\in W} q^{(\rho,2w\lambda)}K_{2w\lambda}
\end{align}
for $\lambda \in P_N^+$
form a basis for the dotted Weyl invariant elements in  $\check U^0(\mathfrak{sl}_N)$.
There will be times that it is useful to rewrite the above formula using the lowest weight term $w_0\lambda$. In particular, the above formula is equal to 
\begin{align*} m_{2\lambda}=m_{2w_0\lambda}=\sum_{w\in W}q^{(\rho, ww_0\lambda)}K_{2ww_0\lambda}.
\end{align*} 
Now the diagonal case is not discussed in \cite{L}.  However, 
it is well-known and straightforward to check that Theorem \ref{L3.1} holds for the simply connected quantized enveloping algebras of semisimple Lie algebras such as $\check{U}_q(\mathfrak{sl}_n\oplus \mathfrak{sl}_n).$  In this case, the basis of dotted Weyl invariant elements takes the same form as above with only difference being $\lambda\in P^+_N\times P^+_N$.

 Set $\tilde{\rho}=(\rho-\theta(\rho))/2$, the restricted weight associated to $\rho$. The dotted action of $W_{\Sigma}$ on elements of 
$\mathbb{C}(q)[\check{\mathcal{A}}]$ is given by
the following formula from \cite{L}, p.24 of Chapter 3:
 \begin{align}\label{dottedaction}
 w\bullet q^{(\tilde{\rho}, \mu)}K_{2\gamma} = q^{(\tilde{\rho},w\gamma)}K_{2w\gamma}
 \end{align}
 for all $w\in W_{\Sigma}$ and $\gamma\in P_{\Sigma}$. Moreover, by \cite{L} Lemma 3.2, given an element $\tilde{w}\in W_{\Sigma}$, there exists $w\in W$ so that the restriction of $w$  to $\Sigma$ equals $\tilde{w}$.
Thus terms of the form \begin{align*} 
m^{\Sigma}_{2\lambda}= \sum_{w\in W_{\Sigma}}q^{(\tilde{\rho},2w\lambda)}K_{2w\lambda}
\end{align*} with $\lambda\in w_0P^+_{\Sigma}$ are dotted $W_{\Sigma}$ invariants.

Noting that elements of the center $Z(\check{U}_q(\mathfrak{sl}_N))$  are invariant under the ordinary dotted Weyl group action (Theorem \ref{L3.1} above) yields the following version of \cite{L}, Chapter 3, Theorem 3.3. Recall the definition of  $\check{\mathcal{A}}$ (see Section \ref{section:HCMs}) and define the related group $\mathcal{A}$ by  $$\mathcal{A} = \{K_{2u}|\mu \in P_{\Sigma}\}$$ (see \cite{L}, middle of page 25).
\begin{theorem}\label{L3.3}Given $f\in Z(\check{U}_q(\mathfrak{sl}_N))$, 
the element $\tilde{\varphi}_{HC}(f)$
is a dotted $W_{\Sigma}$ invariant element of
 $\mathbb{C}(q)[\check{\mathcal{A}}]$.  Moreover, $\tilde{\varphi}_{HC}(Z(\check{U}_q(\mathfrak{sl}_N))$ is a subring of  $\mathbb{C}(q)[ \mathcal{A}]^{W_{\Sigma}\bullet}.$
\end{theorem}

Theorem \ref{L3.3} also extends to the diagonal case. 
First note that the root system in this case consists of the disjoint union of the root systems for each copy of $\mathfrak{sl}_N$.  This leads to two half sums of the positive roots: $\rho_1$ for the first copy and $\rho_2$ for the second.  
 It is also straightforward to check that elements of $W_{\Sigma}$ lift to elements of $ W\times W$, the Weyl group for the root system of $\mathfrak{sl}_N\oplus \mathfrak{sl}_N$.  In particular, we have
\begin{align*}
w_{\tilde{\alpha}_i} 
=w_{\alpha_i}w_{\alpha_{n+i}}{\rm \ and \ }w_{\tilde{\alpha}_{n+i}} =w_{\alpha_{n+i}}w_{\alpha_{i}}
\end{align*}
 for each reflection associated to the simple restricted roots $\tilde{\alpha}_i $ or $\tilde{\alpha}_{n+i}$, for $  i=1,\dots, n-1$ in both cases.

It follows from (\ref{mlambda}) that sums  taken over $\lambda=(\lambda_1,\lambda_2)\in P^+_N\times P^+_N$ 
  \begin{align*}m_{(2\lambda_1, 2\lambda_2)}&= 
\sum_{(w_1,w_2)\in W\times W} q^{((\rho_1,\rho_2),(2w_1\lambda_1,2w_2\lambda_2)}K_{(2w_1\lambda_1,2w_2\lambda_2)}
\end{align*}
form a basis for the dotted invariant elements for images of the center of $\check U_q(\mathfrak{sl}_n\oplus \mathfrak{sl}_n)$ with respect to the ordinary Harish-Chandra map. Moreover, we can write this as a product with two factors:
\begin{align*}m_{(2\lambda_1,2\lambda_2)}=m_{2\lambda_1}m_{2\lambda_2}&= 
(\sum_{w_1\in W} q^{(\rho_1,2w_1\lambda_1)}K_{2w_1\lambda_1})(\sum_{w_2\in W} q^{(\rho_2,2w_2\lambda_2)}K_{2w_2\lambda_2})
\end{align*} Note that $w_1$ is a product of reflections with respect to the roots $\alpha_1,\dots, \alpha_{n-1}$ and $w_2$ is a product of reflections with respect to $\alpha_{n+1}, \dots, \alpha_{2n-1}$.  Set $\tilde{w}_1$ equal to the product where each $w_{\alpha_i}$ is replaced by $w_{\tilde{\alpha}_i}\in W_{\Sigma}$.  It follows that $\tilde{w}_1\lambda_1 = w_1\lambda_1.$  We can define $\tilde{w}_2$ in a similar fashion. Thus we can rewrite the formula for $m_{(2\lambda_1,2\lambda_2)}$ as  
\begin{align*}m_{2\lambda_1}m_{2\lambda_2}&= 
(\sum_{\tilde{w}_1\in W_{\Sigma}} q^{(\rho_1,2\tilde{w}_1\lambda_1)}K_{2\tilde{w}_1\lambda_1})(\sum_{\tilde{w}_2\in W_{\Sigma}} q^{(\rho_2, 2\tilde{w}_2\lambda_2)}K_{2\tilde{w}_2\lambda_2})
\end{align*} 

Recall that the inner product on the restricted root system in the diagonal case satisfies $(\cdot, \cdot)_{\Sigma}=2(\cdot,\cdot).$ 
On the other hand,  $\tilde{\rho}_1=\tilde{\rho}_2 =(\rho_1+\rho_2)/2 = \rho_{\Sigma}$. 
For $i=1,\dots,n-1$, we have $w_{\tilde{\alpha}_i}\lambda_1=w_{{\alpha}_i}\lambda_1$ and $w_{\tilde{\alpha}_i}\theta(\lambda_1)=w_{\alpha_{i+n}}\theta(\lambda_1)
=\theta(w_{\alpha_i}\lambda_1)=\theta(w_{\tilde{\alpha}_i}\lambda_1).
$ Hence $\tilde{w}_1\theta(\lambda_1)= \theta(\tilde{w}_1\lambda_1)$.  Note that this guarantees that $(\rho_1,  \theta(\tilde{w}_1\lambda_1))=0$ 
since $\tilde{w}_1\lambda_1$ is in the first copy of $P_N$ and so 
$\theta(\tilde{w}_1\lambda_1)$ must be in the second.  
Hence $$(\rho_1,2\tilde{w}_1{\lambda}_1)= (\rho_1, 2(\tilde{w}_1\lambda_1
-\theta(\tilde{w}_1\lambda_1))=(\rho_1,4\tilde{w}_1\tilde{\lambda}_1)=(\tilde{\rho}_1,4\tilde{w}_1\tilde{\lambda}_1)=(\rho_{\Sigma}, 2\tilde{w}_1\tilde{\lambda}_1)_{\Sigma}.
$$ 
Thus we can express $m_{2\lambda_1}$ as 
\begin{align*}m_{2\lambda_1}&= 
(\sum_{\tilde{w}_1\in W_{\Sigma}} q^{(\tilde{\rho}_1,2\tilde{w}_1\tilde{\lambda}_1)_{\Sigma}}K_{2(\tilde{w}_1\tilde{\lambda}_1}K_{(2\tilde{w}_1\lambda_1+2\theta(\tilde{w}\lambda_1))/2)})
\cr&=(\sum_{\tilde{w}_1\in W_{\Sigma}} q^{\rho_{\Sigma},2\tilde{w}_1\tilde{\lambda}_1)_{\Sigma}}
K_{2(\tilde{w}_1\tilde{\lambda}_1)}
(K_{(2\tilde{w}_1\lambda_1+2\theta(\tilde{w_1}\lambda_1))/2}-1)+(\sum_{\tilde{w}_1\in W_{\Sigma}} q^{({\rho}_{\Sigma},2\tilde{w}_1\tilde{\lambda}_1)_{\Sigma}}
K_{2(\tilde{w}_1\tilde{\lambda}_1)}).
\end{align*} 
Hence, $\tilde{\mathcal{P}}(m_{\lambda_1}) = \sum_{\tilde{w}_1\in W_{\Sigma}} q^{(\tilde{\rho}_1,2\tilde{w}_1\tilde{\lambda}_1)_{\Sigma}}K_{2(\tilde{w}_1\tilde{\lambda}_1)}$ and a similar result holds for $\tilde{\mathcal{P}}(m_{\lambda_2})$.  Both of these formulas are dotted $W_{\Sigma}$ invariant and hence the diagonal case also satisfies the conclusion of Theorem \ref{L3.3}.

We can translate Theorem \ref{L3.1} and Theorem \ref{L3.3} to the setting of $U^2_q(\mathfrak{g})$ where $\mathfrak{sl}$ is replaced by $\mathfrak{gl}$ everywhere as follows. 
    Recall the isomorphism of $({\rm ad}\ U_q(\mathfrak{sl}_N))$-modules, the first generated by $K_{ 2\lambda}$ and the second by $K_{2\lambda'}$ where
  $\lambda = \sum_{i=1}^{N-1}\lambda_iw_0\omega_i$ and $\lambda'=\sum_{i=1}^{N-1}\lambda_iw_0\hat{\omega}_i$ 
 as  given in (\ref{lambdaprime}).  From the description of the isomorphism in (\ref{Fdecomp10}) including the paragraph below this formula,  we see that  $K_{2\lambda}K_{2\lambda'}^{-1} $ is in the center of $\check U_q(\mathfrak{gl}_N)$.
The diagonal case is very similar where here we express $\lambda=\lambda_{(n)}+\lambda_{(2n)}$ where $\lambda_{(n)} =\sum_{i=1}^{n-1}\lambda_iw_0\omega_i$ and $\lambda_{(2n)} =\sum_{i=1}^{n-1}\lambda_{i+n}w_0\omega_{i+n}$.  Similarly define $\lambda'$ with $w_0\omega_i$ replaced by $w_0\hat{\omega}_i$ for each $i$.%

\begin{theorem}\label{L3.3v2}We have the following analogs of Theorem \ref{L3.1} and Theorem \ref{L3.3} for the algebra $U^2_q(\mathfrak{g})$:
\begin{itemize}
\item[(i)]The ordinary Harish-Chandra map $\varphi_{HC}$ defines an isomorphism from $Z({U}^2_q(\mathfrak{g}))$ 
onto the dotted $W$ invariants $\mathbb{C}(q)[K_{2\lambda}|\  \lambda\in Q^+_N+w_0{\Lambda}^+_N]^{W\circ}$ in Types AI and AII and onto the dotted $W$ invariants
$\mathbb{C}(q)[K_{2\lambda}|\ \lambda\in (Q^+_N\times Q^+_N)+(w_0{\Lambda}^+_N\times w_0{\Lambda}^+_N)]^{W\circ}$ in the diagonal setting.
\item[(ii)]Given $f\in Z({U}^2_q(\mathfrak{g}))$, 
the element $\tilde{\varphi}_{HC}(f)$
is a dotted $W_{\Sigma}$ invariant element of
 $\mathbb{C}(q)[K_{2\lambda}|\ \lambda\in Q^+_{\Sigma}+w_0{\Lambda}^+_{\Sigma}].$ Moreover, $\tilde{\varphi}_{HC}(Z({U}^2_q(\mathfrak{g}))$ is a subring of  $\mathbb{C}(q)[\mathcal{A}_2]^{W_{\Sigma}\bullet}.$
 \end{itemize}
\end{theorem}
\begin{proof} 
Recall that the ordinary Harish-Chandra map is defined via (\ref{decomp7}) as the projection onto the Cartan subalgebra of $U_q^2(\mathfrak{g})$. This Cartan subalgebra is equal to 
$\mathbb{C}(q)[K_{2\lambda}|\ \lambda\in Q^+_N+w_0{\Lambda}^+_N]$ in Types AI and AII and $\mathbb{C}(q)[K_{2\lambda}|\ \lambda\in (Q^+_N\times Q^+_N)+(w_0{\Lambda}^+_N\times w_0{\Lambda}^+_N)]$ in the diagonal setting. Note that the diagonal case follows from the other two so our focus is entirely on the singleton setting. 

Observe that $(G^-_+U^2_q(\mathfrak{g}) + {U}^2_q(\mathfrak{g})U^+_+)$ is a two-sided ideal in $U^2_q(\mathfrak{g})$ since ${U}_2^0(\mathfrak{g})G^-_+=G^-_+{U}_2^0(\mathfrak{g})$ and $U^+_+{U}_2^0(\mathfrak{g})={U}_2^0(\mathfrak{g})U^+_+$.  It follows that the ordinary Harish-Chandra map defines an algebra homomorphism onto ${U}_2^0(\mathfrak{g})$.  Therefore $\varphi_{HC}$ restricts to an algebra homomorphism on the subalgebra $Z(U^2_q(\mathfrak{g}))$.  We show below that this map is injective and hence an isomorphism as stated in 
Theorem \ref{L3.3v2} (i).

 Now $\mathcal{F}(\check U_q(\mathfrak{sl}_N))$  can be written as a direct sum of simple ad-invariant left coideals of the form $({\rm ad}\ U_q(\mathfrak{sl}_N) )K_{2\lambda}$ where $\lambda\in w_0P^+_N$ (see (\ref{Fdecomp})). On the other hand, the same type of 
decomposition for $U^2_q(\mathfrak{g})$ in Theorem \ref{lemma:U2result} gives us 
\begin{align*}\bigoplus_{\lambda'\in w_0\hat{\Lambda}^+_N}({\rm ad}\ U_q(\mathfrak{gl}_N))\cdot K_{2\lambda'}\subset \mathcal{F}(U^2_q(\mathfrak{g}))
=\bigoplus_{\lambda'\in w_0\hat{\Lambda}^+_N, 
s\in\mathbb{N}}
({\rm ad}\ U_q(\mathfrak{gl}_N))\cdot K_{2\lambda'+2s\hat{\omega}_N}.
\end{align*}
(Note that $w_0\hat{\Lambda}_N^+
+ \mathbb{N}
\hat{\omega}_N
=w_0\Lambda_N^+$.)
 As explained right before this theorem, $({\rm ad}\ U_q(\mathfrak{gl}_N))(K_{2\lambda})$ is isomorphic to $({\rm ad}\ U_q(\mathfrak{gl}_N))(K_{2\lambda'})$ via the map sending $({\rm ad}\ a)\cdot (K_{2\lambda})$ to $({\rm ad}\ a)\cdot (K_{2\lambda'})$ for all $a\in U_q(\mathfrak{gl}_N)$.  With respect to this map, the central element $z_{2\lambda}$ of $({\rm ad}\ U_q(\mathfrak{gl}_N))(K_{2\lambda})$   is  sent to the central element $z_{2\lambda'}$ of $({\rm ad}\ U_q(\mathfrak{gl}_N))(K_{2\lambda})$.
 Now $\varphi_{HC}(z_{2\lambda})$ is dotted Weyl invariant.  It follows that $\varphi_{HC}(z_{2\lambda'})$ is also dotted Weyl invariant.

 Note that $\varphi_{HC}(z_{2\lambda'})$ is an element of the Cartan subalgebra $\mathbb{C}(q)[K_{2(\lambda'+\gamma)}|\ \gamma\in Q^+_N {\rm \ and \ }\lambda'\in w_0\hat{\Lambda}^+_N].$ Moreover, $Q^+_N$ does not contain any anti-dominant integral weights and so the only anti-dominant weights come from $w_0\hat{\Lambda}_N^+$. By Theorem \ref{L3.1}, $\{\varphi_{HC}(z_{2\lambda})|\ \lambda \in w_0P^+_N\}$ forms a basis for the dotted invariants of $\mathbb{C}(q)[K_{2\gamma} |\gamma \in P_N]$.
Hence, $\{\varphi_{HC}(z_{2\lambda'})|\ \lambda' \in w_0\hat\Lambda^+_N\}$ forms a basis for the dotted invariants of $\mathbb{C}(q)[K_{2\gamma} |\gamma \in Q^+_N+w_0\hat{\Lambda}^+_N]^{W\circ}$.  Since $\varphi_{HC}(z_{2\lambda'})$ is an element of $({\rm ad}\ U_q(\mathfrak{gl}_N))\cdot K_{2w_0\lambda'}$, the only possible dotted Weyl invariant element (up to nonzero scalar) is
\begin{align*} m_{2\lambda'}=m_{2w_0\lambda'}=\sum_{w\in W}q^{(\rho, 2ww_0\lambda')}K_{2ww_0\lambda'}.
\end{align*} We get a similar equality for $m_{2w_0\lambda'+2s\hat{w}_N}$
\begin{align*} m_{2w_0\lambda'+2s\hat{\omega}_N}=\sum_{w\in W}q^{(\rho, 2ww_0\lambda')}K_{2ww_0\lambda'}K_{\hat{\omega}_N}^{2s}
\end{align*}  for $w_0\lambda'\in w_0\hat{\Lambda}^+_N$ and $ s\in\mathbb{N}.$
Note that these elements span  the image of the center of $U^2_q(\mathfrak{g})$.  Moreover, these terms are linearly independent since they each take the form
\begin{align*} m_{2\lambda'+2s\hat{\omega}_N}\in(q^{(\rho, 2w_0\lambda')}K_{2w_0\lambda'}+\sum_{\beta>w_0\lambda'}\mathbb{C}(q)K_{2\beta})(K_{2\hat{\omega}_N})^s.
\end{align*}  Here the inequality below the summation sign refers to the partial order defined by $\beta>\gamma$ provided $\beta-\gamma\in Q^+_N$.  This finishes the proof for Theorem \ref{L3.3v2} (i).  

The proof of (ii) is similar to that of (i).  In fact,
$\tilde{\mathcal{P}}$ is a projection onto   $\mathbb{C}(q)[\mathcal{A}_2]$ with kernel  the two-sided ideal $\mathbb{C}(q)[\mathcal{A}_2]\mathbb{C}(q)[({T}_{\theta})_2]_+$ inside 
the right hand side of  (\ref{rhcmap13}).  Hence the restricted Harish-Chandra map is an algebra homomorphism of $\mathbb{C}(q)[\mathcal{A}_2]$ onto itself and
$\tilde{\varphi}_{HC}(Z(U^2_q(\mathfrak{g})))$ is a subalgebra of $ \mathbb{C}(q)[\mathcal{A}_2]$. By the discussion preceding the lemma relating $\lambda$ and $\lambda'$ and by Theorem \ref{L3.1}, this image of $Z(U_q^2(\mathfrak{g}))$  under $\tilde{\varphi}_{HC}$ is invariant under the dotted action of $W_{\Sigma}$.  
 \end{proof}
  
  \subsection{Central generators}\label{section:centralgen}
  In \cite{L}, it is shown that $\tilde{\varphi}_{HC}(Z(\check U_q(\mathfrak{sl}_N))$ is  isomorphic to the entire ring of invariants $\mathbb{C}(q)[\mathcal{A}]^{\bullet}$. Indeed we have the following version of \cite{L}, Theorem 8.1.
  \begin{theorem}\label{Theorem:centralgen} Given a symmetric pair of Type AI,  Type AII, or of diagonal type, the image under the restricted Harish-Chandra map   $\tilde{\varphi}_{HC}(Z(U_q(\mathfrak{g}))$
  is isomorphic to the dotted $W_{\Sigma}$ invariants, $\mathbb{C}(q)[\mathcal{A}]^{\bullet}$, of  $\mathbb{C}(q)[\mathcal{A}]$.
  \end{theorem}
  Recall that $\mathcal{A} = \{K_{2\mu}|\ \mu\in P_{\Sigma}\}$.  Hence $\mathbb{C}(q)[\mathcal{A}]^{\bullet}$ is a polynomial ring in variables
  $m^{\Sigma}_{2w_0\eta_i}$, $i=1,\dots, n-1$. Thus we must show that $\tilde{\varphi}_{HC}(Z(U_q(\mathfrak{g}))$
contains these generators.  We start with an overview of the proof and then fill in more of the details below. Eventually we will obtain the analog result  for $Z(U^2_q(\mathfrak{g}))$.  Consider central elements  defined by 
  $$z_i = z_{2w_0\omega_i} {\rm \ for\  }i=1, \dots, n-1$$ 
where $z_{2w_0\omega_j}$ is the central element of $({\rm ad}\ U_q(\mathfrak{sl}_N))\cdot K_{2w_0\omega_j}$ for
Types AI and AII. For the diagonal case, set 
$$z_i = z_{2w_0\omega_i} {\rm \ and \ }z_{i+n}=z_{2w_0\omega_i}  {\rm \ for\  }i=1, \dots, n-1.$$ In this case  the  $z_i$  is the central element of $({\rm ad}\ U_q(\mathfrak{sl}_n\oplus \mathfrak{sl}_n))\cdot K_{2w_0\omega_i}$ and $z_{i+n}$ is the central element of $({\rm ad}\ U_q(\mathfrak{sl}_n\oplus \mathfrak{sl}_n))\cdot K_{2w_0\omega_{i+n}}.$ The argument for Theorem 
\ref{Theorem:centralgen} in Type AI simply shows that $\tilde{\varphi}_{HC}(z_i)=m^{\Sigma}_{2w_0\eta_i}$ for $i=1,\dots, n-1$.  The diagonal case is similar.  On the other hand,
the proof for Type AII is more difficult. It involves an inductive argument that first  establishes $m^{\Sigma}_{2w_0\eta_j}\in \tilde{\varphi}_{HC}(Z(U_q(\mathfrak{g}))$ for all $j<k$ and then  realizes $m^{\Sigma}_{2w\eta_k}$ as a  linear combination of $\tilde{\varphi}_{HC}(z_k))$ plus   products $m^{\Sigma}_{2w_0\eta_j}m^{\Sigma}_{2w_0\eta_i}$ for $j< k$ and $i<k$.

\medskip
\noindent
{\bf Proof of Theorem \ref{Theorem:centralgen}, Type AI and Diagonal Type:} Note that $\tilde{\omega}_i=\eta_i$, the fundamental restricted weight, for $i=1,\dots, n-1$ in both Types.  (We can use either $z_i$ or $z_{n+i}$ in the diagonal case.) Since the restricted root system is of Type $A_{n-1}$, the fundamental weights $\eta_1, \dots, \eta_{n-1}$ are minuscule. 
 In other words,  $\eta_j\not>_{\Sigma}\eta_k$ for any pair $j,k$ and, in addition, $\eta_j\not>_{\Sigma}0$ for any $j$ where the inequalities are defined via the partial  order:  $\beta>_{\Sigma}\gamma$ provided $\beta-\gamma\in Q^+_{\Sigma}$ for $\beta, \gamma\in P^+_{\Sigma}$.
Recall that $z_{2\mu}$ is the unique up to nonzero scalar central element in $({\rm ad}\ U_q(\mathfrak{g})) K_{2\mu}$ for $\mu\in w_0P^+_{\Sigma}$.  
When $\mu=w_0{\omega}_i$ for the two families Type AI and diagonal type under consideration, the image of $\tilde{\varphi}_{HC}(z_{2w_0\omega_i})$ is in \begin{align}\label{Komegai}K_{2w_0\tilde{\omega}_i}+\sum_{\beta\in Q^+_{\Sigma}}\mathbb{C}(q)K_{2w_0\tilde{\omega}_i+2\beta}
=K_{2w_0\eta_i}+\sum_{\beta\in Q^+_{\Sigma}}\mathbb{C}(q)K_{2w_0\eta_i+2\beta}
\end{align} 
up to a nonzero scalar.  
 Since this element is dotted invariant with respect to the restricted root system, it follows that $\tilde{\varphi}_{HC}(z_{2w_0\omega_i})={m}^{\Sigma}_{2w_0\eta_i}$ for each $i$.  Thus $\tilde{\varphi}_{HC}(z_i)$, $i=1,\dots, n-1$ are a set of generators for  $\mathbb{C}(q)[\mathcal{A}]^{\bullet}$. $\qed$

  \medskip
  \noindent
  {\bf Proof of Theorem \ref{Theorem:centralgen}, Type AII:} 
  Note 
  that in Section \ref{section:restricted-root-system} in  the final sentence on Type AII, we see that  $  \tilde{\omega}_1 =\eta_1$.
  This is the same as part of \cite{L}, Lemma 2.4.
  \begin{lemma}(\cite{L}, Lemma 2.4 (i) applied to Type AII) The first fundamental weight  $\omega_1$ in $P^+_N$ restricts to the first fundamental  restricted weight $\tilde{\omega}_1 = \eta_1 $. \end{lemma} 
 
\noindent
 Since $\eta_1$ is minuscule and $\tilde{\omega}_1=\eta_1$, we can use the same argument as used for minuscule fundamental restricted weights for Type AI and the diagonal type.  In particular $\tilde{\varphi}_{HC}(z_{2w_0\omega_1})=m^{\Sigma}_{2w_0\eta_1}$ and so $m^{\Sigma}_{2w_0\eta_1}\in \mathbb{C}(q)[\mathcal{A}]^{\bullet}.$
 
 The next lemma is key in establishing  $m^{\Sigma}_{2w_0\eta_k}\in \mathbb{C}(q)[\mathcal{A}]^{\bullet}$ for $k>1$.
   
  \begin{lemma}\label{L8.8} (\cite{L}, Lemma 8.8) Assume Type AII with $\Sigma$ of Type $A_{n-1}$.
  \begin{itemize} 
  \item[(i)] Let $k$ be a positive integer such that $1\leq 2k\leq n$.  Then $m^{\Sigma}_{2w_0\eta_{2k}}$ is in the span of the set $\{\tilde{\varphi}(z_{2w_0\omega_{2k}})\}\cup \{m^{\Sigma}_{2w_0\eta_{k-j}}m^{\Sigma}_{2w_0\eta_{k+j}}|\   0\leq j<k\}$.
  \item[(ii)] Let $k$ be a positive integer such that $1\leq 2k+1\leq n$.  Then $m^{\Sigma}_{2w_0\eta_{2k+1}}$ is in the span of the set $\{\tilde{\varphi}(z_{2w_0\omega_{2k+1}})\}\cup \{m^{\Sigma}_{2w_0\eta_{k-j}}m^{\Sigma}_{2w_0\eta_{k+1
  +j}}|\  0\leq j< k\}$.
 \end{itemize} \end{lemma}
 
 \noindent
We follow the proof of \cite{L}, Theorem 8.9 which is the Type AII part of Theorem \ref{Theorem:centralgen} of this paper.  Set $R=\tilde{\varphi}_{HC}(Z(\check U_q(\mathfrak{sl}_N))$.
 We already showed that $m^{\Sigma}_{2w_0\eta_1}=\tilde{\varphi}_{HC}(z_{2w_0\eta_1})$ is in $R$.  The strategy is to use induction based on the assumption that $m^{\Sigma}_{2w_0\eta_i}\in R$ for $1\leq i<j$.  Assume first that $j$ is even, say $j=2k$.  By the inductive hypothesis, both $m^{\Sigma}_{2w_0\eta_{k-i}}$ and $m^{\Sigma}_{2w_0\eta_{k+i}}$ are in $R$ for $1\leq i<k$.  Hence $R$ contains all the products  $m^{\Sigma}_{2w_0\eta_{k-i}}m^{\Sigma}_{2w_0\eta_{k+i}}$ for $1\leq i<k$.  By Lemma \ref{L8.8} (i), $m^{\Sigma}_{2w_0\eta_{2k}}$ is in $R$.  A similar induction argument using (ii) applies to $j=2k+1$ and shows that $m^{\Sigma}_{2w_0\eta_{2k+1}}\in R$.  Thus by induction, $m^{\Sigma}_{2w_0\eta_j}$ is in $R$ for $j=1$ (the minuscule element), $j$ takes on all even values between $2$ and $n-1$ by (i), and $j$ takes on all odd integers between $3$ and $n-1$ by (ii).  
  $\qed$

\medskip
Note that a consequence of the above result is that $\mathbb{C}(q)[\mathcal{A}]^{W_{\Sigma}\bullet}$ is a polynomial ring in the variables  $\tilde{\varphi}_{HC}(z_1), \dots, \tilde{\varphi}_{HC}(z_{n-1}).$   The next theorem obtains  similar results for  $Z(U^2_q(\mathfrak{g}))$.  First, we need to define analogs of the $z_i$ that rely on partitions in $\hat{\Lambda}^+_N$ instead of the weight lattice $P^+_N$.  In particular, given $w_0\omega_i\in w_0P^+_N$ with $z_i=z_{2w_0\omega_i}$, set $\hat{z}_i = z_{2w_0\hat{\omega}_i}$.  In other words, the central elements listed above in (i), (ii), (iii) are converted to elements $\hat{z}_1, \dots, \hat{z}_{n-1}$.  There is also an extra generator corresponding to the element 
$z_{2w_0\hat{\omega}_n}=z_{2\hat{\omega}_n}$.  Explicitly, define the central elements $\hat{z}_1, \dots, \hat{z}_{n}$ by 
  \begin{itemize}
  \item[(i)] $\hat{z}_i = z_{2w_0\hat{\omega}_i}=z_{2(\epsilon_{n+1-i}+\cdots + \epsilon_n)}$ for $i=1, \dots, n$ in Type AI and the diagonal case.
  \item[(ii)] $\hat{z}_i = z_{2w_0\hat{\omega}_i}=z_{2(\epsilon_{2n+1-i}+\cdots + \epsilon_{2n})}$ for $i=1, \dots, n$ in Type AII. 
  \item[(iii)] $\hat{z}_i = z_{2w_0\hat{\omega}_i}=z_{2(\epsilon_{n+1-i}+\cdots + \epsilon_n)}$ and $\hat{z}_{i+n}=z_{2(\epsilon_{2n+1-i}+\cdots + \epsilon_{2n})}$ for $i=1, \dots, n$ in the diagonal type.\end{itemize}
 
 \medskip
Recall that  $Q_N\cap w_0\Lambda^+_N=0$ (see Section \ref{subsection:special}).  The same holds for  restricted root systems.  In particular, $Q_{\Sigma}\cap w_0\Lambda^+_{\Sigma}=0$.
  
  \begin{theorem}\label{theorem:dotted_Weyl} The algebra $\tilde{\varphi}_{HC}(Z(U^2_q(\mathfrak{g})))$ equals $\mathbb{C}(q)[\mathcal{A}_2]^{W_{\Sigma}\bullet}$ 
   and is the polynomial ring on the $n$ variables $\tilde{\varphi}_{HC}(\hat{z}_1), \dots, \tilde{\varphi}_{HC}(\hat{z}_{n}).$
     \end{theorem}
\begin{proof} 
Consider Type AI. Recall that $z_{2\mu}$ is the unique (up to nonzero scalar) central element in $({\rm ad}\ U_q(\mathfrak{g})) K_{2\mu}$.  When $\mu=\hat{\omega}_n$, the element $K_{2\hat{\omega}_n}$ is a central element.  Thus we have $z_{2\hat{\omega}_n} = K_{2\hat{\omega}_n}$. Hence $\tilde{\varphi}_{HC}(z_{2\hat{\omega}_n}) =\tilde{\mathcal{P}}\circ {\varphi}_{HC}(z_{2\hat{\omega}_n}) =\sum_{w\in W}q^{(\rho, w\hat{\eta})}K_{2w\hat{\eta}_n}=q^{(\rho, w\hat{\eta})}|W| K_{2\hat{\eta}_n}$ in Type AI. Thus the algebra generated by $\tilde{\varphi}_{HC}(\hat{z}_i),$ $i=1, \dots, n$ in Type AI contains $K_{2\hat{\eta}_n}$. 
A similar statement holds  in the diagonal case.  In this case, we can use either the algebra generated by   $\tilde{\varphi}_{HC}(\hat{z}_i)$ for $i=1,\dots, n$ or the algebra generated by  $\tilde\varphi_{HC}(\hat{z}_{i+n})$ for $i=1, \dots, n$. Both produce the same algebra and that algebra contains $K_{2\hat{\eta}_n}$. 

As explained above, the fundamental weights $\eta_1,\dots, \eta_{n-1}$ are minuscule.  Note that the  same is true for  the fundamental partitions $\hat{\eta}_1,\dots, \hat{\eta}_{n-1}$. 
Moreover, we have the analog of (\ref{Komegai}) in the partition setting.  Namely,
as in the proof of Theorem \ref{Theorem:centralgen} for Type AI and the diagonal type, 
$\tilde{\varphi}_{HC}(z_{2w_0\hat{\omega}_i})$
is an element of 
\begin{align}\label{K2w}K_{2w_0\hat{\eta}_i}+\sum_{\beta\in Q^+_{\Sigma}}\mathbb{C}(q)K_{2w_0{\hat\eta}_i+2\beta_i}
=K_{2w_0\hat{\eta}_i}+\sum_{\beta\in Q^+_{\Sigma}}\mathbb{C}(q)K_{2w_0\eta_i+2\beta}
\end{align} up to a nonzero scalar.   By the discussion preceding the theorem, $Q^+_{\Sigma}\cap w_0\Lambda^+_{\Sigma}=0$. 
 Since $z_{2w_0\hat{\omega}_i}$ is central, its image under $\tilde{\varphi}_{HC}$ must be dotted invariant with respect to $W_{\Sigma}$.  By (\ref{K2w}), the only possible dotted invariant element is $m^{\Sigma}_{2w_0\hat{\eta}_i}$ up to a nonzero scalar.  Hence the theorem holds in Type AI and the diagonal setting.

Now consider Type AII.  As explained in the proof of Theorem \ref{Theorem:centralgen} for Type AII, the first fundamental restricted root $\eta_1$ is minuscule.  Hence, arguing as above, we have
 $\tilde{\varphi}_{HC}(\hat{z}_1)
 =m^{\Sigma}_{2w_0\hat{\eta}_1}.$

By Lemma \ref{L8.8}, $m^{\Sigma}_{2w_0\eta_{2k}}$ can be written as a linear combination of elements in the set 
\begin{align*}\{ \tilde{\varphi}_{HC}(z_{2w_0\omega_{2k}}), m^{\Sigma}_{2w_0\eta_{k-j}}m^{\Sigma}_{2w_0\eta_{k+j}}|\ 1\leq j<k\} 
\end{align*}
for $1<2k\leq n$.  Similarly, $m^{\Sigma}_{2w_0\eta_{2k+1}}$ can be written as a linear combination of elements in the set
\begin{align*}\{ \tilde{\varphi}_{HC}(z_{2w_0\omega_{2k+1}}), m^{\Sigma}_{2w_0\eta_{k+1-j}}m^{\Sigma}_{2w_0\eta_{k+j}}|\ 1\leq j<k\} 
\end{align*}
for $1<2k+1\leq n$.

We recall here some formulas involving  restricted weights from Section \ref{section:restricted-root-system}. 
 In particular,  
 $w_0\omega_{j}=w_0\hat{\omega}_{j}-(j/2n)\hat{\omega}_{2n}$.  As explained in Section \ref{section:restricted-root-system},  $2\hat{\eta}_n=\hat{\omega}_{2n}$ and so $w_0\eta_j  = w_0\hat{\eta}_j-(j/n)\hat{\eta}_{n}$ for $j=1, \dots, n-1$. 
Hence 
\begin{align*}
\tilde{\varphi}_{HC}(z_{2w_0\omega_{j}}) = \tilde{\varphi}_{HC}(z_{2w_0\hat{\omega}_j})K_{(-j/n)\hat{\omega}_{2n}}
\end{align*}
for each $j=1,\dots,n-1$.  
Thus we also have \begin{align*}m^{\Sigma}_{2w_0\eta_{s}}=m^{\Sigma}_{2w_0\hat{\eta}_{s}}K_{(-s/n)2\hat{\eta}_{2n}}=m^{\Sigma}_{2w_0\hat{\eta}_{s}}K_{(-s/n)\hat{\omega}_{2n}}
\end{align*}
for each $s=1,\dots,n-1$ and so
\begin{align}\label{kevenagain}m^{\Sigma}_{2w_0{\eta}_{k-j}}m^{\Sigma}_{2w_0{\eta}_{k+j}}&=m^{\Sigma}_{2w_0\hat{\eta}_{k-j}}m^{\Sigma}_{2w_0\hat{\eta}_{k+j}}K_{((-k+j)/n)\hat{\omega}_{2n}}K_{((-k-j)/n)\hat{\omega}_{2n}}\cr&=m^{\Sigma}_{2w_0\hat{\eta}_{k-j}}m^{\Sigma}_{2w_0\hat{\eta}_{k+j}}K_{(-2k/n)\hat{\omega}_{2n}}
\end{align}
for each $k$ and each $j$ with $1\leq j<k$.
Similarly, \begin{align}\label{koddagain}m^{\Sigma}_{2w_0{\eta}_{k-j}}m^{\Sigma}_{2w_0{\eta}_{k+1+j}}&=m^{\Sigma}_{2w_0\hat{\eta}_{k-j}}m^{\Sigma}_{2w_0\hat{\eta}_{k+1+j}}K_{((-k+j)/n)\hat{\omega}_{2n}}K_{((-k-1-j)/n)\hat{\omega}_{2n}}\cr&=m^{\Sigma}_{2w_0\hat{\eta}_{k-j}}m^{\Sigma}_{2w_0\hat{\eta}_{k+1+j}}K_{((-2k-1)/n)\hat{\omega}_{2n}}
\end{align}
for each $j$ with $1\leq j<k$.  
 Thus a relation of the form 
\begin{align*}
m^{\Sigma}_{2w_0\eta_{2k}} = a_0\tilde{\varphi}_{HC}(z_{2w_0\omega_{2k}}) +\sum_{j=1}^{k-1} a_jm^{\Sigma}_{2w_0\eta_{k-j}}m^{\Sigma}_{2w_0\eta_{k+j}}
\end{align*}
becomes
\begin{align*}m^{\Sigma}_{2w_0\hat{\eta}_{2k}}K_{(-2k/{n})\hat{\eta}_{2n}}
 = a_0\tilde{\varphi}_{HC}(z_{2w_0\hat{\omega}_{2k}})
 K_{(-2k/n)\hat{\eta}_{2n}} 
  +\sum_{j=1}^{l-1} a_jm^{\Sigma}_{2w_0\hat{\eta}_{k-j}}m^{\Sigma}_{2w_0\hat{\eta}_{k+j}}K_{({-2k}/n)\hat{\eta}_{2n}}.
\end{align*}
Multiplying both sides by $K_{({2k/n)\hat{\eta}_{2n}}}$ yields 
\begin{align*}
m^{\Sigma}_{2w_0\hat{\eta}_{2k} }= a_0\tilde{\varphi}_{HC}(z_{2w_0\hat{\omega}_{2k}})  +\sum_{s=1}^{k-1} a_jm^{\Sigma}_{2w_0\hat{\eta}_{k-j}}m^{\Sigma}_{2w_0\hat{\eta}_{k+j}}.\end{align*}
 Hence
$m^{\Sigma}_{2w_0\hat{\eta}_{2k}}$ can be written as a linear combination of elements in the set 
\begin{align}\label{keven2}\{ \tilde{\varphi}_{HC}(z_{2w_0\hat{\omega}_{2k}}), m^{\Sigma}_{2w_0\hat{\eta}_{k-j}}m^{\Sigma}_{2w_0\hat{\eta}_{k+j}}|\ 1\leq j<k\} 
\end{align}
for $1<2k\leq n$.
A similar argument shows that 
$m^{\Sigma}_{2w_0\hat{\eta}_{2k+1}}$ can be written as a linear combination of elements in the set 
\begin{align}\label{kodd}\{ \tilde{\varphi}_{HC}(z_{2w_0\hat{\omega}_{2k+1}}), m^{\Sigma}_{2w_0\hat{\eta}_{k-j}}m^{\Sigma}_{2w_0\hat{\eta}_{k+1+j}}|\ 1\leq j<k\} 
\end{align}
for $1<2k+1\leq n$.

Note that when either $2k=n$ or $2k+1=n$ we see that $m^{\Sigma}_{2w_0\hat{\eta}_{n}}$ is a linear combination of elements in the set (\ref{keven2}) or (\ref{kodd}) depending on the parity of $n$.
Thus  arguing by induction as in the proof of Theorem \ref{Theorem:centralgen}, Type AII,  the algebra generated by  $\tilde{\varphi}_{HC}(\hat{z}_i),$ $i=1, \dots, n$ contains the elements $m^{\Sigma}_{2w_0\hat{\eta}_{j}}$, $j=1, \dots, n$.  The theorem follows from the facts that 
 $w_0\hat{\eta}_n = \hat{\eta}_n$, $K_{2\hat{\eta}_n}$ is invariant with respect to the dotted action of $W_{\Sigma}$ and, thus,  $m_{2w_0\hat{\eta}_{n}}$ is a nonzero scalar multiple of $K_{2\hat{\eta}_n}$.
\end{proof}

Let $Z$ denote the subring of $Z(U^2_q(\mathfrak{g}))$ generated by the elements $\hat{z}_1, \dots, \hat{z}_n$. By  the previous theorem, Theorem \ref{theorem:dotted_Weyl}, the image under the restricted Harish-Chandra map $\tilde{\varphi}_{HC}$ of the algebra generated by $\hat{z}_1, \dots,\hat{z}_n$ is a polynomial ring with these variables.  Hence, since $Z$ is commutative, $\hat{z}_1, \dots, \hat{z}_n$ must also generate a polynomial ring of rank $n$.  We see that the same is true for the image of $Z$ with respect to $\Upsilon$ of 
 Corollary \ref{corollary:upsilon}.

\begin{corollary}\label{corollary:HC}  The algebra $\Upsilon(Z)$  is a polynomial subring of ${\rm End}\ \mathscr{P}_{\theta}$  with variables $\Upsilon(\hat{z}_i), i=1, \dots, n$.  Moreover,  each  $\Upsilon(\hat{z}_r)$ is an element in $\mathscr{PD}_{\theta}$ of degree   less than or equal to $2r$.
\end{corollary} 
\begin{proof}  Recall that $m^{\Sigma}_{2\lambda} = \sum_{w\in W_{\Sigma}}q^{(\tilde{\rho}, 2w\lambda)}K_{2w\lambda}$ for each $\lambda\in \Lambda_{\Sigma}^+.$  Note that $\tilde{\gamma}\in Q^+$ for all $\gamma\in Q^+$.  Hence, given $\lambda\in \Lambda_{\Sigma}^+$ and $w\in W_{\Sigma}$, we have  $w\lambda\in \lambda -Q^+$.  Since $(\rho, \gamma)$ is a positive integer for $\gamma\in Q^+$ so is $(\tilde{\rho}, \tilde{\gamma}) = (\rho, \tilde{\gamma})$.   Hence, for each $\lambda\in \Lambda_{\Sigma}^+$, we have
\begin{align*} m^{\Sigma}_{2\lambda}\in q^{(\tilde{\rho}, 2\lambda)}K_{2\lambda} + \sum_{\gamma\in Q^+}q^{{(\tilde{\rho}, 2\lambda)}-2}\mathbb{C}[q^{-2}] K_{2\lambda-2\gamma}.
\end{align*}
Let $v_{2\beta}$ be a highest weight generating vector  for  $L(2\beta)$ where $\beta\in \Lambda_{\Sigma}^+$.  
It follows that 
\begin{align}\label{q-inequality}  m^{\Sigma}_{2\lambda} \cdot v_{2\beta}\in \left(q^{(\tilde{\rho}+2\beta, 2\lambda)} (1+  q^{-2}\mathbb{C}[q^{-2}])\right)v_{2\beta}.
\end{align}
Note that any $a\in  \mathbb{C}(q)[\mathcal{A}_2]^{W_{\Sigma}\bullet}$ can be expressed as a linear combination $a_1m^{\Sigma}_{2\lambda_1}+\cdots + a_sm^{\Sigma}_{2\lambda_s}$ where each $\lambda_s\in \Lambda_{\Sigma}^+$.
We argue that there exists $\beta\in \Lambda_{\Sigma}^+$ so that $a\cdot v_{2\beta}\neq 0$.  Reordering and multiplying by a nonzero element of $\mathbb{C}[q]$ if necessary, we may assume that $|\lambda_1|\geq |\lambda_j|$ for each $2\leq j\leq s$, $a_i\in \mathbb{C}[q]$ for each $1\leq i\leq s$, and $a_1 = q^d+ $ terms of lower degree in $q$.
It is straightforward to check that there exists $\beta_1\in \Lambda_{\Sigma}^+$ such that $(\lambda_1,\beta_1) >(\lambda,\beta_1)$ for all $\lambda\neq \lambda_1$ satisfying $|\lambda|\leq |\lambda_1|$.  Thus by (\ref{q-inequality}), for a large enough positive integer $r$, $a\cdot v_{2r\beta} = q^{d+  (\tilde{\rho}+2r\beta_1,2\lambda_1)} + $ terms of degree strictly less than $d+  (\tilde{\rho}+2r\beta_1,2\lambda_1)$.   Thus $a\cdot v_{2\beta}\neq 0$ for $\beta =r\beta_1$.

Recall that $H_{2\beta}$ is a highest weight vector in the $U_q(\mathfrak{g})$-module $\mathscr{P}_{\theta}$ for each $\beta\in \Lambda_{\Sigma}^+$.  By (\ref{zvbeta}),  $z\cdot H_{2\beta} = \tilde{\varphi}_{HC}(z)\cdot H_{2\beta}$  for all $z\in Z(U_q(\mathfrak{g}))$.  By Theorem \ref{theorem:dotted_Weyl}, $\tilde{\varphi}_{HC}$ defines an isomorphism from $Z$ onto $\mathbb{C}(q)[\mathcal{A}_2]^{W_{\Sigma}\bullet}$. Hence, by the previous paragraph, given $z\in Z$, we can find $\beta\in 
\Lambda_{\Sigma}^+$ so that $\tilde{\varphi}_{HC}(z)\cdot H_{2\beta}\neq 0$.  It follows that $z\cdot H_{2\beta}\neq 0$.  By Corollary \ref{corollary:upsilon}, $\Upsilon(z) \cdot H_{2\beta} = z\cdot H_{2\beta}$ and so $\Upsilon(z) \neq 0$.  In other words, $\Upsilon$ is injective upon restriction to $Z$. This proves the first assertion of this corollary.

 For the second assertion, note that 
$\hat{z}_{r}$ is in the $({\rm ad}\ U_q(\mathfrak{g}))$-module  generated by $K_{2\epsilon_{N+1-r}+\cdots + 2\epsilon_{N}}$ where $N=n$ in Type AI and the diagonal case and $N=2n$ in Type AII.  Hence, by Proposition \ref{prop:K_epsilon}, $\Upsilon(\hat{z}_{r})$ has degree  less than or equal to 
$2r$. \end{proof}

We return to this degree computation  in Section \ref{subsection:centerandCapelli}.  Indeed, Theorem \ref{theorem:center_and_capelli} establishes the equality $\deg(\Upsilon(\hat{z}_{r}) )=2r$.

\section{Quantum Capelli operators}

\subsection{Definition and description}\label{section:defn-and-desc}
The decompositions in Section \ref{subsection:explicit-module} combined together yield the following module decomposition and related isomorphisms of left $U_q(\mathfrak{g})$-modules:
  \begin{align}\label{decompU-inv2}
\mathscr{PD}_{\theta}=\bigoplus_{\mu,\xi\in \Lambda^+_{\Sigma}} \left(U_q(\mathfrak{g})\cdot H_{2\mu}\right)\otimes \left(U_q(\mathfrak{g})\cdot  H^*_{2\xi}\right)\cong \bigoplus_{\mu,\xi\in \Lambda^+_{\Sigma}} L(2\mu)\otimes L^*(2\xi)
 \end{align}
 where $L^*(2\xi)$ is the  left $U_q(\mathfrak{g})$-module dual of $L(2\xi)$.
This last isomorphism can be made concrete by sending $H_{2\mu}$ to a pre-chosen highest weight generating vector $v_{2\mu}$ of $L(2\mu)$.  Similarly, $H^*_{2\xi}$ is sent to
a nonzero scalar multiple of $v_{2\xi}^*$,
the lowest weight generating vector for  $L^*(2\xi)$ satisfying  $ v^*_{2\xi} (v_{2\xi}) = 1$.  This nonzero scalar is determined  in the next lemma using the bilinear form  $\langle \cdot, \cdot \rangle$ defined by (\ref{inner}) in Section \ref{section:action}.

  Recall that there is a natural isomorphism from $(L(2\mu)\otimes L^*(2\mu))$ to ${\rm End}\ L(2\mu)$ as left $U_q(\mathfrak{g})$-modules and so $ (L(2\mu)\otimes L^*(2\mu))^{U_q(\mathfrak{g})}$ is the one-dimensional subspace consisting of the scalars.  Furthermore, $ (L(2\mu)\otimes L^*(2\xi))^{U_q(\mathfrak{g})} =0$ for $\mu\neq \xi$.
 Hence, the left $U_q(\mathfrak{g})$-module invariants of $ \mathscr{PD}_{\theta}$ satisfy
  \begin{align*}
\mathscr{PD}_{\theta}^{U_q(\mathfrak{g})}\cong \bigoplus_{\mu,\xi\in \Lambda_{\Sigma}^+} (L(2\mu)\otimes L^*(2\xi))^{U_q(\mathfrak{g})} 
= \bigoplus_{\mu\in \Lambda_{\Sigma}^+} (L(2\mu)\otimes L^*(2\mu))^{U_q(\mathfrak{g})}.
 \end{align*}
Let $C_{\mu}$ be the  basis vector for the space $ \left(U_q(\mathfrak{g})\cdot H_{2\mu}\right)\otimes \left(U_q(\mathfrak{g})\cdot  H^*_{2\mu}\right)^{U_q(\mathfrak{g})}$ corresponding to the identity element in 
${\rm End}\ L(2\mu)$ via the isomorphism between $ \left(U_q(\mathfrak{g})\cdot H_{2\mu}\right)\otimes \left(U_q(\mathfrak{g})\cdot  H^*_{2\mu}\right)$ and $L(2\mu)\otimes L^*(2\mu)$ described above.
We refer to the set 
 $\{C_{\mu}|\ \mu\in \Lambda_{\Sigma}^+\}$ as the Capelli operators.  Note that the Capelli operators form a basis for the $U_q(\mathfrak{g})$
 invariant subspace of $\mathscr{PD}_{\theta}$.  
 Since $C_{\mu}\in  \left((U_q(\mathfrak{g})\cdot H_{2\mu})\otimes (U_q(\mathfrak{g})\cdot  H^*_{2\mu})\right)$, it follows that $C_{\mu}$ has degree $2|\mu|$ in terms of the filtration $\mathcal{J}$.  In the lemma below, we drop the tensor product notation and simply right $H_{2\mu}H^*_{2\mu}$.


\begin{lemma}\label{Capelliexpression} The Capelli operator $C_{\mu}$ for  $\mu\in \Lambda^+_{\Sigma}$ lies in
\begin{align}\label{gammaeqn}   \langle H_{2\mu}^*, H_{2\mu}\rangle^{-1} H_{2\mu}H^*_{2\mu}+ (U^-_+\cdot H_{2\mu})(U^+_+\cdot H_{2\mu}^*).
\end{align}
  Moreover $C_{\mu}\cdot H_{2\mu} = H_{2\mu}$  and $C_{\mu}\cdot H_{2{\lambda}} = 0$ for $|\mu|\geq |\lambda|$ and $\mu\neq \lambda$.
\end{lemma}
\begin{proof} Recall that $H_{2\mu}$ is a highest weight vector of weight $2\mu$ and $H_{2\mu}^*$ is a lowest weight vector 
of weight $-2\mu$.  Hence   $$C_{\mu} \in U_q(\mathfrak{g})\cdot (H_{2\mu}H^*_{2\mu} )\subseteq (U^-\cdot H_{2\mu})(U^+\cdot H_{2\mu}^*).$$
Using the augmentation ideals $U^-_+$ and $U^+_+$, this simplifies to 
\begin{align*}  C_{\mu}\in \gamma H_{2\mu}H_{2\mu}^* + (U^-_+\cdot H_{2\mu})(U^+_+\cdot H_{2\mu}^*) 
\end{align*}  for some scalar $\gamma$.  It follows that (\ref{gammaeqn}) is true up to the scalar in front of $H_{2\mu}H^*_{2\mu}$.

The  projection map $\pi$ defined in Section \ref{section:action} can be used to  understand the action of  $C_{\mu}$ on $H_{2\lambda}$ 
In particular we have
\begin{align}\label{mulam} C_{\mu} \cdot H_{\lambda}= \gamma H_{2\mu}\pi (H_{2\mu}^*H_{2\lambda})_0 +(U^-_+\cdot H_{2\mu})\pi ((U^+_+\cdot H_{2\mu}^*H_{2\lambda})_0.
 \end{align} 
 More generally 
 \begin{align}\label{mulam2} C_{\mu} \cdot(u\cdot H_{2\lambda})=\gamma H_{2\mu}\pi (H_{2\mu}^*(u\cdot H_{2\lambda}))_0 +\sum_{u\in U^-_+}(u\cdot H_{2\mu})\pi ((U^+_+\cdot H_{2\mu}^*)(u\cdot H_{2\lambda}))_0 
 \end{align}  where the sum is over weight vectors $u$ in $U^-_+$.
By Lemma \ref{lemma:mugamma}, if $|\mu| \geq |\lambda|$ and $\mu\neq \lambda$ then  $\pi (H_{2\mu}^*H_{2\lambda})=\pi ((U^+_+\cdot H_{2\mu}^*)H_{2\lambda})=0$.
 Hence $C_{\mu}\cdot H_{2\lambda} = 0$ for   $|\mu| \geq |\lambda|$ and $\mu\neq \lambda$ as desired.

Recall that $C_{\mu}$ acts as the identity on $U_q(\mathfrak{g})\cdot H_{2\mu} $. Hence $C_{\mu}\cdot H_{2\mu} = H_{2\mu}$. 
Hence by (\ref{mulam}), $\gamma H_{2\mu}(\pi(H_{2\mu}^*H_{2\mu}))_0=H_{2\mu}$.  (Note here we are taking into account that $U^-_+\cdot H_{2\mu}$ is a sum of terms of  weight strictly less than $2\mu$ so have no contribution to $C_{\lambda}\cdot H_{2\mu})$.
 Recall the definition of the bilinear form
$\langle \cdot, \cdot \rangle $ right before Lemma \ref{lemma:inv}.  In particular, this is a bilinear form on $\mathcal{D}_{\theta}\times \mathcal{P}_{\theta}$ defined by 
$\langle d,p\rangle = \pi(dp)_0$.
  Hence 
$C_{\mu}\cdot H_{2\mu} = H_{2\mu} = \gamma H_{2\mu}(\langle H_{2m}^*, H_{2\mu}\rangle))$.   Thus $\gamma= \langle H_{2\mu}^*, H_{2\mu}\rangle^{-1}$.

\end{proof}

\subsection{ Realization as polynomials}
We start with the twisting relation between elements in $\mathscr{P}_{\theta}$ and $\mathscr{D}_{\theta}$. These relationships will be key to taking products of Capelli operators. Much of the computations involve vector subspaces of $\mathscr{P}_{\theta}$  and vector subspaces of $\mathscr{D}_{\theta}$ of the form 
\begin{align*}\sum_{\substack{\nu<\mu\\ |\nu|=|\mu|}}(\mathscr{P}_{\theta})_{\nu}{\rm \quad and \quad }\sum_{\substack{\nu<\mu\\ |\nu|=|\mu|}}(\mathscr{D}_{\theta})_{-\nu}\end{align*}  Here, the inequality $\nu<\mu$ means that $\mu-\nu\in Q^+_N$.  An equality such as $|\nu|=|\mu|$ ensures that the entire subspace of $\mathscr{P}_{\theta}$ is of degree $|\mu|$ and thus sits inside $\mathcal{J}_{|\mu|}(\mathscr{P}_{\theta})$.  A similar result holds for the subspaces of $\mathscr{D}_{\theta}$.  

Recall  the relation described in Theorem \ref{theorem:relnsWeyl}.
The next lemma provides a version of this relation using subspaces as described above.
\begin{lemma}\label{twistnew}For all $a,b,e,$ and $f$, the relation from Theorem \ref{theorem:relnsWeyl} satisfies
\begin{align}\label{twist2}
d_{ab}x_{ef} - q^{(\epsilon_a+\epsilon_b,\epsilon_e+\epsilon_f )} x_{ef}d_{ab}\in\  q^{-\delta_{ef}}\delta_{ae}\delta_{bf} +\sum_{\substack{\nu<\epsilon_{e}+\epsilon_f\\ |\nu|=2}}\sum_{\substack{\nu'<\epsilon_{a}+\epsilon_b\\ |\nu'|=2}}(\mathscr{P}_{\theta})_{\nu}(\mathscr{D}_{\theta})_{-\nu'}.
\end{align}
\end{lemma}
\begin{proof} Note that the weight of $d_{ab}$ is $-\epsilon_{a}-\epsilon_{b}$ and the weight of $x_{ef}$ is $\epsilon_e+\epsilon_f$.  Hence the exponent of $q$ preceding $x_{ef}d_{ab}$ satisfies  $\delta_{af}+\delta_{ae}+\delta_{bf} +\delta_{be}=(\epsilon_{a}+\epsilon_{b},\epsilon_e+\epsilon_f).$

By \cite{LSS}, Lemma 5.4,
\begin{align*}
F_r\cdot x_{ij}&= \delta_{ir}q^{-\delta_{rj}+\delta_{r,j-1}}x_{i+j,j}+\delta_{jr}x_{i,j+1}\cr
E_r\cdot d_{ij}&=-(q^{-1}\delta_{ir}d_{i+1,j}+q^{1-\delta_{ri}+\delta_{r,i-1}}\delta_{jr}d_{i,j+1}).
\end{align*}
Note that if $F_r\cdot x_{ij}\neq 0$, then the subscripts of  $x_{ij}$  increase upon application of $F_r$.  Similarly, if $E_r\cdot d_{ij}\neq 0$,  the subscripts of  $d_{ij}$ increase upon application of $E_r$.  Therefore \begin{align*}
d_{ab}x_{ef} - q^{(\epsilon_a+\epsilon_b,\epsilon_e+\epsilon_f )} x_{ef}d_{ab}\in\  q^{-\delta_{ef}}\delta_{ae}\delta_{bf} +\sum_{\substack{\nu<\epsilon_{e}+\epsilon_f\\ |\nu|=2}}\sum_{\substack{\nu'<\epsilon_{a}+\epsilon_b\\ |\nu'|=2}}(\mathscr{P}_{\theta})_{\nu}(\mathscr{D}_{\theta})_{-\nu'}.
\end{align*}
as desired.
\end{proof} 

The next lemma gives applications of the relation in Lemma \ref{twistnew}.

\begin{lemma}\label{lemma:XD} Given a weight vector  $X_{2\gamma}$ of weight $2\gamma$ in $\mathscr{P}_{\theta}$  and a weight vector $D_{-2\mu}$ of weight $-2\mu$ in $\mathscr{D}_{\theta}$, we have 
 \begin{align*}{D}_{-2\mu}{X}_{2\gamma} -q^{(2\mu,2\gamma)}{X}_{2\gamma}{D}_{-2\mu}\in \sum_{\substack{\nu<\gamma\\ |\nu'|=|\gamma|}}\  \sum_{\substack{\nu'<2\mu \\|\nu'|=|2\mu|}}(\mathscr{P}_{\theta})_{\nu}(\mathscr{D}_{\theta})_{-\nu'}+\mathcal{J}_{|2\mu|+|2\gamma|-2}(\mathscr{PD}_{\theta}).\end{align*}
Moreover,  \begin{align*}\sum_{\substack{\nu<2\mu \\|\nu|=|2\mu|}}
(\mathscr{D}_{\theta})_{-\nu}\sum_{\substack{\nu'<2\gamma\\|\nu'|=|2\gamma|}}
(\mathscr{P}_{\theta})_{\nu'} \subseteq \sum_{\substack{\nu'<2\gamma\\|\nu'|=|2\gamma|}}
(\mathscr{P}_{\theta})_{\nu'}\sum_{\substack{\nu<2\mu\\|\nu|=|2\mu|}}
(\mathscr{D}_{\theta})_{-\nu} + \mathcal{J}_{|2\mu|+|2\gamma|-2}(\mathscr{PD}_{\theta})
\end{align*}
\end{lemma}
\begin{proof}
Consider  first a term of the form $d_{e_s,f_s}x_{g_1,h_1}x_{g_{2},h_2}$.  Using (\ref{twist2}) as the term $d_{e_s,f_s}$ is moved to the right gives us
\begin{align*}d_{e_s,f_s}x_{g_1,j_1}x_{g_{2},h_2}&=q^{(\epsilon_{e_s}+\epsilon_{f_s}, \epsilon_{g_1}+\epsilon_{g_2})} x_{g_1,h_1}d_{e_s,f_s}x_{g_{2},h_{2}}
\cr&\quad\quad+
\sum_{\substack{\nu<\epsilon_{g_1}+\epsilon_{h_1}\\ |\nu|=2}}\sum_{\substack{\nu'<\epsilon_{e_s}+\epsilon_{f_s}\\ |\nu'|=2}}(\mathscr{P}_{\theta})_{\nu}(\mathscr{D}_{\theta})_{-\nu'}x_{g_{2},h_2}+\mathbb{C}(q)x_{g_2,h_2}
\cr&=
q^{(\epsilon_{e_s}+\epsilon_{f_s}, \epsilon_{g_1}+\epsilon_{h_1}+\epsilon_{g_2}+\epsilon_{h_2})} x_{g_1,h_1}x_{g_{2},h_{2}}d_{e_s,f_s}
\cr&\quad\quad+\sum_{\substack{\nu<\epsilon_{g_2}+\epsilon_{h_2}\\ |\nu|=2}}\sum_{\substack{\nu'<\epsilon_{e_s}+\epsilon_{f_s}\\ |\nu'|=2}}x_{g_1,h_1}(\mathscr{P}_{\theta})_{\nu}(\mathscr{D}_{\theta})_{-\nu'}\cr&\quad\quad+\sum_{\substack{\nu<\epsilon_{g_1}+\epsilon_{h_1}\\ |\nu|=2}}\sum_{\substack{\nu'<\epsilon_{e_s}+\epsilon_{f_s}\\ |\nu'|=2}}(\mathscr{P}_{\theta})_{\nu}(\mathscr{D}_{\theta})_{-\nu'}x_{g_{2},h_2}+\mathbb{C}(q)x_{g_2,h_2}
+\mathbb{C}(q)x_{g_1,h_1}\end{align*} 
 From the relations satisfied by the generators for $\mathscr{P}_{\theta}$ (see Theorem \ref{theorem:polypart} and formula (\ref{quadformula})) we see that $(\mathscr{P}_{\theta})_{2\gamma}(\mathscr{P}_{\theta})_{2\lambda}=(\mathscr{P}_{\theta})_{2\gamma+2\lambda}$.  Similarly $(\mathscr{D}_{\theta})_{-2\gamma}(\mathscr{D}_{\theta})_{-2\lambda}=(\mathscr{D}_{\theta})_{-2\gamma-2\lambda'}$.  Hence 
 \begin{align*}\sum_{\substack{\nu<\epsilon_{g_2}+\epsilon_{h_2}\\ |\nu|=2}}x_{g_1,h_1}(\mathscr{P}_{\theta})_{\nu}\subseteq \sum_{\substack{\nu<\epsilon_{g_1}+\epsilon_{h_1}+\epsilon_{g_2}+\epsilon_{h_2}\\ |\nu|=4}}(\mathscr{P}_{\theta})_{\nu}\ \end{align*}
On the other hand, 
\begin{align*}\sum_{\substack{\nu'<\epsilon_{e_s}+\epsilon_{f_s}\\ |\nu'|=2}}(\mathscr{D}_{\theta})_{-\nu'}
\end{align*}
is spanned by linearly independent vectors of the form $d_{ef}$ with weight $\nu'=\epsilon_{e}+\epsilon_f$ which is strictly less than  $\epsilon_{e_s}+\epsilon_{f_s}$.  Applying 
(\ref{twist2}) results in \begin{align*}\sum_{\substack{\nu'<\epsilon_{e_s}+\epsilon_{f_s}\\ |\nu'|=2}}(\mathscr{D}_{\theta})_{-\nu'}x_{g_2,h_2} \subseteq \sum_{\substack{\gamma\leq\epsilon_{g_2}+\epsilon_{h_2}\\ |\gamma|=2}}\sum_{\substack{\nu'<\epsilon_{e_s}+\epsilon_{f_s}\\ |\nu'|=2}}(\mathscr{P}_{\theta})_{-\gamma}(\mathscr{D}_{\theta})_{-\nu'}+\mathbb{C}(q).
\end{align*}
Hence \begin{align*}\sum_{\substack{\nu<\epsilon_{g_1}+\epsilon_{h_1}\\ |\nu|=2}}(\mathscr{P}_{\theta})_{\nu}\sum_{\substack{\nu'<\epsilon_{e_s}+\epsilon_{f_s}\\ |\nu'|=2}}(\mathscr{D}_{\theta})_{\nu'}x_{g_2,h_2} \subseteq \sum_{\substack{\nu\leq\epsilon_{g_1}+\epsilon_{h_1}+\epsilon_{g_2}+\epsilon_{h_2}\\ |\nu|=4}}\sum_{\substack{\nu'<\epsilon_{e_s}+\epsilon_{f_s}\\ |\nu'|=2}}(\mathscr{P}_{\theta})_{\nu}(\mathscr{D}_{\theta})_{\nu'}+\sum_{\substack{\nu<\epsilon_{g_1}+\epsilon_{h_1}\\ |\nu|=2}}(\mathscr{P}_{\theta})_{\nu}.
\end{align*}
Therefore
\begin{align*}&d_{e_s,f_s}x_{g_1,h_1}x_{g_{2},h_2}-
q^{(\epsilon_{e_s}+\epsilon_{f_s}, \epsilon_{g_1}+\epsilon_{h_1}+\epsilon_{g_2}+\epsilon_{h_2})} x_{g_1,h_1}x_{g_{2},h_{2}}d_{e_s,f_s}
\cr&\subseteq 
\sum_{\substack{\nu\leq\epsilon_{g_1}+\epsilon_{h_1}+\epsilon_{g_2}+\epsilon_{h_2}\\ |\nu|=4}}\sum_{\substack{\nu'<\epsilon_{e_s}+\epsilon_{f_s}\\ |\nu'|=2}}(\mathscr{P}_{\theta})_{\nu}(\mathscr{D}_{\theta})_{\nu'}+\sum_{\substack{\nu<\epsilon_{g_1}+\epsilon_{h_1}\\ |\nu|=2}}(\mathscr{P}_{\theta})_{\nu}.
\end{align*} Note that this final term is in $\mathcal{J}_1(\mathscr{PD}_{\theta})$ while the other terms are in $\mathcal{J}_3(\mathscr{PD}_{\theta})$.  

We now turn our attention in applying such computations to the main assertion of the lemma.  Since $D_{-2\mu}$ is a weight vector, it can be written as a sum of products $d_{e_1,f_1}d_{e_2,f_2}\dots d_{e_s,f_s}$ and each summand has the same terms with the only difference being reordering.  Moreover, since the weight of $d_{ab}$ is $-\epsilon_{a}-\epsilon_{b}$, the weight $-2\mu$ of $D_{-2\mu}$ equals $\sum_i^s{-\epsilon_ {e_i}-\epsilon_{f_i}}.$ A  similar analysis applies to $X_{2\gamma}$ using elements $x_{ab}$ instead of $d_{ab}$.
Repeated applications of (\ref{twistnew}), moving one term of the form $d_{e_i,f_i}$ to the right after the previous one, yields the desired formula.


The argument for the ``moreover" part  is similar.  Start with a weight vector  $D$ of weight $-\nu$ in  the vector space  $\sum_{\substack{\nu<\mu \\|\nu|=|2\mu|}}
(\mathscr{D}_{\theta})_{-\nu}$ and a weight vector $X$ of weight $\nu'$ in the vector space $\sum_{\substack{\nu'<2\gamma \\|\nu'|=|2\gamma|}}
(\mathscr{P}_{\theta})_{\nu'}$. Repeated applications of (\ref{twist2}) yields
\begin{align*}DX-q^{(2\nu,2\nu')}XD \in\sum_{\substack{\lambda'<\nu'<2\gamma\\|\nu'|=|2\gamma|}}
(\mathscr{P}_{\theta})_{\lambda'}\sum_{\substack{\lambda<\nu<2\mu\\|\nu|=|2\mu|}}
(\mathscr{D}_{\theta})_{-\lambda} + \mathcal{J}_{|2\mu|+|2\gamma|-2}(\mathscr{PD}_{\theta}) \end{align*}
Since \begin{align*}XD\in \sum_{\substack{\nu'<2\gamma \\|\nu'|=|2\gamma|}}
(\mathscr{P}_{\theta})_{\nu'}\sum_{\substack{\nu<2\mu \\|\nu|=|2\mu|}}
(\mathscr{D}_{\theta})_{-\nu}\end{align*}
it follows that 
\begin{align*}
DX\in \sum_{\substack{\nu'<2\gamma\\|\nu'|=|2\gamma|}}
(\mathscr{P}_{\theta})_{\nu'}\sum_{\substack{\nu<2\mu\\|\nu|=|2\mu|}}
(\mathscr{D}_{\theta})_{-\nu} + \mathcal{J}_{|2\mu|+|2\gamma|-2}(\mathscr{PD}_{\theta})
\end{align*}
This holds for all weight vectors $D\in \sum_{\substack{\nu<2\mu \\|\nu|=|2\mu|}}
(\mathscr{D}_{\theta})_{-\nu}$ and $X\in \sum_{\substack{\nu'<2\gamma \\|\nu'|=|2\gamma|}}
(\mathscr{P}_{\theta})_{\nu'}$.
\end{proof}

Note that an obvious application of Lemma \ref{lemma:XD} is to $D_{-2\mu} = H^*_{2\mu}$ and $X_{2\gamma}=H_{2\gamma}$.  In this case, we get 
\begin{align}
\label{Hmueqn}
{H}_{2\mu}^*{H}_{2\gamma} -q^{(2\mu,2\gamma)}{H}_{2\gamma}{H}_{2\mu}^*\in \sum_{\substack{\nu<2\gamma\\ |\nu'|=|2\gamma|}}\  \sum_{\substack{\nu'<2\mu \\|\nu'|=|2\mu|}}(\mathscr{P}_{\theta})_{\nu}(\mathscr{D}_{\theta})_{-\nu'}+\mathcal{J}_{|2\mu|+|2\gamma|-2}(\mathscr{PD}_{\theta}).
\end{align}

The next lemma gives another formulation for the Capelli operators described in Lemma \ref{Capelliexpression}.
\begin{lemma}\label{lemma:Cmu}
The Capelli operator $C_{\mu}$ for  $\mu\in \Lambda^+_{\Sigma}$ satisfies
\begin{align*}C_{\mu}-\langle H_{2\mu}^*, H_{2\mu}\rangle^{-1} H_{2\mu}H^*_{2\mu}\in\sum_{\substack{\nu<2\mu\\ |\nu|=|2\mu|}}(\mathscr{P}_{\theta})_{\nu} (\mathscr{D}_{\theta})_{-\nu}
 \end{align*}
\end{lemma}
\begin{proof} It follows from Lemma \ref{Capelliexpression} that the vector $C_{\mu}$ is an element of 
 \begin{align*} \langle H_{2\mu}^*, H_{2\mu}\rangle^{-1} H_{2\mu}H^*_{2\mu}+ (U^-_+\cdot H_{2\mu})(U^+_+\cdot H_{2\mu}^*)
 \end{align*} 
  Recall that ${H}_{2\mu}^*$  is a lowest weight vector and equals a sum of products   $d_{e_1,f_1}d_{e_2,f_2}\dots d_{e_s,f_s}$.  Moreover,  each summand has the same terms with the only difference being reordering. Note that the weight of $d_{ab}$ is $-\epsilon_{a}-\epsilon_{b}$. Hence, the weight $-2\mu$ of  ${H}_{2\mu}^*$ is equal to 
$-2\mu_1\hat{\eta}_1-\cdots-2\mu_j\hat{\eta}_j=
\sum_i^s{-\epsilon_ {e_i}-\epsilon_{f_i}}.$ Similarly, ${H}_{2\mu}$ is a highest weight vector equal to a  a sum of the same products except that each $d_{e_i,f_i}$ is replaced with   $x_{e_i,f_i}$ for $ i=1,\dots, s$ and so the weight is $2\mu$ instead of $-2\mu$. 

By definition of $H_{2\mu}$  we must have $\mu\in \Lambda^+_{\Sigma}$. Recall that $\nu<2\mu$ means that $2\mu-\nu\in Q^+_N$.  Taking the restricted version of $2\mu-\nu$ means that
$2\mu-\tilde{\nu}\in Q^+_{\Sigma}$. As explained in Section \ref{subsection:special}, $Q_{\Sigma}\cap w_0\Lambda^+_{\Sigma}=0$.  Since $w_0Q_{\Sigma} = Q_{\Sigma}$, we also have $Q_{\Sigma}\cap \Lambda^+_\Sigma =0$.  Hence 
\begin{align*} H_{2\mu}H^*_{2\lambda}\notin \sum_{\nu<2\mu, \nu'<2\lambda}(\mathscr{P}_{\theta})_{\nu} (\mathscr{D}_{\theta})_{-\nu'}.
\end{align*}
for any $\mu,\lambda\in \Lambda^+_{\Sigma}$. 

Contributions to $(U^-_+\cdot H_{2\mu})$ take the form of repeated applications of generators $F_1,\dots, F_n$ to $H_{2\mu}$ as described in Lemma \ref{lemma:XD}.
Since $H_{2\mu}\in (\mathscr{P}_{\theta})_{2\mu}$ we see that $(U^-_+\cdot H_{2\mu})\subseteq \sum_{\nu<2\mu}(\mathscr{P}_{\theta})_{\nu}$.  
Note that all these contributions must lie in degree $|2\mu|$ since the action of $U_q(\mathfrak{g})$ on $\mathscr{P}_{\theta}$  preserves degree. Hence 
$$(U^-_+\cdot H_{2\mu})\subseteq \sum_{\substack{\nu<2\mu\\ |\nu|=|2\mu|}}(\mathscr{P}_{\theta})_{\nu}.$$ 
The same reasoning yields 
 \begin{align*}(U^+_+ \cdot H^*_{2\mu})\subseteq \sum_{\substack{\nu<2\mu\\ |\nu|=|2\mu|}} (\mathscr{D}_{\theta})_{-\nu}.
 \end{align*}  
Putting these two inclusions together  yields 
\begin{align*}(U^-_+\cdot H_{2\mu})(U^+_+\cdot H_{2\mu}^*)\subset \sum_{\substack{\nu<2\mu\\ |\nu|=|2\mu|}}\sum_{\substack{\nu'<2\mu\\ |\nu'|=|2\mu|}}(\mathscr{P}_{\theta})_{\nu} (\mathscr{D}_{\theta})_{-\nu'}.
\end{align*} Since $C_{\mu}\in (\mathscr{PD}_{\theta})^{U_q(\mathfrak{g})}$, we can just consider those summands with  $\nu=\nu'$ in the above formula. The lemma follows.
\end{proof}

 It is worth noting that the intersection of $\mathscr{PD}^{U_q(\mathfrak{g})}$ with the space $ \left(U^-\cdot H_{2\mu}\right) \left(U^+\cdot  H^*_{2\mu}\right)$ is one-dimensional. 
Indeed, by Lemma \ref{Capelliexpression} the Capelli operator  $C_{\mu}$ is a basis vector for this space.   As a consequence, we have
  \begin{align}\label{Jinv}
 \mathcal{J}_{2|\mu|}(\mathscr{PD}^{U_q(\mathfrak{g})})\subseteq \mathbb{C}(q)C_{\mu}+\sum_{\substack{\mu'\neq 2\mu \\ |\mu'|=|2\mu|}}\mathbb{C}(q)C_{\mu'}+\mathcal{J}_{2|\mu|-2}(\mathscr{PD}^{U_q(\mathfrak{g})}). \end{align} 
 
\begin{proposition}\label{prop:capelli} Consider a Capelli operator $C_{\mu}$ of degree $2\mu$ where \begin{align*}\mu=m_1\hat{\eta}_1 +\cdots +m_j\hat{\eta}_j
 \end{align*}
 with $|\mu| =\sum_{i=1}^jm_jj$.  It follows that 
 \begin{align*}C_{\mu}=a_{\mu,\mu}C_{\hat{\eta}_1}^{\mu_1}\cdots C_{\hat{\eta}_n}^{\mu_n} +
\sum_{|\mu'|<|\mu|}a_{\mu'}C_{\hat{\eta}_1}^{\mu'_1}\cdots C_{\hat{\eta}_j}^{\mu'_j}+a_0.
 \end{align*}
 where  
 $a_{\mu,\mu} =q^{(2\mu,2\mu)}\left(\prod_{i=1}^j \langle H_{2\hat{\eta}_i},H_{2\hat{\eta}_i}^*\rangle^{-\mu'_i}\right) $, $\mu'=\mu'_1\hat{\eta}_1 +\cdots +\mu'\hat{\eta}_j $ and every $a_{\mu'}$ and $a_0$ are scalars \end{proposition}
\begin{proof} 

Set  $c_{\mu,\gamma} = \langle H_{2\mu}^*, H_{2\mu}\rangle^{-1}  \langle H_{2\gamma}^*, H_{2\gamma}\rangle^{-1}$.
Using (\ref{Hmueqn}) we see that  that \begin{align}\label{CmuCgamma}
C_{\mu}C_{\gamma}&- c_{\mu,\gamma}
{H}_{2\mu}H_{2\mu}^*{H}_{2\gamma}{H}_{2\gamma}^*\in \sum_{\substack{\nu<2\mu\\ |\nu|=|2\mu|}}(\mathscr{P}_{\theta})_{\nu} (\mathscr{D}_{\theta})_{-\nu}\sum_{\substack{\lambda<2\gamma\\ |\lambda|=|2\gamma|}}(\mathscr{P}_{\theta})_{\lambda} (\mathscr{D}_{\theta})_{-\lambda}.
\end{align}
Switching the order of $H_{2\mu}^*$ and $H_{2\gamma}$ using (\ref{Hmueqn}) gives us \begin{align*}
C_{\mu}C_{\gamma}&- c_{\mu,\gamma}
q^{(2\mu,2\gamma)} 
{H}_{2\mu}{H}_{2\gamma}H_{2\mu}^*{H}_{2\gamma}^*\in  {H}_{2\mu}\left(\sum_{\substack{\nu<2\gamma\\ |\nu'|=|2\gamma|}}\  \sum_{\substack{\nu'<2\mu \\|\nu'|=|2\mu|}}(\mathscr{P}_{\theta})_{\nu}(\mathscr{D}_{\theta})_{-\nu'}\right){H}_{2\gamma}^*+{H}_{2\mu}(\mathcal{J}_{|2\mu|+|2\gamma|-2}(\mathscr{PD}_{\theta})){H}_{2\gamma}^*\cr&+\sum_{\substack{\nu<2\mu\\ |\nu|=|2\mu|}}(\mathscr{P}_{\theta})_{\nu} (\mathscr{D}_{\theta})_{-\nu}\sum_{\substack{\lambda<2\mu\\ |\lambda|=|2\mu|}}(\mathscr{P}_{\theta})_{\lambda} (\mathscr{D}_{\theta})_{-\lambda}.
\end{align*}
Note that ${H}_{2\mu}(\mathcal{J}_{|2\mu|+|2\gamma|-2}(\mathscr{PD}_{\theta})){H}_{2\gamma}^*\subseteq \mathcal{J}_{2|2\mu|+2|2\gamma|-2}.$
Also
\begin{align*}H_{2\mu}
 \sum_{\substack{\lambda<2\gamma\\|\lambda|=|2\gamma|}}(\mathscr{P}_{\theta})_{\lambda}
 \subseteq \sum_{\substack{\nu+\nu'<2\mu+2\gamma \\|\nu+\nu'|=|2\mu+2\gamma|}}(\mathscr{P}_{\theta})_{\nu+\nu'}
 \quad{\rm and}\quad
 \sum_{\substack{\lambda'<2\mu \\ |\lambda'|=|2\mu|}}
 (\mathscr{D}_{\theta})_{-\lambda'}
 H_{2\gamma}^*
 \subseteq \sum_{\substack{\nu+\nu'<2\mu+2\gamma \\|\nu+\nu'|=|2\mu+2\gamma|}}(\mathscr{D}_{\theta})_{-\nu-\nu'}.
 \end{align*}
 Hence  
 \begin{align}\label{H2mu} H_{2\mu}
( \sum_{\substack{\nu<2\gamma \\|\nu|=|2\gamma|}}\sum_{\substack{\nu'<2\mu \\|\nu'|=|2\mu|}}(\mathscr{P}_{\theta})_{\nu}
(\mathscr{D}_{\theta})_{-\nu'})
H_{2\gamma}^*\subseteq  \sum_{\substack{\nu+\nu'<2\mu+2\gamma \\|\nu+\nu'|=|2\mu+2\gamma|}}(\mathscr{P}_{\theta})_{\nu+\nu'}(\mathscr{D}_{\theta})_{-\nu-\nu'}.
\end{align}

We now look at the other terms that show up in the formula for $C_{\mu}C_{\gamma}$.  By Lemma \ref{lemma:XD} we have
 \begin{align*}\sum_{\substack{\nu<2\mu \\|\nu|=|2\mu|}}
(\mathscr{D}_{\theta})_{-\nu}\sum_{\substack{\nu'<2\gamma\\|\nu'|=|2\gamma|}}
(\mathscr{P}_{\theta})_{\nu'} = \sum_{\substack{\nu'<2\gamma\\|\nu'|=|2\gamma|}}
(\mathscr{P}_{\theta})_{\nu'}\sum_{\substack{\nu<2\mu\\|\nu|=|2\mu|}}
(\mathscr{D}_{\theta})_{-\nu} + \mathcal{J}_{|2\mu|+|2\gamma|-2}(\mathscr{PD}_{\theta}).
\end{align*}
Therefore
 \begin{align*}&
(\sum_{\substack{\nu<2\mu \\|\nu|=|2\mu|}}(\mathscr{P}_{\theta})_{\nu}
(\mathscr{D}_{\theta})_{-\nu})
(\sum_{\substack{\nu'<2\gamma\\|\nu'|=|\gamma|}}
(\mathscr{P}_{\theta})_{\nu'}(\mathscr{D}_{\theta})_{-\nu'}) \cr&\subseteq \sum_{\substack{\nu<2\mu \\|\nu|=|2\mu|}}\sum_{\substack{\nu'<2\gamma\\|\nu'|=|2\gamma|}}(\mathscr{P}_{\theta})_{\nu}(\mathscr{P}_{\theta})_{\nu'}(\mathscr{D}_{\theta})_{-\nu}(\mathscr{D}_{\theta})_{-\nu'} + \mathcal{J}_{2|2\mu|+2|2\gamma|-2}(\mathscr{PD}_{\theta}).
 \end{align*}
 From the relations satisfied by the generators for $\mathscr{P}_{\theta}$ (see Theorem \ref{theorem:polypart} and formula (\ref{quadformula})) we see that $(\mathscr{P}_{\theta})_{\nu}(\mathscr{P}_{\theta})_{\nu'}=(\mathscr{P}_{\theta})_{\nu+\nu'}$.  Similarly $(\mathscr{D}_{\theta})_{-\nu}(\mathscr{D}_{\theta})_{-\nu'}=(\mathscr{D}_{\theta})_{-\nu-\nu'}$.  Hence 
  \begin{align*}&
 \sum_{\substack{\nu<\mu \\|\nu|=|\mu|}}\sum_{\substack{\nu'<\gamma\\|\nu'|=|\gamma|}}(\mathscr{P}_{\theta})_{2\nu}(\mathscr{P}_{\theta})_{2\nu'}(\mathscr{D}_{\theta})_{-2\nu}(\mathscr{D}_{\theta})_{-2\nu'} + \mathcal{J}_{|\mu|+|\gamma|-2}(\mathscr{PD}_{\theta})\cr&\subseteq \sum_{\substack{\nu+\nu'<\mu +\gamma\\|\nu+\nu'|=|\mu+\gamma|}}(\mathscr{P}_{\theta})_{2\nu+2\nu'}(\mathscr{D}_{\theta})_{-2\nu-2\nu'} +\mathcal{J}_{2|2\mu|+2|2\gamma|-2}(\mathscr{PD}_{\theta}).
 \end{align*}
This formula combined with (\ref{CmuCgamma}) and  (\ref{H2mu}) gives us
\begin{align*}C_{\mu}C_{\gamma}&- c_{\mu,\gamma}q^{(2\mu,2\gamma)} ({H}_{2\mu}{H}_{2\gamma})({H}^*_{2\mu}{H}_{2\gamma}^*)\in \sum_{\substack{\nu+\nu'<\mu+\gamma \\|\nu+\nu'|=|\mu+\gamma|}}(\mathscr{P}_{\theta})_{2\nu+2\nu'}(\mathscr{D}_{\theta})_{-2\nu-2\nu'}+ \mathcal{J}_{2|2\mu|+2|2\gamma|-2}(\mathscr{PD}_{\theta}).\end{align*}

By Proposition \ref{prop:weights}, the elements $H_{2\mu}$ and $  H_{2\mu'}$ commute with each other and the same holds for $H_{2\mu}^*$ and $H_{2\gamma}^*$ where $\mu,\gamma$ are arbitrary weights. Moreover, by weight considerations discussed immediately following Proposition \ref{prop:weights}, $H_{2\mu}H_{2\gamma}=H_{2\mu+2\gamma}$ and  $H^*_{2\mu}H_{2\gamma}^*=H^*_{2\mu+2\gamma}$.  By Lemma \ref{lemma:Cmu}, $C_{\mu+\gamma}$ is the unique (up to nonzero scalar) element of 
\begin{align*}
H_{2\mu+2\gamma}H^*_{2\mu+2\gamma}+
 \sum_{\substack{\nu+\nu'<\mu+\gamma \\|\nu+\nu'|=|\mu+\gamma|}}(\mathscr{P}_{\theta})_{2\nu+2\nu'}(\mathscr{D}_{\theta})_{-2\nu-2\nu'}.
\end{align*}
Hence 
\begin{align*}C_{\mu}C_{\gamma}-  c_{\mu,\gamma}q^{(2\mu,2\gamma)}C_{\mu+\gamma}\in\mathcal{J}_{2|2\mu|+2|2\gamma|-2}(\mathscr{PD}_{\theta}).\end{align*}
Restricting our attention to $U_q(\mathfrak{g})$-invariant elements gives us 
\begin{align}\label{productformula}C_{\mu}C_{\gamma}-  c_{\mu,\gamma}q^{(2\mu,2\gamma)}C_{\mu+\gamma}\in\mathcal{J}_{2|2\mu|+2|2\gamma|-2}(\mathscr{PD}_{\theta}^{(U_q(\mathfrak{g})}).\end{align}

Set $\mu=m_1\hat{\eta}_1 +\cdots +m_j\hat{\eta}_j$ and $\mu'=m'_1\hat{\eta}_1 +\cdots +m'_j\hat{\eta}_j$ for $\mu'\neq \mu$ and $|\mu'|\leq |\mu|$. Using the above analysis and the fact that $(\mathscr{PD}_{\theta})^{U_q(\mathfrak{g})}$ is just the sum of one-dimensional subspaces of Capelli operators, we get 
\begin{align*}
C_{\mu}-a_{\mu,\mu}C_{\hat{\eta}_1}^{\mu_1}\cdots C_{\hat{\eta}_j}^{\mu_j}\in
\sum_{|\mu'|<|\mu|}\mathbb{C}(q)C_{\mu'} 
\end{align*} where  $a_{\mu,\mu} =q^{(2\mu,2\mu)}\left(\prod_{i=1}^j \langle H_{2\hat{\eta}_i},H_{2\hat{\eta}_i}^*\rangle^{-\mu'_i}\right) $.
 Indeed the right hand side is the contribution to the lower degree terms in $\mathcal{J}_{2|2\mu|+2|2\gamma|-2}(\mathscr{PD}_{\theta}^{U_q(\mathfrak{g})})$.
Continuing this process yields \begin{align*}C_{\mu}=a_{\mu,\mu}C_{\hat{\eta}_1}^{\mu_1}\cdots C_{\hat{\eta}_j}^{\mu_j}+\sum_
{|\mu'|<|\mu|} a_{\mu'}C_{\hat{\eta}_1}^{\mu'_1}\cdots C_{\hat{\eta}_j}^{\mu'_j}+a_0
 \end{align*}
 where each  $a_{\mu'}$ and $a_0$ are scalars. Thus the proof of the proposition is complete.
\end{proof}

Since $c_{\mu,\gamma} = c_{\gamma,\mu}$ and $q^{(2\mu,2\gamma)}=q^{(2\gamma,2\mu)}$, the reader can follow the calculations in the previous proposition to conclude that $C_{\mu}C_{\gamma}=C_{\gamma}C_{\mu}$ for all choices of $\mu$ and $\gamma$,  Hence the subalgebra generated by $C_{\hat{\eta}_1},\dots,C_{\hat{\eta}_j}$ is commutative.

 \subsection{The center and Capelli operators} \label{subsection:centerandCapelli}

 Recall the definition of the subring $Z$ generated by the central elements $\hat{z}_1, \dots, \hat{z}_n$ living inside $U^2_q(\mathfrak{g})$ from Section \ref{section:centralgen}.  
 The next theorem relates this subalgebra $Z$ of the center to the algebra of Capelli operators via the mapping $\Upsilon$ of Corollary \ref{corollary:upsilon}.
 
 \begin{theorem}\label{theorem:center_and_capelli}  The $U_q(\mathfrak{g})$-module algebra map $\Upsilon$ defines an algebra isomorphism from the polynomial subring  $Z$ of $Z(U^2_q(\mathfrak{g}))$ to the algebra coinciding with the vector space spanned by the Capelli operators.  Moreover, $\deg \Upsilon(\hat{z}_r)=2r$ for $r=1,\dots, n$.  \end{theorem}
 \begin{proof}

Set $c_r = \Upsilon(\hat{z}_r)$ for $r=1, \dots, n$.  By Corollary \ref{corollary:HC}, the $c_1,\dots, c_n$ generate the polynomial ring  in $\mathscr{PD}_{\theta}$ isomorphic to $Z$ via $\Upsilon$.  Also, $\deg c_r\leq 2r$ for each $r$.  In terms of the degree filtration $\mathcal{J}$, this means that $c_r\in \mathcal{J}_{2r} (\mathscr{PD}^{U_q(\mathfrak{g})})$. 

 We argue by induction on $j$ that 
 \begin{align}\label{inductivestatement}
 \mathbb{C}(q)[c_1, \dots, c_j] = \mathbb{C}(q)[C_{\hat{\eta}_1}, \dots, C_{\hat{\eta}_j}]  \end{align}
for $j=0,\dots,n$ and also show that $\deg c_j=2j$ for $j=1,\dots, n$.
  Note that the equality of algebras holds for  $j=0$ simply because both sides of (\ref{inductivestatement}) are just  the  scalars. Now assume that (\ref{inductivestatement}) is true for some $j$ satisfying $j>0$.  By Corollary \ref{corollary:HC},   $\mathbb{C}(q)[c_1, \dots, c_n]$ is a polynomial ring with $n$ variables. In particular,  the elements $c_1,\dots, c_n$ are algebraically independent.  Hence 
$c_{j+1}\notin  \mathbb{C}(q)[c_1, \dots, c_j]$.

Recall that the elements $ \hat{z}_r, r=1,\dots, n$ are in the center of $U_q(\mathfrak{g})$. Also,  by the definition of the Capelli operators (see the discussion preceding Lemma \ref{Capelliexpression}), $\deg C_{\hat{\eta}_{j+1}}=2|\hat{\eta}_{j+1}|=2j+2.$    By Corollary \ref{corollary:upsilon}, $c_r\in \mathscr{PD}_{\theta}^{U_q(\mathfrak{g})}$
for each $r$.  It follows that $c_{j+1}$ is an element of 
 \begin{align}\label{formula:relating0}
 c_{j+1}\in \mathbb{C}(q)C_{\hat{\eta}_{j+1}} +\sum_{|\mu|=j+1, \mu\neq  \hat{\eta}_{j+1}}\mathbb{C}(q)C_{\mu}+ \mathcal{J}_{2j}(\mathscr{PD}_{\theta}^{U_q(\mathfrak{g})}).\end{align}  By Lemma \ref{Capelliexpression},  each of the Capelli operators in the sum $\mathbb{C}(q)C_{\hat{\eta}_{j+1}} +\sum_{|\mu|=j+1, \mu\neq  \hat{\eta}_{j+1}}\mathbb{C}(q)C_{\mu}$ has degree $2\mu=2j+2$.  Furthermore, this is also true for the entire sum because each Capelli operator $C_{\mu}$   belongs to a subspace of the form  $(U^-H_{\mu})(U^+H_{\mu}^*)$ of degree $2\mu$, for distinct values of $\mu$. Hence $\deg(c_{j+1})=2(j+1)$.

 Consider a term of the form $C_{\mu}$ where $|\mu|=2j+2, \mu\neq  \hat{\eta}_{j+1}$. Since $C_{\mu}$ has degree $2j+2$ but $\mu\neq \hat{\eta}_{j+1}$, we must have
 \begin{align*}\mu=m_1\hat{\eta}_1 +\cdots +m_j\hat{\eta}_j
 \end{align*}
 with $\sum_{i=1}^jm_ii=j+1$.  In particular, the coefficient of $\hat{\eta}_{j+1}$ in $\mu$ is zero.  Moreover the same is true for $\hat{\eta}_s$ for $s>j+1$ because, with this assumption, the degree of $C_{\hat{\eta}_s}=2s>2(j+1).$ 
Using (\ref{formula:relating0}) we see that
 \begin{align}\label{formula:relating} c_{j +1}- a C_{\hat{\eta}_{j+1}} \in \mathbb{C}(q)[C_{\hat{\eta}_1}, \dots, C_{\hat{\eta}_{j}}]=\mathbb{C}(q)[c_1, \dots, c_{j}]
 \end{align} for some nonzero scalar $a$.  
 Assertion (\ref{inductivestatement}) now holds for $j+1$ replacing $j$. 
By induction,  this is true for $j=n$. Hence   the elements $C_{\hat{\eta}_1}, \dots, C_{\hat{\eta}_{n}}$
 are algebraically independent. 

 Now consider $C_{\lambda}$ with $\lambda$ arbitrary.  Since $\hat{\eta}_1\dots, \hat{\eta}_n$ form a basis for $\Lambda^+_{\Sigma}$, we can write 
$\lambda=\lambda_1\hat{\eta}_1+\lambda_2\hat{\eta}_2+\cdots \lambda_n\hat{\eta}_n$. By Proposition \ref{prop:capelli},
  \begin{align*}C_{\lambda}=a_{\lambda,\lambda}C_{\hat{\eta}_1}^{\lambda_1}\cdots C_{\hat{\eta}_n}^{\lambda_n} +
\sum_{|\lambda'|<|\lambda|}a_{\lambda'}C_{\hat{\eta}_1}^{\lambda'_1}\cdots C_{\hat{\eta}_j}^{\lambda'_j}+a_0.
 \end{align*} Hence $C_{\lambda}\in  \mathbb{C}(q)[C_{\hat{\eta}_1}, \dots, C_{\hat{\eta}_n}].$  Thus the vector space spanned by the Capelli operators equals the polynomial algebra $\mathbb{C}(q)[C_{\hat{\eta}_1}, \dots, C_{\hat{\eta}_n}].$   The theorem follows from the fact that this polynomial algebra is equal to  $\mathbb{C}(q)[c_1,\dots, c_n]$. 
   \end{proof}
  
 
   
   \section{Eigenvalues of Capelli operators}
   
  \subsection{Definition and degree}\label{section:dd}

We will be using a polynomial algebra in the restricted root setting  to study the eigenvalues of the Capelli operators.  This will help us in realizing the eigenvalues from two perspectives: as dotted Weyl invariants and as symmetric polynomials.   The second point of view helps in the identification of the eigenvalues with Knop-Sahi interpolation polynomials in Section \ref{section:KS}.

Descriptions of the restricted weights  can be found in Section \ref{section:restricted-root-system}.  The restricted root system for each family is of type $A_{n-1}$ and the  elements $\epsilon^{\Sigma}_i$ for $i=1,\dots, n$ form a fixed orthonormal basis for this system.  Hence passing to the corresponding elements in the Cartan subalgebra, we see that  $K_{2\epsilon^{\Sigma}_i}$ for $i=1,\dots,n$ are algebraically independent.  
Therefore, \begin{align}\label{Cappolyformula}\mathbb{C}(q)[K_{2\epsilon^{\Sigma}_1}, K_{2\epsilon^{\Sigma}_2}, \dots, K_{2\epsilon^{\Sigma}_n}]\end{align} is a polynomial ring with   variables $K_{2\epsilon^{\Sigma}_1},\dots, K_{2\epsilon^{\Sigma}_n}$.
Moreover, symmetric polynomials in \begin{align*}\mathbb{C}(q)[g^{-1}K_{2\epsilon^{\Sigma}_1}, g^{-2}K_{2\epsilon^{\Sigma}_2}, \dots, g^{-n}K_{2\epsilon^{\Sigma}_n}]\end{align*} can be identified with $\mathbb{C}(q)[\mathcal{A}_2]^{W_{\Sigma}\bullet}$ for an appropriate choose of $g$. 
Indeed, we have the following lemma.  The proof involves explicit descriptions of the elements $2\epsilon^{\Sigma}_i$.  In particular,
by Section \ref{section:restricted-root-system}, we have $\epsilon^{\Sigma}_i=\tilde{\epsilon}_i=\epsilon
_i$ in Type AI, $\epsilon^{\Sigma}_i=\tilde{\epsilon}_{2i} = (\epsilon_{2i-1}+\epsilon_{2i})/2$ in Type AII and  $\epsilon^{\Sigma}_i=\tilde{\epsilon}_i=(\epsilon
_i+\epsilon_{n+i})/2$ in the diagonal case.

Set $g^{-1} = q^{-2}$ in Type AI and the diagonal case, and set $g^{-1} =q^{-4}$ in Type AII.  Let $x_1,\dots, x_n$ be indeterminates and consider the algebra isomorphism $\kappa$ from $\mathbb{C}(q)[x_1,\dots, x_n]$ to  $\mathbb{C}(q)[g^{-1}K_{2\epsilon^{\Sigma}_1}, \dots, g^{-n}K_{2\epsilon^{\Sigma}_{n}}]$ defined by $\kappa(x_j)=g^{-j}K_{2\epsilon^{\Sigma}_j}$ for all $j=1,\dots, n$. Let $\mathcal{S}_n$ is the symmetric group acting on the polynomial ring
$\mathbb{C}(q)[x_1,\dots, x_n]$ by permuting the elements $x_1,\dots, x_n$.

\begin{lemma}\label{lemma:explicit}  We have
\begin{align*}\kappa(\mathbb{C}(q)[x_1,\dots, x_n]^{\mathcal{S}_n})=\mathbb{C}(q)[\mathcal{A}_2]^{\mathcal{W}_{\theta}\bullet}.
\end{align*}
\end{lemma}
\begin{proof}Recall that $\rho$ is the half sum of the positive roots for the root system of $\mathfrak{gl}_N$.  Note that for each fixed $i$, the set of roots $\alpha$ such that $(\alpha,\epsilon_i)\neq 0$ can be partitioned into two sets
$\{\alpha_j+\cdots + \alpha_{i-1}|j\leq i-1\}$ and $\{\alpha_i+\cdots +\alpha_j|\ j\geq i\}$.  For elements $\alpha$ in the first set $(\alpha,\epsilon_i) = -1$ and for elements $\alpha$ in the second set, $(\alpha,\epsilon_i) = 1$.  Since the first set has $i-1$ elements and the second set has $N-i$ elements, we have that $(\rho, \epsilon_i) = (N-2i-1)/2$.  Now returning to the restricted root cases under consideration, we set $N=n$. For $i=1, \dots, n$, we have 
\begin{align*}q^{(\tilde{\rho},2{\epsilon}^{\Sigma}_i)}K_{2{\epsilon^{\Sigma}_i}} = q^{(\tilde{\rho},2\tilde{\epsilon}_i)}K_{2\tilde{\epsilon_i}} =q^{(\rho,2{\epsilon}_i)}K_{2{\epsilon_i}} =q^{n-2i-1} K_{2{\epsilon}_i}= q^{n-2i-1} K_{2\epsilon^{\Sigma}_i} =q^{n-1}(q^{-2i} K_{2\epsilon^{\Sigma}_i}) .
\end{align*}
for Type AI,
\begin{align*}q^{(\tilde{\rho},2{\epsilon}^{\Sigma}_i)}K_{2\epsilon^{\Sigma}_i} = q^{(\tilde{\rho},2\tilde{\epsilon}_i)}K_{2\tilde{\epsilon}_i} =q^{(\rho_1,{\epsilon}_i)+(\rho_2,{\epsilon}_{i+n})}K_{2{\epsilon}^{\Sigma}_i} =q^{n-2i-1} K_{2{\epsilon}^{\Sigma}_i}=q^{n-1}(q^{-2i} K_{2{\epsilon}^{\Sigma}_i})
\end{align*} in the diagonal case and 
\begin{align*}
q^{(\tilde{\rho},2{\epsilon}^{\Sigma}_{i})}K_{2\epsilon^{\Sigma}_{i}} = q^{(\rho,2\tilde{\epsilon}_{2i})}K_{2{\epsilon}^{\Sigma}_{i}} 
=q^{(\rho,\epsilon_{2i-1}+\epsilon_{2i})}K_{2{\epsilon}^{\Sigma}_{i}} =q^{2n-4i} K_{2{\epsilon}^{\Sigma}_{i}} =q^{2n}(q^{-4i}K_{2{\epsilon}^{\Sigma}_{i}} )
\end{align*} in Type AII. 

Set $g^{-1} = q^{-2}$ in Type AI and the diagonal type. Set $g^{-1}= q^{-4}$ in Type AII.
It follows that $\mathbb{C}(q)[q^{-2i}K_{2\epsilon^{\Sigma}_i}|\  i=1, \dots, n]=\mathbb{C}(q)[g^{-i}K_{2\epsilon^{\Sigma}_i}|\  i=1, \dots, n] $ in Type AI and the diagonal case. In Type AII, we have $\mathbb{C}(q)[q^{-4i}K_{2\epsilon^{\Sigma}_{i}}|\ i=1, \dots, n]= \mathbb{C}(q)[g^{-i}K_{2\epsilon^{\Sigma}_i}|\  i=1, \dots, n]$.  Thus we are interested in symmetric polynomials in the terms $g^{-i}K_{2\epsilon^{\Sigma}_i}$ for $i=1, \dots, n$.

As explained in the overview part of Section \ref{section:restricted-root-system}, $w_0\hat{\eta}_i=\epsilon^{\Sigma}_{n-i+1}+\epsilon^{\Sigma}_{n-i+2}+\cdots +\epsilon^{\Sigma}_n$.  Hence $K_{2w_0\hat{\eta}_i}=K_{2\epsilon^{\Sigma}_{n-i+1}}K_{2\epsilon^{\Sigma}_{n-i+2}}\cdots K_{2 \epsilon^{\Sigma}_n}$ which is an element of the polynomial ring described by (\ref{Cappolyformula}).  Now consider a partition $\mu_1\geq \mu_2\geq \cdots \geq\mu_n\geq 0$ and the corresponding weight $2w_0\mu=2\mu_nw_0\hat{\eta}_n+ \cdots + 2\mu_1w_0\hat{\eta}_1$.  It follows that 
$K_{2w_0\mu}=K_{2w_0\hat{\eta}_n}^{\mu_n}K_{2w_0\hat{\eta}_{n-1}}^{\mu_{n-1}}\cdots K_{w_0\hat{\eta}_1}^{\mu_1}$ is also an element of the same polynomial ring.  Note that 
\begin{align*}q^{(\rho, 2w_0\mu)} = \left\{\begin{matrix}q^{n-1}(q^{-2\mu_n}q^{-2\mu_{n-1}}\cdots q^{-2\mu_1})& {\rm \ in \ Type\ AI\ and\ the\  diagonal \ type}\\
q^{2n}(q^{-4\mu_n}q^{-4\mu_{n-1}}\cdots q^{-4\mu_1})& {\rm \ in \ Type\ AII}
\end{matrix}\right.
\end{align*}
On the other hand, we have 
\begin{align*}q^{(\rho, 2ww_0\mu)} = \left\{\begin{matrix}q^{n-1}(q^{-2\mu_{w(n)}}q^{-2\mu_{w(n-1)}}\cdots q^{-2\mu_{w(1)}})& {\rm \ in \ Type\ AI\ and\ the\  diagonal \ type}\\
q^{2n}(q^{-4\mu_{w(n)}}q^{-4\mu_{w(n-1)}}\cdots q^{-4\mu_{w(1)}})& {\rm \ in \ Type\ AII}
\end{matrix}\right.
\end{align*}
where $w$ is in the symmetric group $\mathcal{S}_n$ on $n$ symbols.  Hence a basis for symmetric polynomials in $$\mathbb{C}(q)[g^{-1}K_{2\epsilon^{\Sigma}_1}, \cdots, g^{-n}K_{2\epsilon^{\Sigma}_n}]$$ takes the form 
\begin{align*}\sum_{w\in \mathcal{S}_n}q^{(\rho, 2ww_0\mu)}K_{2ww_0\mu}
\end{align*}
as $\mu$ runs over partitions.

Recall the definition of the positive root $\beta_{j,k}= \alpha_j+\alpha_{j+1} +\cdots +\alpha_{k}$ in Section \ref{subsection:Relationships}.  Let $\tilde{\beta}_{j,k}$ denote a restricted root version. Thus $\tilde{\beta}_{j,k}=\alpha^{\Sigma}_j+\alpha^{\Sigma}_{j+1}+\cdots +\alpha^{\Sigma}_k=\epsilon^{\Sigma}_j-\epsilon^{\Sigma}_{k+1}$.   Hence the reflection $s_{\tilde{\beta}_{j,k}}$ in $W_{\Sigma}$ corresponding to $\tilde{\beta}_{j,k}$ gives us $s_{\tilde{\beta}_{j,k}}w_0\hat{\eta}_i=w_0\hat{\eta}_i$ if both $j$ and  $k$ are strictly less than $n-i+1$ or strictly greater than $n-i+1$.  If $j<n-i+1\leq k<n$, then 
\begin{align}\label{sbeta} s_{\tilde{\beta}_{j,k}}w_0\hat{\eta}_i &= s_{\tilde{\beta}_{j,k}}(\epsilon^{\Sigma}_{n-i+1}+\epsilon^{\Sigma}_{n-i+2}+\cdots +\epsilon^{\Sigma}_n)\cr&= \epsilon^{\Sigma}_j+
\epsilon^{\Sigma}_{n-i+1}+\epsilon^{\Sigma}_{n-i+2}+\cdots +\epsilon^{\Sigma}_{k}+(\epsilon^{\Sigma}_{k+1}-\epsilon^{\Sigma}_{k+1})+\epsilon^{\Sigma}_{k+2}+\cdots+\epsilon^{\Sigma}_n
\end{align}
It follows that $s_{\tilde{\beta}_{j,k}}$ in $W_{\Sigma}$ corresponds to the transposition $(j,k)$ in the symmetric group $\mathcal{S}_n$ on $n$ letters. 
This gives us a way to translate between the two isomorphic groups $W_{\Sigma}$ and $\mathcal{S}_n$.  Thus \begin{align*}\sum_{w\in \mathcal{S}_n}q^{(\tilde{\rho}, 2ww_0\hat{\eta}_i)} K_{2ww_0\hat{\eta}_i} =\sum_{w\in W_{\Sigma}}q^{(\tilde{\rho}, 2ww_0\hat{\eta}_i)} K_{2ww_0\hat{\eta}_i}.
\end{align*}
Hence the vector space of symmetric polynomials in $\mathbb{C}(q)[g^{-1}K_{2\epsilon^{\Sigma}_1}, \cdots, g^{-n}K_{2\epsilon^{\Sigma}_n}]$ is equal to the basis of the dotted $W_{\Sigma}$-invariants  $\mathbb{C}(q)[\mathcal{A}_2]^{\mathcal{W}_{\theta}\bullet}$.
\end{proof}

Let $B\in \sum_{\mu}\mathbb{C}(q)C_{\mu}=\mathscr{PD}_{\theta}^{U_q(\mathfrak{g})}$ and $z\in Z$ such that $B= \Upsilon(z)$.  Set $\mathcal{E}(B) = \tilde{\varphi}_{HC}(z)$.  By   Theorem \ref{theorem:dotted_Weyl},  $\mathcal{E}(B)$ is an element of $\mathbb{C}(q)[\mathcal{A}_2]^{W_{\Sigma}\bullet}$. Alternatively, we can use Lemma \ref{lemma:explicit} with its identification of  $\mathbb{C}(q)[\mathcal{A}_2]^{W_{\Sigma}\bullet}$ with the algebra of symmetric polynomials inside of $\mathbb{C}(q)[g^{-i}K_{2\epsilon^{\Sigma}_i}|\  i=1, \dots, n] $.  
   It follows from Theorem \ref{theorem:dotted_Weyl}, Corollary \ref{corollary:HC}, Theorem \ref{theorem:center_and_capelli} and Lemma \ref{lemma:explicit} that $\mathcal{E}$ defines an algebra isomorphism from $ \mathscr{PD}_{\theta
 }^{U_q(\mathfrak{g})}$ to 
$\mathbb{C}(q)[\mathcal{A}_2]^{W_{\Sigma}\bullet}=\mathbb{C}(q)[g^{-i}K_{2\epsilon^{\Sigma}_i}|\  i=1, \dots, n]^{\mathcal{S}_n}$.  
Recall that as  a vector space, $$\mathbb{C}(q)[\mathcal{A}_2]^{W_{\Sigma}\bullet}=\sum_{\lambda\in \Lambda^+_{\Sigma}}\mathbb{C}(q)m_{2w_0\lambda}^{\Sigma}$$
Set ${\rm Stab}(w_0\hat{\eta}_i)=\{w\in W_{\Sigma}|\ ww_0\hat{\eta}_i=w_0\hat{\eta}_i\}$.
Given $w_0\lambda=m_1w_0\hat{\eta}_1+\cdots + m_nw_0\hat{\eta}_n\in w_0 \Lambda_{\Sigma}^+$ 
set  $a_{\lambda}=\prod_{i=1}^na^{m_i}_{w_0\hat{\eta}_i}$ where $a_{w_0\hat{\eta}_i} =  q^{(\tilde{\rho}, 2w_0\hat{\eta}_i)}| {\rm Stab}(w_0\hat{\eta}_i)|$
 The next lemma takes a closer look at this ring of dotted $W_{\Sigma}$.

 \begin{lemma} \label{lemma:Kalgebra}There is an algebra isomorphism $\mathcal{V}$ from $ \mathbb{C}(q)[\mathcal{A}_2]^{W_{\Sigma}\bullet}$ to
$\mathbb{C}(q)[K_{2w_0\lambda}| \ \lambda \in \Lambda_{\Sigma}^+] $ 
which sends 
$$b(m^{\Sigma}_{2w_0\hat{\eta}_1})^{m_1}\cdots(m^{\Sigma}_ {2w_0\hat{\eta}_n})^{m_n}{\rm \quad to\quad }
ba_{\lambda}K_{2w_0\hat{\eta}_1}^{m_1}
\cdots K_{2w_0\hat{\eta}_n}^{m_n}$$ for each $w_0\lambda=m_1w_0\hat{\eta}_1+\cdots + m_nw_0\hat{\eta}_n\in \Lambda_{\Sigma}^+$ and all scalars $b\in \mathbb{C}(q)$.  
  \end{lemma}
\begin{proof}  As explained
in Section \ref{section:HCMs} (see the discussion preceding the inclusion (\ref{rhcmap13})), $$\mathbb{C}(q)[K_{2w_0\hat{\eta}_1},\dots, K_{2w_0\hat{\eta}_n}]\subset\mathbb{C}(q)[\mathcal{A}_2].$$  
 By Lemma \ref{lemma:explicit} and its proof,
$K_{2w_0\hat{\eta}_i}=K_{2\epsilon^{\Sigma}_{n-i+1}}K_{2\epsilon^{\Sigma}_{n-i+2}}\cdots K_{2 \epsilon^{\Sigma}_n}$. It follows that \begin{align*}K_{2w_0\hat{\eta}_i}\in \mathbb{C}(q)[K_{2\epsilon^{\Sigma}_{n-i+1}}|\ i=1,\dots n] {\rm \quad but\quad} K_{2w_0\hat{\eta}_i}\notin  \mathbb{C}(q)[K_{2\epsilon^{\Sigma}_{n-i+1}}|\ i=2,\dots,  n]
\end{align*}
Hence, the elements $K_{2w_0\hat{\eta}_i}$ for $i=1,\dots, n$ are algebraically independent and $\mathbb{C}(q)[K_{2w_0\hat{\eta}_i}|\ i=1,\dots,n]$ is a polynomial ring in $n$ variables 
 $K_{2w_0\hat{\eta}_1},\dots,K_{2w_0\hat{\eta}_n} $.  
 Moreover, 
$$\mathbb{C}(q)[K_{2w_0\lambda}| \ w_0\lambda \in w_0\Lambda_{\Sigma}^+] = \mathbb{C}(q)[K_{2w_0\hat{\eta}_1},\dots K_{2w_0\hat{\eta}_n}].$$

By the proof of Corollary \ref{corollary:HC}, 
\begin{align*}m^{\Sigma}_{2w_0\hat{\eta}_i}&=
\sum_{w\in W_{\Sigma}}q^{(\tilde{\rho},2ww_0\hat{\eta}_i)}K_{2ww_0\hat{\eta}_i}
\in a_{w_0\hat{\eta}_i}K_{2w_0\hat{\eta}_i}+\sum_{\beta\in Q^+_{\Sigma}}\mathbb{C}(q)K_{2w_0\hat{\eta}_i+2\beta}.
\end{align*} 

As explained in Lemma \ref{L8.8}, $Q_{\Sigma}\cap w_0\Lambda^+_{\Sigma}=0$.  Hence
\begin{align*}\left( \sum_{\beta\in Q^+_{\Sigma}}\mathbb{C}(q)K_{2w_0\hat{\eta}_i+2\beta}\right)\cap \mathbb{C}(q)[K_{2w_0\lambda}| \ \lambda \in \Lambda_{\Sigma}^+]=0.
\end{align*}
Now consider $\lambda=m_1w_0\hat{\eta}_1+\cdots + m_nw_0\hat{\eta}_n\in \Lambda_{\Sigma}^+$. Note that $\prod_{i=1}^n K_{2w_0\hat{\eta}_i}^{m_i} = K_{2w_0\lambda}$.  On the other hand, 
\begin{align*}\prod_{i=1}^n \prod_{s=1}^{m_i}(K_{2w_0\hat{\eta}_i+\beta_i^s})^{m_i}=K_{2w_0\lambda + \beta}
\end{align*} where at least one of the $\beta_i^s\in Q^+_{\Sigma}$ on the left hand side and $\beta\in Q^+_{\Sigma}$ on the right hand side.
 It follows that \begin{align*}(m^{\Sigma}_{2w_0\hat{\eta}_1})^{m_1}\cdots(m^{\Sigma}_ {2w_0\hat{\eta}_n})^{m_n}=a_{\lambda} K_{2w_0\lambda}+\sum_{\beta\in Q^+_{\Sigma}}\mathbb{C}(q)K_{2w_0\lambda + \beta}
\end{align*}
In other words, $w_0\lambda=m_1w_0\hat{\eta}_1+\cdots + m_nw_0\hat{\eta}_n$ is the unique partition   in $w_0\Lambda^+_{\Sigma}$ that appears in the expanded expression for $(m^{\Sigma}_{2w_0\hat{\eta}_1})^{m_1}\cdots(m^{\Sigma}_ {2w_0\hat{\eta}_n})^{m_n}$ as given above.
Hence there is a vector space isomorphism, which we refer to as $\mathcal{V}$,  from $ \mathbb{C}(q)[\mathcal{A}_2]^{W_{\Sigma}\bullet}$ to
$\mathbb{C}(q)[K_{2w_0\lambda}| \ \lambda \in \Lambda_{\Sigma}^+] $ which sends $bm^{\Sigma}_{2w_0\lambda}$ to $ba_{\lambda}K_{2w_0\lambda}$ for each $\lambda\in \Lambda^+_{\Sigma}$.
We prove that the map $\mathcal{V}$  also defines an algebra isomorphism. To see this, we consider the product  \begin{align*}
\left((m^{\Sigma}_{2w_0\hat{\eta}_1})^{m_1}\cdots(m^{\Sigma}_ {2w_0\hat{\eta}_n})^{m_n}\right)\left((m^{\Sigma}_{2w_0\hat{\eta}_1})^{\mu_1}\cdots(m^{\Sigma}_ {2w_0\hat{\eta}_n})^{\mu_n}\right)
 \end{align*} where $\sum_{i=1}^nm_i\hat{\eta}_i = \lambda$ and $\sum_{i=1}^n\mu_i\hat{\eta}_i =\mu$. This product equals
\begin{align*} (m^{\Sigma}_{2w_0\hat{\eta}_1})^{m_1+\mu_1}\cdots(m^{\Sigma}_ {2w_0\hat{\eta}_n})^{m_n+\mu_n}
\end{align*}
Now \begin{align*}\mathcal{V}\left((m^{\Sigma}_{2w_0\hat{\eta}_1})^{m_1}\cdots(m^{\Sigma}_ {2w_0\hat{\eta}_n})^{m_n}\right)&\mathcal{V}\left((m^{\Sigma}_{2w_0\hat{\eta}_1})^{\mu_1}\cdots(m^{\Sigma}_ {2w_0\hat{\eta}_n})^{\mu_n}\right)=a_{\lambda}K_{2w_0\lambda}a_{\mu}K_{2w_0\mu}\cr&=a_{\lambda}a_{\mu}K_{2w_0\lambda}K_{w_0\mu}
\cr&=(\prod_{i=1}^{n}a^{m_i}_{w_0\hat{\eta}_i})(\prod_{i=1}^{n}a^{\mu_i}_{w_0\hat{\eta}_i})K_{2w_0(\lambda+\mu)}=(\prod_{i=1}^{n}a^{m_i+\mu_i}_{w_0\hat{\eta}_i})K_{2w_0(\lambda+\mu)}\cr&=\mathcal{V}((m^{\Sigma}_{2w_0\hat{\eta}_1})^{m_1+\mu_1}\cdots(m^{\Sigma}_ {2w_0\hat{\eta}_n})^{m_n+\mu_n}).
\end{align*} Hence the map $\mathcal{V}$ is a vector space isomorphism that also preserves multiplication.  Thus $\mathcal{V}$ defines an algebra isomorphism from  $ \mathbb{C}(q)[\mathcal{A}_2]^{W_{\Sigma}\bullet}$ to
$\mathbb{C}(q)[K_{2w_0\lambda}| \ \lambda \in \Lambda_{\Sigma}^+] $ as stated in the lemma.
\end{proof}


Recall there is a filtration $\mathcal{J}$ on $\mathscr{PD}_{\theta}$. It is inherited by $\mathscr{PD}_{\theta}^{U_q(\mathfrak{g})}$ and thus carried over by the isomorphism $\mathcal{E}$ to $\mathbb{C}(q)[g^{-1}K_{2\epsilon^{\Sigma}_1}, \dots, g^{-n}K_{2\epsilon^{\Sigma}_{n}}]$.
The induced filtration produces the ``filtration degree" which we  denote by $\mathcal{J}{\rm deg}$.  Note that an element $y$ in  $ \mathscr{PD}_{\theta}^{U_q(\mathfrak{g})}$ has filtration degree $t$ provided 
\begin{align*}y\in \mathcal{J}_t( \mathscr{PD}_{\theta}^{U_q(\mathfrak{g})}){\rm \quad\ and\quad}y\notin \mathcal{J}_{t-1}( \mathscr{PD}_{\theta}^{U_q(\mathfrak{g})})\end{align*}
 By Lemma \ref{lemma:epsilon} and Corollary \ref{corollary:upsilon},  $\Upsilon(K_{2\epsilon^{\Sigma}_n}) \in \sum_{ijkl}\mathbb{C}(q)x_{ij}d_{kl}$ and hence  $ \mathcal{J}{\rm deg}K_{2\epsilon^{\Sigma}_n}=2$.  Moreover, applying the reflection $s_{\tilde{\beta}_{in}}$ as in the previous lemma, we see that $K_{2\epsilon^{\Sigma}_i} = s_{\tilde{\beta}_{in}}K_{2\epsilon^{\Sigma}_n}$. Since the action of $U_q(\mathfrak{g})$ on $\mathscr{PD}_{\theta}$ preserves the $\mathcal{J}{\rm deg}$, we get  $\mathcal{J}{\rm deg}K_{2\epsilon^{\Sigma}_i}=2$ all $i$.


Given $\lambda\in \Lambda_{\Sigma}^+$,  let $\mathcal{E}_{\lambda}$ be the polynomial in $\mathbb{C}(q)[x_1,\dots,x_n]$ such that 
 \begin{align}\label{Capkap}\kappa(\mathcal{E}_{\lambda})
 =\mathcal{E}(C_{\lambda}).
 \end{align}   We claim that $\mathcal{E}_{\lambda}$ is a symmetric polynomial. Indeed $C_{\lambda}\in (\mathscr{PD}_{\theta})^{U_q(\mathfrak{g})}$.  Thus 
$ \mathcal{E}(C_{\lambda})\in \mathbb{C}(q)[\mathcal{A}_2]^{W_{\Sigma}\bullet}.$ By (\ref{Capkap}), 
$\kappa(\mathcal{E}_{\lambda})\in \mathbb{C}(q)[\mathcal{A}_2]^{W_{\Sigma}\bullet}.$ Now Lemma \ref{lemma:explicit} tells us that 
\begin{align*}
\mathcal{E}_{\lambda}\in \mathbb{C}(q)[x_1,\dots, x_n]^{\mathcal{S}_n}.
\end{align*}

Note that it makes sense to talk about the degree of $\mathcal{E}_{\lambda}$ using the standard  total degree function for  $\mathbb{C}(q)[x_1,\dots,x_n]$. (We emphasize that  $\deg x_i = 1$ for all $i$.) We refer to this degree  as the ``polynomial degree" and denote it by ${\rm pdeg}$. 

Since $\mathcal{J}{\rm deg}(m^{\Sigma}_{2w_0\hat{\eta}_1})^{m_1}\cdots(m^{\Sigma}_ {2w_0\hat{\eta}_n})^{m_n})= \mathcal{J}{\rm deg}(K_{2w_0\hat{\eta}_1}^{m_1}
\cdots K_{2w_0\hat{\eta}_n}^{m_n})$  it follows that the algebra isomorphism $\mathcal{V}$ of Lemma \ref{lemma:Kalgebra} preserves the filtration degree.   The same holds for the polynomial degree.


  \begin{lemma}\label{lemma:degree} For each $\lambda \in \Lambda_{\Sigma}^+$, the polynomial degree of $\mathcal{E}_{\lambda}$ is $|\lambda|$. 
  \end{lemma}
  
  \begin{proof} 
  We begin by showing ${\rm pdeg}\  \mathcal{E}_{\hat{\eta}_i} = i$ for $i=1,\dots, n$.  
  The reader should observe that we already know that $ |\hat{\eta}_i| = i$.  This will establish a special case of the lemma.  
  By Theorem \ref{theorem:center_and_capelli} and its proof, we have a concurrence of  algebras   \begin{align*}
 \mathbb{C}(q)[c_1, \dots, c_j] = \mathbb{C}(q)[C_{\hat{\eta}_1}, \dots, C_{\hat{\eta}_j}]  \end{align*}
 for $j=1,\dots, n$.  
We also have  (see Theorem \ref{theorem:center_and_capelli}, formula (\ref{formula:relating})) \begin{align}\label{formCap}
c_{j }-a_j C_{\hat{\eta}_{j}} \in \mathbb{C}(q)[C_{\hat{\eta}_1}, \dots, C_{\hat{\eta}_{j-1}}]=\mathbb{C}(q)[c_1, \dots, c_{j-1}]
 \end{align} for nonzero scalars $a_{j}$.  

Recall that $\mathcal{E}$ defines an algebra isomorphism from $(\mathscr{PD}_{\theta})^ {U_q(\mathfrak{g})}$ to  the ring of dotted invariants $\mathbb{C}(q)[\mathcal{A}_2]^{\mathcal{W}_{\theta}\bullet}$ which equals symmetric polynomials in  $\mathbb{C}(q)[g^{-i}K_{2\epsilon^{\Sigma}_i}|\  i=1, \dots, n] $.
   Applying the algebra isomorphism $\mathcal{E}$ to  (\ref{formCap}) yields 
   \begin{align*}
\mathcal{E}(c_j)- a_j \mathcal{E}(C_{\hat{\eta}_j})\in\mathbb{C}(q)[\mathcal{E}(c_1),\dots, \mathcal{E}(c_{j-1})].
  \end{align*} 
  We show 
 \begin{align}\label{polyequal}
  \mathbb{C}(q)[\mathcal{E}(c_1),\dots, \mathcal{E}(c_s)]
  =\mathbb{C}(q)[m^{\Sigma}_{2w_0\hat{\eta}_{1}},m^{\Sigma}_{2w_0\hat{\eta}_{2}},\dots, m^{\Sigma}_{2w_0\hat{\eta}_{s}}]
  \end{align}
 for all $1\leq s\leq n$.

Consider  Type AI and the diagonal case. As explained in the proof of Theorem \ref{theorem:dotted_Weyl}, $\mathcal{E}(c_i)=\tilde{\varphi}_{HC}(\hat{z}_{i}) = m^{\Sigma}_{2w_0\hat{\eta}_{i}}$ (up to a nonzero scalar multiple) for $i=1,\dots,n$. This proves (\ref{polyequal}) for these two types.
For Type AII, first note that $\hat{z}_s=z_{2w_0\hat{\omega}_s}$ for $s=1,\dots, n$. The proof of Theorem \ref{theorem:dotted_Weyl} shows that $m^{\Sigma}_{2w_0\hat{\eta}_{2k}}$
  is a (nonzero) linear combination of  elements in the set 
\begin{align*}
\{ \tilde{\varphi}_{HC}(\hat{z}_{2k}), m^{\Sigma}_{2w_0\hat{\eta}_{k-j}}m^{\Sigma}_{2w_0\hat{\eta}_{k+j}}| 1\leq j<k\}
\end{align*} for $1\leq 2k\leq n$.  Similarly, $m^{\Sigma}_{2w_0\hat{\omega}_{2k+1}}$
  is a linear combination of  elements in the set 
\begin{align*}
\{ \tilde{\varphi}_{HC}(\hat{z}_{2k+1} ), m^{\Sigma}_{2w_0\hat{\eta}_{k-j}}m^{\Sigma}_{2w_0\hat{\eta}_{k+1+j}}| 1\leq j<k\}
\end{align*} for $1\leq 2k+1\leq n$.
Hence 
\begin{align*}
t_{2k}\mathcal{E}(c_k)-\sum_{1\leq j<k}b_{2k}m^{\Sigma}_{2w_0\hat{\eta}_{k-j}}m^{\Sigma}_{2w_0\hat{\eta}_{k+j}} =m^{\Sigma}_{2w_0\hat{\eta}_{2k}}
\end{align*}
for scalars $t_{2k}$ and $b_{2k}$
and \begin{align*}
t_{2k+1}\mathcal{E}(c_{2k+1})-\sum_{1\leq j<k}b_{2k+1}m^{\Sigma}_{2w_0\hat{\eta}_{k-j}}m^{\Sigma}_{2w_0\hat{\eta}_{k+1+j}} =m^{\Sigma}_{2w_0\hat{\eta}_{2k+1}}
\end{align*}
for scalars $t_{2k+1}$ and $b_{2k+1}$.  
Note that here we are using the facts that 
$m^{\Sigma}_{2w_0\hat{\eta}_{k-j}}$ and $m^{\Sigma}_{2w_0\hat{\eta}_{k+j}}$ are elements of 
\begin{align*}\mathbb{C}(q)[m^{\Sigma}_{2w_0\hat{\eta}_{1}},m^{\Sigma}_{2w_0\hat{\eta}_{2}},\dots, m^{\Sigma}_{2w_0\hat{\eta}_{2k-1}}]
\end{align*} and, similarly, $m^{\Sigma}_{2w_0\hat{\eta}_{k-j}}$ and $m^{\Sigma}_{2w_0\hat{\eta}_{k+1+j}}$ are elements of 
\begin{align*}\mathbb{C}(q)[m^{\Sigma}_{2w_0\hat{\eta}_{1}},m^{\Sigma}_{2w_0\hat{\eta}_{2}},\dots, m^{\Sigma}_{2w_0\hat{\eta}_{2k}}]
\end{align*} for $1\leq j<k$. Since $m^{\Sigma}_{2w_0\hat{\eta}_{s}}$ is algebraically independent with  elements in the  polynomial ring $\mathbb{C}(q)m^{\Sigma}_{2w_0\hat{\eta}_{1}},
 m^{\Sigma}_{2w_0\hat{\eta}_{2}},\cdots,
  m^{\Sigma}
  _{2w_0\hat{\eta}_{s-1}}]$, the coefficient $a_s$ must be nonzero. 
This proves  (\ref{polyequal}) in Type AII.

Note that by the above analysis, there exists a polynomial $f_{s}$ in $s-1$ variables and a nonzero scalar $t_s$ such that
$$\mathcal{E} (c_s)-f_{s}(m^{\Sigma}_{2w_0\hat{\eta}_{1}},\dots,m^{\Sigma}_{2w_0\hat{\eta}_{s-1}})=t_sm^{\Sigma}_{2w_0\hat{\eta}_{s}}$$
for all $s=1,\dots,n$. For type AI and the diagonal case, $f_s$ is identically equal to zero. In Type AII, $f_{s}((\mathcal{E}(c_1),\dots, \mathcal{E}(c_{s-1}))$ is a sum of products of the form 
\begin{itemize}
\item $m^{\Sigma}_{2w_0\hat{\eta}_{k-j}}m^{\Sigma}_{2w_0\hat{\eta}_{k+j}}$
for $s=2k$ and $1\leq j<k$
\item
$m^{\Sigma}_{2w_0\hat{\eta}_{k-j}}m^{\Sigma}_{2w_0\hat{\eta}_{k+1+j}}$ 
 for $s=2k+1$ and $1\leq j<k$. 
 \end{itemize}
 Each of these products has polynomial degree less $s$ and filtration degree less than $2s$. We have
\begin{align*}t_sm^{\Sigma}_{2w_0\hat{\eta}_{s}}&-a_s\mathcal{E}(C_{w_0\hat{\eta}_s})=\mathcal{E}(c_s)-f_s(m^{\Sigma}_{2w_0\hat{\eta}_{1}},\dots, m^{\Sigma}_{2w_0\hat{\eta}_{s-1}})-a_s\mathcal{E}(C_{w_0\hat{\eta}_s})\cr&+ p_s(m^{\Sigma}_{2w_0\hat{\eta}_{1}},m^{\Sigma}_{2w_0\hat{\eta}_{2}},\dots, m^{\Sigma}_{2w_0\hat{\eta}_{s}})
\end{align*} for some polynomial $p_s(m^{\Sigma}_{2w_0\hat{\eta}_{1}},\dots, m^{\Sigma}_{2w_0\hat{\eta}_{s-1}})
\in \mathbb{C}(q)[m^{\Sigma}_{2w_0\hat{\eta}_{1}},m^{\Sigma}_{2w_0\hat{\eta}_{2}},\dots, m^{\Sigma}_{2w_0\hat{\eta}_{s-1}}].
$  
 It follows that  \begin{align*}a_s\mathcal{E}(C_{w_0\hat{\eta}_s})=t_sm^{\Sigma}_{2w_0\hat{\eta}_{s}}
- p_s(m^{\Sigma}_{2w_0\hat{\eta}_{1}},m^{\Sigma}_{2w_0\hat{\eta}_{2}},\dots, m^{\Sigma}_{2w_0\hat{\eta}_{s-1}}).\end{align*}
Since \begin{align*}\mathcal{J}{\rm deg}(m^{\Sigma}_{2w_0\hat{\eta}_s})=2s=\mathcal{J}{\rm deg}(C_{w_0\hat{\eta}_s})
\end{align*} we must have\begin{align}\label{Jdegbound}\mathcal{J}{\rm deg}(p_s(m^{\Sigma}_{2w_0\hat{\eta}_{1}},\dots, m^{\Sigma}_{2w_0\hat{\eta}_{s-1}}))\leq 2s.\end{align}

We can write $p_s(x_1,\dots, x_{s-1})$ as a sum $\sum_{\mu} a_{\mu}x_1^{\mu_1}\cdots x_{s-1}^{\mu_{s-1}}$.  Hence 
\begin{align*}a_s\mathcal{E}(C_{w_0\hat{\eta}_s})=t_sm^{\Sigma}_{2w_0\hat{\eta}_{s}}-\sum_{\mu}a_{\mu}(m^{\Sigma}_{2w_0\hat{\eta}_{1}})^{\mu_1}\cdots (m^{\Sigma}_{2w_0\hat{\eta}_{s-1}})^{\mu_{s-1}}.
\end{align*} 
By Lemma \ref{lemma:Kalgebra}, the algebra isomorphism  $\mathcal{V}$ applied to the above yields 
\begin{align*}\mathcal{V}(a_s\mathcal{E}(C_{w_0\hat{\eta}_s}))&=\mathcal{V}(t_sm^{\Sigma}_{2w_0\hat{\eta}_{s}})-\sum_{\mu}a_{\mu}\mathcal{V}((m^{\Sigma}_{2w_0\hat{\eta}_{1}})^{\mu_1}\cdots (m^{\Sigma}_{2w_0\hat{\eta}_{s-1}})^{\mu_{s-1}})\cr&=t_sK_{2w_0\hat{\eta}_{s}}-\sum_{\mu}a_{\mu}(K_{2w_0\hat{\eta}_{1}})^{\mu_1}\cdots (K_{2w_0\hat{\eta}_{s-1}})^{\mu_{s-1}}
\end{align*}   Moreover, as discussed before this lemma, $\mathcal{V}$ preserves both the filtration and the polynomial degree.  Hence
\begin{align*}\mathcal{J}{\rm deg}(\sum_{\mu}a_{\mu}(m^{\Sigma}_{2w_0\hat{\eta}_{1}})^{\mu_1}\cdots (m^{\Sigma}_{2w_0\hat{\eta}_{s-1}})^{\mu_{s-1}})&=\mathcal{J}{\rm deg}\sum_{\mu}a_{\mu}(K_{2w_0\hat{\eta}_{1}})^{\mu_1}\cdots (K_{2w_0\hat{\eta}_{s-1}})^{\mu_{s-1}}\cr&={\rm max}_{\mu}\mathcal{J}{\rm deg}((K_{2w_0\hat{\eta}_{1}})^{\mu_1}\cdots (K_{2w_0\hat{\eta}_{s-1}})^{\mu_{s-1}}).
\end{align*}
Now each term  $(K_{2w_0\hat{\eta}_{1}})^{\mu_1}\cdots (K_{2w_0\hat{\eta}_{s-1}})^{\mu_{s-1}}$ has filtration degree 
$$2\mu_1+4\mu_2+\cdots +2(s-1)\mu_{s-1}$$ which must be less than or equal to $2s$ by (\ref{Jdegbound}).  Similarly, the polynomial degree of $(K_{2w_0\hat{\eta}_{1}})^{\mu_1}\cdots (K_{2w_0\hat{\eta}_{s-1}})^{\mu_{s-1}}$ is $$\mu_1+2\mu_2+\cdots +(s-1)\mu_{s-1}$$ which is less than or equal to $s$ by the previous computation. 
Hence 
\begin{align*}{\rm pdeg} (p_s(m^{\Sigma}_{2w_0\hat{\eta}_{1}},m^{\Sigma}_{2w_0\hat{\eta}_{2}},\dots, m^{\Sigma}_{2w_0\hat{\eta}_{s-1}}))
={\rm max}_{\mu}(\mu_1+2\mu_2+\cdots +(s-1)\mu_{s-1})\leq 2s
\end{align*} Since ${\rm pdeg}(m^{\Sigma}_{2w_0\hat{\eta}_{s}})=s$,  it follows that 
\begin{align*}{\rm pdeg}{\mathcal{E}(C_{w_0\hat{\eta}_s})}=s. 
\end{align*} In other words, ${\rm pdeg}{\mathcal{E}(C_{w_0\hat{\eta}_i})}=i $ for all $i=1,\dots, n$ as desired.

  We now turn to arbitrary $\lambda$ and determine the polynomial degree of $\mathcal{E}(C_{\lambda})$.  By Proposition \ref{prop:capelli} and Theorem \ref{theorem:center_and_capelli}, \begin{align*}C_{\lambda}=a_{\lambda,\lambda}C_{\hat{\eta}_1}^{\lambda_1}\cdots C_{\hat{\eta}_n}^{\lambda_n} +\sum_{|\lambda'|<|\lambda|}a_{\lambda'}C_{\hat{\eta}_1}^{\lambda'_1}\cdots C_{\hat{\eta}_j}^{\lambda'_j}+a_0
 \end{align*} 
 where $\lambda= \lambda_1\hat{\eta}_1+\lambda_2\hat{\eta}_2+\cdots \lambda_n\hat{\eta}_n$ and $a_{\lambda,\lambda}, a_{\lambda'},$ and $ a_0$ are all scalars with $a_{\lambda,\lambda}\neq 0.$  Since it is a product and the polynomial degree of a product is a sum of the degree of the factors, we have \begin{align*}
 {\rm pdeg}(C_{\hat{\eta}_1}^{\lambda_1}\cdots C_{\hat{\eta}_n}^{\lambda_n} )=
 {\rm pdeg}\ C_{\hat{\eta}_1}^{\lambda_1}+\cdots + {\rm pdeg}\ C_{\hat{\eta}_n}^{\lambda_n} =\sum_{i=1}^n\lambda_i|\hat{\eta}_i| = \sum_{i=1}^n\lambda_ii = |\lambda|.
\end{align*}
A similar argument applied to each summand in $\sum_{|\lambda'|<|\lambda|}a_{\lambda'}C_{\hat{\eta}_1}^{\lambda'_1}\cdots C_{\hat{\eta}_j}^{\lambda'_j}$ yields $$ {\rm pdeg} (\sum_{|\lambda'|<|\lambda|}a_{\lambda'}C_{\hat{\eta}_1}^{\lambda'_1}\cdots C_{\hat{\eta}_j}^{\lambda'_j})<|\lambda|.$$  Also, $ {\rm pdeg}\ a_0 = 0$.
 Hence the polynomial degree of $\mathcal{E}_{\lambda}=\mathcal{E}(C_{\lambda})$ satisfies
\begin{align*}      {\rm pdeg}\ \mathcal{E}(C_{\lambda}) = {\rm pdeg}(C_{\hat{\eta}_1}^{\lambda_1}\cdots C_{\hat{\eta}_n}^{\lambda_n} )=|\lambda|.
  \end{align*}
 
   \end{proof}  
   
   \subsection{Vanishing and non-vanishing properties}\label{section:van}

By Proposition \ref{prop:weights} and Theorem \ref{theorem:explicit-module}, $H_{2\mu}$ has weight $2\mu$ with $\mu\in\Lambda^+_{\Sigma}$ as a left $U_q(\mathfrak{g})$-module.
Hence 
\begin{align*}
K_{\lambda}\cdot H_{2\mu} = q^{(\lambda,2\mu)}H_{2\mu}
\end{align*}
for all $K_{\lambda}$ in the Cartan subalgebra of $U_q(\mathfrak{g})$. We expand $\mu$ as a sum of the form  $\mu= (\mu_1\epsilon_1^{\Sigma}+\cdots +  \mu_n\epsilon_n^{\Sigma})$. 
Now consider the action of $K_{2\epsilon_i^{\Sigma}}$ on $H_{2\mu}$ where $2\epsilon_i^{\Sigma}$ is the weight described in Section \ref{section:restricted-root-system} for all three types of symmetric pairs.  We have 
\begin{align*}
K_{2\epsilon_i^{\Sigma}}\cdot H_{2\mu} = q^{(2\epsilon^{\Sigma}_i,2\mu)}H_{2\mu}
\end{align*}
Now expand $\mu$ as a sum of the form  $\mu= (\mu_1\epsilon_1^{\Sigma}+\cdots +  \mu_n\epsilon_n^{\Sigma})$.
It follows that 
\begin{align*}
K_{2\epsilon_i^{\Sigma}}\cdot H_{2\mu} = q^{\mu_i(2\epsilon^{\Sigma}_i,2\epsilon^{\Sigma}_i)}H_{2\mu}=q^{t\mu_i}
\end{align*}
where   $t=4$ in Type AI, and $t=2$ in Types AII and the diagonal case.
Hence
\begin{align*}\kappa(x_i)\cdot H_{2\mu}= g^{-i}q^{t\mu_i} H_{2\mu}
\end{align*}
Since $\kappa$ is an algebra homomorphism, if $P$ is any polynomial in $\mathbb{C}(q)[x_1,\dots, x_n$ then. 
\begin{align*}
\kappa(P)\cdot H_{2\mu}=P(g^{-1}q^{t\mu_1},\cdots, g^{-n}q^{t\,u_n})H_{2\mu}.
\end{align*} 
Thus
\begin{align*}
\kappa(\mathcal{E}_{\lambda})\cdot H_{2\mu}=\mathcal{E}_{\lambda}(g^{-1}q^{t\mu_1},\cdots, g^{-n}q^{t\,u_n})H_{2\mu}.
\end{align*} 

 
 Given a weight $\mu=(\mu_1\epsilon_1^{\Sigma}+\cdots +  \mu_n\epsilon_n^{\Sigma})$  and a polynomial $P$ in $\mathbb{C}(q)[x_1,\dots,x_n]$, set $P(q^{\mu}) =P(q^{t\mu_1},\dots, q^{t\mu_n})$. 
 By (\ref{Capkap})
 \begin{align*}\kappa(\mathcal{E}_{\lambda})= \mathcal{E}_{\lambda}(g^{-1}K_{2\epsilon^{\Sigma}_1}, \dots, g^{-n}K_{2\epsilon^{\Sigma}_{n}})
 \end{align*} 
for each $\lambda$ in $\Lambda_{\Sigma}^+$.  Set $\mathcal{E}_{\lambda}(q^{t\mu})=  \mathcal{E}_{\lambda}(g^{-1}q^{t\mu_1}, \dots, g^{-n}q^{t\mu_n}).$

Consider $H_{2\mu}\in \mathscr{P}_{\theta}$ for $\mu \in \Lambda_{\Sigma}^+$.  Recall Lemma \ref{Capelliexpression} which describes the action of a Capelli operator, say $C_{\lambda}$ on  $H_{2\mu}$.   Hence  \begin{align*} C_{\lambda}\cdot H_{2\mu} = \mathcal{E}(C_{\lambda})\cdot H_{2\mu} = \mathcal{E}_{\lambda}(q^{t\mu})H_{2\mu}.
\end{align*}

   \begin{proposition} \label{prop:vanishing}For each $\lambda,\mu\in \Lambda_{\Sigma}^+$ , the eigenvalue function $\mathcal{E}_{\lambda}$ satisfy
 \begin{itemize}
  \item[(i)] $\mathcal{E}_{\lambda}(q^{t\mu})=1$ for $\lambda = \mu$
 \item[(ii)] $\mathcal{E}_{\lambda}(q^{t\mu}) = 0$ for $\mu\neq \lambda$ and $|\mu|\leq |\lambda|$
 \end{itemize}
 where $t=4$ in Type AI, and $t=2$ in Types AII and the diagonal case. 
 \end{proposition}
 \begin{proof} It follows from the discussion preceding the proposition that $\mathcal{E}_{\lambda}(q^{t\mu}) H_{2\mu}= C_{\lambda}\cdot H_{2\mu}$ for all choices of $\lambda$ and $\mu$ with $|\lambda|\geq |\mu|$. Hence (i) and (ii) are equivalent to the analogous assertions with $\mathcal{E}_{\lambda}(q^{t\mu})$ replaced by $ C_{\lambda}\cdot H_{2\mu}$.  The proposition now follows from Lemma \ref{Capelliexpression} (with the roles of $\mu$ and $\lambda$ switched).
 \end{proof}

\subsection{Knop-Sahi interpolation polynomials}\label{section:KS}
Let  $\mathbb{C}(a,g)[x_1, \dots, x_n]$ be the polynomial in $n$ variables over the field $\mathbb{C}(a,g)$ where $a$ and $g$ are two independent parameters. Given a partition $\mu=\mu_1\geq \mu_2\geq\mu_n\geq 0$ and a polynomial $P(x_1,\dots, x_n)$ in $\mathbb{C}(a,g)[x_1, \dots, x_n]$, set $P(a^{\mu}) = P(a^{\mu_1}, \dots, a^{\mu_n})$. Knop-Sahi interpolation polynomials introduced in \cite{Kn} and \cite{S}, also called shifted Macdonald polynomials in the later paper \cite{O}. They are a family of polynomials $P^*_{\lambda}(x; a,g)$, indexed by partitions ${\lambda}$. In  addition,  they   satisfy both an invariance condition and a vanishing condition.  
In particular, the element $P^*_{\lambda}( x;a,g)$ in 
 $\mathbb{C}(a,g)[x_1, \dots, x_n]$ is the unique (up to nonzero scalar) polynomial in the $x_1, \dots, x_n$ of degree $|\lambda|$  such that
\begin{itemize}
\item $P^*_{\lambda}( x;a,g)$ is symmetric viewed as a polynomial in the $n$ terms $x_1g^{-1}, \dots, x_ng^{-n}$.
\item $P^*_{\lambda}(a^{\mu};a,g) = 0$  for each partition $\mu\neq \lambda$ with $|\mu|\leq |\lambda|$ and $P^*_{\lambda}(a^{\lambda};a,g) \neq 0$. 
\end{itemize}
Here, we are  following Onkoukov's aproach \cite{O}. In particular, these polynomials are symmetric with respect to the choice of variables $x_ig_i^{-1}, i=1,\cdots, n$ (see page 149 of \cite{O} which is part of the introduction). The polynomials in \cite{S} (defined in the introduction via conditions (1) and (2)) are symmetric in the $x_1, \dots, x_n$. Hence these are recovered from the one's defined here by a change of variables. 
 Also, we are using the simpler vanishing condition of Sahi \cite{S} and Knop \cite{Kn} because it is less complicated. The stronger one, which is the one used by Onkounkov (see   \cite{O} (1.3)) is actually shown to be equivalent to the  simpler vanishing condition given above (See \cite{Kn}, Section 4).


Note that in \cite{O}, the initial formulation for the polynomials above is such that they are unique \emph{up to nonzero scalar multiple}. Later, a normalization is provided (Section 4 of \cite{O}, see equality (4.3)) so that these polynomials can be given precisely without worrying about a scalar factor.  The papers \cite{S} and \cite{K} normalize these polynomials in a different way as compared to what is done in this paper.  Namely, they assume that the  symmetric polynomial $m_{\lambda}$ associated to the partition $\lambda$ appears  in the polynomial at $\lambda$ with  coefficient equal to $1$.  Here, we  normalize by assuming the eigenvalue functions $\mathcal{E}_{\lambda}$  arise from elements corresponding to the identity  in the appropriate $U_q(\mathfrak{g})$-module.  This choice is equivalent to the condition  $\mathcal{E}_{\lambda}(q^{m\lambda}) = 1$ of Proposition \ref{prop:vanishing}.  

\begin{theorem}\label{thm:KSpoly} For each $\lambda$, the polynomial $\mathcal{E}_{\lambda}(x_1, \dots, x_n)$ evaluated at $x_i=K_{2{\epsilon}^{\Sigma}_{i}}$ is equal to $c_{\lambda}P^*_{\lambda}(x;a,g)$ 
 where $c_{\lambda} = P^*_{\lambda}(a^{\lambda}; a,g)^{-1}$ and 
\begin{itemize}
\item $(a,g) = (q^4,q^2)$ in Type AI, 
\item $(a,g) = (q^2,q^4)$ in Type AII, and 
\item $(a,g) = (q^2,q^2)$ in the diagonal type.
\end{itemize}
\end{theorem}
\begin{proof}
Set $g=q^2$ in Type AI and the diagonal case and set $g=q^4$ in Type AII. By Theorem \ref{theorem:dotted_Weyl} and Theorem \ref{theorem:center_and_capelli}, $\mathcal{E}_{\lambda}(K_{2{\epsilon}^{\Sigma}_{1}}, \dots, K_{2{\epsilon}^{\Sigma}_{n}})$ is in the dotted Weyl group invariants of 
 $\mathbb{C}(q)[\mathcal{A}_2]$. By Lemma \ref{lemma:explicit}, the dotted invariants in $\mathbb{C}(q)[\mathcal{A}_2]$ equals the  symmetric polynomials in $\mathbb{C}(q)[g^{-1}K_{2\epsilon^{\Sigma}_i}|\  i=1, \dots, n] $.  It follows that $\mathcal{E}_{\lambda}(K_{2{\epsilon}^{\Sigma}_{1}}, \dots,K_{2{\epsilon}^{\Sigma}_{n}})$ is a symmetric polynomial in the 
 variables $$g^{-1}K_{2{\epsilon}^{\Sigma}_{1}}, \dots, g^{-n}K_{2{\epsilon}^{\Sigma}_{n}}.$$ The same claim is true when we replace each $K_{2{\epsilon}^{\Sigma}_{i}}$ with $x_i$ and so $\mathcal{E}_{\lambda}(x_1, \dots, x_n)$ satisfies the same invariance property as $P^*_{\lambda}(x; a,g)=P^*_{\lambda}(x_1,\dots, x_n)$ with this particular choice of $g$.
 
 Now set $a=q^4$ in Type AI and $a=q^2$ in Type AII and the diagonal case.  By Proposition \ref{prop:vanishing} (ii),  
 $\mathcal{E}_{\lambda}(a^{\mu}) =0$ for each partition $\mu\neq \lambda$ with $|\mu|\leq |\lambda|$.   Hence $\mathcal{E}_{\lambda}$ satisfies 
 the same vanishing property as $P^*_{\lambda}(x;a,g)$.  By Lemma \ref{lemma:degree}, the degree of $\mathcal{E}_{\lambda}$ is $|\lambda|$.  Thus $\mathcal{E}_{\lambda}$ must be a nonzero scalar multiple of $P^*_{\lambda}(x;a,g)$ for each $\lambda$.  The fact that this nonzero scalar is $c_{\lambda}$ follows from the fact that   $\mathcal{E}_{\lambda}(a^{\lambda})=1$ which is just Proposition \ref{prop:vanishing} (i).\end{proof}

\begin{remark} \label{remark:finalremark}   Recall that $\mathcal{O}_q(SL_N)$ can be obtained from $\mathcal{O}_q({\rm Mat}_N)$ by modding out by $\det_q-1$. Let $\chi$ be a Hopf algebra automorphism of $U_q(\mathfrak{g})$.  Set $G=SL_n$ in Type AI, $G=SL_{2n}$ in Type AII, and $G=SL_n\times SL_n$ in the diagonal case. The space of zonal spherical functions $\mathcal{H} := \mathcal{H}_{\chi({\mathcal{B}}_{\theta}),\mathcal{B}_{\theta}}$  associated to the pair 
$\chi(\mathcal{B}_{\theta}), \mathcal{B}_{\theta}$  are the  left $\chi(\mathcal{B}_{\theta})$ invariants and right $\mathcal{B}_{\theta}$ invariants of $\mathcal{O}_q[G]$ (see \cite{L2003}). By \cite{L2003} Theorem 4.2, there is a map from 
$\mathcal{O}_q[G]$ to functions on $U^0$  so that for the correct choice of $\chi$, the image of $\mathcal{H}_{\chi({\mathcal{B}}_{\theta}),\mathcal{B}_{\theta}}$  is a subring of $W_{\Sigma}$ invariants. 
Moreover, for this choice of $\chi$, a special basis  $\varphi_{\lambda}$, $\lambda\in P^+(\Phi)$ can be identified with a family of Macdonald polynomials $P_{\lambda}(x;a,g)$ where both $a$ and $g$ are powers of $q$ (see \cite{L2004}, Appendix A).  The values of $a$ and $g$ for the three families of this paper are
\begin{itemize}
\item $(a,g) = (q^4,q^2)$ in Type AI, 
\item $(a,g) = (q^2,q^4)$ in Type AII, and 
\item $(a,g) = (q^2,q^2)$ in the diagonal type.
\end{itemize}
These are the same parameters that appear in the identification of the quantum Capelli eigenvalues with interpolation polynomials in Theorem \ref{thm:KSpoly}.  This connection between the realization of zonal spherical functions as Macdonald polynomials and quantum Capelli operators as Knop-Sahi interpolation polynomials mirrors the classical situation.  Indeed, these parameters agree because in terms of a particular natural grading, the top degree term of Knop-Sahi interpolation polynomials are precisely Macdonald polynomials with the same parameters (see for example \cite{S}, Theorem 1.1).  \end{remark}

\small
\section{Appendix: Commonly used notation}
We list here commonly used symbols and notation along with the first section (post the introduction) in which each item appears.

\medskip
\noindent
{\bf Section 2.1:} $\epsilon_i$, $\alpha_i$, $\omega_i$, $\Phi_N$, $P_N$, $Q_N$, $P_N^+$, $Q_N^+$, $\Lambda_N$, $\Lambda_N^+$, $\hat{\Lambda}_N^+$, $\hat{\omega}_i$, $w_0$, $w_oP^+_N$, $w_0\Lambda_N^+$, $w_0\hat{\Lambda}^+_N$, $\mathfrak{gl}_N$, $\mathfrak{sl}_N$, $(\cdot,\cdot)$, $\gamma\oplus \gamma'$

\medskip
\noindent
{\bf Section 2.2:} $e_i,f_i, h_{\epsilon_i}$, $K_{\epsilon_j}$, $U_q(\mathfrak{gl}_N)$, $K_i, E_i,F_i$, $U_q(\mathfrak{sl}_N)$, $K_{\beta}$, $\Delta$, $\epsilon$, $S$, $U^0(\mathfrak{gl}_N)$, $U^0(\mathfrak{gl}_N\oplus \mathfrak{gl}_N)$, $U^0(\mathfrak{sl}_N),$ $U^+(\mathfrak{gl}_N)$, $U^+(\mathfrak{gl}_N\oplus \mathfrak{gl}_N)$, $U^+$, $U^-$, $\check{U}_q(\mathfrak{sl}_N)$, $U^0(\mathfrak{gl}_N)$, $\check U^0(\mathfrak{sl}_N)$, $U^+_+$,  $L(\lambda)$

\medskip
\noindent
{\bf Section 2.3:} $({\rm ad}\ a)$, $\mathcal{F}(M)$, $({\rm ad}\ E_i)$, $({\rm ad}\ F_i)$, $({\rm ad}\ K)$

\medskip
\noindent
{\bf Section 2.4:} $\mathfrak{g}, \theta, \mathfrak{k}$, $\mathcal{B}_{\theta}$, $J$, $R_{\mathfrak{g}}$, $R_{\mathfrak{g}}^{t_1}$, $R$, $B_i$, Type AI, Type AII, diagonal case 

\medskip
\noindent
{\bf Section 2.5:} $\Phi$, $\tilde{\beta}$, $\Sigma$, $\alpha_i^{\Sigma}$, $\epsilon_i^{\Sigma}$, $\eta_i$, $\hat{\eta}_i$, $w_0\hat{\eta}_i$, $P_{\Sigma}$, $P^+_{\Sigma}$, $\Lambda^+_{\Sigma}$, 
$\hat{\Lambda}^+_{\Sigma}$, $W_{\Sigma}$, $(\cdot, \cdot)_{\Sigma}$, spherical

\medskip
\noindent
{\bf Section 3.1:} ${\rm Mat}_N$, $\mathcal{O}_q({\rm Mat}_N)$,   $\mathcal{O}_q({\rm Mat}_N)^{op}$, $e_{ij}, r^{ij}_{kl}$, $R^{t_s}$, $t_{ij}$, $\partial_{ij}$, $\iota$, $t_{i+N, j+N}$, $T_1$, $T_2$, $\mathscr{P}, \mathscr{D}$

\medskip
\noindent
{\bf Section 3.2:} $J_{r,s}$, $x_{ij}, d_{ij}$, $\mathscr{P}^{\mathcal{B}_{\theta}}$, $\mathscr{P}_{\theta}$, $\mathscr{P}_{\theta}^{op}$, $\mathscr{D}_{\theta}$, $\mathcal{J}_r$,$\mathscr{P}_{\theta}^r$, $\mathscr{D}_{\theta}^r$

\medskip
\noindent
{\bf Section 4.1:} $\det_q(T)$, $\iota(\det_q(T)$, $\det_q(T')$, $H_n$

\medskip
\noindent
{\bf Section 4.2:} $\mathfrak{g}_r$, $U_q(\mathfrak{g}_r)$, $\mathscr{P}(\mathfrak{g}_r)$, $\mathcal{B}_{\theta}^r$, $\mathscr{P}_{\theta}(\mathfrak{g}_r)$, $x(r)_{ij}$,  $\psi_{r,s}$

\medskip
\noindent
{\bf Section 4.3:} $\hat{H}_r$, $H_i$, $H_{2\mu}$

\medskip
\noindent
{\bf Section 4.4:} $H^*_{2\mu}$

\medskip
\noindent
{\bf Section 5.1:} $\mathscr{PD}_q({\rm Mat}_N)$, $\mathscr{PD}_{\theta}$, $\mathcal{J}_r(\mathscr{PD}_{\theta})$

\medskip
\noindent
{\bf Section 5.2:} $\mathcal{L}$, $\pi$, $(b)_0$, $\langle \cdot, \cdot \rangle$,   $\phi$, $\phi_a$

\medskip
\noindent
{\bf Section 5.4:} $|\mu|$

\medskip
\noindent
{\bf Section 6.2:} $T_s$, $\beta_{s,t}$

\medskip
\noindent
{\bf Section 6.3:} $\psi$

\medskip
\noindent
{\bf Section 7.1:} $\check{U}_q(\mathfrak{sl}_N)$, $\check{U}_q(\mathfrak{gl}_N)$, $K_{(\hat{\omega}_N)/N}$, $\zeta$

\medskip
\noindent
{\bf Section 7.2:} $\mathcal{F}(U_q(\mathfrak{gl}_N)$

\medskip
\noindent
{\bf Section 7.3:} $U^2_q(\mathfrak{gl}_N)$, $U^0_2(\mathfrak{gl}_N)$,  $G^-$, $U_q^2(\mathfrak{gl}_N\oplus \mathfrak{gl}_N)$,  $\mathcal{F}(U^2_q(\mathfrak{gl}_N))$,  $\mathcal{F}(U^2_q(\mathfrak{gl}_N \oplus \mathfrak{gl}_N))$

\medskip
\noindent
{\bf Section 7.4:} $U^2_q(\mathfrak{g})$, $\Upsilon$

\medskip
\noindent
{\bf Section 8.1:} $z_{2\mu}$, $Z(\check U_q(\mathfrak{sl}_N))$, $z_{2\mu+2c\hat{\omega}_N}$, $Z( U^2_q(\mathfrak{g}))$

\medskip
\noindent
{\bf Section 8.2:} $\varphi_{HC}$, $G^-$, $\check{\mathcal{A}}$,  $\check{T}_{\theta}$, $(T_{\theta})_2$,  $\tilde{\mathcal{P}}$, $\tilde{\varphi}_{HC}$, $\mathcal{A}_2$

\medskip
\noindent
{\bf Section 8.3:} $\rho$, $w\circ$, $m_{2\lambda}$,  $\tilde{\rho}$, $m_{2w_0\lambda}$, $w\bullet$, $m^{\Sigma}_{2\lambda}$, $\mathcal{A}$

\medskip
\noindent
{\bf Section 8.4:} $\hat{z}_i$, $z_{2w_0\hat{w}}$, $Z$

\medskip
\noindent
{\bf Section 9.1:} $c_i$, $C_{\mu}$

\medskip
\noindent
{\bf Section 9.2:}
 $C_{\hat{\eta}_i}$, $\sum_{\substack{\nu<2\mu \\|\nu|=|2\mu|}}
(\mathscr{D}_{\theta})_{-\nu}$, $\sum_{\substack{\nu<2\mu \\|\nu|=|2\mu|}}
(\mathscr{P}_{\theta})_{\nu}$

\medskip
\noindent
{\bf Section 10.1:} $\mathcal{E}$, $\mathcal{E}_{\lambda}$, $P(q^{\mu})$, $\mathcal{V}$ 

\medskip
\noindent
{\bf Section 10.2:} $\kappa$

\medskip
\noindent
{\bf Section 10.3:} $P_{\lambda}^*(x; a,g)$


\begin{thebibliography}{9}
\bibitem{ASS} A. Alldridge, S. Sahi, and H. Salmasian, \emph{Schur $q$-functions and the Capelli eigenvalue problem for the Lie superalgebra $\mathfrak{q}(n)$.}  In:
Representation theory and harmonic analysis on symmetric spaces, Vol. 714,  Contemp. Math., Amer. Math. Soc., Providence, RI (2018) 1-21..
\bibitem{B2} O. Bershtein, \emph{Regular Functions on the Shilov Boundary}, J. Algebra Appl. {\bf 4}, no. 6 (2005) 613-629.
\bibitem{B} O. Bershtein, \emph{On a $q$-analog of a Sahi result}, J. Math. Phys. {\bf 48} (2007) 043510.
\bibitem{DS} M.S. Dikhhuizen, J.V. Stokman, \emph{Some limit transitions between BC type orthogonal polynomials interpreted on quantum complex Grassmannians}, Publ. Res. Inst. Math. Sci. {\bf 35} (1999) 451-500.
\bibitem{HK} I. Heckenberger, K. Sch\"mudgen, \emph{Classification of Bicovariant Differential Calculi on the Quantum Groups $SL_q(n+1)$ and $Sp_q(2n)$}, J. Reine Angew. Math. {\bf 502} (1997), 141-162.
\bibitem{Jo} A. Joseph, \emph{Quantum Groups and Their Primitive Ideals,} A Series of Modern Surveys in Mathematics,
Vol. 29, Springer-Verlag, Berlin, 2012.
\bibitem{JL} A. Joseph, G. Letzter, \emph{Local finiteness of the adjoint action for quantized enveloping algebras,} J. Algebra {\bf 153} (1992) 289-318.
\bibitem{JL2} A. Joseph, G. Letzter, \emph{Separation of variables for quantized enveloping algebras,} Am. J.  Mathematics {\bf 116} no. 1 (1994) 
127-177.
\bibitem{K} S. Kolb, \emph{Quantum symmetric Kac-Moody pairs}, Adv. Math. {\bf 267} (2014) 395-469.
\bibitem{Kn} F. Knop, \emph{Symmetric and non-symmetric quantum Capelli polynomials},  Comment. Math. Helv.
{\bf 72} Vol 1 (1997) 84-100.
\bibitem{KS} A. Klimyk, K. Schm\"{u}dgen,  \emph{Quantum groups and their representations}, Texts and Monographs in Physics,  Springer,  Berlin,  1997.
\bibitem{KS1991} B. Kostant and S. Sahi, \emph{The Capelli Identity, tube domains, and the generalized Laplace transform} Adv. Math. {\bf 87} no. 1 (1991) 71-92.
\bibitem{KS1993} B. Kostant and S. Sahi, \emph{Jordan algebras and Capelli Identities}, Invent. Math. {\bf 112} no. 3 (1993) 657-664.
\bibitem{L1999} G. Letzter \emph{Symmetric Pairs for Quantized Enveloping Algebras}  J. Algebra {\bf 220}  (1999) 729-767.
\bibitem{L2002} G. Letzter \emph{Coideal Subalgebras and Quantum Symmetric Pairs} In: New Directions in Hopf Algebras, MSRI Publications 
 {\bf 43} (2002) 117-165.
\bibitem{L2003} G. Letzter, \emph{Quantum Symmetric Pairs and their Zonal Spherical Functions}, Transform. Groups {\bf 8} (2003) 261-292.
\bibitem{L2004} G. Letzter, \emph{Quantum zonal spherical functions and Macdonald Polynomials},  Adv. Math. {\bf 189} 88-147 (2004).
\bibitem{L} G. Letzter, \emph{Invariant Differential Operators for Quantum Symmetric Spaces}, Memoirs of the American Mathematical Society, Volume 193, Number 903, 2008.
\bibitem{LSS0} G. Letzter, S. Sahi, H. Salmasian, \emph{Quantized Weyl Algebras, the Double Centralizer Property, and a new first Fundamental Theorem}, preprint 2022.
\bibitem{LSS} G. Letzter, S. Sahi, H. Salmasian, \emph{Weyl Algebras for Quantum Homogeneous Spaces}, Journal of Algebra, January 2024.
\bibitem{M} I.G. Macdonald, \emph{Orthogonal polynomials associated with root systems}, S\'em. Lotharingien Combin.
45 (2000/01) 40.
\bibitem{N} M. Noumi, \emph{Macdonald's symmetric polynomials as zonal spherical functions on quantum homogeneous spaces}, Adv. Math. {\bf 123} (1996) 16-77.
\bibitem{NDS} M. Noumi, M.S. Dijkhuizen, T. Sugitani, \emph{Multivariable Askey-Wilson polynomials and quantum
complex Grassmannians}, Fields Inst. Commun. {\bf14} (1997) 167-177.
\bibitem{NS} M. Noumi, T. Sugitani,\emph{Quantum symmetric spaces and related q-orthogonal polynomials}, In: Group
Theoretical Methods in Physics (ICGTMP), Toyonaka, Japan, 1994, World Science Publishing, River
Edge, NJ (1995)  pp. 28-40.
\bibitem{O} A. Okounkov, \emph{Shifted Macdonald polynomials: q-Integral representation and combinatorial formula}, Compos. Math.  {\bf 112} (1998)147-182.
\bibitem{Ol} G. Olshanski, \emph{Interpolation Macdonald polynomials and Cauchy-type identities}, J. Comb. Theory  Ser. A {\bf 162} (2019) 65-117.
\bibitem{S} S. Sahi, \emph{Interpolation, Integrality, and a Generalization of Macdonald's Polynomials}, Int. Math. Res. Not. {\bf 10} (1996) 457-471.
\bibitem{S2} S. Sahi \emph{The spectrum of certain invariant differential operators associated to a Hermitian symmetric space} In: Lie theory and geometry, Vol. 123, Progr. Math. Birkh\"auser, Boston, MA (1994) 569-576. 
\bibitem{SS2016} S. Sahi and H. Salmasian, \emph{The Capelli problem for $\mathfrak{gl}(m|n)$ and the spectrum of invariant differential operators},  Adv. Math. {\bf 303} (2016) 1-38.
\bibitem{SSS2020} S. Sahi, H. Salmasian, and V. Serganova, \emph{The Capelli eigenvalue problem for Lie superalgebras}, Math. Z. {\bf 294} (2020) 359-395.
\bibitem{SSS2021} S. Sahi, H. Salmasian, and V. Serganova, \emph{Capelli operators for spherical superharmonics and the Dougall-Ramanujan identity}, arXiv:1912:06301, Transform. Groups, to appear.
\bibitem{VSS} D. Shklyarov, S. Sinel'shchikov,  L. Vaksman, \emph{Fock representations and quantum matrices}, Int. J. Math. {\bf15} (2012).
\end{thebibliography}
\end{document}